\documentclass[twoside,11pt,reqno]{amsart}
\usepackage{amsmath,amssymb,amscd,mathrsfs,epic,wasysym,latexsym,tikz,mathrsfs,cite,hyperref,enumerate}
\usepackage{pb-diagram}
\usepackage[matrix,arrow]{xy}
\usepackage{stmaryrd}
\usepackage{centernot}
\usetikzlibrary{matrix,positioning,fit,calc}
\usepackage{array}
\usepackage{upgreek}

\DeclareFontFamily{OT1}{pzc}{}
\DeclareFontShape{OT1}{pzc}{m}{it}{<-> s * [1.10] pzcmi7t}{}
\DeclareMathAlphabet{\mathpzc}{OT1}{pzc}{m}{it}

\makeatletter

\numberwithin{equation}{section}

\allowdisplaybreaks

\newtheorem{Proposition}[equation]{Proposition}
\newtheorem{Lemma}[equation]{Lemma}
\newtheorem{Theorem}[equation]{Theorem}
\newtheorem{Corollary}[equation]{Corollary}
\newtheorem{MainTheorem}{Theorem}

\theoremstyle{definition}  

\newtheorem{Example}[equation]{Example}
\newtheorem{Problem}{Problem}

\let\<\langle
\let\>\rangle

\newcommand\Comment[2][\relax]{\space\par\medskip\noindent%
   \fbox{\begin{minipage}{\textwidth}\textbf{Comment\ifx\relax#1\else---#1\fi}\newline%
        #2\end{minipage}}\medskip
}

\renewcommand{\nmid}{\centernot\mid}

\def\bi{\text{\boldmath$i$}}

\def\b1{\text{\boldmath$1$}}

\newcommand{\da}{{{\downarrow}}}
\newcommand{\ua}{{\uparrow}}

\def\bone{\text{\boldmath$1$}}

\newcommand{\0}{{\bar 0}}
\renewcommand{\1}{{\bar 1}}
\newcommand{\balpha}{\boldsymbol{\upalpha}}
\newcommand{\bbeta}{\boldsymbol{\upbeta}}

\def\v{\bar v}

\def\pmod#1{\text{ }(\text{\rm mod } #1)\,}

\newcommand{\Hom}{\operatorname{Hom}}

\newcommand{\End}{\operatorname{End}}
\newcommand{\ind}{\operatorname{ind}}
\newcommand{\im}{\operatorname{im}}
\newcommand{\id}{\operatorname{id}}
\def\sgn{\mathtt{sgn}}

\def\Reg{\mathtt{Reg}}

\newcommand{\soc}{\operatorname{soc}}
\newcommand{\head}{\operatorname{hd}}
\newcommand{\hd}{\head}

\newcommand{\cha}{\operatorname{char}}

\newcommand{\Z}{\mathbb{Z}}

\newcommand{\cT} {\mathcal{T}}
\newcommand{\ct} {\mathpzc{t}}
\newcommand{\cx} {\mathpzc{x}}
\newcommand{\cy} {\mathpzc{y}}
\newcommand{\cm} {\mathpzc{m}}

\renewcommand{\epsilon}{\varepsilon}
\def\eps{{\varepsilon}}
\def\phi{{\varphi}}

\newcommand{\F}{{\mathbb F}}

\newcommand{\s}{{\mathsf S}}
\newcommand{\ts}{{\hat{\mathsf S}}}
\newcommand{\A}{{\mathsf A}}
\newcommand{\tA}{{\hat{\mathsf A}}}

\newcommand{\ga}{\gamma}

\newcommand{\la}{\lambda}

\newcommand{\al}{\alpha}
\newcommand{\be}{\beta}

\def\Si{\s}
\newcommand{\si}{\sigma}

\newcommand{\Om}{\Omega}

\newcommand{\de}{\delta}

\newcommand{\ka}{\kappa}

\newcommand{\Mull}{{\tt M}}

\def\Mtype{\mathtt{M}}

\def\ttres{\mathtt{res}}

\def\Qtype{\mathtt{Q}}

\newcommand{\Ker}{\operatorname{Ker}}
\newcommand{\Aut}{{\mathrm {Aut}}}

\newcommand{\Irr}{{\mathrm {Irr}}}
\newcommand{\Irrs}{{\mathrm {Irrs}}}

\def\id{\mathop{\mathrm {id}}\nolimits}
\renewcommand{\Im}{{\mathrm {Im}}}
\newcommand{\Ind}{{\mathrm {Ind}}}

\newcommand{\rad}{{\mathrm {rad}\,}}

\newcommand{\Res}{{\mathrm {Res}}}

\newcommand{\C}{{\mathbb C}}
\newcommand{\Q}{{\mathbb Q}}

\newcommand{\ZZ}{{\mathbb Z}}

\newcommand{\D}{{\mathtt D}}

\newcommand{\FF}{{\mathbb F}}

\newcommand{\Cl}{{\mathcal C}}

\newcommand{\T}{\mathcal{T}}

\newcommand{\W}{{\mathsf W}}
\newcommand{\hW}{\hat{\mathsf W}}

\newcommand{\dar}{{\downarrow}}

\newcommand{\JS}{{\tt JS}}

\renewcommand{\setminus}{\smallsetminus}

\def\n{{\mathfrak n}}

\def\d{{\mathtt d}}

\def\Par{{\mathscr P}}
\def\BMOPar{{\mathscr{P}_{\tt{BMO}}}}

\def\RP{{\mathscr {RP}}}
\def\TR{{\mathscr {TR}}}

\def\b{\mathfrak{b}}
\def\k{\Bbbk}

\def\T{\mathcal{T}}

\def\spa{\operatorname{span}}

\def\im{{\mathrm{im}\,}}
\def\onto{{\twoheadrightarrow}}
\def\into{{\hookrightarrow}}

\renewcommand\O{\mathcal O}

\def\iso{\stackrel{\sim}{\longrightarrow}}

\def\col{{\tt col}}
\def\row{{\tt row}}

\def\lan{\langle}
\def\ran{\rangle}

\newcommand{\cc} {\mathpzc{c}}

{\catcode`\|=\active
  \gdef\set#1{\mathinner{\lbrace\,{\mathcode`\|"8000%
  \let|\midvert #1}\,\rbrace}}
}
\def\midvert{\egroup\mid\bgroup}

\colorlet{darkgreen}{green!50!black}
\tikzset{dots/.style={very thick,loosely dotted},
         greendot/.style={fill,circle,color=darkgreen,inner sep=1.5pt,outer sep=0},
         blackdot/.style={fill,circle,color=black,inner sep=1.5pt,outer sep=0},
         graydot/.style={fill,circle,color=gray,inner sep=1.1pt,outer sep=0}
}
\def\greendot(#1,#2){\node[greendot] at(#1,#2){}}
\def\blackdot(#1,#2){\node[blackdot] at(#1,#2){}}
\def\graydot(#1,#2){\node[graydot] at(#1,#2){}}

\newenvironment{braid}{
  \begin{tikzpicture}[baseline=6mm,black,line width=1pt, scale=0.32,
                      draw/.append style={rounded corners},
                      every node/.append style={font=\fontsize{5}{5}\selectfont}]%
  }{\end{tikzpicture}
}

\def\Grid(#1,#2){
  \draw[very thin,gray,step=2mm] (0,0)grid(#1,#2);
  \draw[very thin,darkgreen,step=10mm] (0,0)grid(#1,#2);
}

\newcommand\Tableau[2][\relax]{
  \begin{tikzpicture}[scale=0.5,draw/.append style={thick,black}]
    \ifx\relax#1\relax%
    \else 
      \foreach\box in {#1} { \filldraw[blue!30]\box+(-.5,-.5)rectangle++(.5,.5); }
    \fi
    \newcount\row\newcount\col
    \row=0
    \foreach \Row in {#2} {
       \col=1
       \foreach\k in \Row {
          \draw(\the\col,\the\row)+(-.5,-.5)rectangle++(.5,.5);
          \draw(\the\col,\the\row)node{\k};
          \global\advance\col by 1
       }
       \global\advance\row by -1
    }
  \end{tikzpicture}
}

\newcommand\YoungDiagram[2][\relax]{
  \begin{tikzpicture}[scale=0.5,draw/.append style={thick,black}]
    \ifx\relax#1\relax%
    \else 
    \foreach\box in {#1} {
      \filldraw[blue!30]\box rectangle ++(1,1);
    }
    \fi
    \newcount\row
    \row=0
    \foreach \col in {#2} {
       \draw(1,\the\row)grid ++(\col,1);
       \global\advance\row by -1
    }
  \end{tikzpicture}
}

\newenvironment{Young}{\begingroup
       \def\vr{\vrule height0.89\hoogte width\dikte depth 0.2\hoogte}
       \def\fbox##1{\vbox{\offinterlineskip
                    \hrule height\dikte
                    \hbox to \breedte{\vr\hfill##1\hfill\vr}
                    \hrule height\dikte}}
       \vbox\bgroup \offinterlineskip \tabskip=-\dikte \lineskip=-\dikte
            \halign\bgroup &\fbox{##\unskip}\unskip  \crcr }
       {\egroup\egroup\endgroup}
\def\diagram#1{\relax\ifmmode\vcenter{\,\begin{Young}#1\end{Young}\,}\else%
              $\vcenter{\,\begin{Young}#1\end{Young}\,}$\fi}

\begin{document}

\title[Irreducible restrictions of spin representations]{{\bf Irreducible restrictions of spin representations of symmetric and alternating groups}}

\author{\sc Alexander Kleshchev}
\address{Department of Mathematics\\ University of Oregon\\Eugene\\ OR 97403, USA}
\email{klesh@uoregon.edu}

\author{\sc Lucia Morotti}
\address{Department of Mathematics, University of York, York, YO10 5DD, UK}
\email{lucia.morotti@york.ac.uk} 

\author{\sc Pham Huu Tiep}
\address
{Department of Mathematics\\ Rutgers University\\ Piscataway\\ NJ~08854, USA} 
\email{tiep@math.rutgers.edu}

\subjclass{20C20, 20C30, 20D06, 20E28}

\thanks{The first author was supported by the NSF grant DMS-2346684. The first author thanks ICERM for support during the final stages of this work. 
The second author was supported by the DFG grants MO 3377/1-1 and MO 3377/1-2 and the Royal Society grant URF$\backslash$R$\backslash$221047. The third author gratefully acknowledges the support of the NSF (grant DMS-2200850) and the Joshua Barlaz Chair in Mathematics.}

\begin{abstract}
Let $\F$ be an algebraically closed field and $G$ be an almost quasi-simple group. An important problem in representation theory is to classify the subgroups $H<G$ 
and $\F G$-modules $L$ such that the restriction $L\da_H$ is irreducible. This problem is a natural part of the program of describing maximal subgroups in finite classical groups. In this paper we investigate the case of the problem 
where $G$ is the Schur's double cover of alternating or symmetric group. 
\end{abstract}

\maketitle

\tableofcontents

\section{Introduction}

Let $\F$ be an algebraically closed field of characteristic $p$ and $G$ be an almost
quasi-simple group. 
An important
problem in representation theory is to classify the subgroups $H$ of $G$ and irreducible $\F G$-modules $L$ such that the restriction $L\da_H$ is irreducible. For example, this problem 
is a natural part of the Aschbacher-Scott program of describing maximal subgroups in finite classical
groups; see \cite{Asch, Scott} and \cite{KL, BHRD, Mag, LS}. 

Suppose from now on that $\soc (G/Z(G))$ is the alternating group $\A_n$ with $n\geq 5$. 

\begin{Problem}\label{Problem}
Suppose that $\soc (G/Z(G))$ is the alternating group $\A_n$ with $n\geq 5$. 
Classify the
pairs $(H, L)$, where $H$ is a subgroup of $G$ and $L$ is a faithful irreducible $\F G$-module such that the restriction $L\da_H$ is irreducible.
\end{Problem}

Suppose the center $Z(G)$ is trivial, i.e. $G= \A_n$ 
or $\s_n$\footnote{We ignore the exceptional case $n=6$ which can be easily settled using \cite{ModularAtlas}.}. In this case, Saxl \cite{Saxl} has solved Problem~\ref{Problem} in characteristic $p=0$. In positive characteristic the same has been 
achieved in \cite{BKIrr,KSIrr} for $p>3$, and in \cite{KMTIrr} for $p=2,3$. 

From now on suppose that $Z(G)$ is non-trivial. If $n\neq 6,7$, then $G$ is one of the Schur's double covers $\tA_n$, $\ts_n$ or $\tilde\s_n$.\footnote{We ignore the exceptional $6$-fold double covers for $n=6,7$ as for these small cases Problem~\ref{Problem}  can be easily settled using \cite{ModularAtlas}.} The group algebras $\F
\ts_n$ and 
and $\F\tilde\s_n$ are canonically isomorphic, so we only have to deal with
$\tA_n$ and $\ts_n$.\footnote{Let $A:=\F\ts_n=\F\tilde\s_n$, where we have identified the group algebras via the canonical isomorphism, and $\rho:A\to\End_\F(L)$ be an irreducible representation. If $\hat \pi:\ts_n\to \s_n$ and $\tilde \pi:\tilde\s_n\to \s_n$ are the natural surjections, and $\hat g\in \ts_n$, $\tilde g\in \tilde\s_n$ satisfy $\hat \pi(\hat g)=\tilde\pi(\tilde g)$ then $\rho(\hat g)$ differs from $\rho(\tilde g)$ by a scalar, see for example \cite[p. 93]{stem}. So for subgroups $\hat H<\ts_n$ and $\tilde H<\tilde\s_n$ with $\hat \pi(\hat H)=\tilde\pi(\tilde H)$, we have $L\da_{\hat H}$ is irreducible if and only if $L\da_{\tilde H}$ is irreducible.}

Recall that $\ts_n$ is the double cover of the symmetric group $\s_n$, in which transpositions lift to involutions. It is the group generated by $t_1,\dots,t_{n-1},z$ subject to the following relations:
\begin{align*}
zt_i=t_iz,\ z^2=t_i^2=1,\ t_it_{i+1}t_i=t_{i+1}t_it_{i+1},
\ t_it_j=zt_jt_i\ (\text{for}\ |i-j|>1). 
\end{align*}
We have the natural projection 
$$\pi:\ts_n\to \s_n$$
which maps $t_i$ onto the transposition $(i,i+1)$. Then $\tA_n=\pi^{-1}(\A_n)$, where $\A_n<\s_n$ is the alternating group. 

From now on, let $G=\ts_n$ or $\tA_n$ for $n\geq 5$. 

In characteristic $p=0$, Kleidman and Wales \cite{KW} classify the faithful irreducible $\F G$-modules $L$ and subgroup $H<G$ such that $L\da_H$ is irreducible, provided that either $H$ is quasi-simple or $H$ is a maximal subgroup of $G$.\footnote{A generalization of \cite{KW} can be recovered from the positive characteristic results recorded below by assuming $p>n$.}

Assume from now on that $p=\cha \F$ is positive. 
We may then further assume that $p>2$, as in the case $p=2$  the center $Z(G)$ acts trivially on an irreducible $\F G$-module $L$ and so $L$ is not faithful. 

We now formulate a result by the first and the third authors which deals with the case where $\pi(H)<\s_n$ is a primitive  
subgroup. Recall that 
a subgroup $X \leq \s_n$ is called {\em primitive} if $X$ acts transitively on $\Om:=\{1,2,\dots,n\}$ and the 
only partitions of $\Om$ preserved by $X$ are the partitions into either a single set or into $n$ singleton sets; otherwise $X$ is called {\it imprimitive} 
(imprimitive subgroups may be transitive or 
intransitive). 

The {\em  basic} and {\em second basic} $\F G$-modules are defined in Section~\ref{SSBasic}. Wales \cite{Wales} computed the dimensions of the basic and second basic $\F G$-modules as follows. Let  
$$
\kappa_n:=
\left\{
\begin{array}{ll}
1 &\hbox{if $p|n$,}\\
0 &\hbox{otherwise.}
\end{array}
\right.
$$
Then the dimensions of the basic modules for $\ts_n$ and $\tA_n$  are, respectively, 
$$
2^{\lfloor \frac{n-1-\kappa_n}{2} \rfloor}\quad \text{and}\quad
2^{\lfloor \frac{n-2-\kappa_n}{2} \rfloor};
$$  
and the dimensions of the second basic module for $\ts_n$ and $\tA_n$  are, respectively, 
$$
2^{\lfloor \frac{n-2-\kappa_{n-1}}{2} \rfloor}(n-2-\kappa_n-2\kappa_{n-1})\quad\text{and}\quad 
2^{\lfloor \frac{n-3-\kappa_{n-1}}{2} \rfloor}(n-2-\kappa_n-2\kappa_{n-1}). 
$$

Now the result for the case where $\pi(H)$ is primitive is as follows. 
\footnote{The case (iii)(b) of Theorem~\ref{TAKT}  appeared as the case (iv)(c) in \cite[Theorem B]{KT}, but it is actually a second basic module, so it belongs to part (iii).} 
\footnote{In Theorem~\ref{TAKT}(i)(e), the case 
$\pi(H)=\Aut(\A_6)$ with $p=5$ was missed in \cite[Theorem B]{KT}.} 
\footnote{In the case (ii)(e) of Theorem~\ref{TAKT}, for a subgroup $H$ with ${\sf L}_{2}(8) \lhd \pi(H) \leq \Aut({\sf L}_{2}(8))$, only one of the two basic spin modules of $G$ is irreducible on $H$. See also \cite[p.463]{KW}.}

\begin{MainTheorem}\label{TAKT} {\bf \cite[Theorem B]{KT}} 
Let $G = \ts_n$ or $\tA_n$ with $n \geq 5$, $L$ be a faithful 
irreducible $\F G$-module,
and let $H$ be a subgroup of $G$ such that $\pi(H)<\s_n$ is a {\bf primitive} 
subgroup not containing $\A_n$. Then 
$L\da_{H}$ is irreducible if and only if one of the following holds:

\begin{enumerate}[\rm(i)]
\item $G = \ts_n$, $L$ is a basic module, and one of the 
following holds:

\begin{enumerate}[\rm(a)]
\item $n=5$, $p \neq 5$, and $\pi(H) = \ZZ_{5}\rtimes\ZZ_{4}$;
\item $n=6$, and $\pi(H) = \s_{5}$;
\item $n=6$, $p \neq 3$, and $\pi(H) = \A_{5}$;
\item $n=8$, and $\pi(H)={\sf AGL}_3(2)$;
\item $n=10$. Furthermore, either $p\neq 3,5$ and $\pi(H)=\s_6, {\sf M}_{10}$, or $p\neq 3$ and $\pi(H)=\Aut(\A_6)$;
\item $n=11$, $p=11$, and $\pi(H)={\sf M}_{11}$ (two classes); 
\item $n=12$, $p\neq 3$, and $\pi(H)={\sf M}_{12}$. 
\end{enumerate}

\item $G = \tA_n$, $L$ is a basic module, 
and one of the following holds:

\begin{enumerate}[\rm(a)]
\item $n=5$, $p \neq 5$, and $\pi(H) = \ZZ_{5}\rtimes\ZZ_{2}$;
\item $n=6$, and $\pi(H) = \A_{5}$;
\item $n=7$, and $\pi(H) = {\sf L}_{2}(7)$ (two classes);
\item $n=8$, and $\pi(H)={\sf AGL}_3(2)$ (two classes);
\item $n=9$, $p\neq 3$, and ${\sf L}_{2}(8) \lhd \pi(H) \leq \Aut({\sf L}_{2}(8))$, 
or $3^2\rtimes{\sf Q}_8\lhd\pi(H)\leq 3^2\rtimes {\sf SL}_2(3)$;
\item $n=10$, $p\neq 3$, and $\pi(H)={\sf M}_{10}$;
\item $n=10$, $p=5$, and $\pi(H)=\A_6$;
\item $n=11$, $p\neq 3$, and $\pi(H)={\sf M}_{11}$ (two classes); 
\item $n=12$, $p\neq 3$, and $\pi(H)={\sf M}_{12}$ (two classes). 
\end{enumerate}

\item $G = \tA_n$, $L$ is a second basic module, and 
one of the following holds:

\begin{enumerate}[\rm(a)]
\item $n=6$, $p = 3$, and $\pi(H) = \A_{5}$;
\item $n=7$, $p = 3$, 
$\pi(H) = {\sf L}_{2}(7)$ (two classes); 
\item $n=8$, $p\neq 7$, and $\pi(H)={\sf AGL}_3(2)$ (two classes);
\item $n=12$, $p\neq 3, 11$, and $\pi(H)={\sf M}_{12}$ (two classes). 
\end{enumerate}
\item $L$ is neither a basic nor a second basic module, and 
one of the following holds:

\begin{enumerate}[\rm(a)]
\item $n=5$, $p >5$, $G = \ts_5$, $\pi(H) = \ZZ_{5}\rtimes\ZZ_{4}$, and 
$\dim L = 4$;
\item $n=6$, $p >5$, $G= \ts_6$, $\pi(H) = \s_{5}$, and 
$\dim L = 4$.
\end{enumerate}
\end{enumerate}
\end{MainTheorem}

\vspace{1mm}
The immediate consequence of Theorem~\ref{TAKT} is

\vspace{3mm}
\noindent
{\bf Corollary.}
{\em 
Let $G = \ts_n$ or $\tA_n$ with $n >12$, $L$ be a faithful 
irreducible $\F G$-module,
and let $H$ be a subgroup of $G$ not containing $\tA_n$. If 
$L\da_{H}$ is irreducible then $\pi(H)$ is imprimitive. 
}

\vspace{3mm}

In view of Theorem~\ref{TAKT}, it remains to deal with {\em imprimitive} subgroups, which turns out to be a much more difficult case. Under the assumption that $L$ is not a basic module, this case  is settled in Theorem~\ref{imprimitive} below, which is the main result of this paper. 

To state the theorem we need to recall the classification of the faithful irreducible $\F G$-modules. The reader is refereed to Section~\ref{SDoubleCov} for details on this. We denote by $\RP_p(n)$ the set of all restricted $p$-strict partitions of $n$. For a partition $\la\in \RP_p(n)$, we denote by $h_{p'}(\la)$ the number of parts of $\la$ not divisible by $p$, and set 
\begin{equation*}
a_p(\la):=
\left\{
\begin{array}{ll}
0 &\hbox{if $n-h_{p'}(\la)$ is even,}\\
1 &\hbox{otherwise.}
\end{array}
\right.
\end{equation*}
Then the irreducible $\F G$-modules can be canonically labelled as follows:
\begin{align*}
\Irr(\F\ts_n)&=\{D(\la;0)\mid\la\in\RP_p(n),\,\,a_p(\la)=0\}\sqcup \{D(\la;\pm)\mid\la\in\RP_p(n),\,\,a_p(\la)=1\},
\\
\Irr(\F\tA_n)&=\{E(\la;0)\mid\la\in\RP_p(n),\,\,a_p(\la)=1\}\sqcup \{E(\la;\pm)\mid\la\in\RP_p(n),\,\,a_p(\la)=0\}.
\end{align*}
Thus, when we write $D(\la;\pm)$ it is assumed that $a_p(\la)=1$, when we write $E(\la;\pm)$ it is assumed that $a_p(\la)=0$, etc. 
We will sometimes write $D(\la;\eps)$ to denote $D(\la;0)$ or $D(\la;\pm)$ depending on whether $a_p(\la)=0$ or $1$, and similarly for $E(\la;\eps)$.

Setting $\ell:=(p-1)/2$, for every $i\in I:=\{0,1,\dots,\ell\}$, there is an explicit class $\JS^{(i)}$ of {\em $i$-Jantzen-Seitz}  (restricted $p$-strict) partitions, and we set ${\JS}=\bigsqcup_{i\in I}\JS^{(i)}$. 

For a composition $(\mu_1,\dots,\mu_r)$ of $n$ we have a {\em standard Young subgroup}
$$
\s_{\mu_1,\dots,\mu_r}=\s_{\mu_1}\times\dots\times\s_{\mu_r}<\s_n.
$$
If $n=ab$ for integers $a,b>1$, we also have the {\em standard wreath product subgroup} 
$$
\W_{a,b}:=\s_a\wr\s_b<\s_n.
$$
We set $\A_{\mu_1,\dots,\mu_r}:=\s_{\mu_1}\times\dots\times\s_{\mu_r}\cap\A_n,
$ and 
$$
{\ts}_{\mu_1,\dots,\mu_r}:=\pi^{-1}(\s_{\mu_1,\dots,\mu_r})<{\ts}_n,\ \,
{\tA}_{\mu_1,\dots,\mu_r}:=\pi^{-1}(\A_{\mu_1,\dots,\mu_r})<{\tA}_n,\ \,
\hW_{a,b}:=\pi^{-1}(\W_{a,b}).
$$

\begin{MainTheorem}\label{imprimitive}
Let $G=\ts_n$ or $\tA_n$ with $n\geq 5$, and $H$ be a subgroup of $G$ such that $\pi(H)<\s_n$ is {\bf imprimitive}. Suppose $L$ is a faithful irreducible $\F G$-module, which is {\bf not basic}. Then $L\da_H$ is irreducible if and only if one of the following holds:
\begin{enumerate}[\rm(i)]
\item $H=\ts_{n-1,1}\cap G$ and one of the following holds:
\begin{enumerate}[\rm(a)]
\item $L=D(\la;0)$ or $E(\la;0)$ with $\la\in\JS^{(0)}$;

\item $L=D(\la;\pm)$ or $E(\la;\pm)$ with $\la\in\JS$.
\end{enumerate}

\item $G=\ts_n$, $H=\tA_{n-1,1}$ and $L=D(\la;\pm)$ with $\la\in \JS^{(0)}$;

\item $H=\ts_{n-2,2} \cap G$, and $L=D(\la;\eps)$ or $E(\la;\eps)$ with $\la\in\JS^{(0)}$;

\item $H=\ts_{n-2,1,1}\cap G$, and $L=D(\la;\pm)$ or $E(\la;\pm)$ with $\la\in\JS^{(0)}$;

\item $G=\ts_n$, $H=\tA_{n-2,2}$ and $L=D(\la;\pm)$ with $\la\in\JS^{(0)}$;



\item $L$ is second basic, $p\mid(n-1)$, $n=2b$ is even and one of the following holds:
\begin{enumerate}[\rm(a)]
\item $G=\ts_n$ and $H=\hW_{2,b}$ or $H=\hW_{b,2}$,

\item $G=\ts_n$, $\pi^{-1}(\A_{b}\times \A_{b})< H<\hW_{b,2}$ with  $[\hW_{b,2}:H]=2$ and $H\neq \ts_{b,b}$ (there are two such conjugacy classes of subgroups $H$),

\item $G=\tA_n$ and $H=\hW_{b,2}\cap \tA_n$;
\end{enumerate}



\item $(L,G,\pi(H))$ is as in Tables I or II.
\end{enumerate}
\end{MainTheorem}

\setlength{\extrarowheight}{4pt}
{\small
\[\begin{array}{|c|c|c|c|c|}
\hline
L&\dim L&G&H&p\\
\hline\hline
D((3,2,1);\pm)&4&\ts_6&\hW_{3,2}&p\geq 7\\
\hline
D((3,2,1);\pm)&4&\ts_6&\hW_{2,3}&p\geq 7\\
\hline
E((3,2,1),0)&4&\tA_6&\hW_{3,2}\cap\tA_6&p\geq 7\\
\hline
D((4,3,2,1),0)&96&\ts_{10}&\hW_{5,2}&p\geq 7\\
\hline
E((4,3,2,1);\pm)&48&\tA_{10}&\hW_{5,2}\cap\tA_{10}&p\geq 7\\
\hline
\end{array}\]
\setlength{\extrarowheight}{0pt}

\vspace{1mm}
\centerline{{\sc Table I}: {Non-serial irreducible restrictions to maximal subgroups $\hW_{a,b}\cap G$}}
}

\setlength{\extrarowheight}{4pt}
{\small

\[\begin{array}{|c|c|c|c|c|}
\hline
L&\dim L & G&\pi(H)&p\\
\hline\hline
D((3,2,1);\pm)&4&\ts_6&\pi(H)\cong\Z_5\rtimes\Z_4< \s_{5,1}&p\geq 7\\
\hline
D((3,2,1);\pm)&4&\ts_6&\pi(H)<\W_{3,2}\text{ with }\pi(H)\cap\s_{3,3}=\A_{3,3}\text{ and }\pi(H)\not\leq\s_{3,3}&p\geq 7\\
\hline
D((3,2,1);\pm)&4&\ts_6&\pi(H)=\W_{2,2}\times\s_2&p\geq 5\\
\hline
E((4,2,1);\pm)&6&\tA_7&\pi(H)\cong\A_5\text{ primitive in }\s_{6,1}&p=3\\
\hline
\end{array}\]
\setlength{\extrarowheight}{0pt}

\centerline{{\sc Table II}: {Non-serial irreducible restrictions  to non-maximal imprimitive subgroups}}
}

\vspace{3mm}
Theorem~\ref{imprimitive} substantially strengthens \cite[Theorem D]{KT}. Note also that \cite[Theorem D]{KT} contains a gap---it missed a case corresponding to the case (iii) of Theorem~\ref{imprimitive}, and its corrected and expanded 
version is proved in Theorem \ref{kt-n22}.

Initial considerations indicate that the basic modules may yield many non-maximal imprimitive subgroups with irreducible restrictions, and for this reason we have to exclude them. However, for reader's convenience we cite the following results from \cite{KT}.

\begin{MainTheorem} \cite[Theorem E]{KT} 
Let $G=\ts_n$ or $\tA_n$, 
$L$ be a {\bf basic}  $\FF G$-module, and $H$ be a subgroup of $G$ such that $\pi(H)<\s_n$ is {\bf maximal imprimitive}. 
Then 
$L{\downarrow}_{H}$ is irreducible if and only if one of the following holds:
\begin{enumerate}[\rm(i)]
\item $G=\ts_n$ and one of the following holds:
\begin{enumerate}[\rm(a)]
\item $H=\ts_{n-a,a}$, $a<n/2$, $p{\not{\mid}}\,a$, 
$p{\not{\mid}}\,(n-a)$, and either $2|n$, or $2 \nmid n$ and $p\mid n$.
\item $H=\hW_{a,b}$ for some $a,b\geq 2$ with $n=ab$ and 
$p{\not{\mid}}\,a$.
\end{enumerate}
\item $G=\tA_n$ and one of the following holds:
\begin{enumerate}[\rm(a)]
\item $H=\tA_{n-a,a}$, $a<n/2$, $p{\not{\mid}}\,a$, 
$p{\not{\mid}}\,(n-a)$, and either $2 \nmid n$, or $2p|n$.
\item $H= \hW_{a,b}\cap\tA_n$ for some $a,b\geq 2$ with $n=ab$ and
$p{\not{\mid}}\,a$.
\end{enumerate}
\end{enumerate}
\end{MainTheorem}

Finally, we point out that the irreducible restrictions $L\da_H$ for the case where $H$ almost quasi-simple are classified in \cite[Theorem C]{KT}---that result certainly includes the case where $L$ is basic.

\section{Generalities}

\subsection{Ground field} Throughout the paper we work over an algebraically closed filed $\F$ of characteristic $p>2$. In particular, unless otherwise stated, all representations are over $\F$. Occasionally we will also use complex representations. 

For an $\F$-algebra $A$ denote by $\Irr(A)$ a complete non-redundant set of irreducible $A$-modules up to isomorphism.

\subsection{Groups and modules}
\label{SSGrMod}
Let $G$ be a finite group. All $\F G$-modules are assumed to be finite dimensional. We denote by $\bone_G$ or simply $\bone$ the trivial $\F G$-module. For $\F G$-modules $U,V$ we denote by $\Hom_G(U,V)$ the space of all $\F G$-module homomorphisms from $U$ to $V$, and by $\Hom_\F(U,V)$ the space of all linear maps considered as an $\F G$-module via 
\begin{equation}\label{EHomFMod}
(g\cdot f)(u)=gf(g^{-1}u)\qquad (f\in \Hom_\F(U,V),\ u\in U,\ g\in G).
\end{equation}

 We denote by $M_\F(G)$ the maximal dimension of an irreducible $\F G$-module. We will often use the classical inequalities  $M_\F(G)\leq M_\C(G)$, and 
 \begin{equation}\label{E220925}
 M_\F(\hat G)\leq \sqrt{|G|}
 \end{equation} 
  for a central extension $\hat G$ of $G$.

Let $H$ be another finite group. For an $\F G$-module $V$ and an $\F H$-module $W$ we denote by $V\boxtimes W$ the outer tensor product of $V$ and $W$, which is naturally an $\F(G\times H)$-module. On the other hand, given another $\F G$-module $V'$ we denote by $V\otimes V'$ the inner tensor product of $V$ and $V'$, which is an $\F G$-module via $g(v\otimes v')=gv\otimes gv'$ for all $g\in G,\,v\in V,\,v'\in V'$.

If $H\leq G$ is a subgroup, $V$ is an $\F G$-module, and $W$ is an $\F H$-module, we denote by $V\da^G_{H}$ or simply $V\da_H$ the restriction of $V$ to $H$, and by $W\ua^G_{H}$ or simply $W\ua^G$ the induction of $W$ to $G$. 

Let $V$ be an $\F G$-module. We denote by $V^*$ the dual $\F G$-module. We denote by $V^G$ the set of {\em $G$-invariant vectors} in $V$. 
We write $\soc V$ and $\head V$ for the socle and head of $V$, respectively. 
If $V_1,\dots,V_a$ are $\F G$-modules, we write 
$$
V\sim V_1|\dots|V_a
$$
to indicate that $V$ has a submodule filtration with subquotients $V_1,\dots,V_a$ listed from bottom to top.

For $L\in\Irr(\F G)$ and any $\F G$-module $V$ we denote by $[V:L]$ the composition multiplicity of $L$ in $V$. 

\begin{Corollary} \label{CMin}
Let $V$ be an $\F G$-module and $L\in\Irr(\F G)$ such that $[V:L]=1$. Suppose $W$ is a submodule of $V$ with $\head W\cong L$. Then $W$ is the unique smallest submodule of $V$ having $L$ as a composition factor. 
\end{Corollary}
\begin{proof}
If $X$ is a submodule of $V$ with $[X:L]\neq 0$, then $[V/X:L]=0$, so the composition $W\into V\onto V/X$  has kernel $K$ satisfying $[K:L]\neq 0$, hence $K=W$. Thus $W\subseteq X$.  
\end{proof}

For $m\in\Z_{\geq 0}$, we write $H^m(G,V)$ for the {\em $m$th cohomology space} of $G$ with coefficients in an $\F G$-module $V$, referring the reader for example to \cite{Brown}, \cite[Chapter 1]{Guich} for more information on group cohomology. We will use the following well-known result.

\begin{Lemma} \label{LHS} 
Let $G=A\rtimes B$ be a finite group, $V$ be an $\F G$-module, and $m\in \Z_{\geq 0}$. 
\begin{enumerate}
\item[{\rm (i)}] If $A$ is a $p'$-group then $H^m(G,V)\cong H^m(B,V^A)$;
\item[{\rm (ii)}] If $B$ is a $p'$-group then $H^m(G,V)\cong H^m(A,V)^{B}.$
\end{enumerate}
\end{Lemma}
\begin{proof}
Since $A$ (resp. $B$) is a $p'$-group, 
the corresponding Lyndon-Hochschild-Serre spectral sequences 
\cite[9.1]{Guich}
collapse. 
\end{proof}

\begin{Lemma} \label{LInvSum} 
$V\sim V_1|\cdots| V_t$ be an $\F G$-module. If $H^1(G,V_r)=0$ for all $r=1,\dots,t$ then $\dim V^G=\dim V_1^G+\dots+\dim V_t^G$ and $H^1(G,V)=0$. 
\end{Lemma}
\begin{proof}
The exact sequence $0\to V_1\to V\to W\to 0$ with $W\sim V_2|\cdots| V_t$ yields the exact sequence 
$$0\to V_1^G\to V^G\to W^G\to H^1(G,V_1)\to H^1(G,V)\to H^1(G,W),$$ 
and the result follows by induction on $t$. 
\end{proof}

The following well-known lemma follows from the Clifford theory (using $p\neq 2$). In it, $\sgn$ is the non-trivial $1$-dimensional $\F G$-module with kernel $G_0$, and for an $\F G_0$-module $W$, we denote by $W^\si$ the $\F G_0$-module obtained from $W$ by twisting with the conjugation by $\si\in G\setminus G_0$. 

\begin{Lemma} \label{LClifford}
Let $G_0<G$ be a subgroup of index $2$. Then we can write
$$\Irr(\F G)=\{D_i^{\pm},D_j^0\mid 1\leq i\leq a,\,1\leq j\leq b\}\,\  \text{and}\,\ \Irr(\F G_0)=\{E_i^0,E_j^\pm\mid 1\leq i\leq a,\,1\leq j\leq b\}$$ 
with $D_i^\pm\otimes \sgn\cong D_i^\mp$, $D_i^0\otimes \sgn\cong D_i^0$, $(E_i^\pm)^\si\cong E_i^\mp$, $(E_i^0)^\si\cong E_i^0$, and 
$$
D_i^\pm\da_{G_0}\cong E_i^0,\ D_i^0\da_{G_0}\cong E_i^+\oplus E_i^-,\ E_i^0\ua^G\cong D_i^+\oplus D_i^-,\ E_i^\pm\ua^G\cong D_i^0.
$$ 
\end{Lemma}

\begin{Corollary} \label{CSimpleHeadSnAn}
Let $V$ be an $\F G$-module and ${G_0}<G$ be a subgroup of index $2$. Then $\soc(V\da_{G_0})\cong (\soc V)\da_{G_0}$ and $\head(V\da_{G_0})\cong (\head V)\da_{G_0}$. 
\end{Corollary}
\begin{proof}
Using the notation of Lemma~\ref{LClifford}, we have by the Frobenius reciprocity, 
$$\Hom_{G_0}(E_i^0,V\da_{G_0})\cong \Hom_G(D_i^+\oplus D_i^-,V)\quad\text{and}\quad\Hom_{G_0}(E_i^\pm,V\da_{G_0})\cong \Hom_G(D_i^0,V),
$$ which implies the result on the socle, and the result on the head is proved similarly.
\end{proof}

For an $\F G$-module $V$ we denote by $\chi_V$ its $\F$-valued character, i.e. $\chi_V(x)$ is the trace of $x$ acting on $V$ for all $x\in\F G$. 
If $p>0$ and $V=\bar W$ is a reduction modulo $p$ of a $\C G$-module $W$ (using an appropriate $p$-modular system), then $\chi_V(g)=\bar\chi_W(g)$, reduction modulo $p$ of $\chi_W(g)$ for all $g\in G$.

\subsection{Superalgebras and supermodules}
\label{SSSuperSuper}
A {\em superspace} is a $\Z/2\Z$-graded vector space $V=V_{\0}\oplus V_{\1}$. Let $\eps\in\Z/2\Z$. 
For $v\in V_\eps$, we write $|v|=\eps$. 
Let $V,W$ be superspaces. The tensor product $V\otimes W$ is considered as a superspace via $|v\otimes w|=|v|+|w|$ for all homogeneous $v\in V$ and $w\in W$. For $\de\in\Z/2\Z$, a parity $\de$ homogeneous linear map $f:V\to W$ is a linear map satisfying $f(V_{\eps})\subseteq W_{\eps+\de}$ for all $\eps$. We denote the space of all parity $\de$ homogeneous linear maps  from $V$ to $W$ by $\Hom(V,W)_{\de}$, and set $\Hom(V,W):=\bigoplus_{\de\in\Z/2\Z}\Hom(V,W)_{\de}.$ We write $V\cong W$ (resp. $V\simeq W$) if there is an isomorphism in $\Hom(V,W)$ (resp. $\Hom(V,W)_\0$). 

A~{\em superalgebra} is a superspace $A$ which is a (unital)  algebra with $A_\eps A_\de\subseteq A_{\eps+\de}$ for all $\eps,\de\in\Z/2\Z$. An {\em antiautomorphism} of a superalgebra $A$ is an even linear map $\tau:A\to A$ which satisfies $\tau(ab)=\tau(b)\tau(a)$.

\begin{Example} \label{ExImp} 
{\rm 
An important example is as follow. Let $G$ is be finite group with a subgroup $G_0\leq G$ of index $2$. Then the group algebra $A:=\F G$ is a superalgebra with $(\F G)_\0=\F G_0$, $(\F G)_\1=\spa(G\setminus G_0)$. 
}
\end{Example}

For a superalgebra $A$, a (left) {\em $A$-supermodule} is a superspace $V$ which is a left $A$-module with $A_\eps V_\de\subseteq V_{\eps+\de}$ for all $\eps,\de$. 
Let $V,W$ be graded $A$-supermodules. A 
parity $\de$ homogeneous graded $A$-supermodule homomorphism from $V$ to $W$ is a parity $\de$ homogeneous linear map $f:V\to W$ satisfying $f(av)=(-1)^{\de|a|}af(v)$ for all (homogeneous) $a\in A,\,v\in V$. We denote by $\Hom_A(V,W)_{\de}$ the space of all parity $\de$ homogeneous $A$-supermodule homomorphism from $V$ to $W$, and set 
$
\Hom_A(V,W):=\bigoplus_{\de\in\Z/2\Z}\Hom_A(V,W)_{\de}.
$
We write $V\cong W$ (resp. $V\simeq W$) if there is an isomorphism in $\Hom_A(V,W)$ (resp. $\Hom_A(V,W)_\0$).

In this paper all superalgebras and supermodules are assumed to be finite-dimensional.

If $\tau$ is an antiautomorphism of $A$ and $V$ is an $A$-supermodule, we define the structure of an $A$-supermodule on $V^*$ via $(af)(v)=f(\tau(a)v)$ for all $f\in V^*$, $a\in A$ and $v\in V$. The resulting $A$-supermodule will be denoted $V^\tau$ and called {\em $\tau$-dual of $V$}, or simply {\em dual of $V$} if it is clear which $\tau$ is used.

A {\em subsuperspace} of a superspace $V$ is a subspace $W\subseteq V$ such that $W=(W\cap V_\0)+(W\cap V_\1)$. 
A {\em subsupermodule} of an $A$-supermodule $V$ is a subsuperspace  which is also an $A$-submodule. An {\em irreducible $A$-supermodule} is a supermodule $L$ which has exactly two subsupermodules: $0$ and $L$. 
If $V$ is an $A$-supermodule and $L$ is an irreducible $A$-supermodule, the multiplicity of $L$ in $V$ is denoted $[V:L]$. 
A {\em completely reducible} $A$-supermodule is an $A$-supermodule isomorphic to a direct sum of irreducible $A$-supermodules.

\begin{Lemma} \label{LSelfDual}
Let $A$ be a superalgebra with an antiautomorphism $\tau$, $V$ be an $A$-supermodule, and $W\subsetneq V$ be a proper subsupermodule.  Suppose $V$ and $W$ are $\tau$-self-dual. Then $$\dim\End_A(V)> \dim\End_A(W).$$ 
\end{Lemma}
\begin{proof}
Using $\tau$-self-duality of $V$ and $W$, we see that $W$ is also a quotient of $V$. Now every endomorphism of $W$ gives rise to an endomorphism of $V$ with image contained in $W$. This assignment is injective, but not surjective since the identity on $V$ has image $V$. 
\end{proof}

The {\em socle} (resp. {\em head}) of an $A$-supermodule are defined as the largest completely reducible subsupermodule $\soc V\subseteq V$ (resp. the largest completely reducible quotient module $\head V$ of $V$). 

Let $A$ be a superalgebra. We denote by $|A|$ the algebra $A$ with the superstructure forgotten. If $V$ is an $A$-supermodule, we denote by $|V|$ the $|A|$-module with the superstructure forgotten. We will use without further comment the following equality which comes from \cite[Lemma 12.1.5]{KBook}:
\begin{equation}\label{EDimDim}
\dim\Hom_A(V,W)=\dim \Hom_{|A|}(|V|,|W|). 
\end{equation}

If $V$ is an irreducible $A$-supermodule then either $|V|$ is irreducible or it is the direct sum of two non-isomorphic irreducible $|A|$-modules, see \cite[\S12.2]{KBook}. In the first case we say that $V$ is {\em of type $\Mtype$}, while in the second case we say that $V$ is {\em of type $\Qtype$}. 

For a superalgebra $A$, we denote by $\Irrs(A)$ to be a complete and non-redundant set of irreducible $A$-supermodules up to the isomorphism $\cong$, and we put $\Irr(A):=\Irr(|A|)$. A superalgebra version of Lemma~\ref{LClifford} allows us to relate $\Irrs(A), \Irr(A_\0)$, and $\Irr(A)$ as follows:

\begin{Lemma} \label{PIrrIrr} {\rm \cite[Proposition 12.2.1]{KBook}}
Let $A$ be a (finite-dimensional) superalgebra, and $\Irrs(A)=\{V_1,\dots,V_n\}$, with $V_1,\dots,V_m$ of type $\Mtype$ and $V_{m+1},\dots,V_n$ of type $\Qtype$. Then we have: 
\begin{enumerate}
\item[{\rm (i)}] 
$
\Irr(A)=\{V_1^0,\dots,V_m^0,V_{m+1}^\pm,\dots,V_n^\pm\},
$
where for $i=1,\dots, m$, we have
$V_i^0\cong|V_i|$, and for $j=m+1,\dots, n$ we have 
 $|V_j|\cong V_j^+\oplus V_j^-$.

\item[{\rm (ii)}] 
$
\Irr(A_\0)=\{W_1^\pm,\dots,W_m^\pm,W_{m+1}^0,\dots,W_n^0\},
$
where for $i=1,\dots, m$, we have $\Res^A_{A_\0}V_i\cong W_i^+\oplus W_i^-$, and for $j=m+1,\dots, n$ we have $\Res^{|A|}_{A_\0}V_j^+\cong \Res^{|A|}_{A_\0} V_j^-\cong W_j^0$. 
\end{enumerate}
\end{Lemma}

In the case where the superalgebra $A$ is as in Example~\ref{ExImp}, we get from Lemmas~\ref{LClifford},\,\ref{PIrrIrr}:

\begin{Lemma}\label{lmodules}
Let $G$ is be finite group with a subgroup $G_0< G$ of index $2$, and consider $\F G$ as a superalgebra as in Example~\ref{ExImp}. For $V\in\Irrs(\F G)$, we have: 
\begin{enumerate}[\rm(i)]
\item If $V$ is of type $\Mtype$, then $V^0:=|V|$ is irreducible,  $V^0\otimes\sgn\cong V^0$ and $V^0\da_{G_0}\cong W^+\oplus W^-$ for irreducible $\F G_0$-modules $W^\pm$ satisfying $W^+\not\cong W^-\cong (W^+)^\si$. 

\item If $V$ is of type $\Qtype$, then $|V|\cong V^+\oplus V^-$ with $V^\pm$ irreducible $\F G$-modules such that $V^+\not\cong V^-\cong V^+\otimes\sgn$, and\, $W^0:=V^\pm\da_{G_0}$ is an irreducible $\F G_0$-module satisfying $(W^0)^\si\cong W^0$.
\end{enumerate}
\end{Lemma}

\begin{Corollary} \label{CSuperNonSuper} 
Let $G$ is be finite group with a subgroup $G_0< G$ of index $2$, and consider $\F G$ as a superalgebra as in Example~\ref{ExImp}. Let $H\leq G$ be a subgroup not contained in $G_0$.  In particular, $H_0:=H\cap G_0<H$ is a subgroup of index $2$ and we also  consider $\F H$ as a superalgebra as in Example~\ref{ExImp}. Let $V$ be an irreducible $\F G$-supermodule. Suppose that the supermodule $V\da_H$ has composition length $k$ and that $D_1,\dots,D_k$ be its  composition factors (it could happen that $D_r\cong D_s$ for $r\neq s$). Let $V^\eps$ be an irreducible component of $|V|$ and $W^\eps$ be an irreducible component of $V\da_{G_0}$, with $\eps\in\{0,+,-\}$ as appropriate. Then:
\begin{enumerate}
\item[{\rm (i)}] $V^\eps\da_H$ is irreducible if and only if one of the following happens:
\begin{enumerate}
\item[{\rm (a)}] $k=1$ and $(\text{type of $V$},\text{type of $D_1$})\neq (\Mtype,\Qtype)$. 
\item[{\rm (b)}] $k=2$, $V$ is of type $\Qtype$, and $D_1,D_2$ are of type $\Mtype$; in this case we have $D_1\cong D_2$. 
\end{enumerate}
\item[{\rm (ii)}] $W^\eps\da_{H_0}$ is irreducible if and only if $k=1$. 
\end{enumerate}
\end{Corollary}
\begin{proof}
It is clear from Lemma~\ref{lmodules} that the listed cases produce irreducible restriction.  To see that in case (i)(b) we have $D_1\cong D_2$, note by Lemma~\ref{lmodules} that $|V|=V^+\oplus V^-$, $V^-\cong V^+\otimes\sgn$, $D_1=V^\pm\da_H$, $D_2=V^\mp\da_H$, and so $D_2\cong D_1\otimes\sgn\cong D_1$ since $D_1$ is of type $\Mtype$.

To see that in all other cases the restrictions are reducible, use Lemma~\ref{lmodules} together with the fact that $V^+\da_H$ and  $V^-\da_H$ (resp. $W^+\da_{H_0}$ and $W^+\da_{H_0}$) have the same composition length since $V^-\cong V^+\otimes\sgn$ (resp.  $W^-\cong (W^+)^\si$ for $\si\in H\setminus H_0$). 
\end{proof}

If $V$ is an $A$-supermodule, we say that $V$ {\em admits an odd involution} if there exists $J\in\Hom_A(V,V)_\1$ such $J^2=\id_V$. An irreducible $A$-supermodule admits an odd involution if and only if it is of type $\Qtype$, cf. \cite[Lemma 12.2.3]{KBook}. 

Let $A,B$ be superalgebras. The tensor product $A\otimes B$ is considered as a graded superalgebra via 
$
(a\otimes b)(a'\otimes b')=(-1)^{|b| |a'|}aa'\otimes bb'
$
for all homogeneous $a,a'\in A$ and $b,b'\in B$.

Given an $A$-supermodule $V$ and a $B$-supermodule $W$, we have the $(A\otimes B)$-supermodule $V\boxtimes W$ with the action 
$
(a\otimes b)(v\otimes w)=(-1)^{|b||m|}(av\otimes bw)$ for $a\in A,\, b\in B,\, v\in V,\, w\in W.
$

\begin{Lemma}\label{L071218_4}
Let $A$ and $B$ be superalgebras, $M$ be an $A$-supermodule and $N$ be a $B$-supermodule. If both $M$ and $N$ admit an odd involution then there exists an $A\otimes B$-supermodule $M\circledast N$ such that $M\boxtimes N\cong (M\circledast N)^{\oplus 2}$.
\end{Lemma}

\begin{proof}
If $J_M$ is an odd involution of $M$ and $J_N$ is an odd involution of $N$, then the mapping $J_M\otimes J_N:M\otimes N\to M\otimes N,\ m\otimes n\mapsto (-1)^{|m|}J_M(m)\otimes J_N(n)$ belongs to $\End_{A\otimes B}(M\boxtimes N)_\0$ and $(J_M\otimes J_N)^2=-\id_{M\boxtimes N}$. Now take $M\circledast N$ to be the $\sqrt{-1}$-eigenspace of $J_M\otimes J_N$ on $M\boxtimes N$ and note that $J_M\otimes \id_N$ is an isomorphism between the $\sqrt{-1}$-eigenspace and $-\sqrt{-1}$-eigenspace. We refer the reader to the argument in \cite[Section 2-b]{BK} for details.
\end{proof}

Let $A$ and $B$ be superalgebras, $V$ be an irreducible $A$-supermodule, and $W$ an irreducible $B$-supermodule. If $V$ and $W$ are of type $\Qtype$ then by Lemma~\ref{L071218_4}, there exists an $A\otimes B$-supermodule $V\circledast W$ such that $V\boxtimes W\cong(V\circledast W)^{\oplus 2}$. In all other cases, we denote $V\circledast W\cong V\boxtimes W$.

\begin{Lemma} \label{LBoxTimes} {\rm \cite[Lemma 12.2.13]{KBook}}
Let $A$ and $B$ be superalgebras.
$$
\Irrs(A\otimes W)=\{V\circledast W\mid V\in \Irrs(A),\,W\in \Irrs(B)\}.
$$
Moreover, $V\circledast W$ is of type $\Mtype$ if and only if $V$ and $W$ are of the same type.
\end{Lemma}

The proof of the next lemma does not work in characteristics 0 (but recall  we are assuming $p>2$).

\begin{Lemma}\label{L071218_3}
Let $D$ be an irreducible $A$-supermodule of type $\Qtype$, and let $V$ be an $A$-supermodule with $\head V\cong D$. If\, $\End_A(V)\simeq\End_A(D)^{\oplus [V:D]}$ then $V$ admits an odd involution.
\end{Lemma}

\begin{proof}
We have $V/\rad V\cong D$. 
Since $D$ is of type $\Qtype$, we have $\End_A(D)_\0\cong \End_A(D)_\1\cong\F$. Moreover, the superspace $\Hom_A(V,\rad V)$ embeds into $\End_A(D)^{\oplus([V:D]-1)}$, so there exists   $J\in\End_A(V)_\1$ with $\im J=V$. As $J^2$ is even and $\head V\cong D$, up to rescaling of $J$, we may assume that $J^2=\id_V+f$ for some $f\in\Hom_A(V,\rad V)$. Then there exists $k\in\Z_{>0}$ with $f^k=0$. Take $m\in\Z_{> 0}$ with $p^m\geq k$. Since $p$ is odd, $J^{p^m}$ is odd. Further $(J^{p^m})^2=\id_V^{p^m}+f^{p^m}=\id_V$. Thus, $J^{p^m}$ is an odd involution.
\end{proof}

\begin{Example} \label{ExClifford} 
{\rm 
The rank $n$ {\em Clifford superalgebra\, $\Cl_n$}  is the superalgebra
given by odd generators $\cc_1,\dots,\cc_n$ subject to the relations 
$\cc_r^2=1$ and $\cc_s\cc_t = -\cc_t\cc_s$ for $s\neq t$. 
The superalgebra\, $\Cl_n$ has basis $\{\cc_1^{\eps_1} \dots \cc_n^{\eps_n}\mid \eps_1,\dots,\eps_n\in\{0,1\}\}$, and  $\Cl_n\otimes \Cl_m\cong \Cl_{n+m}.$ The superalgebra $\Cl_1$ has a unique irreducible supermodule $U_1$ which is the regular $\Cl_1$-supermodule. More generally, $\Irrs(\Cl_n)=\{U_n\}$, where the {\em Clifford module}\, $U_n:= U_1^{\circledast n}$ is the irreducible supermodule of dimension $2^{\lceil n/2\rceil}$ and of type $\Mtype$ if and only if $n$ is even. 
}
\end{Example}

\begin{Lemma}\label{tensor1sgn}
Let $A$ be a superalgebra and $V$ be an $(A\otimes \Cl_1)$-supermodule. 
Let $V'=V$ with the new action of $A\otimes \Cl_1$ given by 
$(a\otimes \cc)* v=(-1)^{|\cc|}(a\otimes \cc)v$ 
for all $a\in A,\,\cc\in \Cl_1,\,v\in V$. Then $V'$ is an $(A\otimes \Cl_1)$-supermodule isomorphic to $V$.
\end{Lemma}
\begin{proof}
It is easily checked that the new formula defines an action, and the (odd) isomorphism is given by $V\to V',\ v\mapsto (1\otimes \cc_1)v$, where $\cc_1$ is the canonical generator of $\Cl_1$. 
\end{proof}

\begin{Example} \label{ExSergeev} 
{\rm 
The symmetric group $\s_n$ acts on the generators $\cc_1,\dots,\cc_n$ of the Clifford algebra $\Cl_n$ on the right via place permutations, i.e. $\cc_s\cdot g=\cc_{g^{-1}s}$ for $s=1,\dots,n$ and $g\in\s_n$. This action is extended to the action of $\s_n$ on $\Cl_n$ on the right by superalgebra automorphisms. Considering the group algebra $\F\s_n$ as a purely even superalgebra, we denote by $\F\s_n\ltimes\Cl_n$ the superspace $\F\s_n\otimes\Cl_n$ considered as a superalgebra via $(g\otimes\cc)(g'\otimes\cc')=gg'\otimes (\cc\cdot g')\cc'$ for all $g,g'\in\s_n,\ \cc,\cc'\in\Cl_n$. This is a version of the (rank $n$)  {\em Sergeev superalgebra}, see \cite[\S13.2]{KBook}. For any subgroup $K\leq \s_n$ we have the obvious subsuperalgebra $\F K\ltimes\Cl_n\subseteq \F\s_n\ltimes\Cl_n$
}
\end{Example}

\section{Combinatorics of partitions}

\subsection{Compositions, partitions and tableaux}
A {\em composition} is a sequence $\la=(\la_1,\la_2,\ldots)$ of non-negative integers which are eventually zero. 
For compositions $\la=(\la_1,\la_2,\ldots)$ and $\mu=(\mu_1,\mu_2,\ldots)$, we have the composition
$$
\la+\mu:=(\la_1+\mu_1,\la_2+\mu_2,\ldots). 
$$ 

We let $\unlhd$ denote the {\em dominance order} on compositions, see \cite[\S3]{JamesBook}. For $n\in\Z_{\geq 0}$, we say that $\la$ is a composition of $n$ if $\la_1+\la_2+\ldots=n$. We often omit an infinite tail of zeros and write $\la$  as $\la=(\la_1,\dots,\la_r)$. 

A {\em partition} is a composition whose parts are weakly decreasing. We denote by $\Par(n)$ the set of all partitions of $n$. 
If $\la\in\Par(n)$, we write $|\la|:=n$. 
The only partition of $0$ is denoted $\varnothing$. 
Sometimes we collect equal parts of a partition $\la$ and write it in the form $\la=(l_1^{a_1},\dots,l_s^{a_s})$ for $l_1>\dots>l_s>0$ and $a_1,\dots,a_s>0$. 

We identify a partition $\la\in\Par(n)$ with its {\em Young diagram}  
$\la=\{(r,s)\in\Z_{>0}\times \Z_{>0}\mid s\leq \la_r\}.$ We refer to the elements of $\Z_{>0}\times \Z_{>0}$ as the {\em nodes}. In particular we can speak of nodes of $\la$. 
A {\em $\la$-tableau} is then a bijection $t:\{1,\dots,n\}\to \la$. 

We denote by $h(\la)$ the number of the non-zero parts in the partition~$\la$. We denote by $h_{p'}(\la)$ the number of parts of $\la$ not divisible by $p$, and by $h_p(\la)$ the number of parts of $\la$ divisible by $p$. 

A partition $\la$ is called {\em $p$-regular} if no part of $\la$ is repeated $p$ or more times. We denote by $\Par_{\operatorname{reg}}(n)$ the set of all $p$-regular partitions of $n$.

\subsection{\boldmath $p$-strict and $p$-restricted partitions} 

We denote by $\Par_p(n)$ the set of all {\em $p$-strict partitions}  of $n$, i.e. the partitions $\la=(\la_1,\la_2,\dots)$ of $n$ such that $\la_r=\la_{r+1}$ for some $r$ only if $\la_r$ is divisible by $p$. 
A $p$-strict partition $\la$ is called {\em restricted} if for all $r$ either $\la_r-\la_{r+1}<p$, or $\la_r-\la_{r+1}=p$ and $p\nmid\la_r$. We denote by $\RP_p(n)$ the set of all restricted $p$-strict partitions of $n$. We interpret $\Par_0(n)$ as the set of {\em strict partitions}, i.e. partitions with distinct non-zero parts. 
For a partition $\la\in \Par_p(n)$, we  set 
\begin{equation}
\label{Ea_p}
a_p(\la):=
\left\{
\begin{array}{ll}
0 &\hbox{if $n-h_{p'}(\la)$ is even,}\\
1 &\hbox{otherwise.}
\end{array}
\right.
\end{equation}
We interpret $a_0(\la)$ as 
\begin{equation}
\label{Ea_0}
a_0(\la):=
\left\{
\begin{array}{ll}
1 &\hbox{if $n-h(\la)$ is even,}\\
0 &\hbox{otherwise.}
\end{array}
\right.
\end{equation}

\subsection{\boldmath Addable and removable nodes for $p$-strict partitions}
\label{SSAddRem}
We record some combinatorial notions referring the reader to \cite[\S22.1]{KBook} for details and examples. 

Set $\ell=(p-1)/2$ and $I = \{0,1,\dots,\ell\}$. 
A positive integer $s$ can be written uniquely in the form
$s= mp + \ell +1 \pm k$, with $m,k \in\Z$ and $0 \leq k \leq \ell$.
The {\em residue} of $s$, written $\ttres(s)$ is then defined to be $\ell-k$. The residue of a node $A:=(r,s)$, written $\ttres A$, is defined to be $\ttres(s)$. In particular, the residue of a node
depends only on its column.

\begin{Lemma} \label{LAGa}  
Let $\la\in\RP_p(n)$ and for each $i\in I$ denote by $\ga_i$ the number of nodes of $\la$ of residue $i$. Then $a_p(\la)\equiv \ga_1+\dots+\ga_\ell\pmod{2}$. 
\end{Lemma}
\begin{proof}
This follows from \cite[(22.13),\,(22.14)]{KBook}.
\end{proof}

Let $\la\in\RP_p(n)$ and $i \in I$.
A node $A = (r,s)
\in \la$ is called {\em $i$-removable} (for $\la$) if one of the following
holds:
\begin{itemize}
\item[(R1)] $\ttres A = i$ and
$\la-\{A\}$ is again a $p$-strict partition; such $A$ is also called {\em properly $i$-removable};
\item[(R2)] the node $B = (r,s+1)$ immediately to the right of $A$
belongs to $\la$,
$\ttres A = \ttres B = i$,
and both $\la - \{B\}$ and
$\la - \{A,B\}$ are $p$-strict partitions.
\end{itemize}
Similarly, a node $B = (r,s)\notin\la$ is called 
{\em $i$-addable} (for $\la$) if one of the following holds:
\begin{itemize}
\item[(A1)] $\ttres B = i$ and
$\la\cup\{B\}$ is again an $p$-strict partition; such $B$ is also called {\em properly $i$-addable};
\item[(A2)] 
the node $A = (r,s-1)$
immediately to the left of $B$ does not belong to $\la$,
$\ttres A = \ttres B = i$, and both 
$\la \cup \{A\}$ and 
$\la \cup\{A,B\}$ are $p$-strict partitions.
\end{itemize}
We note that (R2) and (A2) above are only possible in case $i = 0$.
If $A$ is properly $i$-removable and $B$ is properly  $i$-addable  for $\la$, we have the $p$-strict partitions 
$$
\la_A:=\la\setminus\{A\}\in\Par_p(n-1)\quad\text{and}\quad
\la^B:=\la\cup\{B\}\in\Par_p(n+1).
 $$ 

Now label all $i$-addable
nodes of the diagram $\la$ by $+$ and all $i$-removable nodes by $-$.
The {\em $i$-signature} of 
$\la$ is the sequence of pluses and minuses obtained by going along the rim of the Young diagram from bottom left to top right and reading off
all the signs.
The {\em reduced $i$-signature} of $\la$ is obtained 
from the $i$-signature
by successively erasing all neighbouring 
pairs of the form $+-$. 
Note the reduced $i$-signature always looks like a sequence
of $-$'s followed by $+$'s.
Nodes corresponding to a $-$ in the reduced $i$-signature are
called {\em $i$-normal}, nodes corresponding to a $+$ are
called {\em $i$-conormal}.
The rightmost $i$-normal node is called {\em $i$-good}, 
and the leftmost $i$-conormal node is called {\em $i$-cogood}. We define
$$
\eps_i(\la):=\sharp\{\text{$i$-normal nodes in $\la$}\}\quad\text{and}\quad
\phi_i(\la):=\sharp\{\text{$i$-conormal nodes for $\la$}\}.
$$
If $\eps_i(\la)>0$ (resp. $\phi_i(\la)>0$) and $A$ is the $i$-good (resp. $B$ is the $i$-cogood) node for $\la$, we set 
$
\tilde e_i\la:=\la_A
$ (resp. $\tilde f_i\la:=\la^B$). 

Let $\la\in\RP_p(n)$ and $i\in I$. We say that $\la$ is {\em $i$-Jantzen-Seitz}, written $\la\in \JS^{(i)}$, if $\eps_i(\la)=1$ and $\eps_j(\la)=0$ for all $j\neq i$. We say that $\la$ is {\em Jantzen-Seitz}, written $\la\in\JS$, if it is $i$-Jantzen-Seitz for some $i$, i.e. ${\JS}=\bigsqcup_{i\in I}\JS^{(i)}$.\footnote{This terminology comes from \cite{JS} where the similar notion for symmetric groups is considered.}

\begin{Lemma} \label{LJS} 
If $\la\in\JS^{(0)}$ then $\tilde e_0\la\in \JS^{(1)}$. 
\end{Lemma}
\begin{proof}
Note that $\la$ must be of the form $\la=(\la_1,\ldots,\la_{h-1},1)$, with $\tilde e_0\la=(\la_1,\ldots,\la_{h-1})$. Then  $A:=(h-1,\la_{h-1})$ is a $1$-normal node of $\tilde e_0\la$, and it suffices to prove that $A$ is the only normal node of $\tilde e_0\la$. 
 
If $\la_{h-1}=2$, then all addable nodes of $\la$ are in rows $\leq h-1$, so they are also addable in $\tilde e_0\la$. Moreover, the removable nodes of $\tilde e_0\la$ are exactly the node $A$ together with the removable nodes of $\la$ in rows $<h-1$. As removable nodes of $\la$ in rows $<h-1$ cancel in the reduced  signature of $\la$, these nodes cancel also in the reduced signature of $\tilde e_0\la$. Thus $A$ is the only normal node of $\tilde e_0\la$. 

If $\la_{h-1}>2$ then the removable (resp. addable) nodes of $\tilde e_0\la$ in rows $\leq h-2$ (resp. $\leq h-1$) are also removable (resp. addable) in $\la$. In the reduced signature for $\la$, any removable node in rows $\leq h-2$ must cancel with some addable node in rows $\leq h-1$ (as the remaining addable node $(h,2)$ of $\la$ cancels with $A$), so they cancel also in the reduced signature for $\tilde e_0\la$. So again $A$ is the only normal node of $\tilde e_0\la$. 
\end{proof}

\section{Representations of symmetric groups}
\subsection{Symmetric and alternating groups}\label{SSAltSym}
The symmetric group on $n$ letters is denoted $\s_n$ and the alternating group on $n$ letters is denoted $\A_n$. We denote by $\sgn$ the sign representation of $\s_n$ so that $\A_n=\Ker(\sgn)$.

For a composition $\la=(\la_1,\dots,\la_r)$ of $n$ we have a {\em standard Young subgroup}
$$
\s_\la=\s_{\la_1,\dots,\la_r}=\s_{\la_1}\times\dots\times\s_{\la_r}<\s_n.
$$
We set $\A_\la=\A_{\la_1,\dots,\la_r}:=\s_{\la}\cap\A_n$.

For $m\leq n$ we always identify $\s_m$ with the subgroup $\s_{m,1^{n-m}}\leq \s_n$. More generally for a composition $\mu$ of $m$ we identify $\s_\mu$ with a subgroup of $\s_n$ via $\s_\mu\leq \s_m\leq \s_n$. We make similar identifications for the alternating groups. 


If $n=ab$ for integers $a,b>1$, we also have the standard {\em wreath product subgroup} 
\begin{equation}\label{EWreath}
\W_{a,b}:=\s_a\wr\s_b<\s_n.
\end{equation}
Note that $\W_{a,b}\cong \s_a^{\times b}\rtimes \s_b$. 

We will use the following lemma which is a special case of Mackey's Theorem:

\begin{Lemma}\label{L060125_2}
Let $K\leq\s_n$ with $K\not\leq\A_n$. Then $\bone\ua_{K\cap\A_n}^{\A_n}\cong(\bone\ua_K^{\s_n})\da^{\s_n}_{\A_n}$.
\end{Lemma}

\subsection{Modules over symmetric groups}
\label{SSSymMod}
Let $\la$ be a partition of $n$. 
As in \cite[\S4]{JamesBook}, we have the {\em permutation module} $M^\la$ on the set of {\em $\la$-tabloids} $\{t\}$, which are row-equivalence classes of $\la$-tableaux $t$. We have $M^\la\cong\bone_{\s_\la}\ua^{\s_n}$ and $(M^\la)^*\cong M^\la$. 
We also have the {\em Specht module} $S^\la\subseteq M^\la$ spanned by the {\em polytabloids} 
\begin{equation}\label{EPolyTab}
e_t:=\sum_{\si\in {\tt C}_t}(\sgn\, \si)\, \si\cdot\{t\}\in M^\la, 
\end{equation}
where ${\tt C}_t$ denotes the column stabilizer of  the $\la$-tableau $t$. In fact, any $e_t$ generates $S^\la$ as an $\F\s_n$-module \cite[4.5]{JamesBook}, and the polytabloids corresponding to the standard tableaux $t$ form a basis of $S^\la$, see \cite[8.4]{JamesBook}.

Occasionally we work with Specht modules over complex field, in which case we use the notation $S^\la_\C$; this is an irreducible $\C \s_n$-module, and $S^\la$ can be obtained from $S^\la_\C$ using reduction modulo~$p$. 

Let $\langle\cdot,\cdot\rangle$ be the standard invariant bilinear form on $M^\la$ from \cite[\S4]{JamesBook}. Then we have 
\begin{equation}\label{SLaPerpSLaDual}
M^\la/(S^\la)^\perp\cong (S^\la)^*,
\end{equation}
where
$(S^\la)^\perp:=\{v\in M^\la\mid \langle v,w\rangle=0\ \text{for all}\ w\in S^\la\}$.

By \cite{JamesBook}, we have $D^\la:=\head S^\la$ is irreducible if $\la$ is $p$-regular, and 
$$
\Irr(\F\s_n)=\{D^\la\mid \la\in \Par_{\operatorname{reg}}(n)\}. 
$$

The {\em Mullineux bijection} $\Mull:\Par_{\operatorname{reg}}(n)\to\Par_{\operatorname{reg}}(n)$ is defined from $D^\la 
\otimes \sgn \cong D^{\Mull(\la)}$. We usually denote $\la^\Mull:=\Mull(\la)$. 
An explicit combinatorial description of $\Mull$ is known from \cite{FK}, see also \cite{BO}. We refer the reader to these papers for details, noting only that $h(\la^\Mull)=r(\la)-h(\la)+\de$, where $r(\la)$ is the size of the $p$-rim hook of $\la$, $\de=1$ if $p\nmid r(\la)$, and $\de=0$ if $p\mid r(\la)$.

\begin{Lemma} \label{LDomJames} {\rm \cite[12.1,\,12.2]{JamesBook}}
If $\la\in\Par(n)$ and $\mu\in\Par_{\operatorname{reg}}(n)$, then $[S^\la:D^\mu]\neq 0$ implies $\mu\unrhd \la$, and $[S^\mu:D^\mu]=1$. 
\end{Lemma}

We will also use {\em Young modules $Y^\la$} which can be defined using the following well-known facts  contained for example in \cite{JamesArcata} and \cite[\S4.6]{Martin}:

\begin{Lemma} \label{LYoung} 
There exist indecomposable $\F \s_n$-modules $\{Y^\la\mid \la\in\Par(n)\}$ such that $M^\la\cong Y^\la\,\oplus\, \bigoplus_{\mu\rhd\la}(Y^\mu)^{\oplus m_{\mu,\la}}$ for some $m_{\mu,\la}\in\Z_{\geq 0}$. Moreover, each $Y^\la$ is self-dual and can be characterized as the unique indecomposable direct summand of $M^\la$ such that  $S^\la\subseteq Y^\la$. Finally, for each $\la$ we have $Y^\la\sim S^\la| S^{\mu^1}|\cdots |S^{\mu^t}$ for some $\mu^1,\dots,\mu^t\rhd \la$.
\end{Lemma}

\begin{Corollary} \label{CSpechtMin}
If $\la\in \Par(n)$ and $\mu\in\Par_{\operatorname{reg}}(n)$ then $[M^\la:D^\mu]\neq 0$ implies $\mu\unrhd\la$, and $[M^\mu:D^\mu]=1$. 
\end{Corollary}
\begin{proof}
We have for example by Lemma~\ref{LYoung} that $M^\la\sim S^\la| S^{\mu^1}|\cdots |S^{\mu^t}$ for some $\mu^1,\dots,\mu^t\rhd \la$, so the result follows from Lemma~\ref{LDomJames}.
\end{proof}

We will need the following results on cohomology of symmetric groups:

\begin{Lemma} \label{LNakaoka}  
We have $H^m(\s_n,\bone)=0$ for $0<m<2p-3$. 
\end{Lemma}
\begin{proof}
This is \cite[Lemma 5.3(b)]{KN} (and also easily follows from \cite[Proposition 7.3]{Nakaoka}). 
\end{proof}

\begin{Lemma} \label{LBKM} 
Let $\la\in\Par(n)$. 
\begin{enumerate}
\item[{\rm (i)}] \cite[Theorem 2.4]{BKM}\, $H^1(\s_n,(S^\la)^*)=0$ unless $p=3$ and $\la=(1^3),\ (n-3,1^3)$. 
\item[{\rm (ii)}] \cite[Theorem 4.1]{BKM}\, $H^2(\s_n,(S^\la)^*)=0$ unless $p=3$ and $\la=(1^3)$, $(2,1)$, $(2^2)$, $(1^6)$,   $(n-3,1^3)$, $(n-3,2,1)$, $(n-6,1^6)$. 
\end{enumerate}
\end{Lemma}

\subsection{Two-row partitions}
We will often deal with modules corresponding to two-row partitions, so we introduce a special notation 
\begin{equation}\label{ETwoRowNot}
M_k:=M^{(n-k,k)},\quad S_k:=S^{(n-k,k)},\quad D_k:=D^{(n-k,k)},\quad Y_k:=Y^{(n-k,k)}.
\end{equation}
Of course, it is assumed that $k\leq n/2$. 
The notation is convenient when $n$ is understood---otherwise we use the full notation $M^{(n-k,k)},S^{(n-k,k)}$, etc. Recall that throughout the paper we are assuming $p=\cha\F >2$, so the partition $(n-k,k)$ is always $p$-regular.  

Note that if $t$ is a tabloid corresponding to the 2-row partition $(n-k,k)$, then $t$ can be fully recovered from its second row, so 
we can view $M_k$ as the permutation module 
on the set $\Om_k$ of $k$-element subsets of $\{1,2,\dots,n\}$. 
For $k,l\leq n/2$, we will use special homomorphisms between permutation modules:
\begin{equation}\label{EEta}
\eta_{k,l}:M_k\to M_l,\ X\mapsto \sum_{Y\in \Om_l,\,Y\, \text{incident to}\, X}Y,
\end{equation}
where $Y$ is incident to $X$ means $Y\subseteq X$ or $X\subseteq Y$.

\begin{Lemma} \label{LWilson} {\rm \cite[Theorem 1]{Wil}} 
If $k\leq l\leq n/2$ then 
$$\dim\Im\,\eta_{k,l}=\dim\Im\,\eta_{l,k}=\sum \left(\textstyle \binom{n}{r}-\binom{n}{r-1}\right),$$
where the sum is over all $r=0,\dots,k$ such that $\binom{l-r}{k-r}$ is not divisible by $p$ (interpreting $\binom{l}{-1}$ as~$0$).
\end{Lemma}

\begin{Lemma} \label{L131218}
Let $k\leq n/2$. 
\begin{enumerate}[\rm(i)]
\item $M_k\sim S_0^*|S_{1}^*|\ldots|S_k^*$,
\item if $p>k$ then $M_k\sim M_{k-1}|S_k^*$.
\end{enumerate}
\end{Lemma}

\begin{proof}
(i) holds by \cite[Example 17.17]{JamesBook} and (ii) holds for example by \cite[Lemmas 3.1,\,3.2]{BKIrr}.  
\end{proof}

\begin{Lemma} \label{LPerpMax}
Let $0\leq j\leq k$ and $Y\sim S_0^*|S_{1}^*|\ldots|S_k^*$. There exists a unique largest submodule $V\subseteq Y$ such that $[V:D_a]=0$ for all $j\leq a\leq k$. Moreover, $V\sim S_0^*|S_1^*|\dots|S_{j-1}^*$ and $Y/V\sim S_j^*|S_{j+1}^*|\dots|S_k^*$. 
\end{Lemma}
\begin{proof}
There clearly exists $V$ such that $V\sim S_0^*|S_1^*|\dots|S_{j-1}^*$ and $Y/V\sim S_j^*|S_{j+1}^*|\dots|S_k^*$. By Lemma~\ref{LDomJames}, for every $r$ we have $[S_r:D_r]=1$, and $[S_r:D_s]=0$ only if $s\leq r$. So $V$ is a submodule of $Y$ such that $[V:D_a]=0$ for all $j\leq a\leq k$. If $W$ is another such submodule and $W\not\subseteq V$, then $V+W\supsetneq V$ is also such a submodule. This yields a non-zero submodule $(V+W)/V\subseteq Y/V\sim S_j^*|S_{j+1}^*|\dots|S_k^*$. Since $\soc S_r^*\cong D_r$ for all $r$ it follows that some $D_a$ with $j\leq a\leq k$ is a composition factor of $(V+W)/V$, hence of $V+W$, which is a contradiction. 
\end{proof}

\section{Invariants}\label{inv}
In this section, we assume that $n=2b$ for an integer $b>1$  and consider the following natural subgroups of $\s_n$, which are stabilizers in
$\s_n$ of a partition of $\{1,2,\ldots,n\}$ into $b$ pairs, respectively into two $b$-subsets:
$$
\W_{2,b}=\s_2\wr \s_b\cong \s_2^{\times b}\rtimes \s_b\quad \text{and} \quad 
\W_{b,2}=\s_b\wr \s_2\cong (\s_b\times \s_b)\rtimes\s_2.
$$
(There will be one exception to these assumptions---in Lemma~\ref{inv3intr} we consider a different subgroup and do not assume that $n$ is even.) 
The main goal of this section is to obtain information about the invariants $V^\W$ for various special $\F\s_n$-modules $V$ and 
$\W=\W_{2,b}$ or $\W_{b,2}$. 

Recall that throughout the paper we are assuming $p=\cha\F >2$.

\subsection{\boldmath Invariants $(S_k^*)^\W$ for $k\leq p$}
We begin with $\dim M_k^\W$, which is easily seen to be the number of $\W$-orbits on $\Om_k$. So an elementary check shows:

\begin{Lemma} \label{LI} 
Let $\W=\W_{2,b}$ or\, $\W_{b,2}$, and $k\leq b$. Then $\dim M_k^\W=\lceil(k+1)/2\rceil$. In particular, 
\begin{align*}
&\dim M_0^\W=\dim M_1^\W=1,\ \dim M_2^\W=\dim M_3^\W=2,\\  &\dim M_4^\W=\dim M_5^\W=3,\ \dim M_6^\W=4.
\end{align*}
\end{Lemma}


\begin{Lemma} \label{LStabStructure}
Let $\W=\W_{2,b}$ or\, $\W_{b,2}$, and $0\leq k\leq b$. 
Let $X\in\Om_k$ and consider the point stabilizer $\operatorname{Stab}_\W(X)$ for the natural action of $\W<\s_n$ on $\Om_k$. 
\begin{enumerate}
\item[{\rm (i)}] Suppose $\W = \W_{2,b}$. Then there exist non-negative integers $c,d$ such that $2c+d=k$ and 
$\operatorname{Stab}_\W(X) \cong \s_d \times \W_{2,c} \times \W_{2,b-c-d}$.
\item[{\rm (ii)}] Suppose $\W = \W_{b,2}$. Then there exist  non-negative integers $c,d$ such that $c+d=k$, and 
$$
\operatorname{Stab}_\W(X) \cong
\left\{
\begin{array}{ll}
\s_c \times \s_{b-c} \times \s_{d}  \times S_{b-d} &\hbox{if $c\neq d$,}\\
(\s_c \times \s_{b-c}) \wr \s_2 &\hbox{if $c=d$.}
\end{array}
\right.
$$
\end{enumerate}
\end{Lemma}
\begin{proof}
We explain how to choose $c$ and $d$ so that the answer is as predicted. 

(i) The group $\W = \W_{2,b}$ stabilizes the partition $\{1,2,\dots,n\}= \bigsqcup_{r=1}^b\{2r-1,2r\}$. The set $X$ contains exactly $c$ of the pairs $\{2r-1,2r\}$ and intersects exactly $d$ of such pairs at one point.


(ii) The group $\W = \W_{b,2}$ stabilizes the partition  $\{1,2,\dots,n\}= \{1,\dots,b\}\sqcup\{b+1,\dots,n\}$. 
Then $c=|X\cap\{1,\dots,b\}|$ and $d=|X\cap\{b+1,\dots,n\}|$.  
\end{proof}

\begin{Corollary} \label{CPtStabW}
Let $0\leq k\leq b$, $\W=\W_{2,b}$ or\, $\W_{b,2}$, and $K$ is a point stabilizer in $\W$ for the action of $\W<\s_n$ on $\Om_k$. Then:
\begin{enumerate}
\item[{\rm (i)}] $O^p(K)=K$; in particular, $H^1(K,\bone)=0$.
\item[{\rm (ii)}] $H^2(K,\bone)=0$. 
\end{enumerate} 
\end{Corollary}
\begin{proof}
We know the structure of $K$ from Lemma~\ref{LStabStructure}, and this immediately implies (i) (since $p>2$).  For (ii), 
the K\"unneth formula allows us to reduce to proving that 
$$
H^2(\s_r,\bone)=0,\ H^2(\W_{2,r},\bone)=0,\ H^2((\s_r\times \s_t)\wr\s_2,\bone)=0
$$
for various $r,t$. The first equality follows from Lemma~\ref{LNakaoka}. The second equality follows using Lemma~\ref{LHS} and then Lemma~\ref{LNakaoka}. The third equality follows from Lemmas~\ref{LHS},~\ref{LNakaoka} and the K\"unneth formula. 
\end{proof}

\begin{Lemma} \label{LCohW}
Let $\W=\W_{2,b}$ or\, $\W_{b,2}$, and $0\leq k\leq b$. Then $H^1(\W,M_k)=H^2(\W,M_k)=0$. 
\end{Lemma}
\begin{proof}
Let $\O_1,\dots,\O_r$ be the $\W$-orbits on $\Om_k$ with point stabilizers $K_1,\dots,K_r$, respectively. 
By Mackey's theorem, $M_k\da_{\W}\cong \bigoplus_{t=1}^r\bone_{K_t}\ua^\W$. So we have $H^m(\W,M_k)\cong \bigoplus_{t=1}^r H^m(\W,\bone_{K_t}\ua^\W)$. For each $t$, by Eckmann-Shapiro's Lemma,  
$H^m(\W,\bone_{K_t}\ua^\W)\cong H^m(K_t,\bone_{K_t})$ which is zero for $m=1,2$ by Corollary~\ref{CPtStabW}. 
\end{proof}

We can now compute the invariants $(S_k^*)^\W$ for $p$ large enough: 

\begin{Corollary} \label{CWInvLargep}
Let $\W=\W_{2,b}$ or\, $\W_{b,2}$, and $0\leq k\leq b$. If  $p>k$ then 
$$\dim(S_k^*)^\W=
\left\{
\begin{array}{ll}
0 &\hbox{if $k$ is odd,}\\
1 &\hbox{if $k$ is even.}
\end{array}
\right.
$$
\end{Corollary}
\begin{proof}
For $k=0$ we have $S_k\cong\bone$ so the result is clear. Let $k>0$. 
By Lemma~\ref{L131218}(ii), there is an exact sequence 
$0\to M_{k-1}\to M_k\to S_k^*\to 0$, which yields the exact sequence
$$
0\to M_{k-1}^\W\to M_k^\W\to (S_k^*)^\W\to H^1(\W,M_{k-1}).
$$
By Lemma~\ref{LCohW}, $H^1(\W,M_{k-1})=0$, so $\dim(S_k^*)^\W=\dim M_k^\W-\dim M_{k-1}^\W$, and the result follows from Lemma~\ref{LI}. 
\end{proof}

To deal with the case $k=p$ we need a slightly more elaborate cohomological argument.

\begin{Lemma} \label{L131218_5P}
Let $b\geq p$, and $\W=\W_{2,b}$ or\, $\W_{b,2}$. Then 
$(S_p^*)^\W=0$.  
\end{Lemma}
\begin{proof}
By Lemma~\ref{LWilson}, we have $\dim\Im\, \eta_{p-1,p}=\dim M_{p-1}-1$. Let $v:=\sum_{X\in\Om_{p-1}}X\in M_{p-1}$. 
As $v$ spans the trivial submodule $\bone \subseteq M_{p-1}$ and $\eta_{p-1,p}(v)=0$, 
we have short exact sequences
\begin{align}
\label{SES1}
0\to \bone \to M_{p-1}\to \bar M_{p-1}\to 0,
\\
\label{SES2}
0\to \bar M_{p-1}\to M_p\stackrel{\si}{\to} Y\to 0.
\end{align}
The sequence (\ref{SES1}) yields the exact sequence 
$
0\to \bone^{\s_n} \to M_{p-1}^{\s_n}\to \bar M_{p-1}^{\s_n}\to H^1(\s_n,\bone),
$
and since $\bone^{\s_n}\cong M_{p-1}^{\s_n}\cong\F$ and $H^1(\s_n,\bone)=0$,
we deduce that $\bar M_{p-1}^{\s_n}=0$. 
Moreover, $\dim M_p^{\s_n}=1$, so from (\ref{SES2}), we have $Y^{\s_p}\neq 0$, i.e. there is a submodule $\bone\cong I\subseteq  Y$. Further, $\dim Y=\dim M_p-\dim M_{p-1}+1=\dim S_p^*+1$ and $D_p$ is not a composition factor of $\bar M_p$, so using Lemma~\ref{LPerpMax} we deduce that $\si^{-1}(I)\subseteq M_p$ must be the unique maximal submodule of $M_p$ not having $D_p$ as a composition factor and $M_p/\si^{-1}(I)\cong S_p^*$. So we have a short exact sequence
\begin{equation}
\label{SES3}
0\to I \to Y \to S_p^*\to 0.
\end{equation}

The sequences (\ref{SES1})-(\ref{SES3}) yield exact sequences 
\begin{align*}
&0\to \F \to M_{p-1}^\W\to \bar M_{p-1}^\W\to H^1(\W,\bone)\to H^1(\W,M_{p-1})\to H^1(\W,\bar M_{p-1})\to H^2(\W,\bone),
\\
&0\to \bar M_{p-1}^\W\to M_p^\W\to Y^\W\to H^1(\W,\bar M_{p-1}),
\\
&0\to \F \to Y^\W \to (S_p^*)^\W\to H^1(\W,\bone).
\end{align*}
Moreover, $H^1(\W,M_{p-1})=H^1(\W,\bone)=H^2(\W,\bone)=0$ by Lemma~\ref{LCohW}. The equality $(S_p^*)^\W=0$ now follows.  
\end{proof}

\subsection{\boldmath Computing $H^m(\W,S_k^*)$ for $k,m\in\{1,2\}$} 
We now establish some more results on cohomology of wreath products.

\begin{Lemma} \label{LHG2} 
Let $b\geq 5$ and $k=0,1,2$. Then $H^1(\W_{b,2},S_k^*)=H^2(\W_{b,2},S_k^*)=0$. 
\end{Lemma}
\begin{proof}
For $k=0$ we have $S_0^*\cong M_0^*$, so we can use Lemma~\ref{LCohW}. Let $k\neq 0$. 
By Lemma~\ref{LHS}(ii), we have $H^m(\W_{b,2},S_k^*)\cong H^m(\s_b\times \s_b,S_k^*)^{\s_2}$. But by \cite[3.1, 5.5]{JP}, as an $\F(\s_b\times \s_b)$-module, $S_k^*$ has a filtration with subquotients of the form $(S^{(b-i,i)})^*\boxtimes (S^{(b-j,j)})^*$. Now,  $H^m(\s_b\times \s_b,S_i^*\boxtimes S_j^*)=0$ for $m=1,2$ thanks to the K\"unneth formula and Lemma~\ref{LBKM}. 
\end{proof}

To deal with the cohomology $H^m(\W_{2,b},S_k^*)$ we need the following lemma (where the action of the group $\s_b$ on the $\s_2^{\times b}$-invariants comes from the isomorphism $\s_b\cong\W_{2,b}/\s_2^{\times b}$):

\begin{Lemma} \label{LInvComp} 
Let $b\geq 4$. 
Then, as $\F\s_b$-modules, $(S_1^*)^{\s_2^{\times b}}\cong (S^{(b-1,1)})^*$ and 
$(S_2^*)^{\s_2^{\times b}}\cong M^{(b-2,2)}$. 
\end{Lemma}
\begin{proof}
The first isomorphism comes from an easy explicit calculation using $S_1^*\cong M_1/\bone$. 

As $\s_2^{\times b}$ is a $p'$-group and $S_2^*$ is reduction modulo $p$ of $S^{(n-2,2)}_\C$, we have 
\begin{equation}\label{EBasisInvar}
\dim (S_2^*)^{\s_2^{\times b}}=
\dim (S^{(n-2,2)}_\C)^{\s_2^{\times b}}=
b(b-1)/2=\dim M^{(b-2,2)}
\end{equation}
using Littlewood-Richardson rule for the second equality. 

Now, let $\Om_2(n)$ be the set of all $2$-element subsets of $\{1,\dots,n\}$, so that $M_2$ is the permutation module on $\Om_2(n)$. So we have the stadard basis 
$\{A\mid A\in \Om_2(n)\}$ of $M_2$ with the action $gv_A=v_{gA}$ for all $g\in\s_n$ and $A\in\Om_2(n)$. 
By definition, $S_2\subseteq M_2$ is spanned by the polytabloids 
$
v_{\{i,j\}}+v_{\{k,l\}}-v_{\{i,l\}}-v_{\{k,j\}}
$ for distinct $i,j,k,l\in\{1,\dots,n\}$. We have $S_2^*\cong M_2/S_2^\perp$, and it is easy to see that 
$
S_2^\perp:=\spa\{
w_1,\dots,w_n\}
$
where we have set 
$w_i:=\sum_{j\not=i}v_{\{i,j\}}$. 
For $v\in M_2$, denote 
$$\bar v:=v+S_2^\perp\in M_2/S_2^\perp=S_2^*.$$ 

For $1\leq i\leq n-3$, the equality $\bar w_i=0$ implies 
\begin{equation}\label{EV1}
\v_{\{i,n\}}=-\sum_{j\in\{1,\dots,n-1\}\setminus\{i\}}\v_{\{i,j\}}\qquad(1\leq i\leq n-3).
\end{equation}
Then the equality $\bar w_{n-2}+\bar w_{n-1}-\bar w_{n}=0$ and (\ref{EV1}) imply
\begin{equation}\label{EV2}
\v_{\{n-2,n-1\}}=-\sum_{A\in\Om_2(n-2)}\v_A-\sum_{i=1}^{n-3}\v_{\{i,n-1\}}.
\end{equation}
The equation $\bar w_{n-2}=0$ and (\ref{EV2}) imply 
\begin{equation}\label{EV3}
\v_{\{n-2,n\}}=\sum_{A\in\Om_2(n-3)}\v_A+\sum_{i=1}^{n-3}\v_{\{i,n-1\}}.
\end{equation}
The equation $\bar w_{n}=0$ and (\ref{EV1}),(\ref{EV3}) imply 
\begin{equation}\label{EV4}
\v_{\{n-1,n\}}=\sum_{A\in\Om_2(n-2)}\v_A.
\end{equation}
By (\ref{EV1})--(\ref{EV4}), the $n(n-3)/2$ vectors 
\begin{equation}\label{ES2Basis}
\{\v_A\mid A\in\Om_2(n-2)\}\cup\{\v_{\{i,n-1\}}\mid 1\leq i\leq n-3\}
\end{equation}
span $M_2/S_2^\perp$. Since $\dim S_2^*=n(n-3)/2$, we deduce that (\ref{ES2Basis}) is a basis of $M_2/S_2^\perp=S_2^*$. 
So, 
\begin{equation}\label{EInvBasis}
\{\v_{\{2i-1,2i\}}\mid 1\leq i<b\}\cup \{w_{i,j}\mid 1\leq i<j<b\}
\end{equation}
are $b(b-1)/2$ linearly independent elements of $(S_2^*)^{\s_2^{\times b}}$,
where we have set 
$$w_{i,j}:=\v_{\{2i-1,2j-1\}}+\v_{\{2i-1,2j\}}+\v_{\{2i,2j-1\}}+\v_{\{2i,2j\}}.$$ 
Taking into account  (\ref{EBasisInvar}), we deduce that (\ref{EInvBasis}) is a basis of $(S_2^*)^{\s_2^{\times b}}$. 
Moreover, $w_{1,2}$ is invariant with respect to the subgroup $\s_{2,b-2}<\s_b$. 
By the Frobenius reciprocity, there is an $\F \s_b$-module homomorphism $\phi:M_2\cong \bone\ua_{\s_{2,b-2}}^{\s_n}\to (S_2^*)^{\s_2^{\times b}}$ such that $w_{1,2}\in\Im\phi$. Since $\dim M_2=\dim (S_2^*)^{\s_2^{\times b}}$ by (\ref{EBasisInvar}), it remains to prove that $w_{1,2}$ generates $(S_2^*)^{\s_2^{\times b}}$ as an $\F\s_b$-module. Let $W$ be the submodule of  $(S_2^*)^{\s_2^{\times b}}$ generated by $w_{1,2}$. For $\si\in \s_{b-1}<\s_b$ we have $\si\cdot w_{i,j}=w_{\si(i),\si(j)}$ and $\si\cdot \v_{\{2i-1,2i\}}=\v_{\{2\si(i)-1,2\si(i)\}}$, so all $w_{i,j}$ with $1\leq i<j<b$ belong to $W$, and to complete the proof it suffices to prove that $\v_{\{1,2\}}\in W$. Take now $\si$ to be the transposition $(1,b)\in\s_b$. Then 
\begin{equation}\label{E010725}
\si\cdot w_{1,2}=\v_{\{1,n-1\}}+\v_{\{1,n\}}+\v_{\{2,n-1\}}+\v_{\{2,n\}}
\end{equation}
Using (\ref{EV1}) to write $\v_{\{1,n\}}=-\sum_{j=2}^{n-1}\v_{\{1,j\}}$, $\v_{\{2,n\}}=-\v_{\{1,2\}}-\sum_{j=3}^{n-1}\v_{\{2,j\}}$ and simplifying, we see that (\ref{E010725}) equals 
$
-2\v_{\{1,2\}}-\sum_{i=2}^{b-1}w_{1,i},
$
which now implies that $\v_{\{1,2\}}$ belongs to $W$. 
\end{proof}

\begin{Lemma} \label{LHG1} 
Let $b\geq 4$ and $k=0,1,2$. Then $H^1(\W_{2,b},S_k^*)=H^2(\W_{2,b},S_k^*)=0$. 
\end{Lemma}
\begin{proof}
As $S_0^*\cong M_0^*$, for the case $k=0$ we can use Lemma~\ref{LCohW}. Let $k\neq 0$. 
By Lemmas~\ref{LHS}(i) and \ref{LInvComp}, we have 
$$H^m(\W_{2,b},S_1^*)\cong H^m(\s_b,(S_1^*)^{\s_2^{\times b}})
\cong 
H^m(\s_b,S_1^*)=0,
$$
where the last equality follows for example from Lemma~\ref{LBKM}. On the other hand, by Lemmas~\ref{LHS}(i),  \ref{LInvComp} and Eckmann-Shapiro's lemma, we have
$$H^m(\W_{2,b},S_2^*)\cong H^m(\s_b,(S_2^*)^{\s_2^{\times b}})
\cong 
H^m(\s_b,M_2)\cong H^m(\s_{2,b-2},\bone)=0,
$$
the last equality following by the K\"unneth formula and Lemma~\ref{LNakaoka}. 
\end{proof}

\begin{Lemma} \label{LH1S3} 
Let $b\geq 5$, and $\W=\W_{2,b}$ or\, $\W_{b,2}$. Then $H^1(\W,S_3^*)=0$. 
\end{Lemma}
\begin{proof}
By Lemma~\ref{L131218}, there is an exact sequence
$0\to X\to M_3\to S_3^*$ with $X\sim S_0^*|S_1^*|S_2^*$. This yields the exact sequence $H^1(\W,M_3)\to H^1(\W,S_3^*)\to H^2(\W,X)$. But $H^1(\W,M_3)=0$ by Lemma~\ref{LCohW}, and $H^2(\W,X)=0$ because $H^2(\W,S_0^*)=H^2(\W,S_1^*)=H^2(\W,S_2^*)=0$ by Lemmas~\ref{LHG2} and \ref{LHG1}. 
\end{proof}

Lemmas~\ref{LHG1}, \ref{LHG2} and \ref{LH1S3} yield:

\begin{Corollary} \label{CCoh}
Let $b\geq 5$, $0\leq k\leq 3$, and $\W=\W_{2,b}$ or\, $\W_{b,2}$.  Then $H^1(\W,S_k^*)=0$.
\end{Corollary}

\subsection{\boldmath More on the invariants of $\W$}

\begin{Lemma} \label{LLargePS*} 
Let $0\leq k\leq 4$, $b\geq 5$, and $\W=\W_{2,b}$ or\, $\W_{b,2}$. Then $(S_k^*)^\W$ is given by
$$
\dim(S_0^*)^\W=1,\ \dim(S_1^*)^\W=0,\ \dim (S_2^*)^\W=1,\ \dim (S_3^*)^\W=0,\ \dim (S_4^*)^\W=1.
$$
\end{Lemma}
\begin{proof}
In view of Corollary~\ref{CWInvLargep} and Lemma~\ref{L131218_5P}, we only have to deal with the case where $p=3$ and $k=4$. In this case, the exact sequence $0\to Z\to M_4\to S_4^*\to 0$ with $Z\sim S_0^*|S_1^*|S_2^*|S_3^*$ yields the exact sequence $0\to Z^\W\to M_4^\W\to (S_4^*)^\W\to H^1(\W,Z)$. But $H^1(\W,Z)=0$ thanks to Corollary~\ref{CCoh}, and so 
$$
\dim (S_4^*)^\W= \dim M_4^\W-\dim (S_0^*)^\W-\dim (S_1^*)^\W-\dim (S_2^*)^\W-\dim (S_3^*)^\W=1
$$
using Lemmas~\ref{LI},\,\ref{LInvSum} and Corollary~\ref{CCoh} again. 
\end{proof}

From Corollary~\ref{CWInvLargep} and Lemma~\ref{L131218_5P}, we also get:

\begin{Lemma} \label{LP5}
Let $p>3$, $b\geq 5$, and $\W=\W_{2,b}$ or\, $\W_{b,2}$. Then $(S_5^*)^\W=0$.
\end{Lemma}

For $p=3$, it is probably still true that $(S_5^*)^\W=0$, but we cannot prove it, and so will have to do with a little less, see Lemma~\ref{LUInv} below. We need some preliminary work. Recall the homomorphisms $\eta_{k,l}$ from (\ref{EEta}).

\begin{Lemma} \label{LEtas} 
Let $p=3$. We have $\dim\Im\,\eta_{3,5}=n(n-1)(n-5)/6+1=\dim S_3^*+\dim S_0^*$ and the sequence 
$
M_3\stackrel{\eta_{3,5}}{\longrightarrow}M_5\stackrel{\eta_{5,6}}{\longrightarrow}M_6
$
is exact.
\end{Lemma}
\begin{proof}
The equalities $\dim\Im\,\eta_{3,5}=n(n-1)(n-5)/6+1=\dim S_3^*+\dim S_0^*$ follow from Lemma~\ref{LWilson} and the Hook Formula. By Lemma~\ref{LWilson}, we also have 
$\dim\Im\,\eta_{3,5}=\dim M_5-\dim\Im\,\eta_{3,5}.
$ 
So it suffices to prove that $\eta_{5,6}\circ \eta_{3,5}=0$. This is an explicit computation: for $\{a,b,c\}\in \Om_3$ we have 
\begin{align*}
\eta_{5,6}(\eta_{3,5}(\{a,b,c\}))&=\eta_{5,6}\big(\sum_{d,e}\{a,b,c,d,e\}\big)
=\sum_f\sum_{d,e}\{a,b,c,d,e,f\}
\\&=3\sum_{d,e,f}\{a,b,c,d,e,f\}=0,
\end{align*}
where the sum $\sum_{d,e}$ is over all $1\leq d\neq e\leq n$ such that $d,e\neq a,b,c$, the sum $\sum_f$ is over all $f=1,\dots,n$ such that $f\neq a,b,c,d,e$, and the sum $\sum_{d,e,f}$ is over all distinct $d,e,f$ satisfying $1\leq d,e,f\leq n$. 
\end{proof}

We now use Lemma~\ref{LPerpMax} to see that (for any $p\geq 3$) there exist: 
\begin{enumerate}
\item[$\bullet$] the unique largest 
submodule $Z_6\subset M_6$ with $[Z_6:D_6]=0$; moreover, 
$$
M_6/Z_6\cong S_6^*\quad\text{and}\quad Z_6\sim S_0^*|S_1^*|S_2^*|S_3^*|S_4^*|S_5^*.
$$
\item[$\bullet$] the unique largest 
submodule $V\subset Z_6$ with $[V:D_5]=[V:D_4]=0$; moreover, 
$$
Z_6/V\cong S_4^*|S_5^*\quad\text{and}\quad V\sim S_0^*|S_1^*|S_2^*|S_3^*.
$$
\item[$\bullet$] the unique largest 
submodule $Y\subset M_5$ with $[Y:D_5]=[Y:D_4]=0$; moreover, 
$$
M_5/Y\sim S_4^*|S_5^*\quad\text{and}\quad Y\sim S_0^*|S_1^*|S_2^*|S_3^*.
$$
\end{enumerate}

With this notation, we have:

\begin{Lemma} \label{LUIso}
Let $p=3$. Then $Z_6/V\cong M_5/Y$. 
\end{Lemma}
\begin{proof}
As $[M_5:D_6]=0$, we have $[\Im\, \eta_{5,6}:D_6]=0$, so $\Im\, \eta_{5,6}\subseteq Z_6$, and from now on we consider $\eta_{5,6}$ as a map $\eta_{5,6}:M_5\to Z_6$. Composing with the natural projection ${\mathtt p}:Z_6\to Z_6/V$ we get the map $f:= {\mathtt p}\circ \eta_{5,6}:M_5\to Z_6/V$. By Lemma~\ref{LEtas}, $\Ker\,\eta_{5,6}=\Im\,\eta_{3,5}$, so $[\Ker\,\eta_{5,6}:D_5]=[\Ker\,\eta_{5,6}:D_4]=0$. Moreover, $[\Ker\,{\mathtt p}:D_5]=[\Ker\,{\mathtt p}:D_4]=0$. Hence $[\Ker\, f:D_5]=[\Ker\, f:D_4]=0$, hence $\Ker\, f\subseteq Y$. By dimensions, $\Ker\,f=Y$ and $f$ is surjective. This implies the claim. 
\end{proof}

\begin{Lemma} \label{LYW} 
For $\W=\W_{2,b}$ or\, $\W_{b,2}$, we have $H^1(\W,Y)=H^1(\W,V)=0$ and\, $\dim Y^\W=\dim V^\W=2$. 
\end{Lemma}
\begin{proof}
We have $Y\sim S_0^*|S_1^*|S_2^*|S_3^*$ and $V\sim S_0^*|S_1^*|S_2^*|S_3^*$. 
By Corollary~\ref{CCoh}, we also have $$H^1(\W,S_0^*)=H^1(\W,S_1^*)=H^1(\W,S_2^*)=H^1(\W,S_3^*)=0.$$ So $H^1(\W,Y)=H^1(\W,V)=0$, and 
$$\dim Y^\W=\dim V^\W=\dim (S_0^*)^\W+\dim (S_1^*)^\W+\dim (S_2^*)^\W+\dim (S_3^*)^\W=2,$$ using Lemmas~\ref{LInvSum} and~\ref{LLargePS*}. 
\end{proof}

\begin{Lemma} \label{LUInv}
For $\W=\W_{2,b}$ or\, $\W_{b,2}$, we have $\dim(Z_6/V)^\W=1$. 
\end{Lemma}
\begin{proof}
The exact sequence $0\to S_4^*\to Z_6/V\to S_5^*$ 
yields the exact sequence $0\to (S_4^*)^\W\to (Z_6/V)^\W\to (S_5^*)^\W$. For $p>3$, by Lemma~\ref{LLargePS*}, we have $\dim (S_4^*)^\W=1$, and by Lemma~\ref{LP5}, we have $(S_5^*)^\W=0$, which implies $\dim(Z_6/V)^\W=1$. 

We now assume that $p=3$. 
The exact sequence 
$0\to Y\to M_5\to M_5/Y\to 0$ yields the exact sequence
$$0\to Y^\W\to M_5^\W\to (M_5/Y)^\W\to H^1(\W,Y).$$ 
We have $H^1(\W,Y)=0$ by Lemma~\ref{LYW}. Hence 
$\dim (M_5/Y)^\W=\dim M_5^\W-\dim Y^\W=3-2=1$ using  Lemma~\ref{LI} and \ref{LYW}. It remains to use Lemma~\ref{LUIso}. 
\end{proof}


\begin{Lemma} \label{LXW} 
For $\W=\W_{2,b}$ or\, $\W_{b,2}$, we have $\dim Z_6^\W=3$. 
\end{Lemma}
\begin{proof}
From the exact sequence $0\to V\to Z_6\to Z_6/V\to 0$ we get the exact sequence $$0\to V^\W\to Z_6^\W\to (Z_6/V)^\W\to H^1(\W,V).$$ 
We have $H^1(\W,V)=0$, $\dim V^\W=2$ by Lemma~\ref{LYW}, and $\dim(Z_6/V)^\W=1$ by Lemma~\ref{LUInv}. The result follows. 
\end{proof}

\subsection{Some consequences}
By Lemma~\ref{L131218}(i), for every $k$, the dual Specht module $S_k^*$ is a quotient of the permutation module $M_k$. Using Lemma~\ref{LPerpMax}, one can see that in fact there is a unique submodule $Z_k\subseteq M_k$ with $M_k/Z_k\cong S_k^*$. So we have a natural projection 
\begin{equation}\label{ESik}
\si_k:M_k\to S_k^*.
\end{equation}
In this subsection we will show that for some wreath product and parabolic subgroups $H<\s_n$ and special values of $k$, there exists $\phi\in\Hom_{\s_n}(\bone\ua_H^{\s_n},M_k)$ such that the homomorphism $\si_k\circ \phi:\bone\ua_H^{\s_n}\to S_k^*$ is non-zero. 

The approach to the proof is as follows. We need to show that the map 
$$\si_{k,*}:\Hom_{\s_n}(\bone\ua_H^{\s_n},M_k)\to \Hom_{\s_n}(\bone\ua_H^{\s_n},S_k^*),\ \phi\mapsto \si_k\circ\phi
$$ 
is non-zero. The exact sequence 
$$0\to Z_k\to M_k\stackrel{\si_k}{\longrightarrow} S_k^*\to 0$$ yields the exact sequence
$$
0\to \Hom_{\s_n}(\bone\ua_H^{\s_n},Z_k)\to \Hom_{\s_n}(\bone\ua_H^{\s_n},M_k)\stackrel{\si_{k,*}}{\longrightarrow} \Hom_{\s_n}(\bone\ua_H^{\s_n},S_k^*),
$$
which, using the Frobenius reciprocity, can be identified with the exact sequence 
$$
0\to Z_k^H\to M_k^H\stackrel{\bar\si_k}{\longrightarrow} (S_k^*)^H,
$$
where $\bar \si_k$ is the restriction of $\si_k$ to $M_k^H$. To prove that $\bar\si_k\neq0$ it now suffices to see that $\dim Z_k^H<\dim M_k^H$. We have proved:

\begin{Lemma} \label{LGenInv}
Let $H\leq \s_n$, $k\leq n/2$, and $\si_k:M_k\to S_k^*$ be the natural projection with kernel $Z_k$. If $\dim Z_k^H<\dim M_k^H$ then there exists $\phi\in\Hom_{\s_n}(\bone\ua_H^{\s_n},M_k)$ such that $\si_k\circ \phi\neq 0$. 
\end{Lemma}

We now apply this to three special situations, which will be of importance on this paper.

\begin{Lemma}\label{inv2wreath}
Let $b\geq 5$, and $\W=\W_{2,b}$ or\, $\W_{b,2}$. 
Then there exists a homomorphism $\phi\in\Hom_{\s_n}(\bone\ua_\W^{\s_n},M_2)$ such that $\si_2\circ \phi\neq 0$.
\end{Lemma}

\begin{proof}
We have $Z_2\sim S_0^*|S_1^*$. By Corollary~\ref{CCoh}, we  have $H^1(\W,S_0^*)=H^1(\W,S_1^*)=0.$ So 
$\dim Z_2^\W=\dim (S_0^*)^\W+\dim (S_1^*)^\W=1$ using Lemmas~\ref{LInvSum} and~\ref{LLargePS*}. 
On the other hand, 
$\dim M_2^\W=2$
by Lemma~\ref{LI}. An application of Lemma~\ref{LGenInv} completes the proof.
\end{proof}

\begin{Proposition}\label{inv6}
Let $b\geq 6$, and $\W=\W_{2,b}$ or\, $\W_{b,2}$. 
Then there exists a homomorphism $\phi\in\Hom_{\s_n}(\bone\ua_\W^{\s_n},M_6)$ such that $\si_6\circ \phi\neq 0$.
\end{Proposition}

\begin{proof}
By Lemmas~\ref{LXW} and \ref{LI}, we have 
$\dim Z_6^\W=3$ and $\dim M_6^\W=4$. An application of Lemma~\ref{LGenInv} completes the proof.
\end{proof}

We complete this subsection with two results on subgroups $H$ of different kind from $\W$. In particular, in this lemma we do not assume that $n$ is even. 

\begin{Lemma}\label{inv3intr}
Let $H=\s_{n-m,m}$ for $3\leq m\leq n/2$. Then there exists a homomorphism $\phi\in\Hom_{\s_n}(\bone\ua_H^{\s_n},M_3)$ such that $\si_3\circ \phi\neq 0$.
\end{Lemma}

\begin{proof}
It is easy to see that the number of $H$-orbits on $\Om_3$ is $4$, so $\dim M_3^{H}=4$. On the other hand, $Z_3 \sim S_0^*|S_1^*|S_2^*$. Moreover, $\dim (S_k^*)^H\leq1$ for $k=0,1,2$,  thanks to \cite[Lemma 2.12]{KMTAlt}. So $\dim Z_3^H\leq 3$. 
An application of Lemma~\ref{LGenInv} completes the proof.
\end{proof}

\begin{Lemma} \label{L080925} 
Let $n=ab$ for $a,b\geq 3$ and $H=\W_{a,b}$. Then there exists a homomorphism $\phi\in\Hom_{\s_n}(\bone\ua_H^{\s_n},M_3)$ such that $\si_3\circ \phi\neq 0$.
\end{Lemma}
\begin{proof}
It is easy to see that the number of $H$-orbits on $\Om_3$ is $3$,  so $\dim M_3^{H}=3$. On the other hand, $Z_3 \sim S_0^*|S_1^*|S_2^*$. Moreover, $\dim (S_0^*)^H=1\geq \dim (S_2^*)^H$ and  $\dim (S_1^*)^H=0$ by \cite[Theorem 2.13]{KMTAlt}. So $\dim Z_3^H\leq 2$. 
An application of Lemma~\ref{LGenInv} completes the proof.
\end{proof}

\subsection{\boldmath On invariants $(D^\la)^\W$} We need a little more information on $W$-invariants. Throughout the subsection,  we will use the following generalization of the notation (\ref{ETwoRowNot}). Given a partition $\mu=(\mu_1,\dots,\mu_r)\in \Par(m)$ with $\mu_1\leq n-m$, we denote 
$$D_{\mu_1,\dots,\mu_r}:=D^{(n-m,\mu_1,\dots,\mu_r)},\ S_{\mu_1,\dots,\mu_r}:=S^{(n-m,\mu_1,\dots,\mu_r)},\ M_{\mu_1,\dots,\mu_r}:=M^{(n-m,\mu_1,\dots,\mu_r)}.
$$  

In addition to the dimensions of the invariant spaces of permutation modules from Lemma~\ref{LI}, we need to record the dimensions of the invariant spaces $M_{\mu_1,\dots,\mu_r}^{\W_{2,b}}$ for some other special $\mu$'s.

\begin{Lemma} \label{LSpecialInv} 
Let $b\geq 5$ and $\W=\W_{2,b}$. Then
\begin{align*}
&
\dim M_{2,1}^\W=3,\ \dim M_{1^3}^\W=4,\ \dim M_{3,1}^\W=4,\ 
\dim M_{2^2}^\W=6,\ \dim M_{2,1^2}^\W=7,
\\
&\dim M_{4,1}^\W=5,\ 
\dim M_{3,2}^\W=7,\ 
\dim M_{3,1^2}^\W=9,\ 
\dim M_{2^2,1}^\W=12,\ 
\\
&\dim M_{5,1}^\W=6-\de_{b,5},\ \dim M_{4,2}^\W=10-\de_{b,5},\ 
\dim M_{4,1^2}^\W=12-\de_{b,5}
\\
&\dim M_{3^2}^\W=10-\de_{b,5},
\  
\dim M_{3,2,1}^\W=17-\de_{b,5},\ 
\ 
\dim M_{2^3}^\W=24-\de_{b,5}.
\end{align*}
\end{Lemma}
\begin{proof}
We have that $\dim M_{\mu_1,\dots,\mu_r}^\W$ is equal to the number of the $\W$-orbits on the $(n-|\mu|,\mu_1,\dots,\mu_r)$-tabloids. It is now an elementary check that the numbers of the $\W$-orbits are as recorded. 
\end{proof}

\begin{Lemma} \label{LSpecialDirectSumDec} 
Let $p\mid n$.
\begin{enumerate}
\item[{\rm (i)}] If $p=3$ and $n\geq 12$ then $D_{4,2}$ is a direct summand of $M_{4,2}$ and
\begin{equation}
\label{E040725}
D_{4,2}\oplus M_{5,1}\oplus M_{3,2}\oplus M_4\cong M_{4,2}\oplus M_5\oplus M_{3,1}.
\end{equation}
\item[{\rm (ii)}] If $p>3$ and $n\geq 10$ then $D_{2^3}$ is a direct summand of $M_{2^3}$ and
\begin{align*}
&D_{2^3}\oplus M_{3,2,1}^{\oplus 2}\oplus M_{4,2}\oplus M_{2^2,1}\oplus M_{4,1}\oplus M_{2^2}\oplus M_{3,1}\oplus M_{1^3}\oplus M_3
\\
\cong\,&M_{2^3}\oplus M_{3^2}\oplus M_{4,1^2}\oplus M_{3,1^2}\oplus M_{3,2}\oplus M_{2,1^2}\oplus M_4\oplus M_{2,1}^{\oplus 2}.
\end{align*}
\end{enumerate}
\end{Lemma}
\begin{proof}
(i) Note that no hook length in the first four columns of $(n-6,4,2)$ is divisible by $3$. So $S_{4,2}=D_{4,2}$ is irreducible by \cite[Theorem 4.12]{JM}. 
Further $D_{4,2}\cong S_{4,2}\subseteq M_{4,2}$ and $[M_{4,2}:D_{4,2}]=1$. Since $S_{4,2}\subseteq Y_{4,2}$ and $Y_{4,2}$ is indecomposable and self-dual, it follows that $D_{4,2}\cong S_{4,2}\cong Y_{4,2}$.

Denote the left hand side of (\ref{E040725}) by $L$, and the right hand side of (\ref{E040725}) by $R$. 

By the determinantal formula \cite[Theorem 2.3.15]{JK}, we have in the Grothendieck group
$$
[S_{4,2}]+[M_{5,1}]+[M_{3,2}]+[M_4]=[M_{4,2}]+[M_5]+[M_{3,1}],
$$ 
and since $S_{4,2}\cong D_{4,2}$, 
we have 
\begin{equation}\label{E040725_2}
[L:D^\be]=[R:D^\be]\qquad(\text{for all}\ \be\in\Par_{\operatorname{reg}}(n)). 
\end{equation}

As $S_{4,2}\cong Y_{4,2}$, by Lemma~\ref{LYoung}, we have that both $L$ and $R$ are direct sums of Young modules, and it remains to prove the equality of the multiplicities $(L:Y^\al)=(R:Y^\al)$ for all $\al\in\Par(n)$. 
Let $\la=(n-6,4,2)$.  By Lemma~\ref{LYoung}, $(L:Y^\la)=(R:Y^\la)=1$, and all summands in $L$ and $R$ are of the form $Y^\al$ for $\al\unrhd \la$. As all such $\al$ are 3-regular, we have by Lemmas~\ref{LYoung} and \ref{LDomJames} that 
\[([Y^\al:D^\be])_{\al,\be\unrhd \la}=([Y^\al:S^\ga])_{\al,\ga\unrhd \la}\,([S^\ga:D^\be])_{\ga,\be\unrhd \la}\]
is a unitriangular square matrix. So, by (\ref{E040725_2}), we have  $(L:Y^\al)=(R:Y^\al)$  for all $\al$.

(ii) The proof is similar to that of (i).
\end{proof}

\begin{Lemma}\label{inv42}
Let $p\mid b$ and $\W=\W_{2,b}$.
Assume that either $p=3$, $b\geq 6$ and $\al=(n-6,4,2)$, or 
$p\geq 5$, $b\geq 5$ and $\al=(n-6,2^3)$. Then  $\dim (D^\al)^{\W}=1$. Moreover, $S^\al\cong D^\al$ is a direct summand of $M^\al$, and, denoting by $\si_\al$ the projection onto this direct summand, there exists $\phi\in\Hom_{\s_n}(\bone\ua_\W^{\s_n},M^\al)$ such that $\si_\al\circ\phi\neq 0$.

\end{Lemma}

\begin{proof}
Let $p=3$. By Lemmas~\ref{LSpecialDirectSumDec}, \ref{LI}, and \ref{LSpecialInv}, for $b\geq 6$ we have
\begin{align*}
\dim D_{4,2}^\W&=\dim M_{4,2}^\W+\dim M_5^\W+\dim M_{3,1}-\dim M_{5,1}^\W-\dim M_{3,2}^\W- \dim M_4^\W
\\
&=10+3+4-6-7-3=1.
\end{align*}
The rest follows from Lemma~\ref{LSpecialDirectSumDec}. 
The case $p\geq 5$ is proved similarly. 
\end{proof}

\section{Representations of double covers of symmetric and alternating groups}
\label{SDoubleCov}

\subsection{Double covers of symmetric and alternating groups}
\label{SSDoubleCov}

In this paper we work with the double cover $\ts_n$ of the symmetric group $\s_n$, in which transpositions lift to involutions. Precisely, $\ts_n$  is the group generated by $t_1,\dots,t_{n-1},z$ subject to the following relations:
\begin{align*}
zt_i=t_iz,\ z^2=t_i^2=1,\ t_it_{i+1}t_i=t_{i+1}t_it_{i+1},
\ t_it_j=zt_jt_i\ (\text{for}\ |i-j|>1). 
\end{align*}
We have the natural projection 
$$\pi:\ts_n\to \s_n$$
which maps $t_i$ onto the transposition $(i,i+1)$. 
We extend $\pi$ to a homomorphism of group algebras $\pi:\F\ts_n\to \F\s_n$.

For a subgroup $K\leq \s_n$, we have the subgroup 
$\hat K:=\pi^{-1}(K)\leq \ts_n.$  
In particular, we have the double cover $\tA_n=\pi^{-1}(\A_n)$ of the alternating group. For a composition $(\mu_1,\dots,\mu_r)$ of $n$ we have the subgroups
$$
{\ts}_{\mu_1,\dots,\mu_r}:=\pi^{-1}(\s_{\mu_1,\dots,\mu_r})<{\ts}_n\quad\text{and}\quad
{\tA}_{\mu_1,\dots,\mu_r}:=\pi^{-1}(\A_{\mu_1,\dots,\mu_r})<{\tA}_n.
$$
When $n=ab$, we have the subgroup
$$
\hW_{a,b}=\pi^{-1}(\s_a\wr\s_b)\leq \ts_n.
$$

For an element of $g\in\s_n$, we denote by $\hat g$ (or sometimes $g\,\hat{}$\, if it is typographically preferable\footnote{For example we write $((1,2,3)(4,5,6)){}\hat\,\in\tA_6$ instead of $\widehat{(1,2,3)(4,5,6)}$.}) an element of $\ts_n$ such that $\pi(\hat g)=g$ and such that the order of $\hat g$ is odd if the order of $g$ is odd.

\subsection{Spin modules}
\label{SSSpinMod}
Recall that throughout the paper $\F$ is an algebraically closed field of odd characteristic $p$. Every $\F\s_n$-module inflates along $\pi$ to give an $\F\ts_n$-module ${}^\pi V$ {\em with trivial central action. }
On the other hand, an $\F\ts_n$-module $V$ is called a {\em spin module} if the canonical central involution $z$ acts on $V$ as $-\id_V$. The similar terminology and notation is used for $\F\tA_n$-modules. Using the simplicity of $\A_n$ for $n\geq 5$, the following is immediate:

\begin{Lemma} \label{LSpinFaithful} 
If $n\geq 5$, $G\in\{\ts_n,\tA_n\}$ and $L$ be an irreducible spin $\F G$-module, and $H\leq G$ is a subgroup such that $L\da_H$ is irreducible. We have:
\begin{enumerate}[\rm(i)]
\item $Z(G)=\lan z\ran$;
\item $L$ is faithful;
\item $Z(H)=Z(G) \cap H$; 
\item $H/Z(H)\cong\pi(H)$. 
\end{enumerate}
\end{Lemma}
\begin{proof}
(i) and (ii) follow from the simplicity of $\A_n$. For (iii), by Schur's lemma, any $g\in Z(H)$ acts as a scalar on $L$, and so by faithfulness,  $g\in Z(G)$. For (iv), by (i) and (iii), $\pi(H) \cong H/(Z(G) \cap H) \cong H/Z(H)$. 
\end{proof}

\subsection{Twisted group algebras}
To distinguish between the $\F\ts_n$-modules with trivial central action and spin $\F\ts_n$-modules, we consider the central idempotent $e_z^\pm:=(1\pm z)/2$ in the group algebra $\F\ts_n$, and the ideal decomposition $\F\ts_n=\F\ts_ne_z^+\oplus \F\ts_n e_z^-$.
Now the $\F\ts_n$-modules with trivial central action can be identified with the modules over the algebra $\F\ts_ne_z^+\cong\F\s_n$, while the spin $\F\ts_n$-modules can be identified with the modules over the algebra 
$$
\cT_n:=\F\ts_n e_z^-.$$

Denoting $\ct_i:=t_ie_z^-\in\cT_n$, it is easy to see that algebra 
$\cT_n$ is given by generators $\ct_1,\dots,\ct_{n-1}$ subject only to the relations  
\begin{equation*}\label{ET_nRel}
\ct_i^2=1,\quad (\ct_i\ct_{i+1})^3=1,\quad \ct_i\ct_j = -\ct_j\ct_i\ (\text{for }|i-j|>1).
\end{equation*}
Choosing for each $w\in \Si_n$ a reduced decomposition $w=s_{r_1}\cdots s_{r_l}$, we define $\ct_w:=\ct_{r_1}\cdots \ct_{r_l}$ (which depends up to a sign on the choice of a reduced decomposition). Then  $\{\ct_w\mid w\in \Si_n\}$ is a basis of $\cT_n$, and $\cT_n$ is a {\em twisted group algebra of the symmetric group $\s_n$}.

We consider $\cT_n$ as a superalgebra with
$$
(\cT_n)_\0=\spa\{\ct_w\mid w\in\A_n\}\quad\text{and}\quad (\cT_n)_\1=\spa\{\ct_w\mid w\in\s_n\setminus\A_n\}
$$
Note that the spin $\F\tA_n$-modules can be identified with the $(\cT_n)_\0$-modules.

There is an antiautomorphism $\tau$ of $\T_n$ with $\tau(\ct_i)=-\ct_i$ for all $i=1,\dots,n-1$, see \cite[(13.4)]{KBook}. By a dual of a $\T_n$-supermodule we always mean $\tau$-dual.

For any subgroup $H\leq\s_n$ let 
$$\T_H:=\spa\{\ct_w\mid w\in H\}\subseteq\T_n$$ be the corresponding {\em twisted group algebra of $H$}. Denoting $\hat H:=\pi^{-1}(H)$, we identify the spin $\F \hat H$-modules with $\T_H$-modules. Note that $\T_{\s_n}=\T_n$ and $\T_{\A_n}=(\T_n)_\0$. 

For a composition $\mu=(\mu_1,\dots,\mu_r)$ of $n$, we also use the special notation
$$
\T_{\mu}=\T_{\mu_1,\dots,\mu_r}:=\T_{\s_{\mu}}\subseteq \T_{n}.
$$
We note that as superalgebras 
\begin{equation}\label{ETmn}
\T_{\mu_1,\dots,\mu_r}\cong\T_{\mu_1}\otimes\dots\otimes \T_{\mu_r}.
\end{equation}

Occasionally, we will need to work over $\C$, in which case we use the notation $\cT_{n,\C}$, $(\cT_{n,\C})_\0$, etc. For example, we identify the spin $\C \ts_n$-modules with modules over the algebra $\cT_{n,\C}$, and since $\cT_{n,\C}$ is actually a superalgebra, we can speak of spin $\C \ts_n$-supermodules.

\subsection{Twisted group algebras of wreath products}
In this subsection we assume that $n=2b$ is even and collect some informations about the structure of the twisted group algebra $\T_{\W_{2,b}}$ and more generally $\T_{\s_2\wr K}$ for $K\leq\s_{b}$. Let
\begin{align}\label{E160125}
\cx{}_k:=\ct_{2k-1}\ \ (1\leq k\leq b)\quad\text{and}
\quad 
\cy_j:=(\sqrt{-1})^{2j+1}\ct_{2j}\ct_{2j-1}\ct_{2j+1}\ct_{2j}\ \ (1\leq j< b).
\end{align}

Recall the Sergeev superalgebra $\F\s_{b}\ltimes\Cl_{b}$ from Example~\ref{ExSergeev}. 

\begin{Lemma}\label{L160125}
The twisted group algebra $\T_{\W_{2,b}}$ is generated by $\cx{}_1,\ldots,\cx_{b},\cy_1,\ldots,\cy_{b-1}$. Moreover:
\begin{enumerate}[\rm(i)] 
\item There are isomorphisms of superalgebras 
\begin{align*}
\T_{\W_{2,b}}\ \ &\iso\ \ \ \ \ \ \ \F\s_{b}\ltimes\Cl_{b}\ \ \ \ \iso\ \ \ \ \T_{b}\otimes\Cl_{b},
\\
\cx{}_k\ \ &\ \mapsto\ \ \ \ \ \ \ \ \ \ \,1\otimes \cc_k\ \ \ \ \ \,\mapsto\ \ \ \ \ \  1\otimes \cc_k\qquad\qquad\ \ \  (1\leq k\leq b),\\
\cy_j\ \ &\ \mapsto\ \,(j,j+1)\otimes 1\ \ \ \ \ \ \mapsto\ \ \ \ \ \, \ct_j\otimes \frac{\cc_{j+1}-\cc_j}{\sqrt{-2}}\qquad(1\leq j< b).
\end{align*}

\item
If $K\leq\s_{b}$, the above isomorphisms restrict to isomorphisms
$\T_{\s_2\wr K}\iso\F K\ltimes\Cl_{b}\iso\T_K\otimes\Cl_{b}.$
\end{enumerate}
\end{Lemma}

\begin{proof}
As $\pi(x_k)=(2k-1,2k)$ and $\pi(y_j)=\pm\sqrt{-1}(2j-1,2j+1)(2j,2j+2)$, the generation follows. 

(i) It is easy to check that the odd generators $\cx_k$ and the even generators $\cy_j$ satisfy the relations of \cite[Lemma 13.2.4]{KBook}. Since $\T_{\W_{2,b}}$ and $\F\s_{b}\ltimes\Cl_{b}$ have the same dimension, the first isomorphism follows. The second isomorphism comes from \cite[Lemmas 13.2.3 and 13.2.4]{KBook}.

(ii) holds by restricting the explicit isomorphisms in (i). 
\end{proof}

\section{Irreducible spin modules}

\subsection{\boldmath Irreducible spin modules in characteristic $0$}
Classically, the irreducible $\C \ts_n$-supermodules are canonically labeled by the strict partitions of $n$. 
We denote by $S(\la)$ the irreducible spin $\C \ts_n$-supermodules corresponding to $\la\in \Par_0(n)$.\footnote{For typographical reasons, if $\la=(\la_1,\dots,\la_h)$ we write $S(\la_1,\dots,\la_n)$ instead of $S\big((\la_1,\dots,\la_n)\big)$. 
} 
So 
$$
\Irrs(\cT_{n,\C})=\{S(\la)\mid \la\in\Par_0(n)\}. 
$$ 
Moreover, $S(\la)$ has type $\Mtype$ if and only if $a_0(\la)=0$.
By Lemmas~\ref{LClifford},\,\ref{PIrrIrr}, we now have a complete non-redundant set of irreducible spin $\C \ts_n$-modules up to isomorphism given by 
$$
\Irr(\cT_{n,\C})=\{S(\la;0)\mid\la\in\Par_0(n),\,\,a_0(\la)=0\}\sqcup \{S(\la;\pm)\mid\la\in\Par_0(n),\,\,a_0(\la)=1\},
$$
and a complete non-redundant set of irreducible spin $\C \tA_n$-modules up to isomorphism given by 
$$
\Irr((\cT_{n,\C})_\0)=\{T(\la;0)\mid\la\in\Par_0(n),\,\,a_0(\la)=1\}\sqcup \{T(\la;\pm)\mid\la\in\Par_0(n),\,\,a_0(\la)=0\}.
$$
We will refer to the irreducible modules above as $S(\la;\eps),T(\la;\eps)$ with $\eps\in\{0,+,-\}$ as appropriate. For example, if $a_0(\la)=0$ then $\eps$ can only be $0$ in $S(\la;\eps)$, and if $a_0(\la)=1$ then $\eps$ can be $+$ or $-$ in $S(\la;\eps)$.

\begin{Lemma}\label{lcharvan}
Let $g\in\tA_n$, $\la\in\Par_0(n)$, set $T(\la):=T(\la;0)$ if $a_0(\la)=1$, $T(\la):=T(\la;+)\oplus T(\la;-)$ if $a_0(\la)=0$, and let $\chi^\la$ be the character of $T(\la)$. If $\chi^\la(g)\not=0$ then the order of $\pi(g)$ is odd.
\end{Lemma}

\begin{proof}
This follows from \cite[Corollary 7.5]{stem}.
\end{proof}

Since $\T_{n,m,\C}\cong \T_{n,\C}\otimes \T_{m,\C}$, the following lemma follows from Lemmas~\ref{LBoxTimes}.

\begin{Lemma}\label{product}\cite[Lemma 12.2.13]{KBook}
For $\la\in\Par_0(n)$ and $\mu\in\Par_0(m)$, denote 
$$
S(\la,\mu):=S(\la)\circledast S(\mu).
$$
Then 
$S(\la)\boxtimes S(\mu)\cong S(\la,\mu)^{\oplus (1+a_0(\la)a_0(\mu))},$ 
and $S(\la,\mu)$ is of type $\Mtype$ if and only if $a_0(\la)=a_0(\mu)$.  
Moreover 
\begin{align*}
\Irrs(\T_{n,m,\C})=&\,\{S(\la,\mu)\mid \la\in\Par_0(n),\,\mu\in\RP_0(m)\}.
\end{align*}
\end{Lemma}

We now cite some well-known branching results. 

Let $\la=(\la_1,\dots,\la_h)\in\Par_0(n)$ with $\la_h>0$. Define 
$$R'(\la):=\{(\la_1,\dots,\la_{r-1},\la_r-1,\la_{r+1},\dots,\la_h)\mid 1\leq r<h,\  \la_r-\la_{r+1}>1\}\subseteq\Par_0(n-1),
$$
and $R(\la):=R'(\la)\sqcup\{(\la_1,\dots,\la_{h-1},\la_h-1)\}\subseteq \Par_0(n-1)$. 
Define also 
$$
A'(\la):=\{(\la_1,\dots,\la_{r-1},\la_r+1,\la_{r+1},\dots,\la_h)\mid 1\leq r\leq h,\  \la_{r-1}-\la_{r}>1\}\subseteq\Par_0(n+1)
$$
(where $\la_{-1}$ is interpreted as $+\infty$), and $A(\la):=A'(\la)\sqcup\{(\la_1,\dots,\la_{h},1)\}\subseteq \Par_0(n+1)$.

\begin{Lemma} \label{LBr0} {\rm \cite[Theorem 3]{MY}}  
Let $\la=(\la_1,\dots,\la_h)\in\Par_0(n)$ with $\la_h>0$. Then 
\begin{align*}
S(\la)\da_{\ts_{n-1}}&\cong
\left\{
\begin{array}{ll}
\bigoplus_{\mu\in R(\la)}S(\mu) &\hbox{if $a_0(\la)=0$,}\\
\bigoplus_{\mu\in R(\la)}S(\mu)^{\oplus 2} &\hbox{if $a_0(\la)=1$ and $\la_{h}>1$,}\\
S(\la_1,\dots,\la_{h-1})\oplus \bigoplus_{\mu\in R'(\la)}S(\mu)^{\oplus 2} &\hbox{if $a_0(\la)=1$ and $\la_{h}=1$;}
\end{array}
\right.
\\
S(\la)\uparrow^{\ts_{n+1}}&\cong
\left\{
\begin{array}{ll}
\bigoplus_{\nu\in A(\la)}S(\nu) &\hbox{if $a_0(\la)=0$,}\\
S(\la_1,\dots,\la_h,1)\oplus\bigoplus_{\nu\in A'(\la)}S(\nu)^{\oplus 2} &\hbox{if $a_0(\la)=1$.}
\end{array}
\right.
\end{align*}
\end{Lemma}

The next two lemmas follow from \cite[Theorems 8.1 and 8.3]{stem}.

\begin{Lemma}\label{tworowschar0}
Let $0\leq a<n/2$ and $0\leq b\leq n$. Then
\[[S(n-a,a)\da_{\ts_{n-b,b}}:S((n-b-c,c),(b-a+c,a-c))]\neq 0\]
whenever $0\leq c<(n-b)/2$ and $0\leq a-c<b/2$. 
All other composition factors of $S(n-a,a)\da_{\ts_{n-b,b}}$ are of the form $S((n-b-d,d),(b-e,e))$ with $d+e<a$.
\end{Lemma}

\begin{Lemma} \label{L040925}
Let $n=2b\geq 6$ be even. Then, in the Grothendieck group of $\C\ts_{b,b}$-supermodules, 
\begin{align*}
[S(n)\da_{\ts_{b,b}}]&=2[S(b)\circledast S(b)]=(1+\de_{2\nmid b})[S(b)\boxtimes S(b)],\\
[S(n-1,1)\da_{\ts_{b,b}}]&=[S(b-1,1)\circledast S(b)]+[S(b)\circledast S(b-1,1)]+2[S(b)\circledast S(b)]
\\
&=[S(b-1,1)\boxtimes S(b)]+[S(b)\boxtimes S(b-1,1)]+(1+\de_{2\nmid b})[S(b)\boxtimes S(b)].
\end{align*}
\end{Lemma}

\subsection{\boldmath Irreducible spin modules in characteristic $p$}
The classification of the irreducible spin $\F \ts_n$-supermodules  was obtained in \cite{BK,BK2} using two different approaches which were later unified in \cite{KS}. 
The irreducible spin $\F \ts_n$-supermodules are canonically labeled by the restricted $p$-strict partitions of $n$. We denote by $D(\la)$ the irreducible spin $\F \ts_n$-supermodules corresponding to $\la\in \RP_p(n)$.\footnote{If $\la=(\la_1,\dots,\la_h)$ then we often write $D(\la_1,\dots,\la_n)$ instead of $D\big((\la_1,\dots,\la_n)\big)$. 
} 
 So 
$$
\Irrs(\cT_n)=\{D(\la)\mid \la\in\RP_p(\n)\}. 
$$ 
Moreover, $D(\la)$  type $\Mtype$ if and only if $a_p(\la)=0$. 
In particular,
\begin{equation}\label{EEndDDim}
\dim\End_{\T_n}(D(\la))=1+a_p(\la). 
\end{equation}
The supermodules $D(\la)$ are self-dual, see for example \cite[Theorem 22.3.1(i)]{KBook}.

By Lemmas~\ref{LClifford},\,\ref{PIrrIrr}, we now have a complete non-redundant set of irreducible spin $\F \ts_n$-modules up to isomorphism given by 
$$
\Irr(\cT_n)=\{D(\la;0)\mid\la\in\RP_p(n),\,\,a_p(\la)=0\}\sqcup \{D(\la;\pm)\mid\la\in\RP_p(n),\,\,a_p(\la)=1\},
$$
and a complete non-redundant set of irreducible spin $\F \tA_n$-modules up to isomorphism given by 
$$
\Irr((\cT_n)_\0)=\{E(\la;0)\mid\la\in\RP_p(n),\,\,a_p(\la)=1\}\sqcup \{E(\la;\pm)\mid\la\in\RP_p(n),\,\,a_p(\la)=0\}.
$$
We will refer to the irreducible modules above as $D(\la;\eps),E(\la;\eps)$ with $\eps\in\{0,+,-\}$ as appropriate. For example, if $a_p(\la)=0$ then $\eps$ can only be $0$ in $D(\la;\eps)$, and if $a_p(\la)=1$ then $\eps$ can be $+$ or $-$ in $D(\la;\eps)$. 

Note that self-duality of the supermodule $D(\la)$ implies that for $\la\in\RP_p(n)$ and appropriate $\eps$, we have
\begin{equation}\label{E*}
\begin{split}
D(\la,\eps)^*\cong D(\la,\eps)\quad \text{or}\quad D(\la,\eps)^*\cong D(\la,-\eps),
\\
E(\la,\eps)^*\cong E(\la,\eps)\quad \text{or}\quad E(\la,\eps)^*\cong E(\la,-\eps).
\end{split}
\end{equation}

Since $\T_{n,m}\cong \T_n\otimes \T_m$, the following lemma follows from Lemmas~\ref{LBoxTimes},~\ref{LClifford}, and~\ref{PIrrIrr}.

\begin{Lemma}\label{product}\cite[Lemma 12.2.13]{KBook}
For $\la\in\RP_p(n)$ and $\mu\in\RP_p(m)$, denote 
$
D(\la,\mu):=D(\la)\circledast D(\mu).
$
Then 
$D(\la)\boxtimes D(\mu)\cong D(\la,\mu)^{\oplus (1+a_p(\la)a_p(\mu))},$ 
and $D(\la,\mu)$ is of type $\Mtype$ if and only if $a_p(\la)=a_p(\mu)$.  
Moreover: 
\begin{align*}
\Irrs(\T_{n,m})=&\,\{D(\la,\mu)\mid \la\in\RP_p(n),\,\mu\in\RP_p(m)\},
\\
\Irr(\T_{n,m})=&\,\{D(\la,\mu;0)\mid \la\in\RP_p(n),\,\mu\in\RP_p(m)\ \text{with}\ a_p(\la)=a_p(\mu)\}\\&\sqcup\{D(\la,\mu;\pm)\mid \la\in\RP_p(n),\,\mu\in\RP_p(m)\ \text{with}\ a_p(\la)\neq a_p(\mu)\},
\\
\Irr((\T_{n,m})_\0)=&\,\{E(\la,\mu;\pm)\mid \la\in\RP_p(n),\,\mu\in\RP_p(m)\ \text{with}\ a_p(\la)=a_p(\mu)\}\\&\sqcup\{E(\la,\mu;0)\mid \la\in\RP_p(n),\,\mu\in\RP_p(m)\ \text{with}\ a_p(\la)\neq a_p(\mu)\}.
\end{align*}
\end{Lemma}

We will refer to the irreducible modules arising in Lemma~\ref{product} as $D(\la,\mu;\eps),E(\la,\mu;\eps)$ with $\eps\in\{0,+,-\}$ as appropriate.

Recall from (\ref{EHomFMod}) that for an $\F G$-module $L$, the endomorphism space $\End_\F(L)$ is an $\F G$-module. 

\begin{Lemma} \label{LHomPM}
Let $\la\in\Par_p(n)$. 
\begin{enumerate}
\item[{\rm (i)}] If $a_p(\la)=1$ then $\End_\F(D(\la;+))\cong\End_\F(D(\la;-))$ as $\F\ts_n$-modules. 
\item[{\rm (ii)}] If $a_p(\la)=0$ and $\si\in\ts_n\setminus\tA_n$ then $\End_\F(E(\la;+))\cong\End_\F(E(\la;-))^\si$ as $\F\tA_n$-modules. 
\end{enumerate}
\end{Lemma}
\begin{proof}
(i) holds using $D(\la;\pm)\cong D(\la;\mp)\otimes\sgn$, and (ii) follows using $E(\la;\pm)\cong E(\la;\mp)^\si$.
\end{proof}

\begin{Example} \label{ExT2Cl1}
{\rm 
We have $\cT_2\cong \Cl_1$, so, recalling Example~\ref{ExClifford} we can identify $D(2)$ with the Clifford module $U_1$. It follows that $D(2)\cong \cT_2$, the regular $\cT_2$-supermodule. 
So $\dim\End_{\T_{2}}(D(2))=2$, and for any $\T_{n-2}$-supermodule $X$, we have 
\begin{equation}\label{E130825}
\dim\End_{\T_{n-2,2}}(X\boxtimes D(2))=2\dim\End_{\T_{n-2}}(X).
\end{equation}
}
\end{Example}

\begin{Lemma}\label{L161224_3}
Let \,$V$ be a $\T_{n-2,2}$-supermodule. Then
$V\oplus V\cong V\da_{\T_{n-2}}\boxtimes D(2).$ 
In particular,
\[\dim\End_{\T_{n-2}}(V\da_{\T_{n-2}})=2\dim\End_{\T_{n-2,2}}(V).\]
\end{Lemma}

\begin{proof}
We consider the $\s_{n-2,2}$-modules $\bone_{\s_{n-2}}\boxtimes\bone_{\s_2}\cong\bone_{\s_{n-2,2}}$ and $\bone_{\s_{n-2}}\boxtimes\sgn_{\s_2}$ as $\ts_{n-2,2}$-modules via inflation along $\pi$. Moreover, for any composition $\mu$ we always identify $\T_\mu$-(super)modules with spin $\F\ts_\mu$-(super)modules. Then, since $D(2)$ is the regular $\T_2$-supermodule, we have  
\begin{align*}
V\da_{\T_{n-2}}\boxtimes D(2)
&\cong 
V\da_{\T_{n-2}}\boxtimes \T_2
\cong
(V\da_{\T_{n-2}})\ua^{\T_{n-2,2}}
=(V\da_{\ts_{n-2}})\ua^{\ts_{n-2,2}}\\
&\cong V\otimes (\bone\da_{\ts_{n-2}}\ua^{\ts_{n-2,2}})=V\otimes (\bone\da_{\s_{n-2}}\ua^{\s_{n-2,2}})\\
&\cong V\otimes ((\bone_{\s_{n-2}}\boxtimes\bone_{\s_2})\oplus (\bone_{\s_{n-2}}\boxtimes\sgn_{\s_2}))
\\
&\cong V\oplus (V\otimes (\bone_{\s_{n-2}}\boxtimes\sgn_{\s_2}))
\cong V\oplus V,
\end{align*}
using Lemma \ref{tensor1sgn}, for the last isomorphism. 
The second statement now follows from (\ref{E130825}). 
\end{proof}

\section{Reduction modulo $p$}
For $\la\in\Par_0(n)$, we denote by $\bar S(\la)$ a $\cT_n$-supermodule obtained by reduction modulo $p$ from the irreducible $\cT_{n,\C}$-supermodule $S(\la)$. Similarly, for appropriate $\eps$, we denote by $\bar S(\la;\eps)$ and $\bar T(\la;\eps)$ reductions modulo $p$ of the corresponding irreducible modules over $\C\ts_n$ and $\C\tA_n$.

\subsection{Basic and second basic modules} 
\label{SSBasic}
Basic and second basic spin modules are specific spin representations of $\s_n$ or $\A_n$. These modules will play a special role in this paper. 

Let $n\geq 1$. The {\em basic $\C\ts_n$-supermodule} is the irreducible  $\C\ts_n$-supermodule $S(n)$ corresponding to the partition $(n)$. A {\em basic $\C\ts_n$-module} is an irreducible $\C\ts_n$-module of the form  $S((n);\eps)$ (there might be one or two such depending on $a_0((n))$), and a {\em basic $\C\tA_n$-module} is an irreducible $\C\tA_n$-module of the form  $T((n);\eps)$ (there might be one or two such depending on $a_0((n))$). 

Let $n\geq 3$. The {\em second basic $\C\ts_n$-supermodule} is the irreducible  $\C\ts_n$-supermodule $S(n-1,1)$. A {\em second basic $\C\ts_n$-module} is an irreducible $\C\ts_n$-module of the form  $S((n-1,1);\eps)$, and a {\em second basic $\C\tA_n$-module} is an irreducible $\C\tA_n$-module of the form  $T((n-1,1);\eps)$. 

A {\em basic $\F\ts_n$-(super)module} is a composition factor of a reduction modulo $p$ of a complex basic (super)module. A {\em basic $\F\tA_n$-module} is a composition factor of a reduction modulo $p$ of a basic $\C\tA_n$-module. A {\em second basic $\F\ts_n$-(super)module} is a composition factor of a reduction modulo $p$ of a second basic $\C\ts_n$-(super)module not isomorphic to a basic $\F\ts_n$-(super)module. A {\em second basic $\F\tA_n$-module} is a composition factor of a reduction modulo $p$ of a second basic $\C\tA_n$-module not isomorphic to a basic $\F\tA_n$-module.


Denote
\begin{align}
\label{bs}\balpha_n&:=\left\{\begin{array}{ll}
(p^m,b)&\text{if}\  n=pm+b\ \text{with $0<b<p$},\\
(p^{m-1},p-1,1)&\text{if}\  n=pm;
\end{array}\right.\\
\label{sbs}\bbeta_n&:=\left\{\begin{array}{ll}
\balpha_{n-1}+(1)&\text{if}\  n\geq p+2,\\
(p-2,2,1)&\text{if}\  n=p+1\geq 6,\\
(p-2,2)&\text{if}\  n=p\geq 5,\\
(n-1,1)&\text{if}\  3\leq n<p.
\end{array}\right.
\end{align}
Then, in view of \cite[Tables III, IV]{Wales} and \cite[Theorem 3.6]{KT2}, the basic $\F\ts_n$-supermodule is exactly $D(\balpha_n)$, the basic $\F\ts_n$-modules are exactly $D(\balpha_n;\eps)$, the basic $\F\tA_n$-modules are exactly $E(\balpha_n;\eps)$, 
the second basic $\F\ts_n$-supermodule is exactly $D(\bbeta_n)$, the second basic $\F\ts_n$-modules are exactly $D(\bbeta_n;\eps)$, and the second basic $\F\tA_n$-modules are exactly $E(\bbeta_n;\eps)$. 


\begin{Lemma}\label{LBetaGammaChar}
Let $n\geq 5$. Then Table III (resp. Table IV) gives the dimension and type of the supermodules $D(\balpha_n)$ (resp. $D(\bbeta_n)$), as well as the expression of  $[D(\balpha_n)]$ (resp. $[D(\bbeta_n)]$) in terms of $[\bar S(\la)]$'s in the Grothendieck group.
\end{Lemma}

\begin{proof}
This follows from \cite[Tables III, IV]{Wales}.
\end{proof}


\[\begin{array}{|l|l|l|l|}
\hline
\text{cases} &\text{dimension}&\text{type}&\text{$[D(\balpha_n)]$}\\
\hline
p\nmid n\ \text{and $n$ is even}&2^{n/2}&\Qtype&[\bar S(n)]\\
\hline
p\nmid n\ \text{and $n$ is odd}&2^{(n-1)/2}&\Mtype&[\bar S(n)]\\
\hline
p\mid n\ \text{and $n$ is even}&2^{(n-2)/2}&\Mtype&
\frac{1}{2}[\bar S(n)]\\
\hline
p\mid n\ \text{and $n$ is odd}&2^{(n-1)/2}&\Qtype&[\bar S(n)]\\
\hline
\end{array}\]
\centerline{{\sc Table III}: {Basic supermodule $D(\balpha_n)$}}

\[\begin{array}{|l|l|l|l|}
\hline
\text{cases}&\text{dimension}&\text{type}&[D(\bbeta_n)]\\
\hline
p\nmid n,\,p\nmid (n-1),\,\text{and $n$ is even}&2^{(n-2)/2}(n-2)&\Mtype&[\bar S(n-1,1)]
\\
\hline
p\nmid n,\,p\nmid (n-1),\,\text{and $n$ is odd}&2^{(n-1)/2}(n-2)&\Qtype&[\bar S(n-1,1)]
\\
\hline
p\mid n\ \text{and $n$ is even}&2^{(n-2)/2}(n-3)&\Mtype&[\bar S(n-1,1)]-\frac{1}{2}[\bar S(n)]
\\
\hline
p\mid n\ \text{and $n$ is odd}&2^{(n-1)/2}(n-3)&\Qtype&[\bar S(n-1,1)]-[\bar S(n)]
\\
\hline
p\mid (n-1)\ \text{and $n$ is even}&2^{(n-2)/2}(n-4)&\Qtype&
[\bar S(n-1,1)]-[\bar S(n)]
\\
\hline
p\mid (n-1)\ \text{and $n$ is odd}&2^{(n-3)/2}(n-4)&\Mtype&\frac{1}{2}[\bar S(n-1,1)]-[\bar S(n)]\\
\hline
\end{array}\]
\vspace{1mm}
\centerline{{\sc Table IV}: {Second basic supermodule $D(\bbeta_n)$}}

Note that the case $p>n$ in the above tables covers the characteristic $0$ case.

In the following lemma, $D(\mu)\otimes D(\balpha_n)$ refers to the  {\em inner}\, tensor product of $\F\ts_n$-modules (such a tensor product has trivial central action).

\begin{Lemma} \label{L040925_2} 
In the Grothendieck group of $\F\ts_n$-modules, the classes  $\{[D(\mu)\otimes D(\balpha_{n})]\mid \mu\in \RP_p(n)\}$ are linearly independent. 
\end{Lemma}
\begin{proof}
By \cite[Theorem 3.3]{stem} and Lemma~\ref{LBetaGammaChar}, the Brauer character of $D(\balpha_{n})$ does not vanish on any conjugacy class corresponding to the cycle-shape with odd parts. 
Moreover, by \cite[Theorem 7.2]{stem}, the Brauer character of any spin supermodule vanishes on every other conjugacy class.  Since Brauer characters of irreducible modules are linearly independent, it follows that the Brauer characters of the modules $D(\mu)\otimes D(\balpha_{n})$ are linearly independent. 
\end{proof}

We will need the following branching result for the second basic module:

\begin{Lemma}\label{L160125_3}
Let $n=2b\geq 10$ be even. In the Grothendieck group of $\T_{b,b}$-supermodules, denote
$$
\D_{\bbeta,\balpha}:=[D(\bbeta_{b})\boxtimes D(\balpha_{b})],\ \D_{\balpha,\bbeta}:=[D(\balpha_{b})\boxtimes D(\bbeta_{b})],\ \D_{\balpha,\balpha}:=[D(\balpha_{b})\boxtimes D(\balpha_{b})].
$$
Then, in the Grothendieck group, 
$$
[D(\bbeta_n)\da_{\T_{b,b}}]=\left\{
\begin{array}{ll}
\D_{\bbeta,\balpha}+\D_{\balpha,\bbeta}+\D_{\balpha,\balpha} &\hbox{if $n\not\equiv 0,1,2\pmod{p}$ and $b$ is even,}\\
\D_{\bbeta,\balpha}+\D_{\balpha,\bbeta}+2\D_{\balpha,\balpha} &\hbox{if $n\not\equiv 0,1,2\pmod{p}$ and $b$ is odd,}\\
2\D_{\bbeta,\balpha}+2\D_{\balpha,\bbeta}+6\D_{\balpha,\balpha} &\hbox{if $n\equiv 0\pmod{p}$ and $b$ is even,}
\\
\D_{\bbeta,\balpha}+\D_{\balpha,\bbeta}+3\D_{\balpha,\balpha} &\hbox{if $n\equiv 0\pmod{p}$ and $b$ is odd,}
\\
\D_{\bbeta,\balpha}+\D_{\balpha,\bbeta} &\hbox{if $n\equiv 1\pmod{p}$,}
\\
\D_{\bbeta,\balpha}+\D_{\balpha,\bbeta}+3\D_{\balpha,\balpha} &\hbox{if $n\equiv 2\pmod{p}$ and $b$ is even,}
\\
2\D_{\bbeta,\balpha}+2\D_{\balpha,\bbeta}+6\D_{\balpha,\balpha} &\hbox{if $n\equiv 2\pmod{p}$ and $b$ is odd. }
\end{array}
\right.
$$
\end{Lemma}

\begin{proof}
By Lemma~\ref{L040925}, in the Grothendieck group, we have
\begin{align*}
[\bar S(n)\da_{\T_{b,b}}]&=(1+\de_{2\nmid b})[\bar S(b)\boxtimes \bar S(b)],\\
[\bar S(n-1,1)\da_{\T_{b,b}}]&=[\bar S(b-1,1)\boxtimes \bar S(b)]+[\bar S(b)\boxtimes \bar S(b-1,1)]+(1+\de_{2\nmid b})[\bar S(b)\boxtimes \bar S(b)].
\end{align*}
The claim now follows using Lemma~\ref{LBetaGammaChar}. 
\end{proof}

For the following lemma recall the Clifford supermodule $U_b$  from Example~\ref{ExClifford}. We identify $\T_{\W_{2,b}}$ with $\T_{b}\otimes\Cl_{b}$ via the explicit isomorphism 
from Lemma~\ref{L160125}(i). In particular, the $\T_{\W_{2,b}}$-supermodules $D(\la)\circledast U_b$ make sense for any $\la\in\RP_p(b)$. 

\begin{Lemma}\label{L160125_2}
Let $n=2b\geq 10$ be even.
Then, in the Grothendieck group, 
$$[D(\bbeta_n)\da_{\T_{\W_{2,b}}}]=
\left\{
\begin{array}{ll}
2[D(\bbeta_{b})\circledast U_b]+[D(\balpha_{b})\circledast U_b] &\hbox{if $n\not\equiv 0,1,2\pmod{p}$,}\\
2[D(\bbeta_{b})\circledast U_b]+3[D(\balpha_{b})\circledast U_b] &\hbox{if $n\equiv 0,2\pmod{p}$,}
\\
2[D(\bbeta_{b})\circledast U_b]
 &\hbox{if $n\equiv 1\pmod{p}$.}
\end{array}
\right.
$$
\end{Lemma}

\begin{proof}
Recall from Example~\ref{ExClifford} that $U_b$ has dimension $2^{b/2}$ and type $\Mtype$ if $b$ is even, and dimension $2^{(b+1)/2}$ and type $\Qtype$ if $b$ is odd. Denote ${\mathtt B}_\boxtimes:=[D(\bbeta_{b})\boxtimes U_b]$ and ${\mathtt A}_\boxtimes:=[D(\balpha_{b})\boxtimes U_b]$. 
Taking into account Lemma~\ref{LBetaGammaChar}, the claim of the lemma can be re-written as 
$$[D(\bbeta_n)\da_{\T_{\W_{2,b}}}]=
\left\{
\begin{array}{ll}
2{\mathtt B}_\boxtimes+{\mathtt A}_\boxtimes &\hbox{if $n\not\equiv 0,1,2\pmod{p}$ and $b$ is even,}\\
{\mathtt B}_\boxtimes+{\mathtt A}_\boxtimes &\hbox{if $n\not\equiv 0,1,2\pmod{p}$ and $b$ is odd,}\\
2{\mathtt B}_\boxtimes+3{\mathtt A}_\boxtimes &\hbox{if $n\equiv 0\pmod{p}$ and $b$ is even,}
\\
{\mathtt B}_\boxtimes+\frac{3}{2}{\mathtt A}_\boxtimes &\hbox{if $n\equiv 0\pmod{p}$ and $b$ is odd,}
\\
2{\mathtt B}_\boxtimes+3{\mathtt A}_\boxtimes &\hbox{if $n\equiv 2\pmod{p}$,}
\\
2{\mathtt B}_\boxtimes
 &\hbox{if $n\equiv 1\pmod{p}$ and $b$ is even,}
 \\
{\mathtt B}_\boxtimes
 &\hbox{if $n\equiv 1\pmod{p}$ and $b$ is odd.}
\end{array}
\right.
$$

From \cite[Lemma 3.2]{stem}, $U_b$ can be viewed as a $\T_{b}$-supermodule with $\ct_j$ acting as  $(\cc_{j+1}-\cc_j)/\sqrt{-2}$. We denote this $\T_{b}$-supermodule by ${}^{\T_b}U_b$. 
Moreover, since in characteristic $0$ this construction  yields  basic spin modules, reducing modulo $p$, we conclude that in the Grothendieck group we have $[{}^{\T_b}U_b]=c[D(\balpha_b)]$ for some $c\in\Z_{>0}$. Comparing dimensions, we deduce that $c=1$ if $p\nmid b$ and $b$ is even, and $c=2$ otherwise.

Write 
\[[D(\bbeta_n)\da_{\T_{\W_{2,b}}}]
=\sum_{\mu\in\RP_p(b)}e_\mu[D(\mu)\boxtimes U_b]\]
with $e_\mu\in\Q$. 
Let ${\mathsf D}$ be the diagonal embedding of $\s_{b}$ in $\s_{\{1,3,\ldots,n-1\}}\times\s_{\{2,4,\ldots,n\}}$. Then ${\mathsf D}\leq\W_{2,b}$, and by Lemma~\ref{L160125}, we have 
\[[D(\bbeta_n)\da_{\T_{\mathsf D}}]=[D(\bbeta_n)\da_{\T_{\W_{2,b}}}\da_{\T_{\mathsf D}}]=\sum_{\mu\in\RP_p(b)}ce_\mu[D(\mu)\otimes D(\balpha_{b})],\]
where $\otimes$ is the inner tensor product. Denote ${\mathtt B}_\otimes:=[D(\bbeta_{b})\otimes D(\balpha_b)]$ and ${\mathtt A}_\otimes:=[D(\balpha_{b})\otimes D(\balpha_{b})]$. 
In view of Lemmas~\ref{L040925_2}, it suffices prove that 
$$
[D(\bbeta_n)\da_{\T_{\mathsf D}}]=
\left\{
\begin{array}{ll}
2{\mathtt B}_\otimes+{\mathtt A}_\otimes &\hbox{if $n\not\equiv 0,1,2\pmod{p}$ and $b$ is even,}\\
2{\mathtt B}_\otimes+2{\mathtt A}_\otimes &\hbox{if $n\not\equiv 0,1,2\pmod{p}$ and $b$ is odd,}\\
4{\mathtt B}_\otimes+6{\mathtt A}_\otimes &\hbox{if $n\equiv 0\pmod{p}$ and $b$ is even,}
\\
2{\mathtt B}_\otimes+3{\mathtt A}_\otimes &\hbox{if $n\equiv 0\pmod{p}$ and $b$ is odd,}
\\
2{\mathtt B}_\otimes+3{\mathtt A}_\otimes &\hbox{if $n\equiv 2\pmod{p}$ and $b$ is even,}
\\
4{\mathtt B}_\otimes+6{\mathtt A}_\otimes &\hbox{if $n\equiv 2\pmod{p}$ and $b$ is odd,}
\\
2{\mathtt B}_\otimes
 &\hbox{if $n\equiv 1\pmod{p}$ and $b$ is even,}
 \\
2{\mathtt B}_\otimes
 &\hbox{if $n\equiv 1\pmod{p}$ and $b$ is odd.}
\end{array}
\right.
$$

Taking into account that the subgroup $\pi^{-1}(\s_{\{1,3,\ldots,n-1\}}\times\s_{\{2,4,\ldots,n\}})$ is conjugate to the subgroup $\ts_{b,b}$ and that ${\mathsf D}\cong \s_b$ is the diagonal subgroup of $\s_{\{1,3,\ldots,n-1\}}\times\s_{\{2,4,\ldots,n\}}$, the terms $\D_{\bbeta,\balpha}$  and $\D_{\balpha,\bbeta}$ in Lemma~\ref{L160125_3} each contribute ${\mathtt B}_\otimes$ to $[D(\bbeta_n)\da_{\T_{\mathsf D}}]$,  while $\D_{\balpha,\balpha}$ contributes ${\mathtt A}_\otimes$. Now the required expressions for $[D(\bbeta_n)\da_{\T_{\mathsf D}}]$ follow from Lemma~\ref{L160125_3}. 
\end{proof}

We will need the following result about the inner tensor products:

\begin{Lemma}\label{lem:3tens}
Let $n \geq 5$ and $n\not\equiv 0,1\pmod{p}$. Then the tensor product of a basic module and a second basic module of $\tA_{n}$ has composition length at least $3$.
\end{Lemma}

\begin{proof}
We will freely appeal to Tables III,\,IV and Lemma \ref{lmodules} without further reference. We provide details for the case where $n$ is even, the case where $n$ is odd being similar. For even $n$, 
$$(D(\balpha_{n};\pm)\otimes D(\bbeta_{n},0))\da_{\tA_{n}}\cong (E(\balpha_{n},0)\otimes E(\bbeta_{n};+))\oplus (E(\balpha_{n},0)\otimes E(\bbeta_{n};-)),$$
with $E(\balpha_{n},0)\otimes E(\bbeta_{n};+)$ and $E(\balpha_{n},0)\otimes E(\bbeta_{n};-)$ conjugate under the action of $\ts_{n}$ and so having the same composition length. Moreover, the $\F\ts_n$-modules $D(\balpha_{n};+)\otimes D(\bbeta_{n},0)$ and $D(\balpha_{n};-)\otimes D(\bbeta_{n},0)$ differ by $\sgn$, so have the same composition length. So it suffices to prove that the composition length of any $D(\balpha_{n};\eps)\otimes D(\bbeta_{n},0)$ is at least $5$. By the assumption $n\not\equiv 0,1\pmod{p}$, $D(\balpha_{n};\eps)$ is a reducition modulo $p$ of some $S((n);\de)$ (the choice of $\eps$ and $\de$ is not canonical), while $D(\bbeta_{n};0)$ is a reducition modulo $p$ of $S((n-1,1);0)$, and so it suffices to prove that the composition length of $S((n);\de)\otimes S((n-1,1);0)$ is at least $5$. This follows from \cite[Theorem~9.3]{stem}, which guarantees that 
each $S^\la_\C$ with $\la\in\{(n-k,1^k)\mid 1\leq k\leq n-2\}\cup\{(n-k,2,1^{k-2})\mid 2\leq k\leq n-2\}$ 
is a composition factor of $S((n);\de)\otimes S((n-1,1);0)$. 
\end{proof}

\subsection{Two-row reductions}\label{2rows}
In \S\ref{SSBasic}, we have considered the composition factors of $\bar S(n)$ and $\bar S(n-1,1)$ . We now discuss 
the composition factors of reductions modulo $p$ of more general  two-row representations $S(n-a,a)$. 

For $n\in\Z_{>0}$, we set 
\[m_n:=\max\{\lfloor(n-1)/2\rfloor-\de_{p,3}-\de_{n\equiv p\pmod{2p}},0\}.\]

\begin{Lemma} {\rm \cite[Theorems 1.1,\,1.2,\,1.3]{M3}} \label{LM30} 
Let $n\in\Z_{>0}$. For each integer $a$ satisfying $0\leq a\leq m_n$,  there is exactly one $\mu_{n,a}\in\RP_p(n)$ such that $[\bar S(n-a,a):D(\mu_{n,a})]\neq 0$ and $[\bar S(n-b,b):D(\mu_{n,a})]= 0$ for all $0\leq b<a$. Moreover, setting
$$
\TR_p(n):=\{\mu_{n,a}\mid 0\leq a\leq m_n\},
$$
we have that $
\{D(\mu)\mid \mu\in \TR_p(n)\}
$
is a complete and non-redundant set of composition factors of the reductions modulo $p$ of the two-row irreducible $\cT_n$-supermodules $\{\bar S(n-k,k)\mid 0\leq k<n/2\}$. 
\end{Lemma}

We note that \cite{M3} covers only the cases $n\geq p$, but for $n<p$ the group algebra $\F\ts_n$ is semisimple, and so the lemma clearly holds in that case with $\mu_{n,a}=(n-a,a)$ for all $a$. 

\begin{Corollary} \label{CInverse2R} 
Let $0\leq a\leq m_n$. In the Grothendieck group (with coefficients extended from $\Z$ to $\Q$) we have 
\[[D(\mu_{n,a})]=\sum_{b\leq a}c_{b}[\bar S(n-b,b)]\]
for some coefficients $c_{b}\in\Q$, with
$c_{a}=[\bar S(n-a,a):D(\mu_{n,a})]^{-1}\neq 0.$
\end{Corollary}

By definition, we have $\mu_{n,0}=\balpha_n$ and $\mu_{n,1}=\bbeta_n$. Much more generally:

\begin{Lemma} \label{L2RowExplicit} {\rm \cite[Theorems 1.2,\,1.3]{M3}} Let $n>p$. 
\begin{enumerate}
\item[{\rm (i)}] If $p=3$ and $0\leq a\leq m_n$, then $\mu_{n,a}=\balpha_{n-a}+\balpha_a$.
\item[{\rm (ii)}] If $p>3$ and $0\leq 2a\leq n-1-p-\de_{p\mid a}$, then $\mu_{n,a}=\balpha_{n-a}+\balpha_a$. 
\end{enumerate}
\end{Lemma}

In general, the explicit description of the individual partitions $\mu_{n,a}$ can be found in \cite[Theorems 1.1, 1.2, 1.3]{M3}.  Most of the time it will be sufficient to just have description of the set $\TR_p(n)$ given in the Lemma~\ref{LM3} below. 
 
For $p> 3$, we define the explicit set of partitions 
\begin{align*}
\TR'_p&:=\{(p^a,b,c)\mid a\geq 0,\ 1= c<b\leq p-2\text{ or }2\leq c<b\leq p-1\}\\
&\quad\ \cup\{(p^a,p-1,b,1)\mid a\geq 0,\,2\leq b\leq p-2\}\\
&\quad\ \cup\{(p^a,p-1,p-2,2,1),(p^a,p-1,p-2,2),(p^a,p-2,2,1)\mid a\geq 0\}.
\end{align*}
Here, when writing a  partition in the form $(p^a,\dots)$ we mean that the part $p$ is repeated $a$ times. We also set $\TR'_3=\varnothing$. Finally, for $n\in \Z_{\geq 0}$, we let 
$$
\TR'_p(n):=\TR'_p\cap\Par(n).
$$

\begin{Lemma} \label{LM3} {\rm \cite[Theorems 1.1, 1.2, 1.3]{M3}} We have 
$$\TR_p(n)=\{\balpha_{n-k}+\balpha_k\mid k=0\ \text{or}\,\ 0< 2k\leq n-p-\de_{p\mid k}\}\sqcup\TR'_p(n).$$
\end{Lemma}

We now obtain some first results on branching $D(\la)\da_{\ts_{n-1}}$ for $\la\in\TR_p(n)$. 

\begin{Lemma} \label{LRedRes}
If $\la\in\TR_p(n)$ and $[D(\la)\da_{\ts_{n-1}}:D(\mu)]\neq 0$, then then $\mu\in\TR_p(n-1)$.
\end{Lemma}
\begin{proof}
By definition, $D(\la)$ is a composition factor of $\bar S(n-a,a)$ for some $a$. Since reduction modulo $p$ commutes with the restriction to a subgroup, $D(\mu)$ is a composition factor of $\bar S(\nu)$ for some constituent $S(\nu)$ of $S(n-a,a)\da_{\ts_{n-1}}$. 
By Lemma~\ref{LBr0}, $\nu=(n-a-1,a)$ or $(n-a,a-1)$. Hence $\mu\in 
\TR_p(n-1)$. 
\end{proof}

\begin{Lemma}\label{rectworows2}
Let $n\geq 8$ and $\la\in\TR_p(n)\setminus\{\balpha_n,\bbeta_n\}$. Then there exists $\mu\in\TR_p(n-1)\setminus\{\balpha_{n-1},\bbeta_{n-1}\}$ is a composition factor of $D(\la)\da_{\ts_{n-1}}$.
\end{Lemma}

\begin{proof}
This follows immediately from Lemma~\ref{LRedRes} and \cite[Lemma 3.7]{KT2}.
\end{proof}

\begin{Lemma}\label{tworowscharp}
Let $0\leq c\leq a\leq m_n$ and $0\leq b\leq n$. If $c\leq m_{n-b}$ and $a-c\leq m_b$ then
\[[D(\mu_{n,a})\da_{\ts_{n-b,b}}:D(\mu_{n-b,c},\mu_{b,a-c})]\neq 0.\]
\end{Lemma}

\begin{proof}
This follows from Corollary~\ref{CInverse2R} and Lemmas~\ref{tworowschar0},~\ref{product} and~\ref{LM30}.
\end{proof}

\begin{Lemma} \label{L060125-5}
Let $\la\in\RP_p(n)$ be of the form $\la=((2p)^a,2p-1,p+1,p^{2b},p-1,1)$ with $a,b\in\Z_{\geq 0}$. If $p=3$, we assume in addition that  $b> \de_{2\mid a}$. Then $D(\la)\da_{\ts_{n/2,n/2}}$ has at least three non-isomorphic composition factors. 
\end{Lemma}
\begin{proof}
Let $d:=(a+1)p$. By Lemma~\ref{L2RowExplicit}, we have $\la=\mu_{n,d}$. 
By Lemma \ref{tworowscharp}, we conclude that 
$$D(\mu_{n/2,\lfloor d/2\rfloor-1},\mu_{n/2,\lceil d/2\rceil+1}),\   D(\mu_{n/2,\lfloor d/2\rfloor},\mu_{n/2,\lceil d/2\rceil}),\  
D(\mu_{n/2,\lfloor d/2\rfloor+1},\mu_{n/2,\lceil d/2\rceil-1})
$$ 
are composition factors of $D(\mu_{n,d})\da_{\ts_{n/2,n/2}}$. 
\end{proof}

\subsection{Regularization}\label{SSReg} 
In \cite[\S2]{BK3}, certain subsets of the nodes are introduced which are called {\em ladders}. Recall the notation $\ttres(s)$ from \S\ref{SSAddRem}. 
Then, for a positive integer $s$, the {\em $s^{\mathrm{th}}$
ladder} $L_s$ is defined as follows. If $\ttres(s)\neq  0$ then
$L_s:= \{(r,s- (r- 1)p\mid  1\leq r\leq\lceil{s/p}\rceil\}$.
If res $\ttres(s)=0$ then $s= mp$ or $mp + 1$ for some $m \in \Z$, and in this case we set 
$L_s:=\{ r,mp- (r-1)p \mid 1]\leq r\leq m\} \cup \{(r,mp + 1- (r- 1)p\mid  1 \leq r \leq m + 1\}$.

Given $\la\in\Par_p(n)$, we identify as usual $\la$ with its Young diagram. Then the {\em regularization} $\la^\Reg$ of $\la$ is the Young diagram obtained from $\la$ by moving the nodes along the ladders to the left as far as they can go, see \cite[\S2]{BK3} for more details. It is always the case that $\la^\Reg\in \RP_p(n)$, see \cite[Proposition 2.1]{BK3}. Moreover, $\la=\la^\Reg$ if and only if $\la\in \RP_p(n)$.

The following `leading composition factor' result follows from  \cite[Theorem 4.4]{BK3} using 
\cite[Theorem 10.8]{BK2} and \cite[Theorem 10.4]{BK4}:

\begin{Lemma}\label{L020819}
Let $\la\in\Par_0(n)$, and denote by $\bar S(\la)$ a reduction modulo $p$ of the irreducible $\C\ts_n$-supermodule $S(\la)$. Then, in the Grothendieck group, we have 
$$
[\bar S(\la)]=2^{(h_p(\la)+a_0(\la)-a_p(\la^\Reg))/2}[D(\la^\Reg)]+\sum_{\mu\lhd\la^\Reg}[D(\mu)].
$$
\end{Lemma}


The next lemma shows how to compute the regularisation of a partition $\la\in\RP_0(n)$, provided parts are far enough. In it, for every $m,r\in\Z_{>0}$, we denote by $\balpha_m^i$ the set of nodes obtained by shifting the nodes of the Young diagram $\balpha_m$, defined in (\ref{bs}),  to the right by $p(i-1)$ columns:
$$
\balpha_m^i=\{(r,s+p(i-1))\mid (r,s)\in\balpha_m\}. 
$$


\begin{Lemma}\label{reglargedif}
Let $\la\in\Par_0(n)$ with $\la_r-\la_{r+1}\geq p+\de_{p\mid\la_r}$ for all $r=1,2,\dots,h(\la)-1$. Then
\[\la^\Reg=\sum_{r=1}^{h(\la)}\balpha_{\la_r}=\bigsqcup_{r=1}^{h(\la)}\balpha_{\la_r}^r.\]
\end{Lemma}

\begin{proof}
For $r=1,\dots,h(\la)$, let $H_r:=\{(r,s)\mid s\in\Z_{>0}\}$. Observe that 
\begin{equation}\label{E150725}
|\la\cap H_r\cap L_s|=|\balpha_{\la_r}^r\cap L_s| \qquad(\text{for all $s$}).
\end{equation}
Moreover, since the rows of each $(\balpha_{\la_r})$ have length at most $p$, we have $\balpha_{\la_r}^r\cap \balpha_{\la_t}^t=\varnothing$ for all $1\leq r\neq t\leq h(\la)$. Hence, setting $\mu:=\sum_{r=1}^{h(\la)}\balpha_{\la_r}$ and $\nu:=\bigsqcup_{r=1}^{h(\la)}\balpha_{\la_r}^r$, for every $a\in\Z_{>0}$, we have $|H_a\cap\mu|=|H_a\cap\nu|$. So, taking into account (\ref{E150725}), it suffices to prove that $\nu$ is a partition in $\RP_p(n)$. This follows from the definition (\ref{bs}). 
\end{proof}

\subsection{Cyclic defect}
Brauer trees of blocks of $\ts_n$ with cyclic defect were described explicitly in \cite[Theorem 4.4]{Mu}. We will need the following  very special results which follows easily from that description.

\begin{Lemma} \label{LMuller} 
We have 
\begin{enumerate}
\item[{\rm (i)}] $\bar S(p+1,2,1)\cong D(p+1,2,1)$.
\item[{\rm (ii)}] If $p\geq 7$ then  $[\bar S(p,2,1):D(p-3,3,2,1)]\neq 0$.
\item[{\rm (iii)}] $\bar S(2p+1,2,1)\cong D(p+2,p+1,1)$. 
\item[{\rm (iv)}] $\bar S(p+2,2,1)\cong D(p+2,2,1)$.
\end{enumerate}
\end{Lemma}
\begin{proof}
(i) Since $\bar S(p+1,2,1)$ has defect zero, it is irreducible and then $\bar S(p+1,2,1)\cong D(p+1,2,1)$ by Lemma~\ref{L020819}. 

(ii) Assume first that $p\geq 11$. By \cite[Theorem 4.4]{Mu} we have that $\bar S((p-3,3,2,1),0)$ has 2 composition factors, one shared with $\bar S((p,2,1);\pm)$ and one with $\bar S((p-4,4,2,1),0)$. We know from Lemma \ref{L020819} that $D(p-3,3,2,1)$ is a composition factor of $\bar S(p-3,3,2,1)$ but not of $\bar S(p-4,4,2,1)$. So $[\bar S(p,2,1):D(p-3,3,2,1)]\neq 0$.

If $p=7$ then $\bar S((4,3,2,1),0)$ has only 1 composition factors and this composition factor is shared with $\bar S((7,2,1);\pm)$. Since $D(4,3,2,1)$ is a composition of $\bar S(4,3,2,1)$, (ii) holds also in this case.

(iii) It follows from \cite[Theorem 4.4]{Mu} that $\bar S(2p+1,2,1)$ is irreducible. Since $(2p+1,2,1)^\Reg=(p+2,p+1,1)$ the claim follows from Lemma~\ref{L020819}. 

(iv) This is a defect zero case. 
\end{proof}

\section{Branching for spin representations}

\subsection{Modular branching rules}
\label{SSBr}
For $1 \leq r<s \leq n$, we define 
$$[r,s]:= (-1)^{s-r-1} \ct_{s-1}\cdots\ct_{r+1}\ct_r\ct_{r+1} \cdots\ct_{s-1}\in\cT_n,$$ 
and for $s=1,\dots,n$, let 
$
\cm_s:=\sum_{r=1}^{s-1}[r,s]\in\cT_n 
$, 
see \cite[\S13.1]{KBook}. Then the elements $\cm_1^2,\dots, \cm_n^2\in\cT_n$ commute, and for a $\cT_n$-supermodule $V$ and a tuple $\bi=(i_1,\dots,i_n)\in I^n$, 
we consider the simultaneous generalized eigenspace
$$
V_\bi:=\{v\in V\mid (\cm_r^2-i_r(i_r-1)/2)^N=0\ \text{for $N\gg0$ and $r=1,\dots,n$}\}.
$$ 

We consider the set of orbits $\Theta_n:=I^n/\s_n$, where the symmetric group $\s_n$ acts on the $n$-tuples $I^n$ by place permutations. Let $\theta\in\Theta_n$. Pick $\bi\in\theta$ and for every $j\in I$ define $\theta_j:=\sharp\{r\mid 1\leq r\leq n\ \text{and}\ i_r=j\}$. Clearly $\theta_i$ is well-defined and the tuple $(\theta_0,\theta_1,\dots,\theta_\ell)$ determines $\theta$. 
Fix $i\in I$. Define $\theta^{+i}\in\Theta_{n+1}$ from $\theta^{+i}_j=\theta_j+\de_{i,j}$. If $\theta_i>0$, define also $\theta^{-i}\in\Theta_{n-1}$ from $\theta^{-i}_j=\theta_j-\de_{i,j}$. 

Given $\theta\in \Theta_n$ and a $\cT_n$-supermodule $V$, we define $V_\theta:=\bigoplus_{\bi\in\theta}V_\bi$. Then $V=\bigoplus_{\theta\in\Theta_n}V_\theta$ as $\T_n$-supermodules, see \cite[Corollary 22.3.10]{KBook}. The summand $V_\theta$ is actually a superblock component of $V$. In particular, for an irreducible $\cT_n$-supermodule $L$, we always have $L=L_\theta$ for some unique $\theta\in\Theta_n$.

Let $V$ be any $\cT_n$-supermodule with $V=V_\theta$ for some $\theta\in\Theta_n$. 
We define the {\em $i$-induction of $V$} to be the $\cT_{n+1}$-supermodule $
\Ind_i V:=(\Ind^{\cT_{n+1}}_{\cT_{n}}V)_{\theta^{+i}}.
$ 
If $\theta_i>0$, we define the {\em $i$-restriction of $V$} to be the $\cT_{n-1}$-supermodule $
\Res_i V:=(\Res^{\cT_{n}}_{\cT_{n-1}}V)_{\theta^{-i}},
$ 
and we set $\Res_i V=0$ of $\theta_i=0$. For a general $\cT_n$-supermodule $V$, we define $\Res_i V:=\bigoplus_{\theta\in\Theta_n}\Res_i (V_\theta)$ and $\Ind_i V:=\bigoplus_{\theta\in\Theta_n}\Ind_i (V_\theta)$. 
By \cite[Lemma 22.3.12]{KBook}, we have
$$
\Res^{\cT_{n}}_{\cT_{n-1}}V=\bigoplus_{i\in I}\Res_i V\quad\text{and}\quad
\Ind^{\cT_{n+1}}_{\cT_{n}}V=\bigoplus_{i\in I}\Ind_i V
$$

For irreducible $\cT_n$-supermodules, we have a lot of useful information about $i$-induction and $i$-restriction.  Recall the combinatorial notions of \S\ref{SSAddRem}.

\begin{Lemma}\label{branching}\cite[Theorems 22.3.4, 22.3.5]{KBook}
Let $\la\in\RP_p(n)$ and $i\in I$. There exist a self-dual $\cT_{n-1}$-supermodule $e_i D(\la)$ and a self-dual $\cT_{n+1}$-supermodule $f_i D(\la)$, unique up to isomorphism, such that 
$$\Res_iD(\la)\cong (e_iD(\la))^{\oplus (1+\de_{i\neq 0}a_p(\la))}\quad\text{and}\quad\Ind_iD(\la)\cong (f_iD(\la))^{\oplus (1+\de_{i\neq 0}a_p(\la))}.$$
Moreover,  $e_iD(\la)\neq 0$ if and only if $\eps_i(\la)>0$, and 
$f_iD(\la)\neq 0$ if and only if $\phi_i(\la)>0$. 
Further, if $\eps_i(\la)>0$ (resp. $\phi_i(\la)>0$) then:
\begin{enumerate}[\rm(i)]
\item $\soc e_iD(\la)\cong\head e_iD(\la)\cong D(\tilde e_i\la)$ (resp. $\soc f_iD(\la)\cong \head f_iD(\la)\cong D(\tilde f_i\la)$), and $\eps_i(\tilde e_i\la)=\eps_i(\la)-1$ (resp. $\phi_i(\tilde f_i\la)=\phi_i(\la)-1$);

\item $[e_iD(\la):D(\tilde e_i\la)]=\eps_i(\la)$ (resp. $[f_iD(\la):D(\tilde f_i\la)]=\phi_i(\la)$);

\item if $[e_iD(\la):D(\mu)]\neq 0$ (resp. $[f_iD(\la):D(\mu)]\neq 0$) and $\mu\neq\tilde e_i\la$ (resp. $\mu\neq\tilde f_i\la$) then $\eps_i(\mu)<\eps_i(\la)-1$ (resp. $\phi_i(\mu)<\phi_i(\la)-1$); 

\item we have even isomorphisms of superspaces \,$\End_{\cT_{n-1}}(e_iD(\la))\simeq\End_{\cT_{n-1}}(D(\tilde e_i\la))^{\oplus \eps_i(\la)}$
(resp. $\End_{\cT_{n+1}}(f_iD(\la))\simeq\End_{\cT_{n+1}}(D(\tilde f_i\la))^{\oplus \phi_i(\la)}$).
\end{enumerate}
\end{Lemma}

\begin{Lemma}\label{n-1}
Let $\la\in\RP_p(n)$ and $i\in I$. Then we have:
\begin{enumerate}
\item[{\rm (i)}] $\dim\End_{\T_{n-1}}(\Res_i D(\la))=\eps_i(\la)(1+\de_{i\neq 0})(1+a_p(\la))$;
\item[{\rm (ii)}] $\dim\End_{\T_{n-1}}(D(\la)\da_{\ts_{n-1}})=\big(\eps_0(\la)+2\eps_1(\la)+\dots+2\eps_\ell(\la)\big)(1+a_p(\la))$. 
\end{enumerate}
\end{Lemma}
\begin{proof}
(i) By Lemma \ref{branching}, 
\begin{align*}
\dim\End_{\T_{n-1}}(\Res_i D(\la))&=\dim\End_{\T_{n-1}}((e_iD(\la))^{\oplus (1+\de_{i\neq 0}a_p(\la))})
\\
&=(1+\de_{i\neq 0}a_p(\la))^2\dim\End_{\T_{n-1}}(e_iD(\la))
\\
&=\eps_i(\la)(1+\de_{i\neq 0}a_p(\la))^2\dim\End_{\T_{n-1}}(D(\tilde e_i\la))
\\
&=\eps_i(\la)(1+\de_{i\neq 0}a_p(\la))^2(1+a_p(\tilde e_i\la))
\\
&=\eps_i(\la)(1+\de_{i\neq 0})(1+a_p(\la)),
\end{align*}
where we have used (\ref{EEndDDim}) for the fourth equality and Lemma~\ref{LAGa} for the last equality.

(ii) follows from (i), since  $\Hom_{\T_{n-1}}(\Res_i D(\la),\Res_j D(\la))=0$ for $i\neq j$. 
\end{proof}

For the powers of $i$-induction and $i$-restriction on irreducible modules we have the following information, which comes from  \cite[Lemma 22.3.15]{KBook} and Lemma \ref{branching}:

\begin{Lemma}\label{divpowers}
Let $\la\in\RP_p(n)$, $i\in I$, and $r$ be a positive integer. There exist a $\cT_{n-1}$-supermodule $e_i D(\la)$ and a $\cT_{n+1}$-supermodule $f_i D(\la)$, unique up to isomorphism, such that 
\begin{align*}
(\Res_i)^r D(\la)&\cong (e_i^{(r)}D(\la))^{\oplus (r!(1+\de_{i\neq 0})^{\lfloor (r+a_p(\la))/2\rfloor})},\\
(\Ind_i)^r D(\la)&\cong (f_i^{(r)}D(\la))^{(\oplus r!(1+\de_{i\neq 0})^{\lfloor (r+a_p(\la))/2\rfloor})}.
\end{align*}
Moreover, $e_i^{(r)}D(\la)\neq 0$ (resp. $f_i^{(r)}D(\la)\neq 0$) if and only if $\eps_i(\la)\geq r$ (resp. $\phi_i(\la)\geq r$). 
In this case, we have $[e_i^{(r)}D(\la):D(\tilde e_i^r\la)]=\binom{\eps_i(\la)}{r}$ (resp. $[f_i^{(r)}D(\la):D(\tilde f_i^r\la)]=\binom{\phi_i(\la)}{r}$), 
$\eps_i(\tilde e_i^r\la)=\eps_i(\la)-r$ (resp. $\phi_i(\tilde f_i^r\la)=\phi_i(\la)-r$), 
and all other composition factors $D(\mu)$ of $e_i^{(r)}D(\la)$ (resp. $f_i^{(r)}D(\la)$) satisfy $\eps_i(\mu)<\eps_i(\la)-r$ (resp. $\phi_i(\mu)<\phi_i(\la)-r$).
\end{Lemma}

Lemma \ref{branching} gives some information on composition factors of $\Res_iD(\la)$ and $\Ind_iD(\la)$. The next result, which is a rather special case of \cite[Theorem A]{KS}, improves on this.

\begin{Lemma}\label{normal}\cite[Theorem A]{KS}
Let $\la\in\RP_p(n)$. If $A$ is a properly $i$-removable $i$-normal  node of $\la$ and $\la_A\in\RP_p(n-1)$, then $[e_iD(\la):D(\la_A)]\neq 0$.
\end{Lemma}

The following results will be used 
when studying restrictions from $\ts_n$ to $\ts_{n-2}$ and $\ts_{n-2,2}$.

\begin{Lemma}\label{L051218_3}
Let $i,j\in I$ with $i\not=j$. If $V$ is a $\cT_{n}$-supermodule then $\Ind_j\Res_i V\cong \Res_i\Ind_j V$.
\end{Lemma}

\begin{proof}
We may assume that $V=V_\theta$ for some $\theta\in\Theta_n$. Considering $V$ as an $\F\ts_n$-module, by Mackey's theorem, we have $V\ua^{\ts_{n+1}}\da_{\ts_n}\cong V\da_{\ts_{n-1}}\ua^{\ts_n}\oplus V$. Hence $$\Res_{\cT_n}^{\cT_{n+1}}\Ind_{\cT_n}^{\cT_{n+1}}V\cong \Ind_{\cT_{n-1}}^{\cT_{n}}\Res_{\cT_{n-1}}^{\cT_{n}}V\oplus V.$$
Let $\eta:=(\theta^{-i})^{+j}$. Then $(\Res_{\cT_n}^{\cT_{n+1}}\Ind_{\cT_n}^{\cT_{n+1}}V)_\eta\cong (\Ind_{\cT_{n-1}}^{\cT_{n}}\Res_{\cT_{n-1}}^{\cT_{n}}V)_\eta\oplus V_\eta$. It remains to notice that $(\Ind_{\cT_{n-1}}^{\cT_{n}}\Res_{\cT_{n-1}}^{\cT_{n}}V)_\eta\cong \Ind_j\Res_i V$, $(\Res_{\cT_n}^{\cT_{n+1}}\Ind_{\cT_n}^{\cT_{n+1}}V)_\eta\cong \Res_i\Ind_j V$, and $V_\eta=0$.   
\end{proof}

\begin{Lemma}\label{L051218_4}
Let $i,j\in I$ with $i\not=j$. If $\la\in\RP_p(n)$ and $\epsilon_j(\la)>0$, then $\epsilon_i(\tilde e_j\la)\geq\epsilon_i(\la)$.
\end{Lemma}

\begin{proof}
This is well-known and follows easily from the definitions of \S\ref{SSAddRem}. Alternatively, noting that $\tilde f_j\tilde e_j\la=\la$ and using Lemmas \ref{branching} and \ref{L051218_3}, we get 
\[0\not= (\Res_i)^{\epsilon_i(\la)} D(\la)\subseteq (\Res_i)^{\epsilon_i(\la)} \Ind_j D(\tilde e_j\la)\cong \Ind_j(\Res_i)^{\epsilon_i(\la)} D(\tilde e_j\la).\]
In particular $(\Res_i)^{\epsilon_i(\la)} D(\tilde e_j\la)\neq 0$. So the lemma follows from Lemma \ref{divpowers}.
\end{proof}

\begin{Lemma}\label{L101218_2}
Let $\mu\in\RP_p(n-1)$, $i\in I$ and $\eps_i(\mu)>0$. There exists a $\T_{n-2,2}$-supermodule $e_iD(\mu)\circledast D(2)$ such that the following holds:
\begin{enumerate}[\rm(i)] 
\item if $D(\tilde e_i\mu)$ is of type $\Mtype$ then $e_iD(\mu)\circledast D(2)\cong 
e_iD(\mu)\boxtimes D(2)$;

\item if $D(\tilde e_i\mu)$ is of type $\Qtype$ then $(e_iD(\mu)\circledast D(2))^{\oplus 2}\cong
e_iD(\mu)\boxtimes D(2) 
$;

\item $[e_iD(\mu)\circledast D(2):D(\tilde e_i\mu,(2))]=\eps_i(\mu)$;

\item $\soc(e_iD(\mu))\circledast D(2))\cong\head (e_iD(\mu)\circledast D(2))\cong D(\tilde e_i\mu,(2))$;

\item $\dim\End_{\T_{n-2,2}}(e_iD(\mu)\circledast D(2))=\frac{2\eps_i(\mu)}{1+a_p(\tilde e_i\mu)}$;

\item $e_iD(\mu)\circledast D(2)$ is self dual.
\end{enumerate}
\end{Lemma}

\begin{proof}
From Lemma~\ref{branching}, $e_iD(\mu)$ is a self-dual $\T_{n-2}$-supermodule with $\soc e_iD(\mu)\cong\head e_iD(\mu)\cong D(\tilde e_i\mu)$, $[e_iD(\mu): D(\tilde e_i\mu)]=\eps_i(\la)$ and 
$$\End_{\T_{n-2}}(e_iD(\mu))\simeq\End_{\T_{n-2}}(D(\tilde e_i\mu))^{\oplus \eps_i(\la)}.$$ 
If $D(\tilde e_i\mu)$ is of type $\Mtype$, we set $e_iD(\mu)\circledast D(2):= 
e_iD(\mu)\boxtimes D(2)$, so (i) holds. If $D(\tilde e_i\mu)$ is of type $\Qtype$, then by Lemma~\ref{L071218_3}, $e_iD(\mu)$ admits an odd involution, so Lemma~\ref{L071218_4} yields a $\T_{n-2,2}$-supermodule $e_iD(\mu)\circledast D(2)$ such that (ii) holds. Part (iii) also follows. 

By Lemmas \ref{branching} and \ref{product},
\[\soc(e_iD(\mu)\boxtimes D(2))\cong \soc(e_iD(\mu))\boxtimes D(2)\cong D(\tilde e_i\mu)\boxtimes D(2)\cong D(\tilde e_i\mu,(2))^{\oplus 1+a_p(\tilde e_i \mu)},
\]
and similarly $\hd(e_iD(\mu)\boxtimes D(2))\cong D(\tilde e_i\mu,(2))^{\oplus 1+a_p(\tilde e_i(\mu))}.$
So (iv) follows from (i) and (ii).

To prove (v), using (i) and (ii), we get
\begin{align*}
\dim\End_{\T_{n-2,2}}(e_iD(\mu)\circledast D(2))&=(1+a_p(\tilde e_i\mu))^{-2}\dim\End_{\T_{n-2,2}}(e_iD(\mu)\boxtimes D(2))\\
&=(1+a_p(\tilde e_i\mu))^{-2}\dim\End_{\T_{n-2}}(e_iD(\mu))\cdot \dim\End_{\T_2}(D(2))\\
&=(1+a_p(\tilde e_i\mu))^{-2}\cdot \eps_i(\mu)\cdot \dim\End_{\T_{n-2}}(D(\tilde e_i\mu))\cdot 2
\\&=
\frac{2\eps_i(\mu)}{1+a_p(\tilde e_i\mu)}.
\end{align*}

Finally, note that $e_iD(\mu)\boxtimes D(2)$ is self-dual. This implies (vi) by (i),(ii) and Krull-Schmidt. 
\end{proof}

\subsection{Composition factors of some explicit restrictions}

\begin{Lemma}\label{branchingend}
Let $\la\in\RP_p(n)$, $1\leq r\leq h(\la)$, and $\la_r=ap+b$ for integers $a,b$ with $0\leq b<p$. Define
\[\mu:=\left\{\begin{array}{ll}
(\la_1,\ldots,\la_r,(a-1)p+b,(a-2)p+b,\ldots,b)&b>0,\\
(\la_1,\ldots,\la_r,(a-1)p+1,(a-2)p+1,\ldots,1)&b=0.
\end{array}\right.\]
Then $D(\mu)$ is a composition factor of $D(\la)\da_{\ts_{|\mu|}}$.
\end{Lemma}

\begin{proof}
We say that $A$ is the special node for $\la$ if $A$ is the lowest  removable node of $\la$ such that $\la_A\in\RP_p(n-1)$. Note that the special node is always normal, so $D(\la_A)$ is a composition factor of $D(\la)\da_{\ts_{n-1}}$ by 
Lemma \ref{normal}. It remains to observe that one can obtain $\mu$ from $\la$ by successively removing special nodes.
\end{proof}

Recall from Lemmas~\ref{LM30} and \ref{LM3} the set of partitions  $\TR_p(n)$ which label the composition factors of reductions modulo $p$ of the irreducible $\C\ts_n$-supermodules  labeled by  two-row partitions.  

\begin{Lemma}\label{rectworowsprel}
Suppose that $p\geq 5$ and $n\geq 10$. Let $\nu\in\TR_p(n-1)$, and suppose that $B\neq (1,2p+1)$ is an $i$-cogood node for $\nu$ such that $\la:=\nu^B\not\in \TR_p(n)$. 
Then $D(\la)\da_{\ts_{n-1}}$ has a composition factor $D(\mu)$ with $\mu\in \RP_p(n-1)\setminus \TR_p(n-1)$. 
\end{Lemma}
\begin{proof}
By Lemma~\ref{normal}, it suffices to show that for some $j$ there is a properly $j$-removable $j$-normal node $A$ of $\la$ such that $\la_A\in \RP_p(n-1)\setminus \TR_p(n-1)$. We go through different cases and show that most of the time this can be done. When not, we apply some other tricks. 

{\em Case 1:} $\nu=\balpha_{n-1}$. By \cite[Theorem 3.6(iii)]{KT} and Lemma~\ref{branching}, we have $\la\in\{\balpha_{n},\bbeta_{n}\}\subseteq \TR_p(n)$, giving a contradiction.

{\em Case 2:} $\nu=\balpha_{n-1-k}+\balpha_k$ for $0<2k\leq n-1-p-\de_{p\mid k}$. There are four subcases depending on whether $p$ divides $k$ or $p$ divides $n-1-k$. We provide details for the most difficult case where $p\nmid k$ and $p\nmid (n-1-k)$. In this case we have $\nu=((2p)^a,p+c,p^b,d)$ with $0< c,d<p$, and $b>0$ if $c>d$. Then one of the following happens: (a) $c<p-1$ and $\la=((2p)^a,p+c+1,p^b,d)$, (b) $b>0$, $c>1$ and $\la=((2p)^a,p+c,p+1,p^{b-1},d)$, (c) $d<p-1$ and $\la=((2p)^a,p+c,p^b,d+1)$, (d) $d>1$ and $\la=((2p)^a,p+c,p^b,d,1)$.  The cases (a) and (c) are ruled out because in those cases we have $\la\in\TR_p(n)$ or $\la\not\in\RP_p(n)$ (this last case happens if $c=d$ and $b=0$).

In the case (b), by assumption, $B=(a+2,p+1)$ is $0$-cogood for $\nu$, so it is conormal for $\nu$, whence $c=p-1$ or $a=0$. If $c=p-1$, we have $\la\in\TR_p(n)$, so this case is ruled out. Thus $a=0$. If $d\geq 2$ then we can take the normal node $A$ of $\la$ to be $(2+b,d)$. If $d=1$ and $b\geq 2$, we can take $A=(b+1,p)$. Let $b=d=1$, i.e. $\la=(p+c,p+1,1)$. If $c\geq 3$, we can take $A=(1,p+c)$. If $c=2$, we have $D(\la)\cong \bar S(2p+1,2,1)$ by Lemma~\ref{LMuller}(iii). So in the Grothendieck group we have  
$
[D(\la)\da_{\ts_{p+5}}]=[\bar S(2p+1,2,1)\da_{\ts_{p+5}}]
$ contains $[\bar S(p+2,2,1)]$ as a summand thanks to Lemma~\ref{LBr0}. But $[\bar S(p+2,2,1)]=[D(p+2,2,1)]$ by Lemma~\ref{LMuller}(iv). But $(p+2,2,1)\not\in\TR_p(p+5)$, so by Lemma~\ref{LRedRes}, $[D(\la)\da_{\ts_{n-1}}]$ must have a composition factor $D(\mu)$ with $\mu\not\in\TR_p(n-1)$. 

In the case (d), we must have $d<p-1$ since $d=p-1$ implies $\la\in\TR_p(n)$. If $3\leq d\leq p-2$ we can take the normal node $A$ of $\la$ to be $(a+b+1,d-1)$. If $d=2$, then $b=0$ since otherwise $(a+b,p)$ is $0$-normal for $\la$ which contradicts the assumption that $B=(a+b+2,1)$ is $0$-cogood for $\nu$. Then $c\leq b=2$. We now deduce that $c=2$ since otherwise $(a+1,p+c)$ is $0$-normal for $\la$ which contradicts the assumption that $B$ is $0$-cogood for $\nu$. Moreover, now we also deduce that $a=0$ since otherwise $(a,p)$ is $0$-normal for $\la$ which contradicts the assumption that $B$ is $0$-cogood for $\nu$. Thus $\la=(p+2,2,1)$ and now we can take $A:=(1,p+2)$. 

{\em Case 3:} $\nu\in \TR'_p$. Then one of the following happens:   
(3.1) $\nu=(p^a,b,c)$ with $1= c<b\leq p-2$ or $2\leq c<b\leq p-1$; (3.2) $\nu=(p^a,p-1,b,1)$ with $2\leq b\leq p-2$; (3.3) $\nu=(p^a,p-1,p-2,2,1)$; (3.4) $\nu=(p^a,p-1,p-2,2)$; (3.5) $\nu=(p^a,p-2,2,1)$. We provide details for the most difficult case (3.1). Here there are four cases for $\la$: (3.1.a) $a\geq 1$ and $\la=(p+1,p^{a-1},b,c)$; (3.1.b) $\la=(p^{a},b+1,c)$; (3.1.c) $c\leq b-2$ and $\la=(p^{a},b,c+1)$; (3.1.d) $c\geq 2$ and $\la=(p^{a},b,c,1)$. The cases (3.1.b) and (3.1.c) are ruled out because in those cases we have $\la\in\TR_p(n)$. The case (3.1.d) can be ruled out by using that $(a+3,1)$ is conormal in $\nu$ and finding some appropriate normal node $A$ whenever $\la\not\in\TR_p(n)$.

In the case (3.1.a), if either $c\geq 3$, or $c=2$ and $b\leq p-2$, then we can take the normal node $A$ of $\la$ to be $(a+2,c)$. If either $c=2$, $b=p-1$ or $c=1$, $3\leq b\leq p-2$, then we can take the normal node $A$ of $\la$ to be $(a+1,b)$. If $c=1$ and $b=p-1$ then $\la\in\TR_p(n)$. If $c=1$,  $b=2$ and $a\geq 2$,  then we can take the normal node $A$ of $\la$ to be $(a,p)$. Finally, if $c=1$, $b=2$ and $a=1$, then $\la=(p+1,2,1)$ and $n=p+4$. But $[\bar S(p+1,2,1)]=[D(p+1,2,1)]$ by Lemma~\ref{LMuller}(i). So 
$
[D(\la)\da_{\ts_{n-1}}]=[\bar S(p+1,2,1)\da_{\ts_{n-1}}]
$
contains $[\bar S(p,2,1)]$ which contains $[D(p-3,3,2,1)]$ by Lemma~\ref{LMuller}(iii) (note that $p\geq 7$ since by assumption we have $n\geq 10$). Since $(p-3,3,2,1)\not\in \TR_p(n-1)$, we are done. 
\end{proof}

\begin{Lemma}\label{rectworows}
Let $n\geq 13$ and $\la\in\RP_p(n)\setminus\TR_p(n)$ with $6\leq\la_1\leq 2p$. Then $[D(\la)\da_{\ts_{n-1}}:D(\mu)]\neq 0$ for some 
$\mu\in\RP_p(n-1)\setminus \TR_p(n-1)$ with $6\leq\mu_1\leq 2p$.
\end{Lemma}

\begin{proof}
If $p=3$ then $\la_1= 6$. If $\la$ has $r$ nodes in the first three columns, note that 
$\la=\balpha_{r}+\balpha_{n-r}$ and $r\geq n-r+3+\de_{3\mid r}$. So $\la\in\TR_3(n)$ by Lemma~\ref{LM3}. Thus we may assume that $p\geq 5$.

Suppose first that $p\geq 7$ and $n\leq 16$. We use Lemma~\ref{LM3} to list all partitions $\la$ in $\RP_p(n)\setminus \RP_p(n)$ for $n=13,14,15,16$, as well as all partition in $\RP_p(12)\setminus \RP_p(12)$, and check that, with one exception, every such $\la$ has a normal node $A$ such that $\la_A\in \RP_p(n-1)\setminus \RP_p(n-1)$ and $\la_A$ has the first row of length at least $6$; then application of Lemma \ref{normal} completes the proof for the non-exceptional cases. 
The only exception is $\la=(12,2,1)$ for $p=11$. In this case $D(12,2,1)\cong\bar S(12,2,1)$ by Lemma~\ref{LMuller}(i), so, using Lemma~\ref{LBr0}, in the Grothendieck group we have 
$[D(12,2,1)\da_{\ts_{14}}]=[\bar S(11,2,1)]+[\bar S(12,2)]$. As $D(8,3,2,1)$ is a composition factor of $S(11,2,1)$ by Lemma~\ref{LMuller}(ii)  and $(8,3,2,1)\not\in\TR_{11}(14)$, we can take $\mu=(8,3,2,1)$ in this case.

If $p\geq 7$ and $n\geq 17$, then $\mu_1\geq 6$ for every $\mu\in\RP_p(n-1)$. If $p=5$ and $\mu\in\RP_p(n-1)$ with $\mu_1\leq 5$ then $\mu\in\TR_p(n-1)$. So in either case it is enough to prove that there exists $\mu\in\RP_p(n-1)\setminus\TR_p(n-1)$ such that $D(\mu)$ is a composition factor of $D(\la)\da_{\ts_{n-1}}$.

If $\tilde e_j \la\not\in \TR_p(n-1)$ for some $j\in I$ then we are done by Lemma \ref{branching}. So we may assume that for every $j\in I$ either  $\tilde e_j\la=0$ or  $\tilde e_j\la\in \TR_p(n-1)$. As  $\tilde e_i\la\neq 0$ for some $i\in I$, denoting $\nu:=\tilde e_i\la$ we then have $\la=\tilde f_i\nu$ for $\nu\in\TR_p(n-1)$. In other words, 
$\la:=\nu^B$ for an $i$-cogood node $B$ for $\nu$. Since $\la_1\leq 2p$ by assumption, we have $A\neq (1,2p+1)$.
We can now apply Lemma~\ref{rectworowsprel}. 
\end{proof}

\begin{Lemma} \label{L060125-1} 
Let $p\geq 5$, $b\in\Z_{>0}$, $n=p(b+2)$ and $\la=(p^{b},p-1,p-2,2,1)\in\RP_p(n)$.
\begin{enumerate}
\item[{\rm (i)}] $D(p^{b-1},p-1,p-2,2,1)$ is a composition factor of $D(\la)\da_{\ts_{n-p}}$.
\item[{\rm (ii)}] Suppose that $b$ is even, and exclude the case $(b,p)=(0,5)$. Then the supermodule $D(\la)\da_{\ts_{n/2}}$ has at least three non-isomorphic composition factors. 
\end{enumerate}
\end{Lemma}
\begin{proof}
(i) Recursively remove nodes in $\la$ from columns $1, 2$ then $p-2,p-3,\dots,3$, then $p-1,p$, observing that on each step the removed node is normal and using Lemma~\ref{normal}. 

(ii) Let $a:=b/2$ and $\nu:=(p^a,p-1,3,1)\in\RP_p(n/2+3)$. By removing three consecutive normal nodes and using Lemma \ref{normal}, $D(p^a,p-1,1)$, $D(p^a,p-2,2)$, and either $D(p^{a-1},p-1,p-2,2,1)$ if $a>0$, or $D(p-3,3)$ if $a=0$ and $p>5$, are composition factors of $D(\nu)\da_{\ts_{n/2}}$. So it suffices to show that $D(\nu)$ is a composition factor of $D(\la)\da_{\ts_{n/2+3}}$. 

Denote $\mu:=(p^a,p-1,p-2,2,1)$ and 
use (i) to deduce that $D(\mu)$ is a composition factor of $D(\la)\da_{\ts_{n-ap}}$. Now, starting with $\mu$, recursively remove nodes from columns $1, 2$ then $p-2,p-3,\dots,4$ to get $\nu$ (if $p=5$ remove only nodes from columns 1 and 2). Since on each step we removed a normal node, by Lemma~\ref{normal}, we have that $D(\nu)$ is a composition factor of $D(\mu)\da_{\ts_{n/2+3}}$. 
%
\end{proof}

\begin{Lemma} \label{L060125-2}
Let $p=3$, $b\in\Z_{> 0}$, $n=6(b+3)$ and $\la=(6^{b},5,4,3^2,2,1)\in\RP_3(n)$. 
\begin{enumerate}
\item[{\rm (i)}] $D(6^{b-1},5,4,3^2,2,1)$ is a composition factor of $D(\la)_{\ts_{n-6}}$.
\item[{\rm (ii)}] If $b$ is even then the supermodule $D(\la)\da_{\ts_{n/2}}$ has at least three non-isomorphic composition factors. 
\end{enumerate}
\end{Lemma}
\begin{proof}
(i) Recursively remove nodes from columns $1, 2, 4,3,5,6$ of $\la$, observing that on each step the removed node is normal and using Lemma~\ref{normal}. 

(ii) Let $a:=b/2$, denote $\mu=(6^a,5,4,3^2,2,1)$ and 
use (i) to deduce that $D(\mu)$ is a composition factor of $D(\la)\da_{\ts_{n-6a}}$. Let $\nu:=(6^{a},5,3^2,2,1)$. Recursively remove nodes from $\mu$ in columns $1, 2, 4,3$ to get $\nu$. Since on each step the removed node is normal, by Lemma~\ref{normal} we have that $D(\nu)$ is a composition factor of $D(\mu)\da_{\ts_{n/2+5}}$. Applying Lemma \ref{normal} again, we have that $D(6^{a/2},4,3,2)$, $D(6^{a/2},5,3,1)$ and either $D(3^2,2,1)$ if $a=0$ or $D(6^{a/2-1},5,4,3,2,1)$ if $a>0$ are composition factors of $D(\nu)\da_{\ts_{n/2}}$.
\end{proof}

\begin{Lemma} \label{L060125-3}
Let $p=3$, $b\in\Z_{> 0}$, $n=6(b+2)$, and 
$\la=(6^{b},5,4,2,1)\in\RP_3(n)$. 
\begin{enumerate}
\item[{\rm (i)}] $D(6^{b-1},5,4,2,1)$ is a composition factor of $D(\la)\da_{\ts_{n-6}}$. 
\item[{\rm (ii)}] If $b$ is even then the restriction $D(\la)\da_{\ts_{n/2,n/2}}$ has Loewy length at least $3$.
\end{enumerate} 
\end{Lemma}
\begin{proof}
(i) Recursively remove nodes from columns $1, 2, 4,3,5,6$ of $\la$, observing that on each step the removed node is normal and using Lemma~\ref{normal}. 

(ii) Let $a:=b/2$, denote $\mu=(6^a,5,4,2,1)$ and 
use (i) to deduce that $D(\mu)$ is a composition factor of $D(\la)\da_{\ts_{n-6a}}$. Let $\nu:=(6^{a},4,2,1)$. Recursively remove nodes from $\mu$ in columns $1, 2, 4,3,5$ to get $\nu$. Since on each step the removed node is normal, by Lemma~\ref{normal} we have that $D(\nu)$ is a composition factor of $D(\mu)\da_{\ts_{n/2+1}}$. Let $\eta:=(6^{a-1},5,4,2,1)$ and $\theta:=(6^a,4,2)$. 
By Lemmas~\ref{normal} and \ref{branching}, we have that 
$\Res_0D(\nu)\cong e_0D(\nu)$ has Loewy length at least $3$, with socle and head isomorphic to $D(\eta)$ and a composition factor $D(\theta)$. 
Note that $\Res_0D(\nu)$ is a direct summand of $D(\nu)\da_{\ts_{n/2}}$. In particular, $\Res_0D(\nu)$ is a subquotient of $D(\la)\da_{\ts_{n/2}}$. It follows that $D(\la)\da_{\ts_{n/2,n/2}}$ has a subquotient with Loewy length at least $3$ (with socle of the form $D(\eta)\boxtimes D(\al)$, head of the form  $D(\eta)\boxtimes D(\be)$ and a composition factor $D(\theta)\boxtimes D(\ga)$). 
\end{proof}

\begin{Lemma} \label{L060125-4}
Let $p=3$, $a\in\Z_{>0}$ be odd, $n=6(a+2)$, $\la=(6^a,5,4,2,1)\in \RP_3(n)$
and $\al=(6^{(a-1)/2},5,3,1)\in \RP_3(n/2)$. Then $[D(\la)\da_{\ts_{n/2,n/2}}:D(\al,\al)]\geq 6.$
\end{Lemma}
\begin{proof}
In this proof we use the abbreviation $T_d:=[\bar S(n-d,d)]$ with $0\leq d<n/2$ in the Grothendieck group of $\F\ts_n$-supermodules, and $T_{b,c}:=[\bar S((n/2-b,b),(n/2-c,c))]$ with $0\leq b,c<n/4$ in the Grothendieck group of $\F\ts_{n/2,n/2}$-supermodules. 

We have $m_n=3a+4$, and  $\la=\mu_{n,3a+3}$ 
by Lemma~\ref{L2RowExplicit}. 
Further, note that $a_p(\mu_{n,c})=a_0((n-d,d))=0$ for all $0\leq c\leq m_n$ and $0<d<n/2$, so the corresponding  $\F\ts_n$-supermodules $D(\mu_{n,c})$ and $\C\ts_n$-supermodules $S(n-d,d)$ are of type $\Mtype$. 

By \cite[Theorem 1.2]{M3} we have in the Grothendieck group for some $d_{c,b}\in\Z_{\geq 0}$: 
\begin{equation}\label{E050825}
T_c=2^{\de_{3\mid c}}[D(\mu_{n,c})]+\sum_{b=0}^{c-1}d_{c,b}[D(\mu_{n,b})]\qquad(0\leq c\leq 3a+4). 
\end{equation}
Inverting, we have for some $e_{c,b}\in\Q$:
$$
[D(\mu_{n,c})]=2^{-\de_{3\mid c}}T_c+\sum_{b=0}^{c-1}e_{c,b}T_b\qquad(0\leq c\leq 3a+4).
$$
In particular, setting $e_b:=e_{3a+3,b}$, we get
$$
[D(\mu_{n,3a+3})]=\frac{1}{2}T_{3a+3}+\sum_{b=0}^{3a+2}e_{b}T_b.
$$
We need to get more information on the coefficients $e_{3a+2}$ and $e_{3a+1}$. This will come from the following extra information on the decomposition numbers $d_{c,b}$:

\vspace{2mm}
\noindent
{\em Claim 1.} We have $d_{3a+3,3a+1}=0$, $d_{3a+3,3a+2}=2x$ and $d_{3a+2,3a+1}=y$ for $x,y\in\{0,1\}$ with $(x,y)\neq (1,0)$. 

\vspace{2mm}
\noindent
For the proof of Claim 1, we recall that in \cite[Theorem 4.5]{bmo} an alternative (to the one from \cite{BK,BK2}) labeling of spin representations in characteristic $3$ was found. Let
\[\BMOPar(n)=\{\la\in\RP_0(n)\mid \la_r-\la_{r+1}\geq 3+\de_{3\mid\la_r}\text{ for all }1\leq r<h(\la)\}\]
be the labeling set from \cite{bmo} and $D'(\la)$ be the corresponding simple supermodules. Setting $D'(\la):=0$ if $\la\not\in\BMOPar(n)$, by \cite[Theorem 4.5]{bmo}, we have 
\[[\bar S(\la)]=d'_\la[D'(\la)]+\sum_{{\mu\in\BMOPar(n),}\atop{\mu\rhd\la}}d'_{\mu,\la}[D'(\mu)],\]
for some $d'_\la ,d'_{\la,\mu}\in\Z_{\geq 0}$ with $d'_\la>0$ if $\la\in\BMOPar(n)$. It now follows from (\ref{E050825}) that $D'(n-c,c)\cong D(\mu_{n,c})$ for all $c=0,1,\dots,3a+4$.

For $\mu\in\RP_p(n)$, we denote by $P(\mu)$ the indecomposable projective supermodule with head $D(\mu)$.  
For $c=3a+1,3a+2,3a+3$ let $P_c$ be the projective modules constructed in \cite[Theorem 4.1]{bmo} corresponding to $(n-c,c)$.   These projective supermodules have the following properties. 
\begin{equation}\label{E040825-0}
[P_c]=k_c[P(\mu_{n,c})]+\sum_{c<b\leq 3a+4}k_{b,c}[P(\mu_{n,b})]+\sum_{{\mu\in\RP_3(n)\setminus\TR_3(n)}}k_{\mu,c}[P(\mu)]
\end{equation}
for some coefficients $k_c\in\Z_{>0}$ and $k_{b,c},k_{\mu,c}\in\Z_{\geq 0}$, and 
\begin{equation}
\label{E040825}
[P_c]=l_cQ_c+\sum_{\nu\in\RP_0(n-r_c),\,h(\nu)\geq 3}l_{\nu,c}[I_c \bar S(\nu)]
\end{equation}
for some coefficients $l_c\in\Z_{>0}$ and $l_{\nu,c}\in\Z_{\geq 0}$, where 
\begin{eqnarray}
\label{E040825-1}
Q_{3a+1}&=&[\Ind_1(\Ind_0)^2\Ind_1(\Ind_0)^2\Ind_1\bar S(n-8-3a,3a+1)],
\\
\label{E040825-2}
Q_{3a+2}&=&[(\Ind_0)^2\Ind_1(\Ind_0)^2\bar S(n-7-3a,3a+2)],
\\
\label{E040825-3}
Q_{3a+3}&=&[\Ind_0\Ind_1(\Ind_0)^3\bar S(n-7-3a,3a+2)].
\end{eqnarray} 
By Lemma~\ref{LBr0} and (\ref{E040825})-(\ref{E040825-3}), we have  
\begin{align}
\label{E040825-11}
[P_{3a+1}]&=t_{3a+1}(T_{3a+1}+T_{3a+2}
+T_{3a+4}+T_{3a+5})
+\sum_{\nu\in\RP_0(n),\,h(\nu)\geq 3}t_{\nu,3a+1}[\bar S(\nu)],\\
\label{E040825-12}
[P_{3a+2}]&=
t_{3a+2}(T_{3a+2}+2T_{3a+3}+T_{3a+4})+\sum_{\nu\in\RP_0(n),\,h(\nu)\geq 3}t_{\nu,3a+2}[\bar S(\nu)],
\\
[P_{3a+3}]&=t_{3a+3}(T_{3a+3}+T_{3a+4}+T_{3a+5})+\sum_{\nu\in\RP_0(n),\,h(\nu)\geq 3}t_{\nu,3a+3}[\bar S(\nu)]
\end{align}
for some coefficients $t_c\in\Z_{>0}$ and $t_{\nu,c}\in\Z_{\geq 0}$. 

By Brauer Reciprocity, we have 
\begin{align*}
[P(\mu_{n,3a+1})]&=T_{3a+1}+\sum_{b=3a+2}^{3a+5}d_{b,3a+1}T_{b}
+\sum_{\nu\in\RP_0(n),\,h(\nu)\geq 3}d_{\nu,3a+1}[\bar S(\nu)],
\\
[P(\mu_{n,3a+2})]&=T_{3a+2}+\sum_{b=3a+3}^{3a+5}d_{b,3a+2}T_{b}
+\sum_{\nu\in\RP_0(n),\,h(\nu)\geq 3}d_{\nu,3a+2}[\bar S(\nu)],
\\
[P(\mu_{n,3a+3})]&=2T_{3a+3}+\sum_{b=3a+4}^{3a+5}d_{b,3a+3}T_{b}
+\sum_{\nu\in\RP_0(n),\,h(\nu)\geq 3}d_{\nu,3a+3}[\bar S(\nu)].
\end{align*}
Substituting these into (\ref{E040825-0}) with $c=3a+1$, we get 
\begin{align*}
[P_{3a+1}]=\,&k_{3a+1}T_{3a+1}+(k_{3a+1}d_{3a+2,3a+1}+k_{3a+2,3a+1})T_{3a+2}
\\
&+(k_{3a+1}d_{3a+3,3a+1}+k_{3a+2,3a+1}d_{3a+3,3a+2}+2k_{3a+3,3a+1})T_{3a+3}+(*)
\end{align*}
where $(*)$ stands for other terms not involving $T_{3a+1},T_{3a+2},T_{3a+3}$. Comparing with  (\ref{E040825-11}), we deduce that $k_{3a+1}=t_{3a+1}$, $d_{3a+3,3a+1}=0$, $d_{3a+2,3a+1}\in\{0,1\}$, and $d_{3a+2,3a+1}=0$ only if  $d_{3a+3,3a+2}=0$. 

Substituting into (\ref{E040825-0}) with $c=3a+2$, we get 
\begin{align*}
[P_{3a+2}]=\,&k_{3a+2}T_{3a+2}+(k_{3a+2}d_{3a+3,3a+2}+2k_{3a+3,3a+2})T_{3a+3}
+(*)
\end{align*}
where $(*)$ stands for other terms not involving $T_{3a+2},T_{3a+3}$. By \cite[Theorem 10.8]{BK2}, $d_{3a+3,3a+2}$ is even. Comparing with  (\ref{E040825-12}), we deduce that $k_{3a+2}=t_{3a+2}$ and $d_{3a+3,3a+2}\in\{0,2\}$. This completes the proof of Claim 1. 

\vspace{2mm}

Recalling that $\la=\mu_{n,3a+3}$, Claim 1 and (\ref{E050825}) now imply

\vspace{2mm}
\noindent
{\em Claim 2.} There is $z\in\{0,1\}$ such that $-e_{3a+2}=e_{3a+1}=z$, i.e. 
\begin{align*}
[D(\la)]=\frac{1}{2}T_{3a+3}
-zT_{3a+2}+zT_{3a+1}
+\sum_{b=0}^{3a}e_{b}T_b.
\end{align*}

Since $n=6a+12$ with $a\geq 1$ odd, we have that $n\geq 18$ and $n/2$ is odd. From \cite[Theorem 8.1]{stem} it then follows that
\begin{equation}\label{E281224_3}
\begin{split}
[D(\la)\da_{\ts_{n/2,n/2}}]=\,&T_{(n-2)/4,(n-10)/4}
+T_{(n-10)/4,(n-2)/4}+T_{(n-6)/4,(n-6)/4}\\
&+2(1-z)T_{(n-6)/4,(n-10)/4}+2(1-z)T_{(n-10)/4,(n-6)/4}\\
&+2(1-z)T_{(n-10)/4(n-10)/4}+\sum_{\min\{j,k\}\leq (n-14)/4}e_{j,k}T_{j,k}
\end{split}
\end{equation}
for some $e_{j,k}\in\Q$.

Note that $\al=\balpha_{(n+10)/4}+\balpha_{(n-10)/4}$. Moreover, 
by Lemma~\ref{reglargedif}, for integers $j\leq (n-10)/4$, we have 
$
(n/2-j,j)^\Reg=\balpha_{n/2-j}+\balpha_j.
$
But for an integer $j\leq (n-14)/4$ we then have 
$$
(n/2-j,j)^\Reg=\balpha_{n/2-j}+\balpha_j\lhd \balpha_{(n+10)/4}+\balpha_{(n-10)/4}=\al.
$$ 
By Lemma~\ref{L020819}, we conclude that $[T_{j,k}:D(\al,\al)]=0$ whenever  $\min\{j,k\}\leq (n-14)/4$, and so from  (\ref{E281224_3}), taking into account that $(1-z)\geq 0$, we have
\begin{equation}\label{E281224_3New}
\begin{split}
[D(\la)\da_{\ts_{n/2,n/2}}:D(\al,\al)]\geq\,&[T_{(n-2)/4,(n-10)/4}:D(\al,\al)]
+[T_{(n-10)/4,(n-2)/4}:D(\al,\al)]
\\ 
&+[T_{(n-6)/4,(n-6)/4}:D(\al,\al)]
\end{split}
\end{equation}

Note that the partitions $(n/2-j,j)$ for $j=(n-10)/4,(n-6)/4$ and $(n-2)/4$ have the same numbers of notes on each ladder, so for such $j$ we have 
$$(n/2-j,j)^\Reg=(n/2-(n-10)/4,(n-10)/4)^\Reg=\al.
$$ 
So we can apply Lemma \ref{L020819} to get 
\begin{equation}\label{E060826}
[\bar S(n/2-j,j):D(\al)]=\left\{\begin{array}{ll}
1&\text{if}\ j=(n-2)/4\ \text{or}\ j=(n-10)/4,\\
2&\text{if}\ j=(n-6)/4.
\end{array}\right.
\end{equation}
Since $D(\al)$ is of type $\Qtype$, as are $S(n/2-(n-j)/4,(n-j)/4)$ for $j\in\{2,6,10\}$, we have that for $j,k\in\{2,6,10\}$ the multiplicity $[T_{(n-j)/4,(n-k)/4}:D(\al,\al)]$ equals the product of multiplicities 
\begin{align*}
[\bar S(n/2-(n-j)/4,(n-j)/4):D(\al)]\cdot[\bar S(n/2-(n-k)/4,(n-k)/4):D(\al)].
\end{align*}
So taking into account (\ref{E060826}), the first two summands in the right hand side of (\ref{E281224_3New}) equal $1$ and the third summand equals $4$. 
\end{proof}

\begin{Lemma}\label{L060125}
Let $\la\in\RP_p(n)$ have one of the following forms:
\begin{itemize}
\item $\la=((2p)^a,2p-1,p+1,p^{2b},p-1,1)$ for some $a,b\geq 0$,

\item $\la=(p^{2a},p-1,p-2,2,1)$ for some $a\geq 0$ and $p\geq 5$.
\end{itemize}
We assume that $n>10$ if $p> 3$, and $n>12$ if $p=3$. Let $H=\hW_{n/2,2}<\ts_n$. Then the supermodule 
$D(\la)\da_{H}$ has composition length at least $3$ or it has at least two non-isomorphic composition factors.
\end{Lemma}
\begin{proof}
Since $\ts_{n/2,n/2}\leq H$ is of index $2$, to prove that the supermodule $D(\la)\da_H$ has composition length at least $3$, it suffices to prove that the supermodule $D(\la)\da_{\ts_{n/2,n/2}}$ has composition length at least $5$. Similarly, to prove that the supermodule $D(\la)\da_{H}$ has at least two non-isomorphic composition factors it suffices to prove that the supermodule $D(\la)\da_{\ts_{n/2,n/2}}$ has at least $3$ non-isomorphic composition factors, which in turn follows if we can prove that the supermodule $D(\la)\da_{\ts_{n/2}}$ has at least $3$ non-isomorphic composition factors. These facts are established for different cases in Lemmas~\ref{L060125-5}, \ref{L060125-1}, \ref{L060125-2}, \ref{L060125-3}, and ~\ref{L060125-4}. 
\end{proof}

\section{Special homomorphisms and reduction lemmas}

Let $G=\ts_n$ or $\tA_n$, $H$ be a subgroup of $G$, and $L$ be an irreducible $\F G$-module. In this section we develop some sufficient conditions for the restriction $L\da_H$ to be reducible. Recall from (\ref{EHomFMod}) that $\End_\F(L)$ is an $\F G$-module such that 
$$
\Hom_H(\bone,\End_\F(L))\cong \End_\F(L)^H\cong \End_H(L).
$$
It is easy to see that $\End_\F(L)$ is an $\F G$-module with trivial central action. For a partition $\al$ of $n$, we have the permutation module $M^\al$, the Specht module $S_\al\subseteq M^\al$, and the dual Specht module $(S^\al)^*$ which can be considered as a quotient of $M^\al$,  see \S\ref{SSSymMod}, so we have the natural homomorphisms
\begin{equation}\label{ESiAl}
S^\al\stackrel{\iota_\al}{\longrightarrow} M^\al\quad \text{and}\quad M^\al\stackrel{\si_\al}{\longrightarrow}(S^\al)^*.
\end{equation}
(Note that the notation $\si_\al$ agrees with the notation $\si_k$ from (\ref{ESik}) for $\al=(n-k,k)$.) 
These are all modules over $\s_n$, and upon restriction they are also modules over $\A_n$. We inflate these modules along $\pi$ to get the $\F G$-modules ${}^\pi M^\al,\,{}^\pi S^\al,\,{}^\pi (S^\al)^*$ (with trivial central action), cf. \S\ref{SSSpinMod}. Similarly, if $\al$ is $p$-regular, we have the $\F G$-module ${}^\pi D^\al$.

Note that, provided $\al\neq(1^n)$ if $G=\tA_n$, we have ${}^\pi M^{\al}\cong (\bone_{\ts_{\al}\cap G})\uparrow^G$, so by the Frobenius reciprocity, for any $\F G$-modules $V,W$, 
\begin{equation}\label{E220825}
\Hom_G({}^\pi M^{\al},\Hom_\F(V,W))\cong \Hom_{\ts_{\al}\cap G}(V\da_{\ts_{\al}\cap G},W\da_{\ts_{\al}\cap G}).
\end{equation}

\subsection{Reduction lemmas}
\label{sredlemmas}


\begin{Lemma}\label{inv_end}
Let $G=\ts_n$ or $\tA_n$, $H\leq G$ and $L$ be an irreducible $\F G$-module. Let $\al\in\Par_{\operatorname{reg}}(n)$ and suppose that $\al\neq(n)$ if $G=\ts_n$, and $\al\neq (n)$ and\, $\al^\Mull\,{\not{\hspace{-.9mm}\unrhd}}\,\al$ if $G=\tA_n$. If there exist homomorphisms $\phi\in\Hom_G(\bone\ua_H^G,{}^\pi M^\al)$ and $\psi\in\Hom_G({}^\pi M^\al,\End_\F(L))$ such that $\si_\al\circ \phi$ and $\psi\circ\iota_\al$ are non-zero, then $L\da_H$ is reducible.
\end{Lemma}

\begin{proof}
Recall from \S\ref{SSSymMod} that $\hd S^\al\cong D^\al\cong \soc (S^\al)^*$. So by Corollary~\ref{CSimpleHeadSnAn}, we have $\head({}^\pi S^\al)\cong {}^\pi D^\al\cong\soc ({}^\pi( S^\al)^*)$. 
If $G=\tA_n$ then $\al\neq \al^\Mull$ by assumption, so in all cases we have that the $\F G$-module ${}^\pi D^\al$ is irreducible. Moreover, it follows from the assumption $\al^\Mull\,{\not{\hspace{-.9mm}\unrhd}}\,\al$ (if $G=\tA_n$) and Corollary~\ref{CSpechtMin} that $[{}^\pi M^\al:{}^\pi D^\al]=1$.

Now, $\si_\al\circ \phi\neq$ implies that $[\im\phi:{}^\pi D^\al]=1$, and $\psi\circ\iota_\al\neq 0$ implies that $[\ker \psi:{}^\pi D^\al]=0$. Hence the image of the homomorphism $\psi\circ\phi:\bone\ua_H^G\to \End_\F(L)$ has 
the irreducible module ${}^\pi D^\al$ as a composition factor. Moreover, by assumption that $\al\neq (n)$ and in addition $\al\neq(n)^\Mull$ if $G=\tA_n$ since $(n)=((n)^\Mull)^\Mull\unrhd (n)^\Mull$. Thus ${}^\pi D^\al\not\cong\bone_G$. 
On the other hand $\bone_G$ is a quotient of $\bone\ua_H^G$ and a submodule of $\End_\F(L)$, so there is a homomorphism $\bone\ua_H^G\to \End_\F(L)$ with image isomorphic to $\bone_G$. We deduce that $\dim \Hom_G(\bone\ua_H^G,\End_\F(L))\geq 2$. 
Using the Frobenius reciprocity, we get
$$
\Hom_G(\bone\ua_H^G,\End_\F(L))\cong \Hom_H(\bone,\End_\F(L\da_H))\cong \End_H(L\da_H).
$$
So $\dim \End_H(L\da_H)\geq 2$, whence $L\da_H$ is reducible by Schur's Lemma. 
\end{proof}


\begin{Lemma}\label{inv_mixedhom}
Let $\al\in\Par_{\operatorname{reg}}(n)$ and $\la\in\RP_p(n)$. 
Suppose one of the following two assumptions holds:
\begin{enumerate}
\item[{\rm (i)}] $G=\ts_n$, $a_p(\la)=1$, $L=D(\la;\pm)$, $L'=D(\la;\mp)$;
\item[{\rm (ii)}] $G=\tA_n$, $a_p(\la)=0$, $L=E(\la;\pm)$, $L'=E(\la;\mp)$, and $\al^\Mull\,{\not{\hspace{-.9mm}\unrhd}}\,\al$;
\end{enumerate} 
Suppose that $H\leq G$ is a subgroup. 
If there exists $\phi\in\Hom_G(\bone\ua_H^G,{}^\pi M^\al)$ 
such that $\si_\al\circ \phi\neq 0$ and  $\psi_1,\psi_2\in\Hom_G({}^\pi M^\al,\End_\F(L,L'))$ such that $\psi_1\circ\iota_\al,\, \psi_2\circ\iota_\al,\,$ are linearly independent, then $L\da_H$ is reducible.
\end{Lemma}

\begin{proof}
As $\hd S^\al\cong D^\al\cong \soc (S^\al)^*$, by Corollary~\ref{CSimpleHeadSnAn}, we have $\head({}^\pi S^\al)\cong {}^\pi D^\al\cong\soc ({}^\pi( S^\al)^*)$. 
If $G=\tA_n$ then $\al\neq \al^\Mull$ by assumption, so in all cases we have that the $\F G$-module ${}^\pi D^\al$ is irreducible. Moreover, it follows from the assumption $\al^\Mull\,{\not{\hspace{-.9mm}\unrhd}}\,\al$ (if $G=\tA_n$) and Corollary~\ref{CSpechtMin} that $[{}^\pi M^\al:{}^\pi D^\al]=1$. 

Now $\si_\al\circ \phi\neq 0$ implies $[\im\phi:{}^\pi D^\al]\neq 0$. Using Corollary~\ref{CMin}, we deduce that $S^\al\subseteq\im\phi$. So the linear independence of $\psi_1\circ\iota_\al$, $\psi_2\circ\iota_\al$ implies the linear independence of $\psi_1\circ \phi,\,\psi_2\circ \phi\in \Hom_G(\bone\ua_H^G,\Hom_\F(L,L'))$. 
Using the Frobenius reciprocity, we now get 
\begin{align*}
\dim\End_H(L\da_H,L'\da_H)&=\dim\Hom_H(\bone,\Hom_\F(L\da_H,L'\da_H))
\\&=\dim\Hom_G(\bone\ua_H^G,\Hom_\F(L,L'))\geq 2.
\end{align*}
Since $L$ and $L'$ have the same dimension, the lemma follows using Schur's Lemma.
\end{proof}

In constructing homomorphisms $\psi,\psi_1,\psi_2$ as in Lemmas~\ref{inv_end} and \ref{inv_mixedhom} with $\al=(n-2,2)$, the following will be useful:

\begin{Lemma} \label{L220825} 
Let $G\in\{\ts_n,\tA_n\}$ and $V,W$ be $\F G$-modules. 
Then 
$$\dim\Hom_{\ts_{n-2,2}\cap G}(V\da_{\ts_{n-2,2}\cap G},W\da_{\ts_{n-2,2}\cap G})\geq \dim\Hom_{\ts_{n-1}\cap G}(V\da_{\ts_{n-1}\cap G},W\da_{\ts_{n-1}\cap G}).$$
Moreover, there exist homomorphisms $\psi_1,\dots\psi_r\in\Hom_G({}^\pi M^{(n-2,2)},\Hom_\F(V,W))$ such that $\psi_1\circ\iota_{(n-2,2)},\dots,\psi_r\circ \iota_{(n-2,2)}$ are linearly independent  if and only if
$$\dim\Hom_{\ts_{n-2,2}\cap G}(V\da_{\ts_{n-2,2}\cap G},W\da_{\ts_{n-2,2}\cap G})\geq r+\dim\Hom_{\ts_{n-1}\cap G}(V\da_{\ts_{n-1}\cap G},W\da_{\ts_{n-1}\cap G}).$$
\end{Lemma}
\begin{proof}
Since $p>2$, by Lemma~\ref{L131218}, there is an exact sequence 
$$0\rightarrow S^{(n-2,2)}\stackrel{\iota_{(n-2,2)}}{\longrightarrow}M^{(n-2,2)}\longrightarrow M^{(n-1,1)}\rightarrow 0.$$
This yields the exact sequence
\begin{align*}
0&\rightarrow \Hom_G({}^\pi M^{(n-1,1)},\Hom_\F(V,W))
\to \Hom_G({}^\pi M^{(n-2,2)},\Hom_\F(V,W))
\\&\stackrel{\iota_{(n-2,2)}^*}{\longrightarrow} 
\Hom_G({}^\pi S^{(n-2,2)},\Hom_\F(V,W)),
\end{align*}
and the result follows using the isomorphism (\ref{E220825}). 
\end{proof}

\subsection{Special homomorphisms}
Let $G\in\{\ts_n,\tA_n\}$. 
In this subsection, motivated  by the reduction lemmas of the previous subsection, we will construct, for some partitions $\al$ and some irreducible $\F G$-modules $L$, homomorphisms $\psi\in\Hom_G({}^\pi M^\al,\End_\F(L))$ such that $\psi\circ \iota_\al\neq 0$.

We will use the following method to construct such homomorphisms. 
Recall tabloids and polytabloids from \S\ref{SSSymMod}. 
Let $t^\al$ be the standard $\al$-tableau obtained by inserting numbers $1,\dots,n$ into the boxes of the Young diagram $\al$ down the columns starting from the first column and moving to the right. Let $R_\al$ (resp. $C_\al$) be the row (resp. column) stabilizer of $t^\al$, so that $M^\al\cong\bone_{R_\al}{\uparrow}^{\s_n}$, and $C_\al=\s_{\al'}$---the standard parabolic subgroup corresponding to the transposed partition $\al'$. 
Moreover, the Specht module $S^\al\subseteq M^\al$ is generated by the polytabloid 
$$
e^\al:=\sum_{g\in C_\al}\sgn(g)\,g\,\{t^\al\}.
$$

The group $R_\al$ acts on $\tA_n$ via $r\cdot g=\hat r\,g\,(\hat r)^{-1}$ for all $r\in R_\al$ and $g\in\tA_n$. Let $\xi\in\tA_n$ be an element such that $\pi(\xi)$ stabilizes each number in the first row of $t^\al$, and denote by $\O_\xi$ the orbit of $\xi$ under this action. The linear map $$f_\xi:L\to L,\ v\mapsto \sum_{h\in \O_\xi}hv$$ is an $R_\al$-invariant element of $\End_\F(L)$. So, by the Frobenius Reciprocity, there exists an $\F G$-homomorphism 
$$
\psi_\xi:{}^\pi M^\al\to\End_\F(L),\ \{t^\al\}\mapsto f_\xi.
$$
Moreover, for all $v\in L$, we have 
$
\psi(e^\al)(v)=x_\xi v
$
where
$$
x_\xi:=\sum_{g\in C_\al}\sum_{h\in\O_\xi}\sgn(g)\,\hat g\, h\, (\hat g)^{-1}.
$$
So to see that $\psi_\xi\circ\iota_\al\neq 0$, it suffices to prove that $x_\xi L\neq 0$. 
Note that $\O_\xi\subseteq \tA_{n-\al_1+\al_2}$ and $C_\al\leq \ts_{n-\al_1+\al_2}$, so $x_\xi\in\F\tA_{n-\al_1+\al_2}$, and to check that $x_\xi L\neq 0$, it suffices to check that $x_\xi V\neq 0$ for some composition factor $V$ of $L\da_{\tA_{n-\al_1+\al_2}}$. We have proved:

\begin{Lemma} \label{L080825}
Let $G\in\{\ts_n,\tA_n\}$, $L$ be an irreducible spin $\F G$-module, $\al$ be a partition of $n$, and $E(\mu;\eps)$ be a composition factor of $L\da_{\tA_{n-\al_1+\al_2}}$. If $x_\xi E(\mu;\eps)\neq 0$, then $\psi_\xi\circ\iota_\al\neq 0$. 
\end{Lemma} 

In the following lemma we make a minor but useful improvement on Lemma~\ref{L080825}. Recall the $\C \tA_n$-modules $T(\la)$ for $\la\in\RP_0(n)$ from Lemma~\ref{lcharvan}. Similarly, for $\la\in\RP_p(n)$,  we define the $\F\tA_n$-modules 
$$
E(\la):=
\left\{
\begin{array}{ll}
E(\la;0) &\hbox{if $a_p(\la)=1$,}\\
E(\la;+)\oplus E(\la;-) &\hbox{if $a_p(\la)=1$.}
\end{array}
\right.
$$

\begin{Lemma} \label{LRepeat} 
Let $G\in\{\ts_n,\tA_n\}$, $L$ be an irreducible spin $\F G$-module, $\al$ be a partition of $n$, and $E(\mu;\eps)$ be a composition factor of $L\da_{\tA_{n-\al_1+\al_2}}$. If $x_\xi E(\mu)\neq 0$ then there exists $\psi\in\Hom_G({}^\pi M^\al,\End_\F(L))$ such that $\psi\circ \iota_{\al}\neq 0$. 
\end{Lemma}
\begin{proof}
If $\eps=0$ then $E(\mu)=E(\mu;\eps)$, and we are done by Lemma~\ref{L080825}. If $\eps=+$ or $-$, then $E(\mu)=E(\mu;+)\oplus E(\mu;-)$, and so $x_\xi E(\mu)\neq 0$ implies 
$xE(\mu;\eps)\neq 0$ or $xE(\mu;-\eps)\neq 0$. If $xE(\mu;\eps)\neq 0$, we are done. 

Suppose $xE(\mu;-\eps)\neq 0$. If $G=\ts_n$ and $L=D(\la;0)$, or  $G=\tA_n$ and $L=E(\la;0)$, then both  $E(\mu;\pm)$ appear as composition factors of $L\da_{\tA_{n-\al_1+\al_2}}$. If $G=\ts_n$ and   $L=D(\la;\pm)$ then $E(\mu;-\eps)$ is a composition factor of $D(\la;\mp)\da_{\tA_{n-\al_1+\al_2}}$, so the result follows for $D(\la;\mp)$ in place of $L$, and then also for $L$ using Lemma~\ref{LHomPM}(i). If $G=\tA_n$ and   $L=E(\la;\pm)$ then $E(\mu;-\eps)$ is a composition factor of $E(\la;\mp)\da_{\tA_{n-\al_1+\al_2}}$, so the result follows for $E(\la;\mp)$ in place of $L$, and then also for $L$ using Lemma~\ref{LHomPM}(ii) and the fact that $S^{\al}$ and  $M^{\al}$ are $\ts_n$-modules. 
\end{proof}

Recall the $\F$-valued characters from \S\ref{SSGrMod}. 
In the proofs of the remaining results of this subsection, we will be checking the assumption $x_\xi E(\mu)\neq 0$ of Lemma~\ref{LRepeat} by finding $y\in\F \tA_{n-\al_1+\al_2}$ such that $\chi(yx_\xi)\neq 0$ for the $\F$-valued character $\chi$ of $E(\mu)$. 
As in Lemma~\ref{lcharvan}, we denote by $\chi^\la$ the complex  character of $T(\la)$ and  
express $\chi$ as a linear combination of reductions modulo $p$ of such complex characters using decomposition matrices in \cite{ModularAtlas}, \cite{GAP} and \cite[Theorem 4.4]{Mu} (we use \cite[Corollaries 7.3, 7.5 Theorem 3.3]{stem} to identify the rows, we also often use Lemma~\ref{L020819} to identify the columns in the decomposition matrices). 
We then compute the needed characters $\chi^\la$ explicitly on $yx_\xi$, taking into account Lemma \ref{lcharvan} which allows us to ignore the summands  $g$ in $yx_\xi$ with $\pi(g)$ of even order, so we only keep track of the summands $g$ such that all cycles in the cycle type of $\pi(g)$ are odd.  
The coefficients of those summands in the product $yx_\xi$ are determined explicitly using \cite{GAP}. To compute the required values $\chi^\la(g)$, we use the  character tables in \cite{GAP} (and \cite[Corollaries 7.3, 7.5, Theorem 3.3]{stem} to identify the rows in the character tables).

\begin{Lemma}\label{L7}
Let $n\geq 6$, $G\in\{\ts_n,\tA_n\}$ and $L$ be a non-basic irreducible spin $\F G$-module. Then there exists $\psi\in\Hom_G({}^\pi M^{(n-3,3)},\End_\F(L))$ such that $\psi\circ \iota_{(n-3,3)}\neq 0$.
\end{Lemma}

\begin{proof}
If $p>3$ then $M_3\sim S_3|M_2$ by Lemma \ref{L131218} and the lemma holds by \cite[Theorem 7.2]{KT}. Let $p=3$.

Take $\xi=(2,4,6)\,\hat{}$\,. Then, considering $\s_{2^3}<\s_n$ as in \S\ref{SSAltSym}, we have  
$$
x:=x_\xi=\sum_{g\in\s_{2^3}}\sgn(g)\,\hat g((2,4,6)\,\hat{}+(2,6,4)\,\hat{}\,)(\hat g)^{-1}\in\F\tA_6.
$$

By \cite[Lemma 2.4]{KT}, there exists a non-basic composition factor $E$ of $L\da_{\tA_6}$. Then $E\cong E((4,2);\eps)$ for some $\eps\in\{+,-\}$, since $\RP_3(6)=\{\balpha_6,(4,2)\}$. 
By Lemma~\ref{LRepeat}, it suffices to prove that $xE(4,2)\neq 0$. Let $\chi$ be the $\F$-valued character of $E(4,2)$. 
By \cite{ModAtl}, 
we have 
$[E(4,2)]=[\bar T(4,2)]-2[\bar T(6)]$, so $\chi=\bar\chi^{(4,2)}-2\bar\chi^{(6)}$, 
and it suffices to prove that $\bar\chi^{(4,2)}(yx)-2\bar\chi^{(6)}(yx)\neq 0$ for some $y\in\tA_6$. 

There is a lift $y=((1,5)(2,3,4,6))\,\hat{}$\, such that $yx=\sum_{g\in C_+}g-\sum_{g\in C_-}g$ for $C_\pm\subseteq\tA_6$ and the numbers of elements in $C_+$ and $C_-$ with given cycle types and orders are as follows:
\[\begin{array}{l|c|c|c|c|c|c|c|c|c}
\text{cycle type}&(1^6)&(1^6)&(3,1^3)&(3,1^3)&(3^2)&(3^2)&(5,1)&(5,1)&\text{others}\\
\text{order }&1&2&3&6&3&6&5&10&\\
\hline
C_+&0&0&0&0&0&1&1&2&4\\
C_-&0&0&0&0&1&0&2&1&4\\
\end{array}\]
Now, $\chi^{(4,2)}(yx)-2\chi^{(6)}(yx)=-2-2\cdot 0\equiv 1\pmod{3}$.
\end{proof}

\begin{Lemma}\label{L181224_3}
Let $n\geq 12$, $G\in\{\ts_n,\tA_n\}$, $\la\in\RP_p(n)\setminus\TR_p(n)$ with $\la_1\geq 6$, and $L$ be of the form $D(\la;\eps)$ or $E(\la;\eps)$.  
Then there exists $\psi\in\Hom_G({}^\pi M^{(n-6,6)},\End_\F(L))$ such that $\psi\circ \iota_{(n-6,6)}\neq 0$.
\end{Lemma}

\begin{proof}
Take $\xi=((2,4,6)(8,10,12))\,\hat{}$\,. Then, denoting by  $\s_{\{2,4,6,8,10,12\}}$ the symmetric group on the set $\{2,4,6,8,10,12\}$, we have 
$$
x:=x_\xi=\sum_{g\in\s_{2^6}}\,
\sum_{(a,b,c)(d,e,f)\in\s_{\{2,4,6,8,10,12\}}}
\sgn(g)\,\hat g((a,b,c)(d,e,f))\,\hat{}(\hat g)^{-1}\in\F\tA_{12}.
$$

By Lemma \ref{rectworows}, there exists a composition factor of $L\da_{\tA_{12}}$ of the form $E(\mu;\de)$ such that 

\begin{itemize}
\item if $p\geq 13$ then $\mu\in\{(6,3,2,1),(6,4,2),(6,5,1),(7,3,2),(7,4,1),(8,3,1),(9,2,1)\}$,

\item if $p=11$ then $\mu\in\{(6,3,2,1),(6,4,2),(6,5,1),(7,3,2),(7,4,1),(8,3,1)\}$,

\item if $p=7$ then $\mu\in\{(6,3,2,1),(6,4,2),(8,3,1),(9,2,1)\}$,

\item if $p=5$ then $\mu\in\{(6,3,2,1),(6,4,2),(7,3,2),(8,3,1)\}$,

\item if $p=3$ then $\mu=(7,4,1)$.
\end{itemize}
In view of Lemma~\ref{LRepeat}, for all such $\mu$, it suffices to  prove that there exists $y=y_\mu\in\tA_{12}$ such that $\chi(yx)\neq 0$, where $\chi$ is the $\F$-valued character of $E(\mu)$.  Let 
$$
y_1:=((1,3,5,7)(2,9,11)(4,6))\,\hat{}\,\quad\text{and}\quad y_2:=((1,3,5,7,9)(2,4,11)(6,8,10))\,\hat{}\,.
$$ 
Recall that, according to our convention, the lift 
$y_2$ is chosen to be of odd order, while the lift $y_1$ can be  chosen so that the following holds for all $i=1,2$:  we have $y_ix=\sum_{g\in C_{i,+}}g-\sum_{g\in C_{i,-}}g$ for $C_{i,\pm}\subseteq\tA_{12}$ such that the numbers of elements in $C_{i,\pm}$ with given cycle types and orders are as follows:
\[\begin{array}{l|c|c|c|c|c|c|c|c}
\text{cycle type}&(1^{12})&(1^{12})&(3,1^9)&(3,1^9)&(3^2,1^6)&(3^2,1^6)&(3^3,1^3)&(3^3,1^3)\\
\text{order }&1&2&3&6&3&6&3&6\\
\hline
C_{1,+}&0&0&0&0&1&0&2&3\\
C_{1,-}&0&0&0&0&0&0&0&2\\
\hline
C_{2,+}&0&0&0&0&0&1&0&1\\
C_{2,-}&0&0&0&0&0&0&10&0
\end{array}\]
\[\begin{array}{l|c|c|c|c|c|c|c|c}
\text{cycle type}&(3^4)&(3^4)&(5,1^7)&(5,1^7)&(5,3,1^4)&(5,3,1^4)&(5,3^2,1)&(5,3^2,1)\\
\text{order}&3&6&5&10&15&30&15&30\\
\hline
C_{1,+}&0&2&1&0&13&17&6&10\\
C_{1,-}&0&0&1&1&11&13&14&10\\
\hline
C_{2,+}&0&0&0&0&1&11&49&11\\
C_{2,-}&0&11&1&0&2&5&46&27
\end{array}\]
\[\begin{array}{l|c|c|c|c|c|c|c|c}
\text{cycle type}&(5^2,1^2)&(5^2,1^2)&(7,1^5)&(7,1^5)&(7,3,1^2)&(7,3,1^2)&(7,5)&(7,5)\\
\text{order }&5&10&7&14&21&42&35&70\\
\hline
C_{1,+}&10&16&8&5&38&42&0&0\\
C_{1,-}&28&16&7&7&48&40&12&12\\
\hline
C_{2,+}&4&30&0&5&109&15&52&6\\
C_{2,-}&0&42&0&3&54&11&46&4
\end{array}\]
\[\begin{array}{l|c|c|c|c|c|c|c}
\text{cycle type}&(9,1^3)&(9,1^3)&(9,3)&(9,3)&(11,1)&(11,1)&\text{others}\\
\text{order }&9&18&9&18&11&22&\\
\hline
C_{1,+}&58&58&24&24&96&96&750\\
C_{1,-}&66&66&16&16&60&60&774\\
\hline
C_{2,+}&0&28&102&36&156&64&599\\
C_{2,-}&0&51&74&24&198&54&617.
\end{array}\]


By \cite{ModAtl} and \cite[Theorem 4.4]{Mu}, we have $[E(\mu)]=[\bar T(\mu)]$ unless one of the following holds:
\begin{itemize}
\item $p=11$ and $[E(7,4,1)]=[\bar T(7,4,1)]-[\bar T(6,5,1)]$;

\item $p=11$ and $[E(8,3,1)]=[\bar T(8,3,1)]-[\bar T(7,4,1)]+[\bar T(6,5,1)]$;

\item $p=5$ and $[E(7,3,2)]=[\bar T(9,2,1)]-1/2[\bar T(10,2)]+[\bar T(12)]$.
\end{itemize}

From the character tables we have 
\begin{align*}
&\chi^{(6,3,2,1)}(y_1x)=-54,\,\,\chi^{(6,4,2)}(y_1x)=-90,\,\,\chi^{(6,5,1)}(y_1x)=126,
\\
&\chi^{(7,3,2)}(y_1x)=-54,
\chi^{(7,4,1)}(y_1x)=126,\,\,\chi^{(8,3,1)}(y_1x)=-54,\\
&\chi^{(9,2,1)}(y_1x)=-54,\,\,\chi^{(10,2)}(y_1x)=0,\,\,\chi^{(12)}(y_1x)=0.
\end{align*}
Thus $\chi(y_1x)\neq 0$ unless $(p,\mu)\in\{(11,(7,4,1)),(5,(6,4,2)),(3,(7,4,1))\}$. To cover these remaining cases use 
$\chi^{(6,4,2)}(y_2x)=312,\,\,\chi^{(6,5,1)}(y_2x)=-399,\,\,\chi^{(7,4,1)}(y_2x)=-56.$
\end{proof}

\begin{Lemma}\label{end42} 
Let $p=3$, $n\geq 10$, $G\in\{\ts_n,\tA_n\}$, $\la\in\TR_3(n)\setminus\{\balpha_n,\bbeta_n\}$, and $L$ be of the form $D(\la;\eps)$  or $E(\la;\eps)$. Then there exists $\psi\in\Hom_G({}^\pi M^{(n-6,4,2)},\End_\F(L))$ such that $\psi\circ \iota_{(n-6,4,2)}\neq 0$.
\end{Lemma}

\begin{proof}
Take $\xi=((2,5,3)(8,10,6))\,\hat{}$\,. Then
$$
x:=x_\xi=\sum_{g\in\s_{3^2,2^2}}\,
\sum_{h\in\s_{\{2,5,8,10\}}}
\sgn(g)\,\hat g\,\hat h((2,5,3)(8,10,6))\,\hat{}\,(\hat h)^{-1}\,(\hat g)^{-1}\in\F\tA_{10}.
$$

By Lemma \ref{rectworows2}, $E((5,3,2);\pm)$ or $E((5,4,1);\pm)$ is a composition factor of $L\da_{\tA_{10}}$. In view of Lemma~\ref{LRepeat}, it suffices to  prove that there exists $y\in\tA_{10}$ such that $\chi(yx)\neq 0$, where $\chi$ is the $\F$-valued character of $E(\mu)$. 
We take 
$y=((1,5,2,3,6)(4,8,10,9,7))\,\hat{}\,$.
Then $yx=\sum_{g\in C_+}g-\sum_{g\in C_-}g$ for  $C_\pm\subseteq\tA_{10}$ and the numbers of elements in $C_\pm$ with given cycle types and orders are as follows:
\[\begin{array}{l|c|c|c|c|c|c|c|c}
\text{cycle type}&(1^{10})&(1^{10})&(3,1^7)&(3,1^7)&(3^2,1^4)&(3^2,1^4)&(3^3,1)&(3^3,1)\\
\text{order }&1&2&3&6&3&6&3&6\\
\hline
C_+&0&0&0&0&8&0&0&16\\
C_-&0&0&0&0&8&0&0&16
\end{array}\]
\[\begin{array}{l|c|c|c|c|c|c|c|c}
\text{cycle type}&(5,1^5)&(5,1^5)&(5,3,1^2)&(5,3,1^2)&(5^2)&(5^2)&(7,1^3)&(7,1^3)\\
\text{order }&5&10&15&30&5&10&7&14\\
\hline
C_+&0&0&56&16&84&76&12&0\\
C_-&0&0&72&16&96&72&12&0
\end{array}\]
\[\begin{array}{l|c|c|c|c|c}
\text{cycle type}&(7,3)&(7,3)&(9,1)&(9,1)&\text{others}\\
\text{order }&21&42&9&18&\\
\hline
C_+&92&160&280&92&836\\
C_-&88&172&248&124&804.
\end{array}\]
By \cite{ModularAtlas}, we have $[E(5,3,2)]=[\bar T(5,3,2)]$ and $[E(5,4,1)]=\frac{1}{2}[\bar T(6,4)]-[\bar T(10)]$. 
Moreover, $\chi^{(532)}(yx)=-32$, $\chi^{(6,4)}(yx)=-32$ and  $\chi^{(10)}(yx)=0$, and the lemma follows.
\end{proof}

\begin{Lemma}\label{end222}
Let $p\geq 5$, $n\geq 11$, $G\in\{\ts_n,\tA_n\}$, $\la\in\TR_p(n)\setminus\{\balpha_n,\bbeta_n\}$, and $L$ be of the form $D(\la;\eps)$  or $E(\la;\eps)$. 
Then there exists $\psi\in\Hom_G({}^\pi M^{(n-6,2^3)},\End_\F(L))$ such that $\psi\circ \iota_{(n-6,2^3)}\neq 0$.
\end{Lemma}

\begin{proof}
Take $\xi=((2,3,4)(6,7,8))\,\hat{}$\,. Then $x:=x_\xi\in\F\tA_{8}$ is given by 
\begin{align*}
\sum_{g\in\s_{4^2}}
\sgn(g)\,\hat g\big(
((2,3,4)(6,7,8))\,\hat{}\,&+((2,3,8)(6,7,4))\,\hat{}
\\& +((2,7,4)(6,3,8))\,\hat{}\,+((2,7,8)(6,3,4))\,\hat{}\,
\big)(\hat g)^{-1}.
\end{align*}

If $p\geq 7$ then by Lemma \ref{rectworows2}  there exists $\mu\in\TR_p(n-1)\setminus\{\balpha_{n-1},\bbeta_{n-1}\}$ such that $E(\mu;\eps)$ is a composition factor of $L\da_{\tA_8}$. In view of (\ref{bs}), (\ref{sbs}) and Lemma \ref{LM3} we have that $\mu\in\{(6,2),\,(5,3)\}$. If $p= 5$ then  similarly  there exists a partition $\nu\in\{(8,3),\,(7,4),\,(5,4,2),\,(5,3,2,1) \}$ such that $E(\nu;\eps)$ a composition factor of $L\da_{\tA_{11}}$ and then looking at decomposition matrices and using branching in characteristic $0$, we deduce that $E(\mu;\eps)$ is a composition factor of $L\da_{\tA_8}$ for $\mu:=(5,2,1)$.

By Lemma~\ref{LRepeat}, for every $\mu$ as in the previous  paragraph, 
it suffices to prove that $\chi(yx)\neq 0$ for $y:=(1,3,6,7,8,2,5)\,\hat{}\,$ and $\chi$ the $\F$-valued character of $E(\mu)$. 

We have $yx=\sum_{g\in C_+}g-\sum_{g\in C_-}g$ for  $C_\pm\subseteq\tA_{8}$ and the numbers of elements in $C_\pm$ with given cycle types and orders are as follows:
\[\begin{array}{l|c|c|c|c|c|c|c|c}
\text{cycle type}&(1^8)&(1^8)&(3,1^5)&(3,1^5)&(3^2,1^2)&(3^2,1^2)&(5,1^3)&(5,1^3)\\
\text{order }&1&2&3&6&3&6&5&10\\
\hline
C_+&0&0&0&0&102&30&12&30\\
C_-&0&0&0&0&12&12&30&84
\end{array}\]
\[\begin{array}{l|c|c|c|c|c}
\text{cycle type}&(5,3)&(5,3)&(7,1)&(7,1)&\text{others}\\
\text{order }&15&30&7&14&\\
\hline
C_+&60&24&168&132&594\\
C_-&132&96&204&96&486.
\end{array}\]
By \cite{ModularAtlas}, we have $[E(6,2)]=[\bar T(6,2)]$ and $[E(5,3)]=[\bar T(5,3)]$ if $p\geq 7$ and $[E(5,2,1)]=[\bar T(5,2,1)]$ if $p=5$. 
Moreover, $\chi^{(6,2)}(yx)=-72$, $\chi^{(5,3)}(yx)=360$ and  $\chi^{(5,2,1)}(yx)=-252$. The lemma follows.
\end{proof}

\section{Restrictions to $\ts_{n-2}$ and $\ts_{n-2,2}$}\label{sn-2}
In this section we find lower bounds for $\dim\End_{\ts_{n-2,2}}(D(\la)\da_{\ts_{n-2,2}})$. These  bounds will then be used to compare $\dim\End_{\ts_{n-2,2}}(D(\la)\da_{\ts_{n-2,2}})$ with $\dim\End_{\ts_{n-1}}(D(\la)\da_{\ts_{n-1}})$ given by Lemma~\ref{n-1}. This will allow 
us in many cases to apply the reduction lemmas of \S\ref{sredlemmas} with $\al=(n-2,2)$ and $L=D(\la;\eps)$ or $L=E(\la;\eps)$ with $\la\in\TR_p(n)$. In view of Lemma~\ref{L161224_3}, it will suffice to understand $\dim\End_{\ts_{n-2}}(D(\la)\da_{\ts_{n-2}})$. 

To save space, for a composition $\al$ and a $\T_\al$-supermodules $V,W$, in this section we denote
\begin{equation}\label{EDNot}
\d_\al(V):=\dim\End_{\T_\al}(V)\qquad\text{and}\qquad
\d_\al(V,W):=\dim\Hom_{\T_\al}(V,W).
\end{equation}

\subsection{\boldmath Bounding $\dim\End_{\ts_{n-2}}(D(\la)\da_{\ts_{n-2}})$} Recall the material of \S\ref{SSBr}. In the next three lemmas we obtain lower bounds for dimensions of endomorphism algebras of the summands $\Res_i\Res_j D(\la)$ of $D(\la)\da_{\T_{n-2}}$. Recall the notation (\ref{EDNot}).

\begin{Lemma}\label{L051224}
Let $\la\in\RP_p(n)$ and $i\in I$. Then
\[\d_{n-2}(\Res_i^2D(\la))\geq 4\d_n(D(\la))(\eps_i(\la)-1)(1+\de_{i\neq 0})^2+4\de_{\eps_i(\la)\geq 3}(1+\de_{i\neq 0})^2.\]
\end{Lemma}

\begin{proof}
We may assume that $\eps_i(\la)\geq 2$ for otherwise $\Res_i^2D(\la)=0$ and the result follows. By Lemma~\ref{LAGa}, we have $a_p(\tilde e_i^2\la)=a_p(\la)$, hence 
$\d_{n-2}(D(\tilde e_i^2\la))=
\d_n(D(\la))$ by (\ref{EEndDDim}). By Lemma \ref{divpowers}, we have $\Res_i^2D(\la)\cong (e_i^{(2)}D(\la))^{\oplus 2(1+\de_{i\neq 0})}$, so 
$\d_{n-2}(\Res_i^{2}D(\la))=4\d_{n-2}(e_i^{(2)}D(\la))(1+\de_{i\neq 0})^2$,   
and it suffices to prove that
\begin{equation}\label{E130825_2}
\d_{n-2}(e_i^{(2)}D(\la))\geq \d_n(D(\la))(\eps_i(\la)-1)+\de_{\eps_i(\la)\geq 3}.
\end{equation}
If $\eps_i(\la)=2$ then  $e_i^{(2)}D(\la)\cong D(\tilde e_i^2\la)$ by Lemma \ref{divpowers}, and we have the equality in (\ref{E130825_2}). Assume that $\eps_i(\la)\geq 3$. 
By Lemma \ref{branching}, we have 
$$
e_i D(\tilde e_i\la)\subseteq\Res_i D(\tilde e_i\la)\subseteq \Res_i e_iD(\la)\subseteq \Res_i^2 D(\la), 
$$
and $e_i D(\tilde e_i\la)$ has an irreducible socle. 
So 
$e_i D(\tilde e_i\la)\subseteq e_i^{2}D(\la)\cong (e_i^{(2)}D(\la))^{\oplus 2}$ implies 
$e_i D(\tilde e_i\la)\subseteq e_i^{(2)}D(\la)$. 
Moreover, by Lemmas \ref{branching} and \ref{divpowers}, we have 
\[[e_i D(\tilde e_i\la):D(\tilde e_i^2\la)]=\eps_i(\la)-1<\eps_i(\la)(\eps_i(\la)-1)/2=[e_i^{(2)}D(\la):D(\la)],\]
so $e_i D(\tilde e_i\la)\subsetneq e_i^{(2)}D(\la)$, and  $\d_{n-2}(D( e_i^{(2)}\la))>\d_{n-2}(e_iD(\tilde e_i\la))$ by  Lemma~\ref{LSelfDual}. By Lemma \ref{branching}, 
\begin{align*}
\d_{n-2}(e_iD(\tilde e_i\la))
=\d_{n-2}(D(\tilde e_i^2\la))(\eps_i(\la)-1)
=\d_n(D(\la))(\eps_i(\la)-1).
\end{align*}
We deduce that $\d_{n-2}(e_i^{(2)}D(\la))> \d_n(D(\la))(\eps_i(\la)-1)$, which implies (\ref{E130825_2}). 
\end{proof}

\begin{Lemma}\label{L171224}
Let $\la\in\RP_p(n)$, and $i, j\in I$ with $i\neq j$. If $\eps_j(\la)>0$ then
\begin{align*}
\d_{n-2}(\Res_i\Res_j D(\la))&\geq \eps_i(\tilde e_j\la)\d_n(D(\la))(1+\de_{i\neq 0})(1+\de_{j\neq0})+\de_{\eps_i(\tilde e_j\la)>0}\de_{\eps_j(\la)\geq 2}\\
&\geq \eps_i(\la)\d_n(D(\la))(1+\de_{i\neq 0})(1+\de_{j\neq0})+\de_{\eps_i(\tilde e_j\la)>0}\de_{\eps_j(\la)\geq 2}.
\end{align*}
\end{Lemma}

\begin{proof}
By Lemma \ref{L051218_4}, it is enough to prove the first inequality. We may assume that $\eps_i(\tilde e_j\la)\geq 1$.

By Lemma \ref{branching},  
\[(e_iD(\tilde e_j\la))^{\oplus (1+\de_{j\neq 0}a_p(\la))(1+\de_{i\neq 0}a_p(\tilde e_j\la))}
\subseteq \Res_i\Res_j D(\la),\] 
and the containment is strict if $\eps_j(\la)\geq 2$. So, by Lemma~\ref{LSelfDual}, 
$$
\d_{n-2}(\Res_i\Res_j D(\la))\geq
\big((1+\de_{j\neq 0}a_p(\la))(1+\de_{i\neq 0}a_p(\tilde e_j\la))\big)^2\d_{n-2}(e_iD(\tilde e_j\la))+\de_{\eps_j(\la)\geq 2}.
$$
By Lemma~\ref{branching}(iv), 
$$\d_{n-2}(e_iD(\tilde e_j\la))=\eps_i(\tilde e_j\la)\d_{n-2}(D(\tilde e_i\tilde e_j\la))=\eps_i(\tilde e_j\la) (1+a_p(\tilde e_i\tilde e_j\la)).
$$
So it remains to observe that
$$
\big((1+\de_{j\neq 0}a_p(\la))(1+\de_{i\neq 0}a_p(\tilde e_j\la))\big)^2(1+a_p(\tilde e_i\tilde e_j\la))=(1+\de_{i\neq 0})(1+\de_{j\neq 0})(1+a_p(\la)),
$$
which follows easily using Lemma~\ref{LAGa}, and apply (\ref{EEndDDim}). 
\end{proof}

Recall the notation (\ref{EDNot}).

\begin{Lemma}\label{L081218_2}
Let $\la\in\RP_p(n)$, and $i, j\in I$ with $i\neq j$. Then
\[
\d_{n-2}\big(\Res_i\Res_jD(\la),\Res_j\Res_i D(\la)\big)\geq \eps_i(\la)\eps_j(\la)\d_n(D(\la))(1+\de_{i\neq 0})(1+\de_{j\neq 0}).
\]
\end{Lemma}
\begin{proof}
We may assume that $\eps_i(\la),\eps_j(\la)>0$.
By adjointness of $\Res_k$ and $\Ind_k$, and Lemma \ref{L051218_3},
\begin{equation}\label{E140825}
\d_{n-2}(\Res_i\Res_jD(\la),\Res_j\Res_i D(\la))
= \d_n(\Ind_j\Res_jD(\la),\Ind_i\Res_i D(\la)).
\end{equation}
By adjointness again and Lemma \ref{n-1}, 
\begin{align*}
\d_n(D(\la),\Ind_i\Res_i D(\la))&=\d_{n-1}(\Res_iD(\la))=\eps_i(\la)(1+\de_{i\neq 0})(1+a_p(\la)),
\end{align*}
so, using (\ref{EEndDDim}), we deduce that  $D(\la)^{\oplus \eps_i(\la)(1+\de_{i\neq 0})}\subseteq \Ind_i\Res_i D(\la)$. Similarly we have that $D(\la)^{\oplus \eps_j(\la)(1+\de_{j\neq 0})}$ is a quotient of $\Ind_j\Res_jD(\la)$. 
So by (\ref{E140825}), we have 
\begin{align*}
\d_{n-2}(\Res_i\Res_jD(\la),\Res_j\Res_i D(\la))&\geq 
\d_{n}(D(\la)^{\oplus \eps_j(\la)(1+\de_{j\neq 0})},D(\la)^{\oplus \eps_i(\la)(1+\de_{i\neq 0})})
\\
&=\eps_i(\la)\eps_j(\la)(1+\de_{i\neq 0})(1+\de_{j\neq 0})\d_{n}(D(\la)),
\end{align*}
as desired. 
\end{proof}

Lemmas~\ref{L171224} and \ref{L081218_2} immediately give

\begin{Corollary} \label{C180825}
Let $\la\in\RP_p(n)$, and $i, j\in I$ with $i\neq j$. If $\eps_i(\la),\eps_j(\la)>0$ then
\begin{align*}
&\d_{n-2}\big(\Res_i\Res_j D(\la)\oplus\Res_j\Res_i D(\la)\big)\\\geq\, &(\eps_i(\la)+\eps_j(\la)+2\eps_i(\la)\eps_j(\la))\d_n(D(\la))(1+\de_{i\neq 0})(1+\de_{j\neq 0}).
\end{align*}
\end{Corollary}

We combine the above lower bounds to get:

\begin{Lemma}\label{L171224_2}
Let $\la\in\Par_p(n)$, $X:=\{i\in I\mid \eps_i(\la)>0\}$ and $x:=|X|$. Then
\begin{align*}
&\d_{n-2,2}(D(\la)\da_{\ts_{n-2,2}})
\geq 
2\de_{\eps_0(\la)\geq 3}+8\sum_{i\neq 0}\de_{\eps_i(\la)\geq 3}
\\
&\quad+\d_n\big(D(\la)\big)\big(4(x-1)(x-\de_{0\in X})+\de_{0\in X}(2\eps_0(\la)-2)+\sum_{i\in X,\,i\neq 0}(8\eps_i(\la)-8)\big).
\end{align*}
\end{Lemma}
\begin{proof}
From Lemma \ref{L161224_3} we have to prove that
\begin{align*}
&\d_{n-2}(D(\la)\da_{\ts_{n-2}})
\geq 
4\de_{\eps_0(\la)\geq 3}+16\sum_{i\neq 0}\de_{\eps_i(\la)\geq 3}
\\
&\ +\d_n\big(D(\la)\big)\big(8(x-1)(x-\de_{0\in X})+\de_{0\in X}(4\eps_0(\la)-4)+\sum_{i\in X,\,i\neq 0}(16\eps_i(\la)-16))\big).
\end{align*}
Note that 
$$
D(\la)\da_{\ts_{n-2}}\cong \bigoplus_{i\in X}\Res_i^2 D(\la)\,\oplus\bigoplus _{i\neq j}\Res_i\Res_j D(\la),
$$
and 
$$\Hom_{\T_{n-2}}\Big(\bigoplus_{i\in X}\Res_i^2 D(\la),\bigoplus _{i\neq j}\Res_i\Res_j D(\la)\Big)=0.
$$
So 
$$
\d_{n-2}(D(\la)\da_{\ts_{n-2}})=\d_{n-2}\Big(\bigoplus_{i\in X}\Res_i^2 D(\la)\Big)
+\d_{n-2}\Big(\bigoplus _{i\neq j}\Res_i\Res_j D(\la)\Big).
$$
By Lemma \ref{L051224},
\begin{align*}
\d_{n-2}\Big(\bigoplus_{i\in X}\Res_i^2 D(\la)\Big)\geq\, 
&\d_n\big(D(\la)\big)\big(\de_{0\in X}(4\eps_0(\la)-4)+\sum_{i\in X,\,i\neq 0}(16\eps_i(\la)-16)\big)
\\
&+4\de_{\eps_0(\la)\geq 3}
+16\sum_{i\neq 0}\de_{\eps_i(\la)\geq 3}.
\end{align*}
So it is enough to prove that
 \begin{equation*}\label{E171224}
 \d_{n-2}\Big(\bigoplus _{i\neq j}\Res_i\Res_j D(\la)\Big)\geq 8\d_n(D(\la))(x-1)(x-\de_{0\in X}),
 \end{equation*}
 which in turn follows from
 \begin{equation}\label{E171224}
\sum_{i,j\in X,\,i> j} \d_{n-2}\big(\Res_i\Res_j D(\la)\oplus \Res_j\Res_i D(\la)\big)\geq 8\d_n\big(D(\la)\big)(x-1)(x-\de_{0\in X}).
 \end{equation}
We may assume that $x>1$. Let $i,j\in X$ with $i\neq j$. If $i,j\neq 0$, Lemmas \ref{L171224} and \ref{L081218_2} give
 \begin{align*}
 \d_{n-2}(\Res_i\Res_j D(\la))&\geq 4\d_n(D(\la)),\\
 \d_{n-2}(\Res_i\Res_j D(\la),\Res_j\Res_i D(\la))&\geq 4\d_n(D(\la)).
 \end{align*}
 So the pair $(i,j)$ with $i,j\in X$ and $i>j>0$ contributes $16\d_n(D(\la))$ to the sum in the left hand side of (\ref{E171224}). 
On the other hand,  if $j=0$ then Lemmas \ref{L171224} and \ref{L081218_2} give 
 \begin{align*}
 \d_{n-2}(\Res_i\Res_0 D(\la))&\geq 2\d_n(D(\la)),\\
  \d_{n-2}(\Res_0\Res_i D(\la))&\geq 2\d_n(D(\la)),\\
 \d_{n-2}(\Res_i\Res_0 D(\la),\Res_0\Res_i D(\la))&\geq 2\d_n(D(\la)),\\
 \d_{n-2}(\Res_0\Res_i D(\la),\Res_i\Res_0 D(\la))&\geq 2\d_n(D(\la)).
 \end{align*}
So, if $0\in X$, then the pair $(i,0)$ with $i\in X$ and $i\neq 0$ contributes $8\d_n(D(\la))$ to the sum in the left hand side of (\ref{E171224}). 
Now (\ref{E171224}) follows.
\end{proof}


\subsection{\boldmath Comparing $\dim\End_{\ts_{n-2,2}}(D(\la)\da_{\ts_{n-2,2}})$ and $\dim\End_{\ts_{n-1}}(D(\la)\da_{\ts_{n-1}})$}
In the next two lemmas we will show that  $\d_{n-2,2}(D(\la)\da_{\T_{n-2,2}})>\d_{n-1}(D(\la)\da_{\T_{n-1}})+\d_n(D(\la))$ in most cases. For $\la\in\TR_p(n)$ this will be used in Lemma \ref{L181224_2} to show that the assumptions of the reduction lemmas of \S\ref{sredlemmas} are satisfied in some important situations.

\begin{Lemma}\label{L051224_3}
Let $\la\in\RP_p(n)$. Then
\[\d_{n-2,2}(D(\la)\da_{\T_{n-2,2}})>\d_{n-1}(D(\la)\da_{\T_{n-1}})+\d_n(D(\la))\]
unless one of the following holds:
\begin{itemize}
\item $\eps_0(\la)\leq 1$, $\eps_j(\la)=1$ for some $j\neq 0$ and  $\eps_i(\la)=0$ for all $i\neq 0,j$;

\item $\eps_0(\la)\leq 2$ and $\eps_i(\la)=0$ for all $i\neq 0$.
\end{itemize}
\end{Lemma}
\begin{proof}
Let $X:=\{i\in I \mid \eps_i(\la)>0\}$, set $x:=|X|$ and 
\begin{align*}
S&:=4(x-1)(x-\de_{0\in X})+\de_{0\in X}(\eps_0(\la)-2)+\sum_{i\in X,\,i\neq 0}(6\eps_i(\la)-8).
\end{align*}
In view of (\ref{EEndDDim}) and Lemmas \ref{n-1} and \ref{L171224_2}, it suffices to prove that $S\geq 2$. 
Moreover, if $\eps_i(\la)\geq 3$ for some $i\neq 0$ it is enough to prove that $S\geq \de_{i,0}$. 

If $x\geq 3$, then $S\geq 11$. If $x=2$ and $0\not\in X$ then $S\geq 4$. If $X=\{0,j\}$ with $j\neq 0$ then 
$
S=\eps_0(\la)+6\eps_j(\la)-6
$, so $S\geq 2$ if $\eps_0(\la)+\eps_1(\la)\geq 3$. If $X=\{j\}$ for $j\neq 0$ then $S=6\eps_j(\la)-6$, so $S\geq 6$ if $\eps_j(\la)\geq 2$. Finally, if $X=\{0\}$ then $S=\eps_0(\la)-2$, so $S\geq 1$ if $\eps_0(\la)\geq 3$. 
\end{proof}

\begin{Lemma}\label{L171224_3}
Let $n\geq 6$ and $\la\in\TR_p(n)\setminus\{\balpha_n\}$. Then
\begin{equation}\label{E150825}
\d_{n-2,2}(D(\la)\da_{\T_{n-2,2}})>\d_{n-1}(D(\la)\da_{\T_{n-1}})+\d_n(D(\la))
\end{equation}
unless one of the following holds:
\begin{itemize}
\item[(1)] $\la=((2p)^a,2p-1,p+1,p^b,p-1,1)$ for some $a,b\geq 0$,

\item[(2)] $\la=(p+1,p^b,p-1)$ for some $b\geq 0$,

\item[(3)] $\la=((2p)^a,2p-1,p+1,p^b,p-1)$ for some $a,b\geq 0$,

\item[(4)] $\la=((2p)^a,p+1,p^b,p-1,1)$ for some $a,b\geq 0$,

\item[(5)] $p>5$ and $\la=(p-2,2)$.

\item[(6)] $p>3$ and $\la=(p^a,p-1,p-2,2,1)$ for some $a\geq 0$,

\item[(7)] $p>3$ and $\la=(p^a,p-2,2,1)$ for some $a\geq 0$,

\item[(8)] $p>3$ and $\la=(p^a,p-1,p-2,2)$ for some $a\geq 0$.

%
\end{itemize}
\end{Lemma}

\begin{proof}
Denote $D:=D(\la)$, $\eps_i:=\eps_i(\la)$, $\d_{n-k}:=\d_{n-k}(D\da_{\T_{n-k}})$ for $k=0,1,2$, and $\d_{n-2,2}:=\d_{n-2,2}(D\da_{\T_{n-2,2}})$. 
We have 
$ 
\d_{n-1}=\d_n(\eps_0+2\eps_1+\dots+2\eps_\ell)
$ 
by (\ref{EEndDDim}) and Lemma~\ref{n-1}. 
So, taking into account Lemma~\ref{L161224_3},  it suffices to prove
\begin{equation}\label{E180825_3}
\d_{n-2}>2\d_n(1+\eps_0+2\eps_1+\dots+2\eps_\ell).
\end{equation}

By Lemma~\ref{LM3}, we have $\la\in A\sqcup B$, where 
$$A:=\{\balpha_{n-k}+\balpha_k\mid 0< 2k\leq n-p-\de_{p\mid k}\}\quad \text{and}\quad B:=\TR'_p(n).$$ 

\vspace{1mm}
\noindent
{\em Case 1: $\la\in A$, i.e. $\la=\balpha_{n-k}+\balpha_k$ with $0< 2k\leq n-p-\de_{p\mid k}$}. 

\vspace{1mm}
\noindent
{\em Case 1.1: $p\mid (n-k)$ and $p\mid k$}.  In this case $\la$ is of the form (1) which we have excluded. 

\vspace{1mm}
\noindent
{\em Case 1.2: $p\mid (n-k)$ and $p\nmid k$}. In this case we can write $\la=((2p)^a,p+c,p^b,p-1,1)$ for $a,b\geq 0$ and $1\leq c\leq p-1$. The case $c=1$ is excluded in (4). So we may assume that $2\leq c\leq p-1$. Let $i=\ttres(c)$. Then $\eps_0=\eps_i=1$ and $\eps_j=0$ for all $j\neq 0,i$, so the right hand side of (\ref{E180825_3}) equals $8\d_n$. On the other hand, by Corollary~\ref{C180825}, we have $\d_{n-2}(\Res_0\Res_i D\oplus \Res_i\Res_0 D)\geq 8\d_n$. If $c\neq 2,p-1$ then $i\neq 1$, and to see (\ref{E180825_3}), it remains to note that $\Res_1\Res_0 D\neq 0$. If $c=2$ or $c=p-1$, then $\eps_1(\tilde e_0\la)=2$, and so 
$\d_{n-2}(\Res_1\Res_0 D)\geq 4\d_n$, $\d_{n-2}(\Res_0\Res_1D)\geq 2\d_n$ by Lemma~\ref{L171224}, while 
$\d_{n-2}(\Res_1\Res_0,\Res_0\Res_1D)\geq 2\d_n$ and $\d_{n-2}(\Res_0\Res_1,\Res_1\Res_0D)\geq 2\d_n$ by Lemma~\ref{L081218_2}, so $\d_{n-2}\geq 10 \d_n$, proving (\ref{E180825_3}).

\vspace{1mm}
\noindent
{\em Case 1.3: $p\nmid (n-k)$ and $p\mid k$}. In this case we can write $\la=((2p)^a,2p-1,p+1,p^b,d)$ for $a,b\geq 0$ and $1\leq d\leq p-1$. The case $d=p-1$ is excluded in (3). If $d=1$ then $\eps_0=3$, and we are done by Lemma~\ref{L051224_3}. Let $2\leq d\leq p-2$, and set $i:=\ttres(d)$. Then $\eps_0=\eps_i=1$ and $\eps_j=0$ for all $j\neq 0,i$, so the right hand side of (\ref{E180825_3}) equals $8\d_n$. On the other hand, by Corollary~\ref{C180825}, we have $\d_{n-2}(\Res_0\Res_i D\oplus \Res_i\Res_0 D)\geq 8\d_n$. If $d\neq 2$ then $i\neq 1$, and to see (\ref{E180825_3}), it remains to note that $\Res_j\Res_i D\neq 0$ for $j:=\ttres(d-1)$. If $d=2$, then  $\eps_0(\tilde e_1\la)=3$, and so 
$\d_{n-2}(\Res_0\Res_1 D)\geq 6\d_n$, 
$\d_{n-2}(\Res_1\Res_0D)\geq 2\d_n$ by Lemma~\ref{L171224}, 
while 
$\d_{n-2}(\Res_1\Res_0,\Res_0\Res_1D)\geq 2\d_n$ and $\d_{n-2}(\Res_0\Res_1,\Res_1\Res_0D)\geq 2\d_n$ by Lemma~\ref{L081218_2}, so $\d_{n-2}\geq 10 \d_n$, proving (\ref{E180825_3}).

\vspace{1mm}
\noindent
{\em Case 1.4: $p\nmid (n-k)$ and $p\nmid k$}. In this case we can write $\la=((2p)^a,p+c,p^b,d)$ for $a,b\geq 0$ and $1\leq c, d\leq p-1$. Set $i:=\ttres(c)$ and $j:=\ttres(d)$. 

If $2\leq c,d\leq p-1$ with $c\neq d+1$ and $c\neq p-d$ then $i,j\neq 0$ and either $i\neq j$ and $\eps_i=\eps_j=1$, or $i=j$ and $\eps_i=2$. In both cases we are done by Lemma~\ref{L051224_3}. 

If $2\leq c,d\leq p-1$ and $c= d+1$ then $b>0$ since otherwise $\la\not\in A$. Moreover, $\eps_j=1$ and $\eps_k=0$ for all $k\neq j$, so the right hand side of (\ref{E180825_3}) equals $6\d_n$. If $c=2$, then $j=1$, and $\eps_0(\tilde e_1\la)=2$ and $\eps_2(\tilde e_1\la)=1$, so $\d_{n-2}(\Res_0\Res_1 D)\geq 4\d_n$ and $\d_{n-2}(\Res_2\Res_1 D)\geq 4\d_n$ by Lemma~\ref{L171224}, so the left hand side of (\ref{E180825_3}) is at least $8\d_n$. If $c>2$, then setting $k=\ttres(c-1)$, we have $k\neq 0$. Moreover, either $k\neq i$ and $\eps_k(\tilde e_j\la)=\eps_i(\tilde e_j\la)=1$, or  $k=i$ and $\eps_i(\tilde e_j\la)=2$. In both cases Lemma~\ref{L171224} implies that the left hand side of (\ref{E180825_3}) is at least $8\d_n$.

The case $2\leq c,d\leq p-1$ and $c= p-d$ is similar to the case $2\leq c,d\leq p-1$ and $c= d+1$ considered in the previous paragraph; one just needs to take into account that $b>0$ when  $c=2$ since otherwise $\la\not\in A$.

Let $c=1$. Set $i=\ttres(d)$. If $d=1$ then $\eps_0(\la)\geq 3$. 
If $2\leq d<p-1$ and $a>0$ then $\eps_i=1$ and $\eps_0=2$. So in both cases we are again done by Lemma~\ref{L051224_3}. If $d=2<p-1$ and $a=0$ then $\eps_1=\eps_0=1$, the right hand side of (\ref{E180825_3}) equals $8\d_n$, $\eps_0(\tilde e_1\la)=3$, so $\d_{n-2}(\Res_0\Res_1 D)\geq 6\d_n$ and $\d_{n-2}(\Res_1\Res_0 D)\geq 2\d_n$ by Lemma~\ref{L171224}, and $\d_{n-2}(\Res_0\Res_1 D,\Res_1\Res_0 D)\geq 2\d_n$ by Lemma~\ref{L081218_2}, so the left hand side of (\ref{E180825_3}) is at least $10\d_n$. If $2<d<p-1$ and $a=0$ then the right hand side of (\ref{E180825_3}) equals $8\d_n$, $\d_{n-2}(\Res_0\Res_i D\oplus\Res_i\Res_0 D)\geq 8\d_n$ by Corollary~\ref{C180825}, and $\d_n(\Res_k\Res_iD)>0$ for $k=\ttres(d-1)$ so the left hand side of (\ref{E180825_3}) is greater than $8\d_n$. If $d=p-1$ we may assume that $a>0$ since we have excluded the case (2);  if $p>3$ this case is similar to the case $2\leq c,d\leq p-1$ with $c\neq d+1$, while if $p=3$ then $\eps_1(\la)=1$, $\eps_0(\la)=0$ and $\eps_0(\tilde e_1\la)=4$ and we can again conclude by Lemma~\ref{L171224}. 

Let now $d=1$ and $2\leq c\leq p-1$. If $3\leq c\leq p-2$ then $\eps_0(\la)=2$ and $\eps_i(\la)=1$, so we can conclude by Lemma~\ref{L051224_3}. If $c=2$ or $p-1$ then $b\geq 1$ since $\la\in\RP_p(n)$, $\eps_0(\la)=2$ and $\eps_k(\la)=0$ for $k\neq 0$, so the right hand side of (\ref{E180825_3}) equals $6\d_n$. Further $\eps_1(\tilde e_0\la)=1$, so $\d_{n-2}(\Res_0^2 D)\geq 4\d_n$ by Lemma \ref{L051224} and $\d_{n-2}(\Res_1\Res_0 D)>2\d_n$ by Lemma \ref{L171224}. In particular the left hand side of (\ref{E180825_3}) is $>6\d_n$.


\vspace{1mm}
\noindent
{\em Case 2: $\la\in B$.} In this case, by definition, we have $p>3$.  As we have excluded the cases (6),(7),(8), we are left with the following two subcases.

\vspace{1mm}
\noindent
{\em Case 2.1: $\la\in \{(p^a,b,c)\mid a\geq 0,\ 1= c<b\leq p-2\text{ or }2\leq c<b\leq p-1\}$.} 
\vspace{1mm}

If $c=1$ and $b=2$ then $a>0$ by the assumption $n\geq 6$. So  $\eps_0=2$ and $\eps_i=0$ for all $i\neq 0$. We have $\d_{n-2}(\Res_0^2 D)\geq 4\d_n$ by Lemma~\ref{L051224}, and 
$\d_{n-2}(\Res_1\Res_0 D)>2\d_n$ 
by Lemma~\ref{L171224}, so 
$
\d_{n-2}\geq \d_{n-2}(\Res_0^2 D)+\d_{n-2}(\Res_1\Res_0 D)>6\d_n,
$
and we have verified (\ref{E180825_3}).

If $c=1$ and $a=0$, then $b\geq 5$ by the assumption $n\geq 6$.   In this case $\eps_0=1$, $\eps_i=1$ for some $i\neq 0,1$ and $\eps_j=0$ for all $j\neq 0,1$ since $d<p-1$ as $c=1$.  Moreover, $\d_{n-2}(\Res_i\Res_0 D\oplus \Res_0\Res_i D)\geq 8\d_n$ by Corollary~\ref{C180825}, and $\Res_{i+1}\Res_i D\neq 0$ or $\Res_{i-1}\Res_i D\neq 0$, so $
\d_{n-2}>8\d_n$, proving  (\ref{E180825_3}).

If $c=1$, $b>2$ and $a>0$ then $\eps_0=2$ and $\eps_i=1$ for some $i\neq 0$, so we are done by Lemma~\ref{L051224_3}. 

If $2\leq c\leq b-2$ and $b\neq p-c$, then there exist distinct $i,j\neq 0$ with $\eps_i=\eps_j=1$, and we are done by Lemma~\ref{L051224_3}. 

Suppose $2\leq c=b-1$. Setting $i:=\ttres(c)$, we have $\eps_i=1$ and $\eps_j=0$ for all $j\neq i$. If $c=2$ then $a>0$ by the assumption $n\geq 6$, and $\eps_0(\tilde e_1\la)=2$, $\eps_2(\tilde e_1\la)=1$, so  
$\d_{n-2}(\Res_0\Res_1 D)\geq 4\d_n$ and $\d_{n-2}(\Res_2\Res_1 D)\geq 4\d_n$ by Lemma~\ref{L171224}, hence $\d_{n-2}\geq 8\d_n$, proving (\ref{E180825_3}). If $3\leq c\neq \ell+1$ then there exists distinct non-zero $j,k$ with $\eps_j(\tilde e_i\la)=1$ and $\eps_k(\tilde e_i\la)=1$, so  
$\d_{n-2}(\Res_j\Res_i D)\geq 4\d_n$ and $\d_{n-2}(\Res_k\Res_i D)\geq 4\d_n$ by Lemma~\ref{L171224}, hence $\d_{n-2}\geq 8\d_n$, proving (\ref{E180825_3}). If $c=\ell+1$ then there exists $j\neq 0$ with $\eps_j(\tilde e_i\la)=2$, so $\d_{n-2}(\Res_j\Res_i D)\geq 8\d_n$ by Lemma~\ref{L171224}, hence $\d_{n-2}\geq 8\d_n$, proving (\ref{E180825_3}).

Suppose $2\leq c$ and $b=p-c$. Setting $i:=\ttres(c)$, we have $\eps_i=1$ and $\eps_j=0$ for all $j\neq i$. If $c=2$, we have $a>0$, since the case $\la=(p-2,2)$ has been excluded in (5), and then $\eps_0(\tilde e_1\la)=2$, $\eps_2(\tilde e_1\la)=1$, so  
$\d_{n-2}(\Res_0\Res_1 D)\geq 4\d_n$ and $\d_{n-2}(\Res_2\Res_1 D)\geq 4\d_n$ by Lemma~\ref{L171224}, hence $\d_{n-2}\geq 8\d_n$, proving (\ref{E180825_3}). If $c>2$ then $\eps_{i-1}(\tilde e_i\la)=1$, $\eps_{i+1}(\tilde e_i\la)=1$, so  
$\d_{n-2}(\Res_{i-1}\Res_i D)\geq 4\d_n$ and $\d_{n-2}(\Res_{i+1}\Res_i D)\geq 4\d_n$ by Lemma~\ref{L171224}, hence $\d_{n-2}\geq 8\d_n$, proving (\ref{E180825_3}).

\vspace{1mm}
\noindent
{\em Case 2.2: $\la\in \{(p^a,p-1,b,1)\mid a\geq 0,\,2\leq b\leq p-2\}$}

If $b=2$, then $\eps_0=\eps_1=1$ and $\eps_j=0$ for all $j\neq 0,1$. Moreover, $\eps_1(\tilde e_0\la)=2$, so $\d_{n-2}(\Res_1\Res_0 D)\geq 4\d_n$ 
and $\d_{n-2}(\Res_0\Res_1 D)\geq 2\d_n$ 
by Lemma~\ref{L171224}. By Lemma~\ref{L081218_2}, we also have $\d_{n-2}(\Res_0\Res_1 D,\Res_1\Res_0 D)\geq 2\d_n$ and  $\d_{n-2}(\Res_1\Res_0 D,\Res_0\Res_1 D)\geq 2\d_n$. Thus $\d_{n-2}\geq 10\d_n$, proving (\ref{E180825_3}).

If $b=p-2$, then $\eps_0=\eps_2=1$ and $\eps_j=0$ for all $j\neq 0,2$. By Corollary~\ref{C180825}, we get $\d_{n-2}(\Res_0\Res_2 D\oplus \Res_2\Res_0 D)\geq 8\d_n$. Moreover, by Lemma~\ref{L081218_2} we have $\d_{n-2}(\Res_3\Res_2D)\geq 2\d_n$ if $p>5$, and  $\d_{n-2}(\Res_1\Res_2D)\geq 2\d_n$ if $p=5$. Thus $\d_{n-2}\geq 10\d_n$, proving (\ref{E180825_3}).

If $2<b<p-2$, then, setting $i:=\ttres(c)$, we have $i>2$, $\eps_0=\eps_i=1$, and $\eps_j=0$ for all $j\neq 0,i$. By Corollary~\ref{C180825}, we get $\d_{n-2}(\Res_0\Res_i D\oplus \Res_i\Res_0 D)\geq 8\d_n$. Moreover, 
by Lemma~\ref{L081218_2} we have $\d_{n-2}(\Res_1\Res_0D)\geq 2\d_n$. Thus $\d_{n-2}\geq 10\d_n$, proving (\ref{E180825_3}).
\end{proof}

We next consider most of the exceptional partitions in Lemma \ref{L171224_3}. 
For these we obtain special homomorphisms $\psi,\psi_1,\psi_2$ as in Lemmas~\ref{inv_end} and \ref{inv_mixedhom} with  $\al=(n-2,2)$. In order to do this, we now need to work with modules instead of supermodules. 

\begin{Lemma}\label{L181224}
Let $n\geq 6$, $G\in\{\ts_n,\tA_n\}$ and $L$ be an irreducible spin $\F G$-module labeled by a partition $\la$ which has one of  the following forms:
\begin{itemize}
\item[(1)] $(p+1,p^a,p-1)$ for some $a\geq 0$,

\item[(2)] $((2p)^a,2p-1,p+1,p^b,p-1)$ for some $a,b\geq 0$,

\item[(3)] $((2p)^a,p+1,p^b,p-1,1)$ for some $a\geq 1$ and $b\geq 0$,           

\item[(4)] $(p^a,p-2,2,1)$ for some $a\geq 1$ and $p>3$,

\item[(5)] $(p^a,p-1,p-2,2)$ for some $a\geq 0$ and $p>3$,

\item[(6)] $(p-2,2)$ for some $p>5$.
\end{itemize}
Then there exists $\psi\in\Hom_G({}^\pi M^{(n-2,2)},\End_\F(L))$ such that $\psi\circ\iota_{(n-2,2)}\neq 0$.
\end{Lemma}

\begin{proof}
By assumption, $L=D(\la;\eps)$ or $E(\la;\eps)$ for $\eps\in\{0,+,-\}$ with $\la$ of the form (1)-(6). By Lemma~\ref{L220825}, we have to prove 
\begin{equation}\label{E220825_1}
\dim\End_{\ts_{n-2,2}\cap G}(L\da_{\ts_{n-2,2}\cap G})>\dim\End_{\ts_{n-1}\cap G}(L\da_{\ts_{n-1}\cap G}).
\end{equation}

\vspace{2mm}
\noindent
{\em Claim: If $\eps=0$ then it suffices to prove that }
\begin{equation}\label{E220825_2}
\d_{n-2,2}(D(\la)\da_{\ts_{n-2,2}})>\d_{n-1}(D(\la)\da_{\ts_{n-1}}).
\end{equation}

\vspace{2mm}
\noindent
Indeed, if $L=D(\la;0)$ then  $L=|D(\la)|$ by Lemma~\ref{lmodules}, and so for $k=1,2$, we have that $\d_{n-k,k}(D(\la)\da_{\ts_{n-k,k}})=\dim\End_{\ts_{n-k,k}}(D\da_{\ts_{n-k,k}})$, hence (\ref{E220825_1}) is equivalent to (\ref{E220825_2}). On the other hand, if   $L=E(\la;0)$ then 
by Lemma~\ref{lmodules}, we have $L^{\oplus 2}\cong D(\la)\da_{\tA_n}$, hence $\End_\F(D(\la)\da_{\tA_n})\cong \End_\F(L)^{\oplus 4}$ as $\F\tA_n$-modules. 
Given (\ref{E220825_2}), we deduce from Lemma~\ref{L220825} that there exists $\in\Hom_{\ts_n}( {}^\pi M^{(n-2,2)},\End_\F(D(\la)))$ such that $\psi\circ\iota_{(n-2,2)}\neq 0$. Restricting to $\tA_n$, there exists an $\F\tA_n$-homomorphism ${}^\pi M^{(n-2,2)}\to\End_\F(D(\la)\da_{\tA_n}))\cong \End_\F(L)^{\oplus 4}$ such that $\psi\circ\iota_{(n-2,2)}\neq 0$. Hence there exists an $\F\tA_n$-homomorphism ${}^\pi M^{(n-2,2)}\to\End_\F(L)$ such that $\psi\circ\iota_{(n-2,2)}\neq 0$. 

\vspace{2mm}
We now go through different cases. 

\vspace{2mm}
\noindent
{\em Case 1: $\la$ is of the forms (5) or (6), or $\la$ is of the forms (1),(2) and $p>3$.} 

In this case we have by Lemmas \ref{branching}, \ref{LAGa} and  \ref{product},
\begin{equation}\label{E220825_3}
\begin{split}
D(\la)\da_{\ts_{n-1}}&\cong D(\tilde e_1\la)^{\oplus (1+a_p(\la))}
\\ 
D(\la)\da_{\ts_{n-2,2}}&\cong D(\tilde e_0\tilde e_1\la,(2))\oplus D(\tilde e_2\tilde e_1\la,(2))^{\oplus (1+a_p(\la))},
\end{split}
\end{equation}
and (\ref{E220825_2}) easily follows using Lemmas~\ref{LAGa} and  \ref{product}. So by the Claim, we may assume that $\eps=\pm$. 
In this case, by Lemma \ref{lmodules}, either $a_p(\la)=0$ and $L=E(\la;\pm)$, or $a_p(\la)=1$ and $L=D(\la;\pm)$. 
From (\ref{E220825_3}) and Lemmas~\ref{LAGa}, \ref{product}, in the first case we deduce  
\begin{align*}
E(\la;\pm)\da_{\tA_{n-1}}&\cong E(\tilde e_1\la;0),\\
E(\la;\pm)\da_{\tA_{n-2,2}}&\cong E(\tilde e_0\tilde e_1\la,(2);\de)\oplus E(\tilde e_2\tilde e_1\la,(2);0),
\end{align*}
where $\de=\pm$ or $\mp$, while in the second case we deduce 
\begin{align*}
D(\la;\pm)\da_{\ts_{n-1}}&\cong D(\tilde e_1\la;0),\\
D(\la;\pm)\da_{\ts_{n-2,2}}&\cong D(\tilde e_0\tilde e_1\la,(2);\de)\oplus D(\tilde e_2\tilde e_1\la,(2);0),
\end{align*}
where $\de=\pm$ or $\mp$. 
In both cases (\ref{E220825_1}) follows immediately. 

\vspace{2mm}
\noindent 
{\em Case 2: $\la$ is of the forms (1),(2) and $p=3$.} 

In this case, by Lemma \ref{branching}, 
\begin{equation}\label{E220825_4}
D(\la)\da_{\ts_{n-1}}\cong D(\tilde e_1\la)^{\oplus (1+a_p(\la))}
\quad\text{and}\quad
D(\la)\da_{\ts_{n-2}}\cong (e_0D(\tilde e_1\la))^{\oplus (1+a_p(\la))}.
\end{equation}
Moreover, $\eps_0(\tilde e_1\la)=3$ implies by Lemmas~\ref{branching} and \ref{LAGa} that 
$$\dim\End_{\ts_{n-2}}(e_0D(\tilde e_1\la))=3\dim\End_{\ts_{n-2}}(D(\tilde e_0\tilde e_1\la))=3\dim\End_{\ts_{n-1}}(D(\tilde e_1\la)).
$$
So Lemma~\ref{L161224_3} and (\ref{E220825_4}) imply  (\ref{E220825_2}), so by the Claim, we may assume that $\eps=\pm$. So either $a_p(\la)=1$ and $L=D(\la;\pm)$, or $a_p(\la)=0$ and $L=E(\la;\pm)$. In both cases, it follows from the first isomorphism in (\ref{E220825_4}) that $L\da_{\ts_{n-1}\cap G}$ is irreducible, so the right hand side in (\ref{E220825_1}) equals $1$. We show that the left hand side in (\ref{E220825_1}) is at least $2$.

By Lemmas~\ref{L161224_3} and \ref{L101218_2}, 
\[(D(\la)\da_{\ts_{n-2,2}})^{\oplus 2}\cong D(\la)\da_{\ts_{n-2}}\boxtimes D(2)\cong (e_0D(\tilde e_1\la)\circledast D(2))^{\oplus 2},\]
with 
\begin{enumerate}
\item[$\bullet$] $\soc(e_0D(\tilde e_1\la)\circledast D(2))\cong \head(e_0D(\tilde e_1\la)\circledast D(2))\cong D(\tilde e_0\tilde e_1\la,(2))$,
\item[$\bullet$] $[e_0D(\tilde e_1\la)\circledast D(2):D(\tilde e_0\tilde e_1\la,(2))]=3$, 
\item[$\bullet$] $\dim\End_{\ts_{n-2,2}}(e_0D(\tilde e_1\la)\circledast D(2))=\frac{6}{1+a_p(\tilde e_0\tilde e_1\la)}=3(1+a_p(\la))$, 
\end{enumerate}  
using  
Lemma~\ref{LAGa} for the last equality. By Krull-Schmidt, $D(\la)\da_{\ts_{n-2,2}}\cong e_0D(\tilde e_1\la)\circledast D(2)$. 

If $a_p(\la)=1$ this implies 
$$
D(\la;+)\da_{\ts_{n-2,2}}\oplus D(\la;-)\da_{\ts_{n-2,2}}
\cong |e_0D(\tilde e_1\la)\circledast D(2)|.
$$
Since $D(\la,\pm)\otimes\sgn\cong  D(\la,\mp)$, we deduce that 
$\soc (D(\la;\pm)\da_{\ts_{n-2,2}})\cong D(\tilde e_0\tilde e_1\la,(2);\pm\de)$ and $\head (D(\la;\pm)\da_{\ts_{n-2,2}})\cong D(\tilde e_0\tilde e_1\la,(2);\pm\de')$ for some $\de,\de'\in\{+,-\}$ are simple.  So there exists $\eps\in\{+,-\}$ such that
\begin{align}\label{E050925}
D(\la;\pm)\da_{\ts_{n-2,2}}&\sim \overbrace{D(\tilde e_0\tilde e_1\la,(2);\pm\de)}^{\text{socle}}|B|D(\tilde e_0\tilde e_1\la,(2);\pm\eps)|A|\overbrace{D(\tilde e_0\tilde e_1\la,(2);\pm\de')}^{\text{head}},
\end{align}
where $A$ and $B$ have no composition factor of the form $D(\tilde e_0\tilde e_1\la,(2);+)$ or $D(\tilde e_0\tilde e_1\la,(2);-)$. 
Note by (\ref{E*}) that 
$$
(D(\la;+)\da_{\ts_{n-2,2}})^*\cong D(\la;+)^*\da_{\ts_{n-2,2}}\cong
D(\la;\pm)\da_{\ts_{n-2,2}}.
$$
If $(D(\la;+)\da_{\ts_{n-2,2}})^*\cong D(\la;+)\da_{\ts_{n-2,2}}$, then (\ref{E050925}) implies that $D(\tilde e_0\tilde e_1\la,(2);\pm\eps)\cong D(\tilde e_0\tilde e_1\la,(2);\pm\eps)$ so 
$$D(\tilde e_0\tilde e_1\la,(2);\pm\de)\cong D(\tilde e_0\tilde e_1\la,(2);\pm\de),$$ hebce $\de=\de'$. If $(D(\la;+)\da_{\ts_{n-2,2}})^*\cong D(\la;-)\da_{\ts_{n-2,2}}$, a similar analysis shows that again $\de=\de'$. Now it is clear that $\dim\End_{\ts_{n-2,2}}(D(\la;\pm)\da_{\ts_{n-2,2}})\geq 2$.

If $a_p(\la)=0$ then 
$$
E(\la;+)\da_{\tA_{n-2,2}}\oplus E(\la;-)\da_{\tA_{n-2,2}}
\cong (e_0D(\tilde e_1\la)\circledast D(2))\da_{\tA_{n-2,2}},
$$
and a similar argument shows that $\dim\End_{\tA_{n-2,2}}(E(\la;\pm)\da_{\tA_{n-2,2}})\geq 2$. 

\vspace{2mm}
\noindent
{\em Case 3: $\la$ is of the forms (3),(4).} 

We have $\eps_0(\la)=2$ and $\eps_i(\la)=0$ for all $i\neq 0$.
So by Lemmas~\ref{branching} and \ref{LAGa}, we have 
\begin{align}
\label{E1}
&\dim\End_{\ts_{n-1}}(D(\la)\da_{\ts_{n-1}})=2(1+a_p(\tilde e_0\la))=2(1+a_p(\la)),
\\
&[D(\la)\da_{\ts_{n-1}}:D(\tilde e_0\la)]=2,
\\
\label{E2}
&\soc (D(\la)\da_{\ts_{n-1}})\cong \head(D(\la)\da_{\ts_{n-1}})\cong D(\tilde e_0\la),
\\
\label{E3}
&\text{$\eps_0(\mu)=0$ whenever $\mu\neq \tilde e_0\la$ and $[D(\la)\da_{\ts_{n-1}}:D(\mu)]\neq 0$,}
\\
\label{E4}
&\Res_0^2D(\la)\cong D(\tilde e_0^2\la)^{\oplus 2}.
\end{align}
Moreover, if $A$ is the bottom normal node of $\la$, then $\la_A\in\RP_p(n)$ and $\eps_1(\la)=1$. 
So by Lemma~\ref{normal}, we have $\Res_1\Res_0D(\la)\neq 0.$
By Lemma \ref{product}, we have 
\begin{equation}\label{E5}
D(\la)\da_{\ts_{n-2,2}}\cong D(\tilde e_0^2\la,(2))^{\oplus (1+a_p(\la))}\oplus V
\end{equation}
for some self-dual supermodule $V$ with composition factors of the opposite type than $D(\tilde e_0^2\la,(2))$. 
If $a_p(\la)=0$ then the first summand has the endomorphism algebra of dimension $2$. If $a_p(\la)=1$ then the first summand has the endomorphism algebra of dimension $4$. Since $\End_{\ts_{n-2,2}}(V)\neq 0$, in both cases we get that $\dim \End_{\ts_{n-2,2}}(D(\la)\da_{\ts_{n-2,2}})$ is greater than the right hand side of (\ref{E1}), proving (\ref{E220825_2}). So by the Claim, we may assume that $\eps=\pm$. 

We give details for the case $a_p(\la)=1$, the case $a_p(\la)=0$ being similar. By Lemma \ref{lmodules}, in the case $a_p(\la)=1$ we have 
$$
D(\la;\pm)\da_{\ts_{n-2,2}}\cong D(\tilde e_0^2\la,(2);0)\oplus V_\pm
$$
for some $V_\pm\neq 0$. So $\dim\End_{\ts_{n-2,2}}(D(\la;\pm)\da_{\ts_{n-2,2}})\geq 2$.
It will be enough to show that $\dim\End_{\ts_{n-1}}(D(\la;\pm)\da_{\ts_{n-1}})=1$. By (\ref{E1}),(\ref{E2}),(\ref{E3}), we have 
\begin{align*}
&\soc (D(\la;\pm)\da_{\ts_{n-1}})\cong D(\tilde e_0\la;\pm\de),\ 
\head(D(\la;\pm)\da_{\ts_{n-1}})\cong D(\tilde e_0\la;\pm\de'),
\\
&[D(\la;\pm)\da_{\ts_{n-1}}:D(\tilde e_0\la;+)]+[D(\la;\pm)\da_{\ts_{n-1}}:D(\tilde e_0\la;-)]=2,
\\
&\text{$\eps_0(\mu)=0$ whenever $\mu\neq \tilde e_0\la$ and $[D(\la;\pm)\da_{\ts_{n-1}}:D(\mu;\kappa)]\neq 0$,}
\\
&[\Res_0^2D(\la;\pm)]= D(\tilde e^2_0\la;\pm\kappa)\oplus D(\tilde e^2_0\la;\pm\kappa').
\end{align*}
with $\kappa=\kappa'$ if $\de=\de'$. Moreover, by (\ref{E5}), 
\[\Res_0^2 D(\la;+)\cong D(\tilde e_0^2\la,(2),0)\da_{\ts_{n-2}}\cong\Res_0^2 D(\la;-).\]
So $\ka\neq \ka'$, hence $\de\neq \de'$, and it follows that 
$\dim\End_{\ts_{n-1}}(D(\la;\pm)\da_{\ts_{n-1}})=1$.
\end{proof}

We are now ready to check that key assumptions of Lemmas \ref{inv_end} or \ref{inv_mixedhom} hold for $\al=(n-2,2)$ and most $\la\in\TR_p(n)$.

\begin{Lemma}\label{L181224_2}
Let $n\geq 6$, $G\in\{\ts_n,\tA_n\}$ and $L$ be an irreducible spin $\F G$-module labeled by a
partition $\la\in\TR_p(n)$. Assume that the following conditions hold:
\begin{itemize}
\item $\la\not\in\{\balpha_n,\, ((2p)^a,2p-1,p+1,p^b,p-1,1),\,  (p+1,p^b,p-1,1)\mid a,b\geq 0\}$;



\item if $p>3$ then $\la\not\in\{(p-2,2,1),\, (p^a,p-1,p-2,2,1)\mid a\geq 0\}$.

\end{itemize}
Then one of the following holds:
\begin{enumerate}[\rm(i)] 
\item there exists 
$\psi\in\Hom_G({}^\pi M^{(n-2,2)},\End_\F(L))$ such that $\psi\circ\iota_{(n-2,2)}\neq 0$. 

\item $G=\ts_n$, $L=D(\la;\pm)$ and there exist
$$\psi_1,\psi_2\in\Hom_G\big({}^\pi M^{(n-2,2)},\Hom_\F(D(\la;\pm),D(\la;\mp))\big)$$ such that $\psi_1\circ\iota_{(n-2,2)}$ and $\psi_1\circ\iota_{(n-2,2)}$ are linearly independent. 

\item $G=\tA_n$, $E=E(\la;\pm)$ and there exists 
$$\psi_1,\psi_2\in\Hom_G\big({}^\pi M^{(n-2,2)},\Hom_\F(E(\la;\pm),E(\la;\mp))\big)$$ such that $\psi_1\circ\iota_{(n-2,2)}$ and $\psi_1\circ\iota_{(n-2,2)}$ are linearly independent. 

\end{enumerate}
\end{Lemma}

\begin{proof}
For $H$-modules $V,W$, we denote  $\d(V,W):=\dim\Hom_H(V,W)$ and $\d(V):=\dim\End_H(V)$. 

Note that either Lemma \ref{L171224_3} or Lemma \ref{L181224} applies. If Lemma \ref{L181224} applies, we have (i). So we may assume that Lemma \ref{L171224_3} applies. 
Then
\begin{equation}\label{E260825}
\d(D(\la)\da_{\ts_{n-2,2}})>\d(D(\la)\da_{\ts_{n-1}})+\d(D(\la)).
\end{equation}

\vspace{2mm}
\noindent
{\em Case 1: $a_p(\la)=0$ and $G=\ts_n$.} We have $L=D(\la;0)=|D(\la)|$, and (i) follows from Lemma~\ref{L220825}.

\vspace{2mm}
\noindent
{\em Case 2:  $a_p(\la)=1$ and $G=\tA_n$.} In this case $L=E(\la;0)$ and $D(\la)\da_{\tA_n}\cong L^{\oplus 2}$, so by (\ref{E260825}). In this case (i) holds as in the claim in the proof of Lemma \ref{L181224}. 

\vspace{2mm}
\noindent
{\em Case 3: $a_p(\la)=1$ and $G=\ts_n$}. We have  $L=D(\la;\pm)$ and $D(\la)=D(\la;+)\oplus D(\la;-)$. 
By Lemma~\ref{LHomPM}, it suffices to prove the lemma for $L=D(\la,\eps)$ for any choice of $\eps$. 
By (\ref{E260825}), 
\[\d((D(\la;+)\oplus D(\la;-))\da_{\ts_{n-2,2}})\geq\d((D(\la;+)\oplus D(\la;-))\da_{\ts_{n-1}})+3.\]
If there exists $\eps$ with
$\d(D(\la;\eps)\da_{\ts_{n-2,2}})>\d(D(\la;\eps)\da_{\ts_{n-1}})$
then by Lemma~\ref{L220825}, we have (i) for $L=D(\la,\eps)$. 
So we may assume that no such $\eps$ exists. As $\d(L\da_{\ts_{n-2,2}})\geq\d(L\da_{\ts_{n-1}})$
by Lemma~\ref{L220825}, we deduce that $\d(D(\la;\eps)\da_{\ts_{n-2,2}})=\d(D(\la;\eps)\da_{\ts_{n-1}})$ for all $\eps$. 
Therefore 
\begin{align*}
&\d(D(\la;+)\da_{\ts_{n-2,2}},D(\la;-)\da_{\ts_{n-2,2}})+\d(D(\la;-)\da_{\ts_{n-2,2}},D(\la;+)\da_{\ts_{n-2,2}})\\
\geq\,&\d(D(\la;+)\da_{\ts_{n-1}},D(\la;-)\da_{\ts_{n-1}})+\d(D(\la;-)\da_{\ts_{n-1}},D(\la;+)\da_{\ts_{n-1}})+3.
\end{align*}
So for some $\eps$, we must have 
\begin{align*}
&\d(D(\la;\eps)\da_{\ts_{n-2,2}},D(\la;-\eps)\da_{\ts_{n-2,2}})\geq\d(D(\la;\eps)\da_{\ts_{n-1}},D(\la;-\eps)\da_{\ts_{n-1}})+2,
\end{align*}
and (ii) holds for $L=D(\la;\eps)$ by Lemma~\ref{L220825}. 

\vspace{2mm}
\noindent
{\em Case 4:  $a_p(\la)=0$ and $G=\tA_n$}. We have $L=E(\la;\pm)$, $D(\la)\da_{\tA_n}=E(\la;+)\oplus E(\la;-)$ and $E(\la;\pm)\ua_{\tA_n}^{\ts_n}\cong D(\la)$. So, using Frobenius reciprocity and Mackey's theorem, we get 
\begin{align*}
\d(E(\la;\pm)\da_{\tA_{n-2,2}},(E(\la;+)\oplus E(\la;-))\da_{\tA_{n-2,2}})
&=\d(E(\la;\pm)\da_{\tA_{n-2,2}},D(\la)\da_{\tA_{n-2,2}})\\
&=\d(E(\la;\pm)\da_{\tA_{n-2,2}}\ua^{\ts_{n-2,2}},D(\la)\da_{\ts_{n-2,2}})\\
&=\d(E(\la;\pm)\ua^{\ts_n}\da_{\ts_{n-2,2}},D(\la)\da_{\ts_{n-2,2}})\\
&=\d(D(\la)\da_{\ts_{n-2,2}}).
\end{align*}
Similarly
\begin{align*}
\d(E(\la;\pm)\da_{\tA_{n-1}},(E(\la;+)\oplus E(\la;-))\da_{\tA_{n-1}})=
\d(D(\la)\da_{\ts_{n-1}}).
\end{align*}
So by (\ref{E260825}), 
\begin{align*}
&\d(E(\la;\pm)\da_{\tA_{n-2,2}},(E(\la;+)\oplus E(\la;-))\da_{\tA_{n-2,2}}) 
\\
\geq\,&\d(E(\la;\pm)\da_{\tA_{n-1}},(E(\la;+)\oplus E(\la;-))\da_{\tA_{n-1}})+2.
\end{align*}
If (i) does not hold, then by Lemma~\ref{L220825}, we have $\d(E(\la;\pm)\da_{\tA_{n-2,2}})=\d(E(\la;\pm)\da_{\tA_{n-1}})$, so 
\begin{align*}
\d(E(\la;\pm)\da_{\tA_{n-2,2}},E(\la;\mp)\da_{\tA_{n-2,2}})\geq\d(E(\la;\pm)\da_{\tA_{n-1}},E(\la;\mp)\da_{\tA_{n-1}})+2,
\end{align*}
from which (iii) follows by Lemma~\ref{L220825}.
\end{proof}

\section{Restrictions to maximal imprimitive subgroups}

\subsection{Restrictions to maximal intransitive subgroups}
\label{sSMaxIntr}

In this subsection we classify irreducible restrictions of spin representations to maximal intransitive subgroups. For $p>3$ (and non-basic representations), this is contained in \cite[Theorem 5.16]{P}.\footnote{Note that \cite[Lemma 3.14(i)]{P} contains an error. That lemma is used in the proof of the crucial \cite[Lemma 5.10]{P}. The statement of \cite[Lemma 5.10]{P} is correct, but to fix the proof one needs more work. We will pursue this elsewhere. Here, we will reprove the main results from \cite{P} by a different method since at any rate we need to extend them to the case $p=3$.} 

Recall the definition of Jantzen-Seits partitions from \S\ref{SSAddRem}, in particular the sets $\JS^{(i)}$.

\begin{Theorem}\label{thmintr}
Let $G\in\{\ts_n,\tA_n\}$, $L$ be an irreducible spin $\F G$-module, and $H=\ts_{n-k,k}\cap G$ for some $1\leq k\leq n/2$.  Then $L\da_H$ is irreducible if and only if one of the following holds:
\begin{enumerate}[\rm(i)]
\item $L$ is basic, $p\nmid k$, $p\nmid (n-k)$ and one of the following holds:
\begin{enumerate}[\rm(a)]
\item $G=\ts_n$ and $p\mid n$ if $n$ is odd,

\item $G=\tA_n$ and $p\mid n$ if $n$ is even;
\end{enumerate}

\item $k=1$ and one of the following holds:
\begin{enumerate}[\rm(a)]
\item $L=D(\la;\eps)$ or $E(\la;\eps)$ for $\la\in\JS^{(0)}$,

\item $L=D(\la;\pm)$ or $E(\la;\pm)$ for $\la\in\JS^{(i)}$ with $i\neq 0$;
\end{enumerate}

\item $k=2$ and $L=D(\la;\eps)$ or $E(\la;\eps)$ for $\la\in\JS^{(0)}$.
\end{enumerate}
\end{Theorem}

\begin{proof}
The case where $L$ is basic is covered by \cite[Corollary 4.2]{KT}. So we may assume that $L$ is not basic, i.e. $L=D(\la;\eps)$ or $L=E(\la;\eps)$ with $\la\neq \balpha_n$. We set $h:=h(\la)$. 

For $n\leq 7$ the lemma is checked using the decomposition matrices \cite[Theorem 4.4]{Mu}, \cite{GAP}, and branching in characteristic $0$  \cite[Theorems 8.1, 8.3]{stem}. So we may assume that $n\geq 8$. This assumption guarantees that $(n-3,3)\in\Par_p(n)$, and 
 $h((n-3,3)^\Mull)\geq 3$ hence  
$(n-3,3)^\Mull\not\hspace{-1.1mm}\unrhd(n-3,3)$.

Recall the homomorphisms $\iota_\al$ and $\si_\al$ from (\ref{ESiAl}).  
By Lemma \ref{L7}, there exists 
$$\psi\in\Hom_G({}^\pi M^{(n-3,3)},\End_\F(L))$$ such that $\psi\circ\iota_{(n-3,3)}\neq 0$.

If $3\leq k\leq n/2$ then by Lemmas \ref{L060125_2} and \ref{inv3intr} there exists 
$\phi\in\Hom_G(\bone\ua_{H}^G,{}^\pi M^{(n-3,3)})$ such that $\si_3\circ \phi\neq 0$. 
So $L\da_{H}$ is reducible by Lemma \ref{inv_end}.
Thus we may assume that $k=1$ or $2$.

Let $k=1$. If $\la$ has normal nodes of two different residues, then by Lemma \ref{branching}, the restriction $D(\la)\da_{\ts_{n-1}}$ has non-zero components in at least two superblocks. So the same holds for $D(\la;\eps)$ and $E(\la;\eps)$. So we may assume that all normal nodes of $\la$ have the same residue $i$. If $\eps_i(\la)>1$ then by Lemma \ref{branching}, the restriction $D(\la)\da_{\ts_{n-1}}$ is not semisimple. So the same holds for $D(\la;\eps)\da_{\ts_{n-1}}$ and $E(\la;\eps)\da_{\tA_{n-1}}$. So we may assume that $\la\in\JS^{(i)}$. In this case $D(\la)\da_{\s_{n-1}}\cong D(\tilde e_i\la)^{\oplus (1+\de_{i\neq 0}a_p(\la))}$ by Lemma \ref{branching}, so Theorem \ref{thmintr} for $k=1$ follows from  Corollary \ref{CSuperNonSuper}.

Let $k=2$. By Lemma~\ref{L131218}(ii), 
$M^{(n-1,1)}$ is a quotient of $M^{(n-2,2)}$. So, using (\ref{E220825}), we have 
\begin{align*}
\dim\End_{\ts_{n-2,2}\cap G}(L\da_{\ts_{n-2,2}})&=\dim\Hom_G(M^{(n-2,2)},\End_\F(L))\\
&\geq\dim\Hom_G(M^{(n-1,1)},\End_\F(L))=\dim\End_{\ts_{n-1}\cap G}(L\da_{\ts_{n-1}}).
\end{align*}
Hence if $L\da_{\ts_{n-2,2}\cap G}$ is irreducible then so is $L\da_{\ts_{n-1}\cap G}$. So we may assume that $\la\in\JS$.

Assume first that $\la\in\JS^{(i)}$ with $i\neq 0$. Since $\la\neq \balpha_n$, one of the following happens:
\begin{itemize}
\item $\la=(\ldots,b,a)$ with $2\leq a<b<p$ and $\ttres(b)=\ttres(a+1)$ (it could be that $b=a+1$);

\item $\la=(\ldots,b,p^c,a)$ with $2\leq a<p<b<2p$, $c\geq 0$ and $\ttres(b)=\ttres(a+1)$.
\end{itemize}
Setting $c:=0$ in the first case, it can be checked that the nodes $(h-1-c,b)$ and $(h,a-1)$ are normal in $\tilde e_i\la=(\dots,b,a-1)$. Let $j:=\ttres(b)$ and $k:=\ttres(a-1)$.
If $j\not=k$ then $D(\la)\da_{\ts_{n-2}}$ has non-zero components in at least two superblocks. As $D(2)$ is the only supermodule of $\ts_2$, the same holds for $D(\la)\da_{\ts_{n-2,2}}$ and then also for $D(\la;\eps)\da_{\ts_{n-2,2}}$ and $E(\la;\eps)\da_{\tA_{n-2,2}}$. 
So we may assume that $j=k$, in which case $i=\ell$ and $j=k=\ell-1$. If $p> 3$ then $\ell-1>0$, and since $\eps_\ell(\la)=1$ and $\eps_{\ell-1}(\tilde e_\ell\la)\geq 2$, we have by Lemma \ref{branching} that 
$[D(\la)\da_{\ts_{n-2}}:D(\tilde e_{\ell-1}\tilde e_\ell\la)]\geq 4$, and 
 with $D(\la)$ and $D(\tilde e_{\ell-1}\tilde e_\ell\la)$ of the same type by Lemma~\ref{LAGa}. So
\[[D(\la)\da_{\ts_{n-2}}:D(\tilde e_{\ell-1}\tilde e_\ell\la,(2))]\geq 2^{1+a_p(\la)}.\]
If $p=3$ then $a=2$ and $b=4$, so 
$\eps_0(\tilde e_1\la)\geq 3$. Hence 
$[D(\la)\da_{\ts_{n-2}}:D(\tilde e_0\tilde e_1\la)]\geq 2^{a_3(\la)}\cdot 3$ by Lemma \ref{branching}, 
and $D(\la)$ and $D(\tilde e_0\tilde e_1\la)$ have different types by Lemma~\ref{LAGa}. So
\[[D(\la)\da_{\ts_{n-2}}:D(\tilde e_{\ell-1}\tilde e_\ell\la,(2))]\geq 3.\]
In both cases, the restrictions $L\da_{\ts_{n-2,2}\cap G}$ is reducible by Corollary~\ref{CSuperNonSuper}.


Finally, assume that $\la\in\JS^{(0)}$. 
By Lemma~\ref{LJS}, we have 
$\tilde e_0\la\in\JS^{(1)}$. 
So by Lemmas~\ref{branching} and \ref{LAGa}, we have $D(\la)\da_{\ts_{n-2}}\cong D(\tilde e_1\tilde e_0\la)^{\oplus 1+a_p(\la)}$. 
By Lemma~\ref{LAGa}, the supermodules $D(\la)$ and $D(\tilde e_1\tilde e_0\la,(2))$ have the same type, and 
\begin{equation}\label{E110925}
D(\la)\da_{\ts_{n-2,2}}\cong D(\tilde e_1\tilde e_0\la,(2)). 
\end{equation} 
So  $L\da_{\ts_{n-2,2}\cap G}$ is irreducible by Corollary~\ref{CSuperNonSuper}. 
\end{proof}

\subsection{Restrictions to wreath product subgroups}
\label{SWreathRes}

In this subsection we classify irreducible restrictions of spin representations to maximal wreath product subgroups.

\begin{Theorem}\label{thmwreath}
Let $n\geq 5$, $G\in\{\ts_n,\tA_n\}$, $L$ be an irreducible spin $\F G$-module, and $H=\hW_{a,b}\cap G$ for some $a,b\geq 2$ with $n=ab$. Then $L\da_H$ is irreducible if and only if one of the following holds:
\begin{enumerate}[\rm(i)]
\item $L$ is basic and $p\nmid a$;

\item $L$ is second basic, $p\mid(n-1)$ and one of the following holds:
\begin{enumerate}[\rm(a)]
\item $G=\ts_n$, and either $a$ or $b$ equals $2$, 

\item $G=\tA_n$ and $b=2$;
\end{enumerate}

\item $L$, $G$, $\pi(H)$ are as in Table I.
\end{enumerate}
\end{Theorem}

\begin{proof}
Write $L=D(\la;\eps)$ or $L=E(\la;\eps)$ for some $\la\in\RP_p(n)$ and appropriate $\eps\in\{0,+,-\}$.  

\vspace{2mm}
\noindent
{\em Claim 1. The theorem holds for $L$ basic, i.e. when $\la=\balpha_n$.} 

\vspace{1mm}
\noindent
This follows by \cite[Theorem E]{KT}. 


\vspace{2mm}
\noindent
{\em Claim 2. The theorem holds if $a,b\geq 3$.}

\vspace{1mm}
\noindent
We have $n\geq 9$ and 
$(n-3,3)^\Mull\,{\not\hspace{-0.9mm}\unrhd}\,(n-3,3)$ as $h((n-3,3)^\Mull)\geq 3$. Recall $\iota_\al$ and $\si_\al$ from (\ref{ESiAl}).  
By Lemma \ref{L7}, there exists 
$\psi\in\Hom_G({}^\pi M^{(n-3,3)},\End_\F(L))$ such that $\psi\circ\iota_{(n-3,3)}\neq 0$. 
Moreover, by Lemmas \ref{L060125_2} and \ref{L080925} there exists 
$\phi\in\Hom_G(\bone\ua_{H}^G,{}^\pi M^{(n-3,3)})$ such that $\si_3\circ \phi\neq 0$. 
So $L\da_{H}$ is reducible by Lemma \ref{inv_end}.

\vspace{2mm}
So from now on we assume that $L$ is not basic and $a=2$ or $b=2$; in particular, $n$ is even. 

\vspace{2mm}
\noindent
{\em Claim 3. The theorem holds for $n\leq 10$.} 

\vspace{1mm}
\noindent
We have $n=6,8$ or $10$. 
Using GAP to compute restrictions $S(\al)\da_{\hW_{a,b}}$ for every $\al\in\RP_0(n)$ and decomposition matrices for $\ts_n$, one can write every $[D(\la)\da_{\hW_{a,b}}]$ as a linear combination with non-negative coefficients of reductions modulo $p$ of irreducible modules $\C\hW_{a,b}$-modules. Checking the number of these modules and their types  
(and using the semisimplicity of $\F\hW_{a,b}$ for $p>n/2$), 
the claim follows, except possibly for the cases where $\la=(4,3,2,1)$, $(G,a,b)\in\{(\ts_{10},5,2),(\ts_{10},2,5),(\tA_{10},5,2)\}$, and $p=3$ or $5$.
Using decomposition matrices for $\ts_{10}$ and $\ts_5$ and Lemma \ref{lmodules} it can be checked that in the Grothendieck group we have 
\begin{align*}
[D(4,3,2,1)\da_{\ts_{5,5}}]&=[\bar S(\la)\da_{\ts_{5,5}}]=[\bar S((4,1),(3,2))]+[\bar S((3,2),(4,1))]\\
&=\left\{\begin{array}{ll}
[D((4,1),(3,2))]+[D((3,2),(4,1))]&\text{if}\ p=3,\\
{}[D((4,1),(3,2))]+[D((3,2),(4,1))]+2[D((4,1),(4,1))]&\text{if}\ p=5.
\end{array}\right.
\end{align*}
So $L\da_{\hW_{5,2}\cap G}$ is reducible for $p=5$. For $p=3$, we have that the supermodules  $D(4,3,2,1)$, $D((4,1),(3,2))$ and $D((3,2),(4,1))$ are of type $\Qtype$. Also, the supermodules $D((4,1),(3,2))$ and $D((3,2),(4,1))$ are exchanged by the wreath product action of $\hW_{5,2}$. It follows that $L\da_{\hW_{5,2}\cap G}$ is irreducible for $p=3$. Note that in this case $(4,3,2,1)=\bbeta_{10}$, and we get a contribution to the case (ii) in the statement of the theorem. 
Using the isomorphism of Lemma \ref{L160125} to identify $\T_{\W_{2,5}}=\T_5\otimes\Cl_5$ and the decomposition matrices for $\ts_{10}$ and $\ts_5$, we get 
\begin{align*}
[D(4,3,2,1)\da_{\hW_{2,5}}]=[\bar S(\la)\da_{\hW_{2,5}}]
=\left\{\begin{array}{ll}
2[D(4,1)\circledast U_5]&\text{if}\ p=3,\\
2[D(4,1)\circledast U_5]+2[D(3,2)\circledast U_5]&\text{if}\ p=5,
\end{array}\right.
\end{align*}
since in characteristic $0$ we have $[S(\la)\da_{\hW_{2,5}}]
=2[S(4,1)\circledast U_5]$. 
So $D\da_{\hW_{5,2}}$ is reducible for $p=5$. For $p=3$ we have that $(4,3,2,1)=\bbeta_{10}$, $D(4,3,2,1)$ is of type $\Qtype$ and $D(4,1)\circledast U_5$ of type $\Mtype$ and so $E((4,3,2,1);0)\da_{\hW_{2,5}\cap \tA_n}$ is reducible, while $D((4,3,2,1);\pm)\da_{\hW_{2,5}}$ is irreducible by Corollary~\ref{CSuperNonSuper}. This gives a contribution to the case (ii)(a) in the statement of the theorem. Claim 3 is proved.

\vspace{2mm}
\noindent
{\em Claim 4. The theorem holds for $n=12$ and $p=3$.} 

\vspace{1mm}
\noindent
By Corollary~\ref{CSuperNonSuper}, if $L\da_{\hW_{a,b}\cap G}$ is irreducible, then the supermodule $D(\la)\da_{\hW_{a,b}}$ has composition length at most $2$, and if $D(\la)\da_{\hW_{a,b}}$ has composition length $2$ then its two composition factors are isomorphic. Apart from the case where $\la=\bbeta_{12}$ and $(a,b)=(6,2)$, such situations can be excluded using dimensions of irreducible supermodules over $\ts_{12}$, $\hW_{6,2}$ and $\T_6\otimes\Cl_6\cong\T_{\W_{2,6}}$ (to compute the dimensions for $\hW_{6,2}$ and $\T_6\otimes\Cl_6\cong\T_{\W_{2,6}}$ we use the dimensions of irreducible supermodules over 
$\ts_6$). In the exceptional case, by \cite[Theorems 8.1, 8.3]{stem} and \cite[Tables III, IV]{Wales}, any composition factor of $D(\bbeta_{12})\da_{\ts_{6,6}}$ is of the form $D(\balpha_6,\balpha_6)$, $D(\balpha_6,\bbeta_6)$ or $D(\bbeta_6,\balpha_6)$. Now the composition length of $D(\bbeta_{12})\da_{\hW_{6,2}}$ is greater than $2$ by dimensions. Claim 4 is proved.

\vspace{2mm}
\noindent
{\em Claim 5. The theorem holds for $p>5$ and $\la=(5,4,3)$, $(5,4,2,1)$ or $(5,4,3,2)$.} 

\vspace{1mm}
\noindent
Using decomposition matrices from \cite{GAP,ModAtl}, we have $L=D(\la;\eps)\cong \bar S(\la;\eps)$ or $L=E(\la;\eps)\cong \bar T(\la;\eps)$. In either case, if $L\da_{\hW_{a,b}\cap G}$ is irreducible, then so is $S(\la;\eps)\da_{\hW_{a,b}}$ or $T(\la;\eps)\da_{\hW_{a,b}\cap \tA_n}$. But these restrictions are reducible by \cite[Theorems 1.1,\,1.2]{KW}. (Alternatively, we could use the observation that $11$ divides $\dim L$ but not $|\hW_{a,b}|$.) Claim 5 is proved. 

\vspace{2mm}
Using the above claims we will now complete the proof of the theorem. This will involve some case analysis. Taking into account the cases considered so far in Claims 1--5, we may assume that
$n\geq 12+2\de_{p,3}$. We may assume that either $\la_1\geq 6$, or $p\leq 5$ and $\la\in \TR_p$. Indeed, suppose $\la_1\leq 5$. If $p>5$ then either $\la$ is as in Claims 5, or $n\leq 10$ which is covered by Claim 3. If $p=5$, then $\la$ is of the form $(5^a,\mu_1,\mu_2,\dots)$ where $(\mu_1,\mu_2,\dots)$ a strict partition with $\mu_1\leq 4$, hence $\la\in\TR_5$ by Lemma~\ref{LM3}. If $p=3$ then $\la\in\TR_3(n)$.  

Since $n\geq 12+2\de_{p,3}$, we have $h((n-k,k)^\Mull)>2$ for $1\leq k\leq 6$, so
\begin{equation}\label{EHD1}
(n-k,k)^\Mull\,{\not\hspace{-0.9mm}\unrhd}\,(n-k,k)\qquad(\text{for}\ 1\leq k\leq 6).
\end{equation}
If $p=3$ then $h((n-6,4,2)^\Mull)=6$, so
\begin{equation}\label{EHD2}
(n-6,4,2)^\Mull\,{\not\hspace{-0.9mm}\unrhd}\,(n-6,4,2)\qquad(\text{for}\ p=3).
\end{equation}
If $p>3$ then $h((n-6,2^3)^\Mull)>4$, so
\begin{equation}\label{EHD2}
(n-6,2^3)^\Mull\,{\not\hspace{-0.9mm}\unrhd}\,(n-6,2^3)\qquad(\text{for}\ p>3).
\end{equation}

\vspace{2mm}
\noindent
{\em Case 1: $\la\not\in\TR_p(n)$}. 

\vspace{1mm}
\noindent
By Lemma \ref{L181224_3} there exists a homomorphism $\psi\in\Hom_G({}^\pi M^{(n-6,6)},\End_\F(L))$ with $\psi\circ\iota_{(n-6,6)}\neq 0$. Furthermore, by Lemma~\ref{L060125_2} and Proposition~\ref{inv6} there exists a homomorphism $\phi\in\Hom_G(\bone\ua_{\hW_{a,b}\cap G}^G,{}^\pi M^{(n-6,6)})$ with $\si_{(n-6,6)}\circ\phi\neq 0$. Taking into account (\ref{EHD1}), we may now apply Lemma \ref{inv_end} with $\al=(n-6,6)$ to deduce that  $L\da_{\hW_{a,b}\cap G}$ is reducible.

\vspace{2mm}
\noindent
{\em Case 2: $\la$ satisfies the assumptions of Lemma~\ref{L181224_2}}.

\vspace{1mm}
\noindent
By Lemmas~\ref{L060125_2} and \ref{inv2wreath}, there exists $\phi\in\Hom_G(\bone\ua_{\hW_{a,b}\cap G}^G,{}^\pi M^{(n-2,2)})$ with $\si_{(n-2,2)}\circ\phi\neq 0$. Now, $L\da_{\hW_{a,b}\cap G}$ is reducible by Lemma~\ref{L181224_2} and Lemmas~\ref{inv_end}, \ref{inv_mixedhom} with $\al=(n-2,2)$.

\vspace{2mm}
\noindent
{\em Case 3: $\la\in\TR_p(n)\setminus\{\balpha_n,\bbeta_n\}$ and $\la$ does not satisfy the assumptions of Lemma~\ref{L181224_2}.}

\vspace{2mm}
\noindent
Taking into account that $n$ is even and that  $(p+1,p^b,p-1,1)$ or $(p-2,2,1)$ is $\bbeta_n$, we conclude that 
$\la$ is of the form $((2p)^a,2p-1,p+1,p^{2b},p-1,1)$ or $(p^{2a},p-1,p-2,2,1)$. By Lemma \ref{L060125} and Corollary~\ref{CSuperNonSuper}, $L\da_{\hW_{n/2,2}\cap G}$ is reducible. So we may assume that $a=2$. If $p=3$ then by Lemma~\ref{end42}, there exists $\psi\in\Hom_G({}^\pi M^{(n-6,4,2)},\End_\F(L))$ such that $\psi\circ \iota_{(n-6,4,2)}\neq 0$. Moreover, by  Lemmas~\ref{inv42} and \ref{L060125_2}, there exists $\phi\in\Hom_G(\bone\ua_{\hW_{2,n/2}\cap G}^G,{}^\pi M^{(n-6,4,2)})$ with $\si_{(n-6,4,2)}\circ\phi\neq 0$. So $L\da_{\hW_{2,n/2}\cap G}$ is reducible by Lemma~\ref{inv_end} with $\al=(n-6,4,2)$. The case $p>3$ is similar using $(n-6,2^3)$ in place of $(n-6,4,2)$ and Lemma~\ref{end222} in place of Lemma~\ref{end42}.

\vspace{2mm}
\noindent
{\em Case 4: $\la=\bbeta_n$ and $a=2$.}

\vspace{1mm}
\noindent
This case follows from Lemma \ref{L160125_2} and Corollary~\ref{CSuperNonSuper}. 

\vspace{2mm}
\noindent
{\em Case 5: $\la=\bbeta_n$, $b=2$.} 

\vspace{2mm}
\noindent
By Lemma \ref{L160125_3}, we may assume that $n\equiv 1\pmod{p}$, in which case 
$$[D(\bbeta_n)\da_{\ts_{n/2,n/2}}]=[D(\bbeta_{n/2},\balpha_{n/2})]+[D(\balpha_{n/2},\bbeta_{n/2})].$$ 
Moreover, using the information on types contained in Tables III,\,IV, we conclude that $D(\bbeta_n)$, $D(\bbeta_{n/2},\balpha_{n/2})$ and $D(\balpha_{n/2},\bbeta_{n/2})$ are of all of type $\Qtype$. So
\begin{align*}
[D(\bbeta_n;\pm)\da_{\ts_{n/2,n/2}}]&=([D(\bbeta_{n/2},\balpha_{n/2};\pm\eps)]+[D(\balpha_{n/2},\bbeta_{n/2};\pm\de)],\\
[E(\bbeta_n,0)\da_{\tA_{n/2,n/2}}]&=([E(\bbeta_{n/2},\balpha_{n/2},0)]+[E(\balpha_{n/2}\bbeta_{n/2},0)].
\end{align*}
Since $\si\in\hW_{n/2,2}\setminus\ts_{n/2,n/2}$ resp. $\si\in\hW_{n/2,2}\cap\tA_n\setminus\tA_{n/2,n/2}$ exchanges $D(\bbeta_{n/2},\balpha_{n/2};\pm)$ and $D(\balpha_{n/2},\bbeta_{n/2};\pm)$ (resp. $E(\bbeta_{n/2},\balpha_{n/2},0)$ and $E(\balpha_{n/2}\bbeta_{n/2},0)$), we have that the restriction $L\da_{\hW_{n/2,2}\cap G}$ is irreducible.
\end{proof}

\section{Restrictions to non-maximal imprimitive subgroups}

\subsection{First reductions} 

In this subsection, we give a corrected version of \cite[Theorem~D]{KT} and extend it to include the case $p=3$.

\begin{Lemma} \label{L150925} 
Let $n\geq 8$, $K\leq X\in\{\ts_{n-2},\tA_{n-2}\}$, 
$\la\in\RP_p(n)\setminus\{\balpha_n\}$ satisfy $\la\in\JS(0)$, 
and $V$ be an irreducible $\F X$-module labeled by the partition $\tilde e_1\tilde e_0\la$. If $V\da_K$ is irreducible then $\pi(K)$ acts 3-homogeneously on $\{1,\ldots,n-2\}$.
\end{Lemma}
\begin{proof}
By \cite[Lemma 3.7(i)]{KT}, $V$ is non-basic.

If $\pi(K)\leq \s_{n-2}$ is primitive on $\{1,\ldots,n-2\}$, then the possibilities for $(\pi(K),V)$ are listed in Theorem~\ref{TAKT}.  Going through the list and taking into account that $n-2\geq 6$ and $V$ is non-basic, we are left with the cases where $\pi(K)$ is $3$-homogeneous on $\{1,2,\dots,n-2\}$ or the following three cases:
(1) $X=\tA_8$, $p=3$, $\pi(K)=\A_5$, $V$ is second basic; (2) $X=\tA_9$, $p=3$, $\pi(K)\cong L_2(7)$, $V$ is second basic; (3) $X=\ts_8$, $p>5$, $\pi(K)\cong \s_5$, and $V$ is neither basic nor second basic with $\dim V=4$. The cases (1)-(3) are ruled out since in those cases $V$ is labeled by a partition not of the form $\tilde e_1\tilde e_0\la$ for $\la\in\JS^{(0)}$.

If $\pi(K)$ is intransitive on $\{1,2,\dots,n-2\}$ then $\tilde e_1\tilde e_0\la$ is $\JS$ by Theorem~\ref{thmintr}. Since $\la$ is $\JS(0)$, $\la\neq\balpha_n$ and $n\geq 8$ we have that one of the following holds: (a) $\la=(\ldots,3,2,1)$ and $\tilde e_1\tilde e_0\la=(\ldots,3,1)$; (b) $p\geq 7$, $\la=(\ldots,p-2,2,1)$ and $\tilde e_1\tilde e_0\la=(\ldots,p-2,1)$; (c) $\la=(\ldots,p+1,p^a,p-1,1)$ and $\tilde e_1\tilde e_0\la=(\ldots,p+1,p^a,p-2)$ for some $a\geq 0$. This contradicts $\tilde e_1\tilde e_0\la$ being $\JS$.

If $\pi(K)$ is imprimitive but transitive on $\{1,2,\dots,n-2\}$, then by Theorem~\ref{thmwreath}, either $p\mid (n-3)$ and  $V$ is second basic labeled by $\bbeta_{n-2}$, or $p\geq 7$, $n=12$ and $V$ is labeled by $(4,3,2,1)$. None of these is of the form $\tilde e_1\tilde e_0\la$ for $\la\in\JS$. 
\end{proof}

We now extend \cite[Theorem~7.4]{KT} to the case $p=3$ (with a slightly larger lower bound on $n$). 

\begin{Proposition}\label{P050525}
Let $n\geq 8$, $H\leq G\in\{\ts_n,\tA_n\}$, and $L$ be a non-basic  irreducible spin $\F G$-module. 
If $L\da_H$ is irreducible, then one of the following holds:
\begin{enumerate}[\rm(i)]
\item $\pi(H)$ is 3-homogeneous on $\{1,2,\dots,n\}$;

\item $\pi(H)$ has an orbit of length $n-1$ or $n-2$ on $\{1,2,\dots,n\}$ and $\pi(H)$ acts 3-homogeneously on that orbit;

\item $n$ is even, $\pi(H)$ is transitive on $\{1,2,\dots,n\}$, and $H\leq\hW_{n/2,2}$ or $H\leq\hW_{2,n/2}$;

\end{enumerate}
\end{Proposition}

\begin{proof}
Write $L=D(\la,\eps)$ or $L=E(\la,\eps)$ for some $\la\in\RP_p(n)\setminus\{\balpha_n\}$ and $\eps\in\{0,+,-\}$.

If $\pi(H)$ is primitive then by Theorem~\ref{TAKT} (which is \cite[Theorem B]{KT}), we have that $\pi(H)$ contains $\A_n$, and so $\pi(H)$ satisfies (i). If $\pi(H)$ is  imprimitive and transitive, then by Theorem \ref{thmwreath} we are in (iii). So we may assume that $\pi(H)$ is intransitive. By Theorem \ref{thmintr} we have that $H\leq\ts_{n-1}$ or $H\leq\ts_{n-2,2}$.

\vspace{2mm}
\noindent
{\em Case 1: $\pi(H)\leq\s_{n-1}$ and $\pi(H)\not\leq\s_{n-2,2}$.}

\vspace{1mm}
\noindent 
As $L\da_H$ is irreducible, the $\F(\ts_{n-1}\cap G)$-module $L':=L\da_{\ts_{n-1}\cap G}$ is also irreducible. By Theorem~\ref{thmintr}, $\la\in\JS^{(i)}$ for some $i$, and then by Lemma \ref{branching}, the irreducible $\F(\ts_{n-1}\cap G)$-module $L'$ is labeled by the partition $\tilde e_i\la$.  Moreover, by \cite[Lemma 3.7(i)]{KT}, since $L$ is non-basic, $L'$ is also non-basic. If the subgroup $\pi(H)\leq \s_{n-1}$ is primitive then the possibilities for $(\pi(H),L')$ are listed in Theorem~\ref{TAKT}.  Going through the list and taking into account that $n-1\geq 7$ and $L'$ is non-basic, we are left with the cases where $\pi(H)$ is $3$-homogeneous on $\{1,2,\dots,n-1\}$ which contributes to (ii), or the case where $G=\tA_8$, $p=3$, $\pi(H)\cong L_2(7)$ and $L'$ is  second basic in which case $L'$ is labeled by the partition $\bbeta_7=(4,2,1)$, which is not of the form $\tilde e_i\la$ for $\la\in\JS$. So we may assume that the subgroup $\pi(H)\leq \s_{n-1}$ is imprimitive. If $\pi(H)$ is intransitive on $\{1,2,\dots,n-1\}$ then by Theorem~\ref{thmintr}, $\pi(H)\leq\s_{n-2,2}$. If $\pi(H)$ is transitive on $\{1,2,\dots,n-1\}$, then by Theorem~\ref{thmwreath}, either $p\mid (n-2)$ and  $L'$ is second basic labeled by $\bbeta_{n-1}$, or $p\geq 7$, $n=11$ and $L'$ is labeled by $(4,3,2,1)$. None of these is of the form $\tilde e_i\la$ for $\la\in\JS$.

\vspace{2mm}
\noindent
{\em Case 2: $H\leq\ts_{n-2}$.} 

\vspace{1mm}
\noindent
As $L\da_H$ is irreducible, the $\F(\ts_{n-2}\cap G)$-module $L':=L\da_{\ts_{n-2}\cap G}$ is also irreducible. By Theorem \ref{thmintr}, we have $\la\in\JS^{(0)}$, and then by Lemma \ref{branching}, the irreducible $\F(\ts_{n-2}\cap G)$-module $L'$ is labeled by the partition $\tilde e_1\tilde e_0\la$  and $L'\da_H$ is irreducible. By Lemma~\ref{L150925}, $\pi(H)$ is 3-homogenous on $\{1,\ldots,n-2\}$.

\vspace{2mm}
\noindent
{\em Case 3: $H\leq\ts_{n-2,2}$ but $H\not\leq\ts_{n-2}$.}

\vspace{1mm}
\noindent
Let $K=H\cap\ts_{n-2}$ so that $[H:K]=2$, and fix $h\in H\setminus K$. 
Write $h=g t_{n-1}$ for $g\in\ts_{n-2}$. 
Define the subgroups $K^+=\langle K,g,z\rangle\leq \ts_{n-2}$ and $H^+=\langle H,g,z\rangle\leq \ts_{n-2,2}$. Note that $\pi(K^+)$ is transitive on $\{1,\ldots,n-2\}$, since otherwise $H\leq \langle K^+, t_{n-1}\rangle$ is contained (up to isomorphism) in some $\ts_{n-a,a}$ with $3\leq a\leq n/2$,  contradicting Theorem \ref{thmintr}.


\vspace{2mm}
\noindent
{\em Case 3.1: $G=\ts_n$.} 

\vspace{1mm}
\noindent
As $H\leq H^+
\leq \ts_{n-2,2}$ the modules $L\da_{H^+}$ and $L\da_{\ts_{n-2,2}}$ are  irreducible. In particular, $\la\in\JS^{(0)}$ by Theorem~\ref{thmintr}. Moreover, by (\ref{E110925}), in the Grothendieck group we have 
\begin{equation}\label{E130925}
[D(\la)\da_{\T_{n-2,2}}]=[D(\tilde e_1\tilde e_0\la,(2))]=2^{a_p(\la)-1}[D(\tilde e_1\tilde e_0\la)\boxtimes D(2)].
\end{equation}
Since $\T_{\pi(H^+)}=\T_{\pi(K^+)\times\s_2}\cong \T_{\pi(K^+)}\otimes \T_2$, we deduce 
\begin{equation}\label{E110925_2}
[D(\la)\da_{H^+}]=2^{a_p(\la)-1}[D(\tilde e_1\tilde e_0\la)\da_{K^+}\boxtimes D(2)].
\end{equation}

Note that $H^+\not\leq \tA_n$, so we can apply  Corollary~\ref{CSuperNonSuper} to deduce that $L\da_{H^+}$ is irreducible if and only if one of the following happens: (1) the supermodule $D(\la)\da_{H^+}$ is irreducible and $(\text{type of $D(\la)$},\text{type of $D(\la)\da_{H^+}$})\neq (\Mtype,\Qtype)$; (2) $D(\la)$ is of type $\Qtype$ and the supermodule $D(\la)\da_{H^+}$ has length two, with two isomorphic composition factors of type $\Mtype$. If $a_p(\la)=1$ then by (\ref{E110925_2}) and Lemma~\ref{LBoxTimes}, (1) is satisfied if and only if the supermodule $D(\tilde e_1\tilde e_0\la)\da_{K^+}$ is irreducible of type $\Mtype$, while (2) is satisfied if and only if the supermodule $D(\tilde e_1\tilde e_0\la)\da_{K^+}$ is irreducible of type $\Qtype$. 
On the other hand, if $a_p(\la)=0$ then (1) is satisfied if and only if the supermodule $D(\tilde e_1\tilde e_0\la)\da_{K^+}$ is irreducible of type $\Qtype$, while (2) never happens. So one of the following happens:

\begin{enumerate}[\rm(a)]
\item $D(\la)$ is of type $\Mtype$ and the supermodule $D(\tilde e_1\tilde e_0\la)\da_{K^+}$ is irreducible of type $\Qtype$;

\item $D(\la)$ is of type $\Qtype$ and the supermodule $D(\tilde e_1\tilde e_0\la)\da_{K^+}$ is irreducible of type $\Mtype$;

\item $D(\la)$ is of type $\Qtype$ and the supermodule $D(\tilde e_1\tilde e_0\la)\da_{K^+}$ is irreducible of type $\Qtype$.

\end{enumerate}
 
\noindent
In the cases (a),(b), by Lemma~\ref{LAGa}, the $\F\ts_{n-2}$-supermodule $D(\tilde e_1\tilde e_0\la)$ is of the same type as the  $\F K^+$-supermodule $D(\tilde e_1\tilde e_0\la)\da_{K^+}$. So the $\F K^+$-module $D(\tilde e_1\tilde e_0\la;\eps)\da_{K^+}$ is irreducible for every $\eps$. 
By Lemma~\ref{L150925}, $\pi(K^+)$ is 3-homogeneous, whence so is $\pi(H)$.

In the case (c), by Lemma~\ref{LAGa}, the $\F\ts_{n-2}$-supermodule $D(\tilde e_1\tilde e_0\la)$ is of type $\Mtype$ and the  $\F K^+$-supermodule $D(\tilde e_1\tilde e_0\la)\da_{K^+}$ is of type $\Qtype$. In particular, $K^+\not\leq\tA_{n-2}$.  By Corollary~\ref{CSuperNonSuper} it follows that $E(\tilde e_1\tilde e_0\la,\pm)\da_{K^+\cap\tA_{n-2}}$ is irreducible. So $\pi(K^+\cap\tA_{n-2})$ and then also $\pi(H)$ is 3-homogeneous by Lemma~\ref{L150925}. 

\vspace{2mm}
\noindent
{\em Case 3.2: $G=\tA_n$ and $L=E(\la;0)$.} 

\vspace{1mm}
\noindent
Then $L$ extends to $\ts_n$ by Lemma \ref{lmodules} and we can apply Case 3.1.

\vspace{2mm}
\noindent
{\em Case 3.3: $G=\tA_n$ and $L=E(\la;\pm)$.} 

\vspace{1mm}
\noindent
Since $H\leq G=\tA_n$ and $g t_1\in H$, we have that $g\in\ts_{n-2}\setminus\tA_{n-2}$. Replace if necessary $H$ by $\langle H,z\rangle$, we may assume that $z\in H$. Note that
\[\langle H\cap\ts_{n-2},g\rangle=\langle K,g\rangle=K^+=\pi^{-1}({\tt p}\circ \pi(H)),\]
where ${\tt p}$ is the projection $\s_{n-2,2}\to\s_{n-2}$. Since $g\not\in\A_n$ but $H\leq\tA_n$, it follows that $H\leq\pi^{-1}(\pi(K^+)\times\s_2)=H^+$ is normal of index $2$ and $H=H^+\cap\tA_n$.

We have $D(\la;0)\da_{\tA_n}\cong E(\la;+)\oplus E(\la;-)$. 
Since the $\F H$-module $E(\la;\pm)\da_H$ is irreducible, 
by Corollary \ref{CSuperNonSuper} it follows that $D(\la)\da_{H^+}$ is irreducible as supermodule.

Since $D(\la)$ is of type $\Qtype$, $D(\la)\da_{H^+}$ must also be of type $\Qtype$ and then $D(\la;\pm)\da_{H^+}=D(\la;\pm)\da_{\ts_{n-2,2}}\da_{H^+}$ is irreducible. So by Theorem \ref{thmintr} we have as in (\ref{E110925_2}) that
\[[D(\la)\da_{H^+}]=2^{-1}[D(\tilde e_1\tilde e_0\la)\da_{K^+}\boxtimes D(2)].\]
We deduce that 
the $\F K^+$-supermodule 
$D(\tilde e_1\tilde e_0\la)\da_{K^+}$ is irreducible of type $\Qtype$ and 
$$D(\la)\da_{H^+}\cong (D(\tilde e_1\tilde e_0\la)\da_{K^+})\circledast D(2)$$ 
is an irreducible supermodule of type $\Mtype$. In particular the $\F K^+$-module $D(\tilde e_1\tilde e_0\la;\pm)\da_{K^+}$ is irreducible. 
We can then conclude by Lemma~\ref{L150925}. 
\end{proof}


\begin{Theorem}\label{kt-n22}
Let $n\geq 8$, $H\leq G\in\{\ts_n,\tA_n\}$, 
and $L$ be a non-basic irreducible spin $\F G$-module such that $L\da_H$ is irreducible. 
If $\pi(H)$ is not almost simple and $\pi(H)$  is imprimitive on $\Omega=\{1,2,\ldots,n\}$ 
then one of the following holds:
\begin{enumerate}[\rm(i)]
\item $G=\ts_n$ and $\pi(H)=\s_{n-2,2}$, 
\item
$n$ is even, $\pi(H)$ is transitive on $\{1,2,\dots,n\}$, and 
 $H\leq\hW_{n/2,2}$ or $H\leq\hW_{2,n/2}$.
\end{enumerate}
\end{Theorem}

\begin{proof}

By Proposition~\ref{P050525}, one of the following happens: 
(1) $\pi(H)$ is 3-homogeneous on $\{1,2,\dots,n\}$;
(2) $\pi(H)$ has an orbit of length $n-1$ or $n-2$ on $\{1,2,\dots,n\}$ and $\pi(H)$ acts 3-homogeneously on that orbit;
(3) $n$ is even, $\pi(H)$ is transitive on $\{1,2,\dots,n\}$, and $H\leq\hW_{n/2,2}$ or $H\leq\hW_{2,n/2}$. Now (3) is precisely the conclusion (ii) of the current theorem, and (1) contradicts the assumption that $\pi(H)$ is imprimitive since $n \geq 6$.
So we may assume that $H$ is as in (2), and let $\Omega_1$ be the `long' orbit, with $r:=|\Omega_1|$, 
so that $r=n-1$ or $n-2$.
Denote $K:=\pi(H)$, and let $X<\s_r$ be the 
image of $K$ with respect to its $3$-homogeneous action on $\Omega_1$. So $X\cong K/J$, where $J$ is the kernel of the action of $K$ on $\Omega_1$, and $|J|\leq 2$.
Let $S:=\soc(X)$.

If $S$ is non-simple, we apply \cite[Proposition~7.5]{KT} to the action of $X$ on $\Omega_1$, and the argument on \cite[p. 1996]{KT} shows that $L\da_H$ is reducible. So we may assume that $S$ is simple. If moreover $S$ is abelian, then $S \cong \Z_r$, $r \geq n-2 \geq 6$ is prime, and $X \leq \mathrm{AGL}_1(r)$ acting on $\FF_r$. In this case $|X|$ divides $r(r-1)$ which is not divisible by $\binom{r}{3}$, so $X$ cannot act $3$-homogeneously on $\Omega_1$. 
Hence $S$ is non-abelian\footnote{This case was missed in the proof of \cite[Theorem D]{KT}.}, $X$ is almost simple, while $K$ is not almost simple by assumption; in particular, $K\neq X$ and $|J|=2$.  

As $r=n-1$ implies $K=X$, we must have 
$r=n-2$. So we may assume that $\Om_1=\{3,4,\ldots,n\}$ and $K$ acts non-trivially on $\{1,2\}$. Let $K_1$ be the stabilizer of $1$ in $K$, so that $[K:K_1]=2$.
Applying \cite[Theorem 1]{Kantor} to $X$, we see 
that $X$ is $2$-transitive, in fact $3$-transitive unless we are in case (a) listed below. 
Now, by the main result of \cite{Cameron} applied to $X$, one of the following holds: 
\begin{enumerate}[\rm(a)]
\item $S = \mathrm{PSL}_2(q)$ for a prime power $q=p^f$, and $n-2=q+1$;
\item $(S,n-2) = (M_{11},11)$, $(M_{11},12)$, $(M_{12},12)$, $(M_{22},22)$, $(M_{23},23)$ or $(M_{24},24)$;
\item $S=\A_{n-2}$.
\end{enumerate}
In particular, $S$ is $2$-transitive on $\Omega_1$, which implies by \cite[Proposition 5.2]{Cameron} that $C_{\s_r}(S)=1$.

Let $S^+\leq K$ be the full preimage of $S$ in $K$ under the natural projection $K\to X\cong K/J$. 
Then $J$ is a normal subgroup of $S^+$ of order $2$, so 
$J \leq Z(S^+)$ and $|S^+|=2|S|$. Now, if $S^+$ is perfect then, considering its action on $\{1,2\}$, we see that 
$S^+ \leq K_1$, so 
$J$ acts trivially on both $\Omega_1$ and $\{1,2\}$, hence $J=1$, giving a contradiction. So $S^+$ is not 
perfect, whence $S^+ \cong J \times R$ with $R \cong S$. Then $R = [R,R] \lhd K_1$. Thus $K_1$ contains a subgroup $R \cong S$ which is normal in $K$. 

By the assumption that $K$ is not almost simple, we have $\soc(K) = R \times T$ for a subgroup $T\neq 1$. Then the image $TJ/J$ of $T$ in $X$ centralizes 
the action of $S = \soc(X)$ on $\Omega_1$, so $T$ acts trivially on $\Omega_1$. Hence $T = \langle (1,2) \rangle$; in particular, 
$G=\ts_n$. 
By the same argument applied to $C_K(R)$, we have $C_K(R) = T$. 
In particular, $|K| \leq 2|\Aut(S)|$. By Lemma \ref{LSpinFaithful},  
$K \cong H/Z(H)$, so  
the degree of any complex irreducible character of $H$ is at most $\sqrt{|H/Z(H)|}=\sqrt{|K|}$. But $H$ acts irreducibly on $L$, hence
\begin{equation}\label{eq10b}
  (\dim L)^2 \leq |K|  \leq 2|\Aut(S)|.
\end{equation}

Suppose $S=\A_{n-2}$ as in (c). Then $\soc(K) = R \times \lan (1,2)\ran$ implies that $K=\A_{n-2} \times \s_2$ or $K=\s_{n-2,2}$. In the former case,   
$\lan H,z\ran=\pi^{-1}(\A_{n-2} \times \s_2)$ is centralized by the element $t_1$ which is not in $Z(H)$ by  Lemma~\ref{LSpinFaithful}, so $H$ cannot act irreducibly 
on $L$, a contradiction. Thus $K=\s_{n-2,2}$ and we have arrived  at conclusion (i). 

It remains to consider the possibilities listed in (a) and (b). Under the additional assumption that $n \geq 12$, by \cite[Main Theorem]{KT2} we have
\begin{equation}\label{eq12}
\dim L\geq 2^{\lfloor (n-3)/2 \rfloor}(n-4).
\end{equation}
which implies by \eqref{eq10b} that
\begin{equation}\label{eq11}
  |\Aut(S)| \geq (\dim L)^2/2 = 2^{2 \cdot \lfloor (n-3)/2 \rfloor -1}(n-4)^2\geq 2^{n-5}(n-4)^2.
\end{equation}

Suppose we are in case (a), in particular $n=q+3$. Assume first that $n \geq 13$. If  $q = p^f \geq 11$ and $p > 2$, then $f \leq q/9$, and
$|\Aut(S)| = fq(q^2-1) < 2^{q-3}(q-1)^2$, violating \eqref{eq11}. If $q = 2^f \geq 16$, then $f \leq q/4$, and 
$|\Aut(S)| = fq(q^2-1) < 2^{q-3}(q-1)^2$, again violating \eqref{eq11}. In the remaining cases $8 \leq n \leq 12$ 
we have $q=5,7,8$ or $9$, and 
 $|\Aut(S)| = 120, 336,1512$ or $1440$, respectively. By \eqref{eq10b}, $\dim L \leq 15,25, 54$ or $53$, respectively. But since $G=\ts_n$,  \cite{GAP} yields 
$\dim L \geq 16, 48, 112$ or $128$, respectively, for a non-basic irreducible spin module $L$, giving a contradiction.  

Suppose now that we are in case (b). 
By \eqref{eq11}, we have  $(S,n-2)=(M_{12},12)$. Since $M_{12}\cdot 2$ does not embed into $\s_{12}$, we have $R=S$ and $K=R\times T$. Since $R$ and $T$ act on disjoint sets of numbers, $\T_K\cong \T_R\otimes\T_T$ as superalgebras. 
As $\T_T\cong\Cl_2$ and $\Cl_2$ has a unique irreducible supermodule $U_2$ which is 2-dimensional and of type $\Qtype$, the maximal dimension of an irreducible $\F H$-module is equal to the maximal dimension of an irreducible $\F\hat R$-module. But 
the maximal dimension of an irreducible module over $\hat M_{12}= 2\cdot M_{12}$  is $176$, contradicting \eqref{eq12}.
\end{proof}

\subsection{Subgroups of $\hW_{n/2,2}$} 

Let $G\in\{\ts_n,\tA_n\}$. 
In view of Theorem~\ref{thmwreath}(ii), we need to study the irreducible restrictions $L\da_H$ of the second basic $\F G$-module $L$ to the subgroups $H$ contained in $\hW_{n/2,2}$ or $\hW_{2,n/2}$ (for even $n$) when $p\mid(n-1)$. In this subsection we deal with the subgroups of $\hW_{n/2,2}$, and in the next subsection we deal with the subgroups of $\hW_{2,n/2}$.

Throughout this subsection, we assume that  $n=2b\geq 10$ is even and $p\mid(n-1)$.
By Table~IV and Lemma~\ref{L160125_3}, the second basic supermodule $D(\bbeta_n)$ is of type $\Qtype$, 
\begin{equation}\label{E170925_2}
\dim D(\bbeta_n;\pm)=\dim E(\bbeta_n;0)
=2^{(n-4)/2}(n-4).
\end{equation}
and 
\begin{equation}\label{E170925}
[D(\bbeta_n)\da_{\T_{b,b}}]=[D(\bbeta_{b})\boxtimes
D(\balpha_{b})]+[D(\balpha_{b})\boxtimes D(\bbeta_{b})].
\end{equation}
As $p\mid(n-1)$, we have that $b\not\equiv 0,1\pmod{p}$, the supermodules $D(\balpha_b)$ and $D(\bbeta_b)$ are of different types by Tables~III,\,IV, so $D(\bbeta_{b})\boxtimes
D(\balpha_{b})=D(\bbeta_b,\balpha_b)$ and $D(\balpha_{b})\boxtimes
D(\bbeta_{b})=D(\balpha_b,\bbeta_b)$. 

If $b$ is even then $D(\bbeta_{b})$ is of type $\Mtype$, and  $D(\balpha_{b})$ is of type
$\Qtype$, so by Lemma~\ref{lmodules}, 
$$D(\bbeta_{b};0)\da_{\tA_{b}}=E(\bbeta_{b};+)\oplus
E(\bbeta_{b};-)\quad\text{and}\quad 
D(\balpha_{b};\pm)\da_{\tA_{b}}=E(\balpha_{b},0).$$
Now, restricting (\ref{E170925}) to $\T_{\A_b\times \A_b}\cong \T_{\A_b}\otimes \T_{\A_b}$, we get 
\begin{equation}\label{E170925_3}
\begin{split}
[D(\bbeta_n;\pm)\da_{\T_{\A_b\times \A_b}}]  =\,&[E(\bbeta_n;0)\da_{\T_{\A_b\times \A_b}}]\\
 =\,&[E(\bbeta_{b};+)\boxtimes
E(\balpha_{b};0)]+[E(\bbeta_{b};-)\boxtimes
E(\balpha_{b};0)]\\
& +[E(\balpha_{b};0)\boxtimes
E(\bbeta_{b};+)]+[E(\balpha_{b};0)\boxtimes E(\bbeta_{b};-)].
\end{split}
\end{equation}
If $b$ is odd, then $D(\bbeta_{b})$ is of type $\Qtype$, $D(\balpha_{b})$ is of type $\Mtype$, and, as in the case $b$ even,  we get
$$
D(\bbeta_{b};\pm)\da_{\tA_{b}}=E(\bbeta_{b},0)\quad\text{and}\quad 
D(\balpha_{b};0)\da_{\tA_{b}}=E(\balpha_{b};+)\oplus E(\balpha_{b};-),
$$
\begin{equation}\label{E170925_4}
\begin{split}
[D(\bbeta_n;\pm)\da_{\T_{\A_b\times \A_b}}] =\,&[E(\bbeta_n;0)\da_{\T_{\A_b\times \A_b}}]\\
 =\,&[E(\bbeta_{b};0)\boxtimes
E(\balpha_{b};+)]+[E(\bbeta_{b};0)\boxtimes
E(\balpha_{b};-)]\\
& +[E(\balpha_{b};+)\boxtimes
E(\bbeta_{b};0)]+[E(\balpha_{b};-)\boxtimes E(\bbeta_{b};0)].
\end{split}
\end{equation}
Taking into account (\ref{E170925_2}), we have proved:

\begin{Lemma} 
Let\, $2\mid n \geq 10$ and $p\mid (n-1)$. Then the module 
$D(\bbeta_n;\pm)\da_{\T_{\A_b\times \A_b}}=E(\bbeta_n,0)\da_{\T_{\A_b\times \A_b}}$ has 4 composition
factors, each of dimension $2^{(n-8)/2}(n-4)$.
\end{Lemma}

In the next two lemmas we will use the notation $\hW_{q,r,t}:=\pi^{-1}(\s_q\wr\s_r\wr\s_t)\leq \ts_{qrt}$. 

\begin{Lemma} \label{L180925_0}
Let $n=4t \geq 12$ be even, $p\mid (n-1)$, and $L$ be a second basic $\F \ts_n$-module. Then $L\da_{\hW_{2,t,2}}$ is reducible. 
\end{Lemma}
\begin{proof}
By Lemmas \ref{L160125_3} and \ref{L160125_2}, 
\begin{align*}
[D(\bbeta_n)\da_{\pi^{-1}(\W_{2,t}\times\W_{2,t})}]=\,&2[(D(\bbeta_{t})\circledast U_t)\boxtimes D(\balpha_{b})\da_{\hW_{2,t}}]
\\
&+2[D(\balpha_{b})\da_{\hW_{2,t}}\boxtimes( D(\bbeta_{t})\circledast U_t)]
\\
&+m[(D(\balpha_{t})\circledast U_t)\boxtimes D(\balpha_{b})\da_{\hW_{2,t}}]
\\&+m[D(\balpha_{b})\da_{\hW_{2,t}}\boxtimes( D(\balpha_{t})\circledast U_t)]
\end{align*}
with $m>0$. In particular, the composition length of the supermodule $D(\bbeta_n)\da_{\pi^{-1}(\W_{2,t}\times\W_{2,t})}$ is at least $6$. Since $\pi^{-1}(\W_{2,t}\times\W_{2,t})$ is normal of index 2 in $\hW_{2,t,2}$ and $D(\bbeta_n;+)\da_{\hW_{2,t,2}}$ is irreducible if and only if $D(\bbeta_n;-)\da_{\hW_{2,t,2}}$ is (the two modules differing by $\sgn$), $L\da_{\hW_{2,t,2}}$ is reducible.
\end{proof}

The next lemma deals with certain subgroups of both $\hW_{b,2}$ and $\hW_{2,b}$, and will be used in this subsection as well as the next one.

\begin{Lemma} \label{L180925} 
Let $n=2b \geq 10$ be even, $p\mid (n-1)$, $K\leq\s_b$,  $H=\pi^{-1}(K\wr\s_2)$ or $\pi^{-1}(\s_2\wr K)$, and $L$ be a second basic $\F \ts_n$-module. If  $L\da_H$ is irreducible then $K$ is primitive on $\{1,2,\dots,b\}$. 
\end{Lemma}
\begin{proof}
First, assume that $K$ is intransitive. In this case we may assume that $K\leq\s_{b-a,a}$ for some $1\leq a\leq b/2$. Then $H\leq\ts_{n-2a,2a}$. By Theorem \ref{thmintr}, $a=1$ and so $H\leq\pi^{-1}(\W_{b-1,2}\times\s_2)$ resp. $H=\pi^{-1}(\W_{2,b-1}\times\s_2)$. So $L\da_H$ is reducible by Proposition~\ref{P050525} 
giving a contradiction.

Next assume that $K$ is transitive but imprimitive, in which case $K\leq \W_{r,t}$ for some $r,t\geq 2$ with $rt=b$. Then $H\leq \hW_{r,t,2}\subseteq\hW_{r,2t}$ (resp. $H\leq \hW_{2,r,t}\subseteq\hW_{2r,t}$). 
By Theorem \ref{thmwreath}, we may assume that $b$ is even and $r=2$ (resp. $t=2$), that is $\hW_{r,t,2}=\hW_{2,b/2,2}$ (resp. $\hW_{2,r,t}=\hW_{2,b/2,2}$), so $L\da_H$ is reducible 
by  Lemma~\ref{L180925_0} giving a contradiction.
\end{proof}

The main result of this subsection is the following theorem (the second basic $\F\tA_n$-module is covered by the theorem as it lifts to a second basic $\F \ts_n$-module).

\begin{Theorem}\label{thm:wr2}
Let $n=2b \geq 10$ be even, $p\mid (n-1)$, $L$ be a second basic $\F \ts_n$-module, and $H\leq \hW_{b,2}< \ts_n$. Then $L\da_H$ is irreducible if and only if one of the following happens
\begin{enumerate}[\rm(i)]
\item $H = \hW_{b,2}$.
\item $\pi^{-1}(\A_b\times \A_b)\leq H\leq \hW_{b,2}$, $[\hW_{b,2}:H]=2$ and $H\neq \ts_{b,b}$ (there are two such subgroups $H$).
\end{enumerate}
\end{Theorem}
\begin{proof}
We use conjugation by the element 
$$t:= (1,b+1)(2,b+2) \ldots (b,n) \in \s_{b,b}.$$
to identify the second factor $\s_{b}$ of $\s_{b,b}=\s_b\times\s_b$ with the first factor $\s_{b}$, and write 
$\s_{b,b} = \left\{ (u,v) \mid u,v \in \s_{b} \right\}.$ Let ${\mathtt p}_1$ (resp. ${\mathtt p}_2$) be the projection of $\s_{b,b}$ to the first (resp. second) factor. 

Note that $\ts_{b,b}\lhd{\hW_{b,2}}$ is a normal subgroup of index $2$, and $A:= \pi^{-1}(\A_{b} \times \A_{b}) \lhd 
\hW_{b,2}$ is a normal subgroup of index $8$ isomorphic to the central product $\tA_{b} \ast \tA_{b}$. 
Set 
\begin{align*}
&J :={\mathtt p}_1(\pi(H \cap \ts_{b,b}))\leq\s_b,
 &K :={\mathtt p}_2(\pi(H \cap \ts_{b,b}))\leq\s_b,
\\ 
&J_1:={\mathtt p}_1(\pi(H \cap A)) \unlhd J \cap \A_{b},
&K_1:= {\mathtt p}_2(\pi(H \cap A)) \unlhd K \cap \A_{b}.
\end{align*}
We have $H \cap A \leq \pi^{-1}(J_1 \times K_1)\cong \hat J_1\ast \hat K_1$. 

By Theorem \ref{thmwreath}(ii), the restrictions $L\da_{\hW_{b,2}}$ and $L\da_{{\hW_{b,2}} \cap \tA_n}$ are irreducible. Now, (\ref{E170925_3}), (\ref{E170925_4}), and Clifford's Theorem imply that
\begin{equation}\label{eq:fort10}
  L\da_A \cong X_+ \boxtimes Y \,\,\oplus\,\, X_- \boxtimes Y \,\,\oplus\,\, Y \boxtimes X_+ \,\,\oplus\,\, Y \boxtimes X_-,
\end{equation}
for irreducible $\F\tA_{b}$-modules $X_\pm,Y$ such that $X_+$ and $X_-$ are conjugate under $\ts_{b}$ and 
$Y$ extends to $\ts_{b}$.

\vspace{2mm}
\noindent
{\em Claim 1. The restriction $L\da_{\ts_{b,b}}$ is reducible.}

\vspace{1mm}
\noindent 
Indeed, under the conjugation action of $\ts_{b,b}$ on (the isomorphism classes of) $X_+ \boxtimes Y,~X_- \boxtimes Y,Y \boxtimes X_+,~Y \boxtimes X_-$ there are two orbits (namely $X_+ \boxtimes Y,~X_- \boxtimes Y$ and $Y \boxtimes X_+,~Y \boxtimes X_-$).

\vspace{2mm}
\noindent
{\em Claim 2. There is $s\in H$ such that $\pi(s)= (x,y)t$ for $x,y \in \s_{b}$.}

\vspace{1mm}
\noindent 
Indeed, by Claim 1, $L\da_{\ts_{b,b}}$ is reducible, while $L\da_H$ is irreducible by assumption, so $H\not\leq \ts_{b,b}$.

\vspace{2mm}
\noindent
{\em Claim 3. We have $yJy^{-1}=K$ and $xKx^{-1}= J.$} 

\vspace{1mm}
\noindent 
Indeed, suppose $a \in J$. Then there exists $h \in H \cap \ts_{b,b}$ such that $\pi(h)=(a,v)$ for some $v \in \s_{b}$. Now $shs^{-1} \in H \cap \ts_{b,b}$ and 
$$\pi(shs^{-1}) = (x,y)t(a,v)t(x^{-1},y^{-1}) = (x,y)(v,a)(x^{-1},y^{-1}) = (xvx^{-1},yay^{-1}),$$
so $yay^{-1} \in K$. Thus $yJy^{-1} \leq K$.
Similarly $xKx^{-1} \leq J$, and the claim follows by comparing orders. 


\vspace{2mm}
\noindent
{\em Claim 4. We have that $K \leq \s_{b}$ is a primitive subgroup.}

\vspace{1mm}
\noindent 
As $\pi(s)^2= (x,y)t(x,y)t= (xy,yx)$, we have $s^2 \in H \cap \ts_{b,b}$ ,  $xy \in J$ and $yx \in K$. Now, for any $(j,k) \in J \times K$, we get using Claim 3:
\begin{align*}
\pi(s)(j,k)\pi(s)^{-1} &= (x,y)t(j,k)t(x^{-1},y^{-1}) = (x,y)(k,j)(x^{-1},y^{-1}) 
\\&= (xkx^{-1},yjy^{-1}) \in J \times K.
\end{align*}
So $J \times K$ is a normal subgroup of $Y:=\langle J \times K,\pi(s) \rangle$ of index $2$. Moreover, using Claim 3 again,
$$(x^{-1},1)(J \times K)(x,1) = xJx^{-1} \times K = K \times K,$$
and
$$(x^{-1},1)\pi(s)(x,1) = (x^{-1},1)(x,y)t(x,1) = (x^{-1},1)(x,y)(1,x)t = (1,yx)t.$$
We have shown above that $yx\in K$, so we now deduce that 
$(x^{-1},1)Y(x,1)=\langle K \times K,t \rangle$, which is precisely 
the wreath product $K \wr \s_2$ inside $\W_{b,2}=\s_{b} \wr \s_2$. Picking an element $\tilde{x} \in \ts_n$ with 
$\pi(\tilde{x})=(x,1)$, we now have that $L$ is irreducible over $\tilde{x}^{-1}H\tilde{x} \leq \pi^{-1}(K \wr \s_2)$. Hence $L\da_{\pi^{-1}(K \wr \s_2)}$ is irreducible, and 
$K$ is primitive by Lemma~\ref{L180925}.

\vspace{2mm}
Recall the notation $M_\F(G)$ from \S\ref{SSGrMod}. 

\vspace{2mm}
\noindent
{\em Claim 5. We have $M_\F(\hat K_1)\geq 2^{\lfloor (b-5)/2\rfloor}(b-2)$.}

\vspace{1mm}
\noindent 
Let $W$ be a second basic module among the modules $X_\pm,Y$ appearing in the right hand side of \eqref{eq:fort10}. By Table~IV, we have $\dim W = 2^{\lfloor (b-3)/2\rfloor}(b-2)=2\cdot 2^{\lfloor (b-5)/2\rfloor}(b-2)$.  
So if the claim fails, then $M_\F(\hat K_1)<(\dim W)/2$ and so  $W\da_{\hat K_1}$ has composition length at least $3$. It follows that the restriction to $H \cap A$ of each of the four summands in the decomposition \eqref{eq:fort10} has composition length at least $3$, and so $L\da_{H \cap A}$ has composition length at least $12$. Recalling that $L\da_H$ is irreducible and $H \cap A \lhd H$, we see that 
$L\da_{H \cap A}$ is a direct sum of simple modules of the same dimension $d \leq (\dim L)/12$. Choosing $U$ to be one of these simple modules,
by Frobenius' reciprocity we have that $L\da_H$ is a quotient of $\ind^H_{H \cap A}(U)$, a module of dimension $[H:H \cap A] \dim U \leq 8d < \dim L$,
a contradiction. 

\vspace{2mm}
\noindent
{\em Claim 6. We have $|K| \geq |K_1| \geq 2^{b-6}(b-2)^2$.}

\vspace{1mm}
\noindent 
Indeed, since $M_\F(\hat K_1) \leq |K_1|^{1/2}$, Claim 6 follows from Claim 5.

\vspace{2mm}
\noindent
{\em Claim 7. If $K \geq \A_{b}$, then 
$\pi(H \cap A) = \A_{b} \times \A_{b}$.}

\vspace{1mm}
\noindent 
We have $K_1 = \A_{b}$ and then $J_1 = \A_{b}$ by Claim 3. Thus $\pi(H \cap A)$ is a subgroup of $\A_{b} \times \A_{b}$ which projects onto $\A_{b}$ via ${\mathtt p}_1$ and ${\mathtt p}_2$. If $\pi(H \cap A) \neq \A_{b} \times \A_{b}$, since $\A_{b}$ is simple, by Goursat's lemma we have that 
$\pi(H \cap A) = \{ (v,\sigma(v) \mid v \in \A_{b} \}$ for some automorphism $\sigma$ of $\A_{b}$. If $b = 6$, then $H \cap A \cong {\sf C}_2 \times \A_6$ or $\tA_6$, whence the maximal dimension of an irreducible representation of 
$H \cap A$ is $\leq 10$ by \cite{Atl}, whereas $\dim L = 128$ and $[H:H \cap A] \leq 8$, contrary to the irreducibility of $L\da_H$. So  $b \neq 6$ and $\sigma$ is  a conjugation by an element of $\s_{b}$. By Lemma \ref{lem:3tens}, the restriction of any of the four summands in \eqref{eq:fort10} to 
$\pi^{-1}(\pi(H \cap A))\leq A$ has composition length at least $3$. So $L\da_{H \cap A}$ has composition length at least $12$, again contradicting the irreducibility of $L\da_H$.

\vspace{2mm}
Suppose $K \geq \A_{b}$. Then $\pi(H \cap A) = \A_{b} \times \A_{b} = \pi(A)$ by Claim 7. Lifting back to $\ts_n$ we get that $H \geq A$, since $b\geq 5$ and lifts of double transpositions square to $z$. The quotient group ${\hW_{b,2}}/A \cong D_8$ 
is generated by the cosets of three elements $s_1,s_2,s_3 \in \ts_n$ with $\pi(s_1) = (1,2)$, $\pi(s_2) = (b+1,b+2)$, $\pi(s_3)=t$. This group acts faithfully and transitively on the four summands of the decomposition \eqref{eq:fort10}, and we can label these  summands as $\bar 1,\bar 2,\bar 3,\bar 4$ so that $s_3$ acts via $(\bar 1,\bar 2)$, $s_2$ acts via $(\bar 3,\bar 4)$, and $s_1$ acts via $(\bar 1,\bar 3)(\bar 2,\bar 4)$. Every subgroup of order $4$ in ${\hW_{b,2}}/A$ will contain
its center, generated by $s_1s_2$. Among these three subgroups of order $4$, the subgroup $\langle s_1,s_2 \rangle$, corresponds to $\ts_{b,b}$, is intransitive, and the other two are transitive, one of which being $\hW_{b,2} \cap \tA_n$. In  particular, we arrive exactly at the exceptional cases described in (ii), and it remains to prove that $L\da_H$ is reducible when $K \not\geq \A_{b}$. In that case, by Claims 4 and 5, we can apply \cite[Proposition 6.2]{KW} to the subgroup $K$ of $\s_{b}$. We arrive at the following possibilities, where we denote $S:=\soc(K)$.

\vspace{2mm}
\noindent 
{\em Case 1: $(b,S) = (24, {\sf M}_{24})$.} In this case we have $K={\sf M}_{24}$ and $\hat K \cong {\sf C}_2 \times K$. It follows from \cite{Atl} that $M_\F(\hat K_1)\leq M_\C(\hat K) = 10395 < 2^9 \cdot 22$, contradicting Claim 5.

\vspace{2mm}
\noindent 
{\em Case 2: $b= 16$ and ${\sf C}_2^4 \cong E \lhd K \leq {\sf ASL}_4(2)$.} Using Claim 6 we have $|K/E| > |{\sf SL}_4(2)|/2$, and so $K = {\sf ASL}_4(2)$. In particular, $K$ is perfect, so $K \leq \A_{b}$ and 
$J \leq \A_{b}$ as well by Claim 3. It follows that $H \cap \ts_{b,b} = H \cap A$, and hence $[H: H \cap A] \leq 2$, since $|{\hW_{b,2}}/\ts_{b,b}|=2$. However, the decomposition \eqref{eq:fort10} shows that 
$L\da_{H \cap A}$ has at least four summands, contrary to the irreducibility of $L\da_H$.

\vspace{2mm}
\noindent 
{\em Case 3: $b=12$ and $S = {\sf M}_{11}$ or ${\sf M}_{12}$.} In either case, $K=S$ (since ${\sf M}_{12} \cdot 2$ does not embed in $\s_{12}$). It follows that $K \leq \A_{b}$, whence $H \cap \ts_{b,b} = H \cap A$, and 
we arrive at a contradiction as in Case 2.

\vspace{2mm}
\noindent 
{\em Case 4: $(b,S)=(11,{\sf M}_{11})$}. In this case $K=S$ is simple, so $K \leq \A_{b}$, $H \cap \ts_{b,b} = H \cap A$, and 
we arrive at a contradiction as in Case 2.

\vspace{2mm}
\noindent 
{\em Case 5: $b = 10$ and $S=\A_6$ in its primitive action in $\s_{10}$.} In this case we have $\dim W = 64$ but $M \leq 20$ by \cite{Atl}, and we arrive at a contradiction as in Case 1.

\vspace{2mm}
\noindent 
{\em Case 6: $b = 9$ and $K \leq {\sf AGL}_2(3)$ or $K \leq {\sf SL}_2(8) \cdot 3$.} In this case we have $\dim W = 56$ but $M < 28$, and we arrive at a contradiction as in Case 1.

\vspace{2mm}
\noindent 
{\em Case 7: $b = 16$ and $S = {\sf SL}_3(2)$ or $K \leq {\sf ASL}_3(2)$.} Here $\dim W = 24$. In the former case we have $M = 8$ by \cite{Atl}, and we arrive at a contradiction as in Case 1. In the latter case,
as ${\sf ASL}_3(2)$ is perfect we have $K \leq \A_{b}$, and we arrive at a contradiction as in Case 2.

\vspace{2mm}
\noindent 
{\em Case 8: $b = 7$ and $S = {\sf SL}_3(2)$.} Here $\dim W = 20$ and $M = 8$ by \cite{Atl}, so we arrive at a contradiction as in Case 1. 

\vspace{2mm}
\noindent 
{\em Case 9: $b = 6$, $p=11$, and $S = \A_5$.} Here, $\A_5 \lhd K \leq \s_5$ in its primitive action in $\s_6$. Since $\s_5$ does not embed in $\A_6$, we have $K_1=\A_5$, a maximal subgroup of $\A_6$, which lifts to a maximal subgroup
$\tA_5$ of $\tA_6$. Now using \cite{GAP} we can check that $W\da_{\tA_5}$ is a sum of two modules of dimension $2$ and $6$. This implies that $L\da_{H \cap A}$ has 
a summand of dimension $\leq 4 \cdot 2$ (since basic modules of $\tA_6$ have dimension 4), whereas $\dim L = 128$, contrary to the irreducibility of $L\da_H$.

\vspace{2mm}
\noindent 
{\em Case 10: $b = 5$, $p=3$, and $K \leq {\sf C}_5 \rtimes {\sf C}_4$.}  Here $|\hat K|=40$ which is coprime to $p=3$. But $\dim W = 6$, so $W\da_{\hat K}$ must have a simple summand of dimension 
$\leq 2$. This implies that $L\da_{H \cap A}$ has 
a summand of dimension $\leq 2 \cdot 2$ whereas $\dim L = 48$, contrary to the irreducibility of $L\da_H$.
\end{proof}

\subsection{Subgroups of $\hW_{2,n/2}$} 


Throughout this subsection, we assume that  $n=2b$ is even, and study restrictions of second basic modules to subgroups of $\hW_{2,b}$ (under the additional assumption $p\mid(n-1)$ coming from Theorem~\ref{thmwreath}(ii)). 

For $1\leq i<n$, we denote the simple transposition $(i,i+1)\in\s_n$ by $s_i$. We will use the notation $\s_{(2^b)}:=\s_{2,\ldots,2}$ for the Young subgroup of $\s_n$ corresponding to the partition $(2^b)$, and $\A_{(2^b)}:=\s_{(2^b)}\cap\A_n$. We have $\s_{(2^b)}=\{s_1^{a_1}s_3^{a_2}\cdots s_{2b-1}^{a_b}\mid a_1,\dots a_b\in \Z/2\}$ The wreath product subgroup $\W_{2,b}=\s_{(2^b)}\rtimes \s_b\leq \s_n$ yields an embedding and a projection 
$$\iota:\s_b\into \A_n\quad{and}\quad {\mathtt p}:\W_{2,b}\to \s_b$$ 
with ${\mathtt p}\circ\iota=\id$ and $ws_{2i-1}w^{-1}=s_{2{\mathtt p}(w)(i)-1}$ for all $w\in \W_{2,b}$ and $1\leq i\leq b$.

\begin{Lemma}\label{L310325}
Let $n=2b\geq 10$ be even, and $H\leq\W_{2,b}$ be a subgroup not contained in $\A_{n}$ and such that ${\mathtt p}(H)=\s_b$. Set $c:=s_1s_3\cdots s_{n-1}\in C$. 
Then one of the following holds:
\begin{enumerate}[\rm(i)] 
\item $H=\W_{2,b}$;

\item $H=\{x\iota(y)\mid x\in\A_{(2^b)},\,y\in \A_b\}\cup\{x\iota(y)\mid x\in\s_{(2^b)}\setminus\A_{(2^b)},\,y\in \s_b\setminus \A_b\}$;

\item $b$ is odd and $H=g(\lan c\ran\times \iota(\s_b))g^{-1}$ for 
 some $g\in\s_{(2^b)}$.

\item $b$ is odd and 
$H=g(\iota(\A_b)\cup\{c\iota(y)\mid y\in\s_b\setminus \A_b\})g^{-1}
$
for some $g\in\s_{(2^b)}$. 
\end{enumerate}
\end{Lemma}
\begin{proof}
Let $B:=H\cap\s_{(2^b)}$. Since ${\mathtt p}(H)=\s_b$, we have $|H|=|\s_b|\,|B|$. 

Suppose first that $B\not\leq\langle c\rangle$. Then there exist $1\leq i\neq j\leq b$, and $a_k\in\{0,1\}$ for $k\in\{1,2,\dots,b\}\setminus\{i,j\}$ such that 
$$
t:=s_{2i-1}\prod_{k\in\{1,2,\dots,b\}\setminus\{i,j\}}s_{2k-1}^{a_k}\in B
$$
As ${\mathtt p}(H)=\s_b$ by assumption, there exists $h\in H$ with ${\mathtt p}(h)=(i,j)$. 
Then
\begin{align*}
hth^{-1}=s_{2j-1}\prod_{k\in\{1,2,\dots,b\}\setminus\{i,j\}}s_{2k-1}^{a_k}\in B. 
\end{align*}
We conclude that $s_{2i-1}s_{2j-1}\in B$.
Conjugating with elements of $H$ we then deduce that $s_{2r-1}s_{2t-1}\in B$ for all $1\leq r<t\leq b$. It follows that $\A_{(2^b)}\leq B$. If $B=\s_{(2^b)}$ then $R=\W_{2,b}$. So we may assume that $B=\A_{(2^b)}$, in which case we must have 
\[H=\{x\iota(y)\mid x\in\A_{(2^b)},\, y\in L\}\cup\{x\iota(y)\mid x\in\s_{(2^b)}\setminus\A_{(2^b)},\, y\in \s_b\setminus L\}\]
for some subgroup $L\leq \s_b$ of index at most 2. Then $L\in\{\A_b,\s_b\}$ as $b\geq 5$. If $L=\s_b$ we get a contradiction since then $H=\{x\iota(y)\mid x\in\A_{(2^b)},\,y\in \s_b\}\leq\A_{n}$. So $L=\A_b$, and $H$ is as in (ii).

We now assume that $B\leq\langle c\rangle$. 
As ${\mathtt p}(H)=\s_b$, there exists $w\in H$ with ${\mathtt p}(w)=s_1$.

Write $w=u\iota(s_1)$ for $u\in \s_{(2^b)}$. We may assume that $u$ is odd, for if $u$ is even then so is $w$ and so all $hwh^{-1}\in \A_n$, but $H\not\leq\A_{2b}$ by assumption, so in this case we must have that $c\in B$ and that $c$ is odd, and we can replace $u$ by $cu$.

If exactly one of $s_1,s_3$ appears in $u$, i.e. $u =s_1\prod_{i=3}^bs_{2i-1}^{c_i}$ or $u =s_3\prod_{i=3}^bs_{2i-1}^{c_i}$  for some $c_i\in\{0,1\}$ then  
$w^2=(u\iota(s_1))^2=s_1s_3\in B$, which contradicts  $B\leq\langle s\rangle$ since $b\geq 5$. So we may assume that 
$u =(s_1s_3)^c\prod_{i=3}^bs_{2i-1}^{c_i}$ for some $c,c_i\in\{0,1\}$.

For any $x\in\s_{(2^b)}$ and $y\in\s_{1,1,b-2}$, we have 
\[x\iota(y)w(x\iota(y))^{-1}=x\iota(y)u\iota(s_1)\iota(y)^{-1}x^{-1}
=(\iota(y)u\iota(y)^{-1})(x\iota(s_1)x^{-1}\iota(s_1)^{-1})\iota(s_1)
.\]
Note that $x\iota(s_1)x^{-1}\iota(s_1)^{-1}$ has form $(s_1s_3)^{d}$ for some $d\in\{0,1\}$. Since ${\mathtt p}(H)=\s_b$ there are $x$ and $y$ as above with $x\iota(y)\in H$, so for some $d\in\{0,1\}$, we have 
\begin{equation}\label{E240925}
\iota(y)u\iota(y)^{-1}(s_1s_3)^{d}\iota(s_1)\in H.
\end{equation}
Therefore, we may now assume that $u$ is as in the Cases 1-3 below.

\vspace{2mm}
\noindent
{\em Case 1: some but not all of the $s_{2i-1}$ with $i\geq 3$ appear in $u$.} By (\ref{E240925}) we may assume that $u =s_1^as_3^as_5^cs_9\prod_{i\geq 6}s_{2i-1}^{d_i}$ for some $a,c,d_i\in\{0,1\}$. Let $h$ be an element of $H$ with ${\mathtt p}(h)=(1,3)s_4$. We can write $h$ in the form
\[h:=s_1^es_3^fs_5^gs_7^ks_9^l\prod_{i\geq 6}s_{2i-1}^{m_i}\,\iota((1,3)s_4)\]
Then we have the element of $H$ 
\begin{align*}
whwh^{-1}&=s_1^{f+g}s_3^{a+c}s_5^{a+c+f+g}s_7s_9
 \,
\iota(s_1s_2),
\end{align*}
so $(whwh^{-1})^3=s_7s_9\in B$, 
again contradicting $B\leq\langle s\rangle$.

\vspace{2mm}
\noindent
{\em Case 2: either all or none of the $s_{2i-1}$ with $i\geq 3$ appear in $u$.} Then $u =s_1^as_3^a\prod_{i\geq 3}s_{2i-1}^c$ for some $a,c\in\{0,1\}$. Since $u $ is odd, we must have that $b$ is odd and $c=1$. 
We will use the notation 
$$s_{\neq j,j+1}:=\prod_{i\in\{1,2,\dots,b\}\setminus\{j,j+1\}}s_{2i-1}\qquad(1\leq j<b).$$
For $2\leq j<b$, let $h_j\in H$ be an element with ${\mathtt p}(h_j)=(1,j)(2,j+1)$. For each $j$, we can write  $h_j=v_j\,\iota((1,j)(2,j+1))$ for some $v_j\in\s_{(2^b)}$. Then for some $d_j\in\{0,1\}$, we have elements of $H$:
$$h_jwh_j^{-1}=h_ju\iota(s_1)h_j^{-1}=(s_{2j-1}s_{2j+1})^{d_j}s_{\neq j,j+1}\,\iota(s_j)\qquad(2\leq j<b).$$
Setting, $d_1:=a$ note that
\[H=\langle B, (s_{2j-1}s_{2j+1})^{d_j}s_{\neq j,j+1}\,\iota(s_j)\mid 1\leq j<b\rangle.\]
since the group in the right hand side is contained in $H$ and has the same order $|\s_b|\,|B|$. 
But
\[(s_{2j-1}s_{2j+1})^{d_j}s_{\neq j,j+1}\,\iota(s_j)
=
\left\{
\begin{array}{ll}
c\iota(s_j) &\hbox{if $d_j=1$,}\\
cs_{2j-1}s_{2j+1}\iota(s_j) &\hbox{if $d_j=0$,}
\end{array}
\right.
\]
and $cs_{2j-1}s_{2j+1}\iota(s_j)=s_{2j-1}(c\iota(s_j))s_{2j-1}$. 
So
\begin{align*}
\langle (s_{2j-1}s_{2j+1})^{d_j}s_{\neq j,j+1}\,\iota(s_j)\mid 1\leq j<b\rangle
=y\langle c\iota(s_j)\mid 1\leq j<b\rangle y^{-1}
\end{align*}
for some $y\in\s_{(2^b)}$. Note that $c\in Z(\W_{2,b})$. So if $B=\lan c\ran$, we are in case (iii), and if $B=\{1\}$, we are in the case (iv).
\end{proof}

\begin{Theorem}\label{T180625}
Let $n=2b\geq 10$ be even, $p\mid(n-1)$, $L$ be a second basic $\F \ts_n$-module, and $H$ be a subgroup of\, $\hW_{2,b}$. Then $L\da_H$ is irreducible if and only if $H=\hW_{2,b}$.
\end{Theorem}

\begin{proof}
%
By assumption and Table IV, we have that $D(\bbeta_{n})$ is of type $\Qtype$, and 
$L=D(\bbeta_n;\eps)$ for $\eps\in\{+,-\}$. Under the isomorphism $\T_{\W_{2,b}}\cong\T_{b}\otimes\Cl_{b}$ from Lemma \ref{L160125}, we have by Lemma \ref{L160125_2} that 
$[D(\bbeta_n)\da_{\T_{\W_{2,b}}}]=2[D(\bbeta_{b})\circledast U_b]$ in the Grothendieck group.  
Moreover, by Table IV and Example~\ref{ExClifford}, both $D(\bbeta_{b})$ and $U_b$ are of type $\Mtype$ if $b$ is even, and of type $\Qtype$ if $b$ is odd. So the $\T_{\W_{2,b}}$-supermodule $D(\bbeta_{b})\circledast U_b$ is always of type $\Mtype$, so $|D(\bbeta_{b})\circledast U_b|$ is an irreducible $\T_{\W_{2,b}}$-module. Since $|D(\bbeta_n)|\cong D(\bbeta_n;+)\oplus D(\bbeta_n;-)$, we deduce 
\begin{equation}\label{E260925}
L\da_{\T_{\W_{2,b}}}\cong D(\bbeta_n;\pm)\da_{\T_{\W_{2,b}}}\cong  |D(\bbeta_{b})\circledast U_b|.
\end{equation}

Let $K\leq\s_{b}$ be a minimal subgroup with $\pi(H)\leq \s_2\wr K$. Since $L\da_H$ is irreducible, so is $L\da_{\T_{\s_2\wr K}}$. 
By Lemma \ref{L180925}, $K$ is primitive on $\{1,2,\dots,b\}$. 
Moreover, by Lemma \ref{L160125}(ii),
\begin{align*}
[L\da_{\T_{\s_2\wr K}}]&=[(L\da_{\T_{\W_{2,b}}})\da_{\T_{\s_2\wr K}}]=[(D(\bbeta_{b})\circledast U_b)\da_{\T_K\otimes\Cl_{b}}]\\
&=2^{-\de_{2\nmid b}}[(D(\bbeta_{b})\boxtimes U_b)\da_{\T_K\otimes\Cl_{b}}]=2^{-\de_{2\nmid b}}[D(\bbeta_{b})\da_{\T_K}\boxtimes U_b].
\end{align*}
Since the module $L\da_{\T_{\s_2\wr K}}$ is irreducible, 
we deduce using Lemmas~\ref{lmodules},\,\ref{LBoxTimes} that the supermodule $D(\bbeta_{b})\da_{\T_K}$ is irreducible and of the same type as $D(\bbeta_{b})$; in particular, $K\not\leq \A_b$. 
This also implies that the modules $D(\bbeta_{b};\de)\da_{\T_K}$ are irreducible (for appropriate $\de\in\{+,-,0\}$).

Since $K$ is primitive on $\{1,2,\dots,b\}$ and $D(\bbeta_{b};\pm)\da_{\T_K}$ is irreducible, we have $\A_{b}\leq K$ by Theorem~\ref{TAKT}. So $K=\s_b$ since $K\not\leq \A_b$ by the previous paragraph. 
Moreover, $H\not\leq\tA_n$ since $L\da_{\tA_n}=E(\bbeta_n;0)$ and $E(\bbeta_n;0)\da_{\hW_{2,b}\cap\tA_n}$ is reducible by Theorem \ref{thmwreath}. So $\pi(H)$ is as in the cases (i)-(iv) of Lemma~\ref{L310325}. We now analyze each of these cases. 

\vspace{2mm}
\noindent
{\em Case (i).} In this case $\pi(H)=\W_{2,b}$, so $H$ contains a lift of $(1,2)(3,4)$. Any such lift squares to $z$, so $z\in H$ and $H=\hW_{2,b}$.

\vspace{2mm}
\noindent
{\em Case (ii).} In this case $X:=\{x\iota(y)\mid x\in\A_{(2^b)},\,y\in \A_b\}$ is a normal subgroup of index $2$ in $\pi(H)$, which in turn is normal of index $2$ in $\W_{2,b}$. As $\T_{\A_{(2^{b})}}=(\T_{\s_{(2^{b})}})_\0\cong(\Cl_{b})_\0$ (resp. $\T_{\A_{b}}\cong (\T_{b})_\0$), every irreducible supermodule over $\T_{\s_{(2^{b})}}$ (resp. $\T_{b}$) splits as a direct sum of two irreducible supermodules of type $\Mtype$ when restricted to $\T_{\A_{(2^{b})}}$ (resp. $\T_{\A_{b}}$). 

Moreover, under the isomorphism $\T_{\W_{2,b}}\cong\T_{b}\otimes\Cl_{b}$ of Lemma \ref{L160125},
$\T_{X}$ embed into $\T_{\A_{b}}\otimes (\Cl_{b})_\0$. 
But $\dim \T_{X}=\dim \T_{\A_{b}}\otimes (\Cl_{b})_\0$, so we can identify $\T_{X}=\T_{\A_{b}}\otimes (\Cl_{b})_\0$ under the isomorphism. 

If $b$ is even, then the module 
$L\da_{\T_{X}}=D(\bbeta_{b})\da_{\T_{\A_{b}}}\boxtimes (U_b)\da_{(\Cl_{b})_\0}$ 
has 4 irreducible direct summands. 
Since $L\da_{\hW_{2,b}}$ is irreducible, this is only possible if $L\da_{\pi^{-1}(\pi(H))}$ has 2 composition factors, both of which split further when restricted to $\hat X$. In particular, $L\da_{\pi^{-1}(\pi(H))}$ and so also $L\da_H$ is reducible.

If $b$ is odd, then the supermodule $D(\bbeta_{b})$ (resp. $U_b$) is of type $\Qtype$, and so by Lemma~\ref{PIrrIrr}, we have $D(\bbeta_{b})\da_{\T_{\A_{b}}}\cong B^{\oplus 2}$ (resp. $(U_b)\da_{(\Cl_{b})_\0}\cong C^{\oplus 2}$) for some irreducible module $B$ (resp. $C$) over $\T_{\A_{b}}$ (resp. $(\Cl_{b})_\0$). So, in the Grothendieck group, we have 
\[[L\da_{\T_{X}}]=[(D(\bbeta_{b})\circledast U_b)\da_{\T_{\A_{b}}\otimes (\Cl_{b})_\0}]=2[B\boxtimes C]\]
with $B\boxtimes C$ an irreducible module over $\T_{\A_{b}}\otimes (\Cl_{b})_\0$. As $\T_{\pi(H)}$ has both even and odd part and $(\T_{\pi(H)})_\0=\T_{X}\cong\T_{\A_{b}}\otimes (\Cl_{b})_\0$, it follows again from Lemma~\ref{PIrrIrr} that $L\da_{\T_{\pi(H)}}\cong(D(\bbeta_{b})\circledast U_b)\da_{\T_{\pi(H)}}$ is an irreducible supermodule of type $\Qtype$; in particular, it is reducible as a module.

\vspace{2mm}
\noindent
{\em  Cases (iii) and (iv).} In these cases $b$ is odd, and it suffices to prove that $L\da_H$ is reducible for $\pi(H)=\langle c\rangle\times\iota(\s_{b})$. Then $\T_{\pi(H)}\cong\T_{\lan c\ran}\otimes\T_{\iota(\s_{b})}$ as superalgebra. 
As $\T_{\iota(\s_{b})}$ is a purely even algebra, all its irreducible supermodules are of type $\Mtype$. Since $b$ is odd, $\T_{\lan c\ran}\cong\Cl_{1}$ and so its only irreducible supermodule is of type $\Qtype$. Thus all irreducible supermodules of $\T_{\pi(H)}$ are of type $\Qtype$.

We have 
$L\da_H\cong (L\da_{\hW_{2,b}})\da_H$ 
and $L\da_{\hW_{2,b}}\cong |D(\bbeta_{b})\circledast U_b|$ 
by (\ref{E260925}). 
So $L\da_{\hW_{2,b}}$ and hence $L\da_H$ can be viewed as a supermodule. As the composition factors of the supermodule $L\da_H$ are all of type $\Qtype$ by the previous paragraph, it is reducible as a module.
\end{proof}

\section{Proof of Theorem \ref{imprimitive}}


Note that $\pi^{-1}(\pi(H))=\langle H,z\rangle$ and $z$ acts as $-1$ on any irreducible spin representation, $L\da_H$ is irreducible if and only if $L\da_{\pi^{-1}(\pi(H))}$ is irreducible. So we may assume that $H=\pi^{-1}(\pi(H))$ or, equivalently, $z\in H$.

If $H$ is as in Theorem \ref{imprimitive}(iv) then $n\geq 6$ since by assumption $\la\neq\balpha_n$ is $\JS(0)$. Thus for subgroups $H$ appearing in Theorem \ref{imprimitive}(i)-(vi), $\pi(H)$ contains a commuting product of two simple transpositions and then $z \in H$ in those cases. 

\smallskip
(a) Assume first that $\pi(H)$ is almost simple. Then Theorem \ref{imprimitive} holds by \cite[Theorem C]{KT}, taking into account Lemma~\ref{LJS} which show that the condition $\tilde e_0\la\in\JS^{(1)}$ appearing in \cite[Theorem C]{KT} is redundant (and recalling that $L$ is not basic spin by assumption).
Indeed, conclusion (i) of \cite[Theorem C]{KT} leads to cases (i) and (ii) in Theorem \ref{imprimitive}, whereas conclusion (ii) of \cite[Theorem C]{KT} leads 
to cases (iii) (with $G = \tA_n$), (iv), and (v) in Theorem \ref{imprimitive}.

\smallskip
(b) Henceforth we may assume that $\pi(H)$ is not almost simple. Assume in addition that $n\geq 8$. Then 
Theorem \ref{kt-n22} applies, so we are in the situation described either by conclusion (i) or by conclusion (ii) of Theorem \ref{kt-n22}. In the former case, by Theorem~\ref{thmintr}(iii) we arrive at conclusion (iii) (with $G =\ts_n$) of Theorem~\ref{imprimitive}.
In the latter case, 
$n=2b$ is even and $H\leq\hW_{b,2}$ or $H\leq\hW_{2,b}$. In this case, 
if $n\geq 10$ and $L$ is second basic, then, by Theorems~\ref{thmwreath}, \ref{thm:wr2} and  \ref{T180625}, 
we arrive at conclusion (vi) of Theorem~\ref{imprimitive}.

\smallskip
In view of Theorems \ref{thmintr} and \ref{thmwreath}, this leaves us with the following cases
\begin{enumerate}[(i)]
\item $5\leq n\leq 7$, $H<\ts_{n-1,1}\cap G$ or $H<\ts_{n-2,2}\cap G$,

\item $3\leq b\leq 4$, $n=2b$, $p\mid(n-1)$, $L$ is second basic, and $H<\hW_{b,2}\cap G$ or $H<\hW_{2,b}\cap G$,

\item $n=6$ or $10$, $(G,L)$ as in Table I, $H<\hW_{b,2}\cap G$ (resp. $H<\hW_{2,b}\cap G$), where $\hW_{b,2}\cap G$ (resp. $\hW_{2,b}\cap G$) is the subgroup appearing in Table I.
\end{enumerate}
We will now assume $L\dar_H$ is irreducible.  
Note that for $4\leq b\leq 5$ subgroups of $\hW_{b,2}$ and $\hW_{2,b}$ need to be transitive by Theorem~\ref{kt-n22}(ii). Some subgroups $H$ can also be excluded since they are contained in maximal imprimitive subgroups which are ruled out
by Theorems \ref{thmintr} and \ref{thmwreath}. 

\smallskip
(c) Assume $n=10$. Then $p \geq 7$, $H < K:=\hW_{5,2} \cap G$, and $K$ is the subgroup listed in Table I. First suppose that $G = \tA_{10}$, so that
$\dim L = 48$. Then $\pi(H) \leq \pi(K) = W_{5,2}$ has order at least $48^2$ and is transitive in $\s_{10}$. If $N := [\pi(K),\pi(K)] = \A_5 \times \A_5$ then 
$\pi(K)/N \cong \Z_4$ is generated by $s_1t$ where we can take $s_1=(1,2)$ and $t=(1,6)(2,7)(3,8)(4,9)(5,10)$. Now $\pi(H) \cap N$ has order $\geq 48^2/4=576$.
Since proper subgroups of $\A_5$ has order $\leq 12$, it follows that $\pi(H) \cap N$ projects onto at least one of the two factors $\A_5$ of $N$. The transitivity of
$\pi(H)$ implies that $\pi(H)$ permutes the two $\A_5$-factors of $N$ transitively, hence $\pi(H) \cap N$ projects onto both factors. By Goursat's lemma, 
either $\pi(H) \cong \A_5$ or $\pi(H) \geq N$. In the former case,  $|H|$ divides $120$ and hence is not divisible by $48$, a contradiction. 
So $\pi(H) \geq N$, whence $z \in H$ and 
$H \geq [K,K]$. As $K/[K,K] \cong \Z_4$ and $H < K$, we obtain $\pi(H) = \langle N,(s_1t)^2= (1,2)(6,7) \rangle$, which is intransitive, a contradiction.

Next suppose that $G = \ts_{10}$. Then $\dim L=96$ and $L\dar_{\tA_{10}} = L_1\oplus L_2$ with $L_i$ irreducible of dimension $48$.
As $H \cap \tA_{10}$ has index at most $2$ in $H$ and $L\dar_H$ is irreducible, it follows that each $L_i$ is irreducible over $H \cap \tA_{10}$. By the preceding 
result, $H$ contains $\hW_{5,2} \cap \tA_{10}$ which has index $2$ in $\hW_{5,2}$. As $H < \hW_{5,2}$, we obtain $H = \hW_{5,2} \cap \tA_{10}$, 
again a contradiction.

\smallskip
(d) We now use GAP \cite{GAP} to study the remaining cases $5 \leq n \leq 8$ without further reference. We also use Lemma \ref{LBetaGammaChar} and 
the modular character tables in \cite{GAP} to get a lower bound (or the exact value) for the dimension of the non-basic spin module $L$.

\smallskip
(d1) First assume that $n=5$. Then $L<\ts_{4,1}\cap G$ or $L<\ts_{3,2}\cap G$, and $\dim L\geq 4$. 
But proper subgroups of $\s_{4,1}$ or $\s_{3,2}$ have order $\leq 12$. Thus $L\da_H$ is reducible, a contradiction.

\smallskip
(d2) Next assume that $n=6$. If $p=3$ then by Theorems \ref{thmintr} and \ref{thmwreath} we have that $G=\tA_{6}$, $L=E((4,2);\pm)$, $H<\tA_{5,1}$ and $\dim L=6$. Since $|\pi(H)|<36$ for any $H$ with $H<\tA_{5,1}$, 
it follows that $L\da_H$ is reducible, a contradiction. 

Consider now $p\geq5$. Then it can be checked from Theorems \ref{thmintr} and \ref{thmwreath} that $L=D((3,2,1);\pm)\da_G$ if $G=\ts_6$ and $H<K$ for some $K\in\{\ts_{5,1},\ts_{4,2},\hW_{3,2},\hW_{2,3}\}$ but $H\not\leq\ts_{3,3}$ and  $H\not\leq\hW_{2,3}\cap\tA_6$. Similarly if $G=\tA_6$ and $H<K$ for some $K\in\{\tA_{5,1},\tA_{4,2},\hW_{3,2}\cap\tA_6\}$ but $H\not\leq\tA_{3,3}$ and $H\not\leq\hW_{2,3}\cap\tA_6$. Note that if $p=5$ then $(3,2,1)=\bbeta_6$, while if $p\geq 7$ then $(3,2,1)\not\in\{\balpha_6,\bbeta_6\}$. Further by Lemma \ref{L020819} we have that $D((3,2,1);\pm)=\bar S((3,2,1);\pm)$ and $E((3,2,1);0)=\bar T((3,2,1);0)$. Thus $\dim L=4$, so that $|\pi(H)|\geq 16$. If $G=\tA_6$ one can checks with GAP that no such group $H$ exists. Consider now $G=\ts_6$. Since $H$ is not almost simple, one can compute with GAP that there are 10 $\s_6$-conjugacy classes of possible subgroups $\pi(H)$. For each one of them, one can choose a representative for one $\pi(H)$ in the corresponding conjugacy class and lift generators to obtain generators of some subgroup $H_1$ with $\pi(H)=\pi(H_1)$. As mentioned in the introduction, we have that $L\da_H$ is irreducible if and only if $L\da_{H_1}$ is, so we may assume that $H=H_1$. GAP is able to compute character tables for each of these subgroups (also the modular character tables, when $p\mid |H|$). The result then follows by comparing characters.

\smallskip
(d3) Assume now that $n=7$. Then $L<\ts_{6,1}\cap G$ or $L<\ts_{5,2}\cap G$,
and $\dim L\geq 12$ if $G=\ts_7$ and $\dim L\geq 6$ if $G=\tA_7$. Consider first $G=\ts_7$. Then $|\pi(H)|\geq 144$ and $\pi(H)$ is a proper subgroup of $\s_{6,1}$ or $\s_{5,2}$. The only such subgroup is $\tA_{6,1}$, which contradicts the assumption of $\pi(H)$ not being almost simple. Consider now $G=\tA_7$. Then $|\pi(H)|\geq 36$ and $\pi(H)$ is a proper subgroup of $\A_{6,1}$ or $\A_{5,2}$. It can then be checked with GAP that 
either $\pi(H)\cong\A_5$ 
or $\pi(H)=\A_{3,3,1}$. The former case is excluded since $\pi(H)$ is then almost simple, while the latter is excluded in view of Theorem \ref{thmintr} since $\A_{3,3,1}\leq\A_{4,3}$.

\smallskip
(d4) Finally, we consider the case $n=8$. Then $p=7$, $L=D(\bbeta_8,\pm)\da_G$ and $H<\hW_{4,2}\cap G$ or $H<\hW_{2,4}\cap G$. By Lemma \ref{LBetaGammaChar} we have that $\dim L=16$. The case $H<\hW_{2,b}\cap G$ is excluded since proper subgroups of $\W_{2,4}$ have order $\leq 2^4\cdot 12 < 16^2$. So $H<\hW_{4,2}\cap G$. Since $|\pi(H)|\geq 256$ and it is transitive, we obtain using GAP that $\pi(H)\geq\langle \A_{4,4},s\rangle$ where $s\in\s_{4,4}$ is an element that permutes the sets $\{1,2,3,4\}$ and $\{5,6,7,8\}$. Note since $(1,2)(3,4)\in\pi(H)$, it follows that $z\in H$, so $H=\pi^{-1}\pi(H)\geq\langle \tA_{4,4},\hat{s}\rangle$.

In view of Lemmas \ref{L040925} and \ref{LBetaGammaChar} and since $p=7>4$ we have that
\begin{equation}\label{E071025}
[D(\bbeta_8)\da_{\ts_{4,4}}]=[\bar S(3,1)\boxtimes S(4)]+[\bar S(4)\boxtimes S(3,1)]=[D(3,1)\boxtimes D(4)]+[D(4)\boxtimes D(3,1)].
\end{equation}
It follows that $D(\bbeta_8;+)\da_{\tA_{4,4}}\oplus D(\bbeta_8;-)\da_{\tA_{4,4}}$ has 8 composition factors. Since $D(\bbeta_8;+)$ and $D(\bbeta_8;-)$ only differ by tensoring with $\sgn$, it follows that each of $D(\bbeta_8;\pm)\da_{\tA_{4,4}}$ has 4 composition factors. Since $\tA_{4,4}$ is normal in $\hW_{4,2}$ and $[\hW_{4,2}:\tA_{4,4}]=8$, we then have that from the assumption $H<\hW_{4,2}\cap G$ that $G=\ts_8$ and $\pi^{-1}(\A_4\times \A_4)\leq H\leq \hW_{4,2}$, $[\hW_{4,2}:H]=2$. Further $H\neq \ts_{4,4}$ since $\pi(H)$ is transitive. Note that there are two such subgroups $H$ and for either of these subgroups we have that $H\cap\ts_{4,4}=\tA_{4,4}$. Since 
\[[D(\bbeta_8;\pm)\da_{\tA_{4,4}}]=[E((3,1),(4);0)]+[E((4),(3,1);0)]\]
by \eqref{E071025}, $\hat{s}\in H$ interchanges the sets $\{1,2,3,4\}$ and $\{5,6,7,8\}$ and thus $E((3,1),(4);0)^{\hat{s}}=E((4),(3,1);0)$, it follows that $L\da_H$ is irreducible for either of these two choices of $H$. Thus we arrive at case (vi)(b) of Theorem \ref{imprimitive}.

\end{document}